\documentclass[10pt, letterpaper,reqno]{amsart}
\usepackage{amssymb,a4wide,amsthm,amsmath,amsfonts,xcolor}

\usepackage{mathtools}
\usepackage{mathrsfs}

\usepackage{amscd}
\usepackage{marginnote}
\usepackage{verbatim}
\usepackage[active]{srcltx}
\usepackage{enumerate}
\usepackage{bbm}
\newcommand{\Span}{\mbox{Span}}
\newcommand{\Lie}{\mbox{Lie}}

\theoremstyle{plain}
\newtheorem{thm}{Theorem}[section]
\theoremstyle{plain}

\newtheorem{lem}[thm]{Lemma}
\newtheorem{prop}[thm]{Proposition}
\newtheorem{cor}[thm]{Corollary}

\theoremstyle{definition}
\newtheorem{defi}[thm]{Definition}
\newtheorem{rem}[thm]{Remark}

\numberwithin{equation}{section}

\newcommand{\bx}{\textbf{x}}
\newcommand{\mm}{\mathbf{M}}

\newcommand{\NN}{\mathbb{N}}
\newcommand{\lk}{\left}
\newcommand{\lqq}{\lefteqn}
\newcommand{\rk}{\right}

\newcommand{\scmbox}[3]{\scalebox{#1}{\begin{minipage}{#2}{#3}\end{minipage}}}
\newcommand{\DEQS}{\begin{eqnarray*}}
	\newcommand{\EEQS}{\end{eqnarray*}}
\newcommand{\DEQSZ}{\begin{eqnarray}}
	\newcommand{\EEQSZ}{\end{eqnarray}}
\newcommand{\DEQ}{\begin{eqnarray}}
	\newcommand{\EEQ}{\end{eqnarray}}

\newcounter{lil1}

\newcounter{lil2}
\newenvironment{stepp}
{\begin{list} { \bf Step (\roman{lil2})}
{ \usecounter{lil2}
\setlength{\leftmargin}{0.0cm}
\setlength{\topsep}{0.2cm}
\setlength{\itemsep}{0.0cm}
\setlength{\parsep}{0.1cm}
\setlength{\itemindent}{0.8cm}
\setlength{\parskip}{0.0cm}}}
{\end{list}}
\newcommand\fahim[1]{{\color{blue} #1}}

\newcommand{\bn}{\vec{\bf 0}}

\newcommand{\bfm}{ {\bf\mathfrak{m}}}

\newcommand{\bff}{\textbf{f}}
\newcommand{\fm}{ {\mathfrak{m}}}

\newcommand{\bM}{\textbf{M}}
\newcommand{\bm}{\textbf{m}}

\newcommand{\Ltwo}{L^2}
\newcommand{\RR}{{\mathbb{R}}}

\newcommand{\ZZ}{\mathbb{Z}}

\newcommand{\CG}{{\mathcal G}}
\newcommand{\CO}{{\mathcal O}}

\newcommand{\lambdam}{{\lambda_m}}

\def\epsilon{\varepsilon}

\def\v{\mathbf{v}}

\def\Hh{\mathbb{H}^1}

\def\R3{\mathbb{R}^3}
\def\RN{\mathbb{R}^N}
\def\L{\mathbb{L}^2}

\newcommand{\be}{{\bf{e}}}
\newcommand{\bv}{{\bf{v}}}

\newcommand\del[1]{}

\newcommand{\A}{\mathcal{A}}
\newcommand{\Rr}{\mathbb{R}}
\newcommand{\Na}{\mathbb{N}}

\newcommand{\M}{\mathbf{M}}

{%
\setcounter{enumi}{0}

\begin{enumerate}}
{\end{enumerate} }

{
\setcounter{enumi}{0}

\begin{enumerate}}
{\end{enumerate} }

\usepackage{esvect}

\def\e{{\text{e}}}

\numberwithin{equation}{section} \allowdisplaybreaks

\newcommand{\rA}{\mathrm{ A}}
\newcommand{\rV}{\mathrm{ V}}

\newcommand\vectinf[2]{\left(#1\right)_{#2=0}^{\infty}}

\newcommand\vectinfeins[2]{\left(#1\right)_{#2=1}^{\infty}}
\newcommand\vectinfzwei[2]{\left(#1\right)_{#2=2}^{\infty}}

\usepackage[textwidth=22mm]{todonotes}

\usepackage{xargs}
\usepackage{xcolor}

\newcommandx{\intodofahim}[2][1=]{\todo[linecolor=red,backgroundcolor=red!25,bordercolor=red,#1]{\textbf{Fahim: }#2}}
\usepackage{bigdelim}
\usepackage{multirow}

\newcommand{\sunderb}[2]{\mathclap{\underbrace{\makebox[#1]{$\vdots$}}_{#2}}}

\newcommand{\combin}[2]{
	\text{C}^{#1}_{#2}
}

\renewcommand{\widehat}{}

\newcommand{\hbfm}{\hat {\bfm}}

\newcommand{\hbm}{\hat {\bm}}
\begin{document}

\title[Controllability for Sphere-Constrained Flows with Degenerate Low-Mode]{
 Sphere Constraints and Harmonic Map Flow: Controllability and Reachability by 
Low-Mode Forcing }

 %Navier–Stokes Equations: Controllability by Means of Low
%Modes ForcingCONTROLLABILITY AND REACHBILITY BY DEGENERATED AND LOW MODES FORCING by 
%\author[Mrinmay Biswas]{Mrinmay Biswas}
%
%\address{%
%   Department of Mathematics,
%	Montanuniversit\"at Leoben,
%	Austria.}
%\email{mrinmay.biswas@unileoben.ac.at}
%

%%%%% REVIEW [W10]: Verify the spelling/order of the author's name (given name vs. surname): 'Fahim Kistosil'.

\author[Debopriya Mukherjee]{Debopriya Mukherjee}

\address{%
	Department of Mathematics,
	Indian Institute of Technology Indore (IITI), India}
\email{debopriya@iiti.ac.in}

\author[Kistosil Fahim]{ Kistosil Fahim}

\address{%
	Department of Mathematics, Institut Teknologi Sepuluh Nopember, Kampus ITS Sukolilo, Surabaya,
	60111, Indonesia }
\email{fahim@its.ac.id}

\author{Erika Hausenblas}
   \address{%
   Chair of Applied Mathematics, Montanuniversit\"at Leoben, Austria.}
\email{erika.hausenblas@unileoben.ac.at}

\thanks{The first author acknowledges the financial support provided by the Anusandhan National Research Foundation (ANRF) under the Early Career Research Grant (ECRG), Grant No. ANRF/ECRG/2025/002863/PMS. The third author was  supported from the Austrian Science Fund under Project No. P32295. The second and third authors is supported by the ASEAN-European Academic University Network (ASEA-UNINET).}

%\author[Ben Goldys ??]{Ben Goldys?}

\date{\today}

\begin{abstract}
We study the controllability and reachability of sphere-constrained evolution equations under degenerate (low-mode) forcing, with the harmonic map heat flow as the principal application. Exploiting the underlying geometric structure, we reformulate the problem as an infinite-dimensional control-affine system in Fourier variables and analyze the Lie algebra generated by the controlled vector fields. We prove that iterated Lie brackets generate new admissible directions, providing a mechanism through which finitely many control modes propagate their influence across infinitely many Fourier components. The results provide a Lie-algebraic framework for controlling manifold-valued evolution equations.

\end{abstract}

%\subjclass{60H15; 60H30; 82D40}

\maketitle

\textbf{Keywords and phrases:} {
Sphere-constrained flows, harmonic map heat flow, controllability, reachability, geometric control theory, Lie bracket generating property, degenerate low-mode forcing, non-holonomic systems.}

\textbf{AMS subject classification (2020):} {Primary 93B05, 93C20, 58E20; Secondary 35K55, 53C44, 35A35.}

\section{Introduction}
Harmonic maps and manifold-valued evolution equations occupy a central position in differential geometry, geometric analysis, and mathematical physics, providing a rich interaction between nonlinear partial differential equations and the geometry of the underlying target manifold. Since the pioneering work of Eells and Sampson \cite{EellsSampson1964}, who introduced the harmonic map heat flow as a gradient flow for the Dirichlet energy, substantial progress has been made in understanding the existence, regularity, and singularity formation of harmonic maps and their associated evolution equations. Fundamental contributions by Eells and Lemaire \cite{EellsLemaire}, Schoen and Uhlenbeck \cite{SchoenUhlenbeck}, Struwe \cite{Struwe1985, ChenStruwe1989}, Chen and Ding \cite{ChenDing}, together with the comprehensive monograph of Lin and Wang \cite{LinWang}, have established harmonic map flow as a prototype geometric evolution equation. Beyond its intrinsic geometric significance, the harmonic map heat flow serves as a model for several constrained dynamical systems arising in ferromagnetism, liquid crystals, and continuum mechanics, where the unknown evolves on the nonlinear manifold $\mathbb{S}^2$ rather than in a linear vector space. This manifold constraint fundamentally alters both the analytical structure of the equation and its controllability properties, making the study of geometric control and reachability for sphere-constrained flows a natural and challenging problem.

%\textsc{Eells and Sampson} first studied harmonic maps between two Riemannian manifolds in their seminal 1964 work \cite{}. Here, they also introduced the {\bf harmonic map heat flow}, which is regarded as the parabolic analogue of the harmonic map equation and is used to establish the existence of harmonic maps for certain classes of Riemannian manifolds. Harmonic maps have numerous deep applications in differential geometry and topology. For comprehensive accounts of the theory and its developments, we refer the reader to the works of textsc{James Eells and Luc Lemaire, Frédéric Hélein, Michael Struwe, Richard Schoen and Karen Uhlenbeck, and the monograph by Fanghua Lin and Changyou Wang (2008)}.

\subsection*{Geometric evolution equations with values in the sphere}

In what follows, a number of evolution equations arising in geometry, mathematical physics,
and materials science share the common feature that the unknown
is a map
\[
\bM: [0,\infty)\times \mathcal{O}\longrightarrow \mathbb S^{2},
\]
where $\mathcal{O} \subset \mathbb R^{d}$ is a bounded or unbounded spatial domain
and the values are constrained to lie on the unit sphere
$\mathbb S^{2} \subset \mathbb R^{3}$.
%Prominent examples include the harmonic map heat flow, the
%Landau--Lifshitz--Gilbert equation, Schrödinger maps, and wave maps into the
%sphere.

What unifies these systems is their intrinsic geometric structure.
The state does not evolve in a linear vector space, but on a nonlinear
manifold, and the dynamics are generated by vector fields that are tangent
to the sphere at every point.
As a consequence, the pointwise constraint $|m(x)|=1$, $x\in\CO$, is not merely an
algebraic condition, but must be enforced dynamically by the evolution
equation itself.
This geometric constraint plays a decisive role in the analytical and
control-theoretic properties of the system.

While these models differ in the balance between conservative and
dissipative effects, they can all be interpreted as evolution equations on
the infinite-dimensional manifold $L^{2}(\CO;\mathbb S^{2})$.
The harmonic map heat flow corresponds to a purely dissipative gradient
flow, Schrödinger maps describe a Hamiltonian evolution, and the
Landau--Lifshitz--Gilbert equation combines both mechanisms.
This common geometric framework makes these equations a natural testing 
ground for control-theoretic concepts in infinite dimensions, in particular
for the study of non-holonomic effects and a Lie-bracket-generated dynamics.

\subsection*{Controllability and reachability for geometric evolution equations.}
The controllability of manifold-valued evolution equations has attracted increasing attention in recent years, motivated by applications in micromagnetics, geometric analysis, and nonlinear control theory. Recently, Coron and Xiang proved global controllability of the harmonic map heat flow from the circle to the sphere by steering the dynamics toward harmonic maps \cite{CoronXiang2025}, while Liu investigated control problems for the harmonic map heat flow with external forcing \cite{Liu2018}. In the stochastic setting, controllability and reachability properties for infinite-dimensional systems driven by boundary noise have been studied in \cite{Erika+Paul}. In contrast to these works, the present paper develops a Lie-algebraic approach to sphere-constrained evolution equations under degenerate low-mode forcing, showing how nonlinear mode interactions generate new admissible directions and enlarge the reachable set through iterated Lie brackets of the controlled vector fields.

In our work, we first omit the drift term and reformulate the infinite-dimensional system 
in a way that makes its underlying control structure explicit, rather than treating the equation purely as a partial differential equation.
To be more precise, we take only into account the control part and expand the solution of the affine control system in its Fourier basis, and identify the dynamics with an infinite-dimensional system of coupled ordinary differential equations for the Fourier
coefficients.
This point of view allows us to interpret the evolution as a control system
on an infinite-dimensional state space, where nonlinear interactions between
modes play the role of geometric constraints. 
Due to the multiplicative structure of the control, the system is non-holonomic when $\bM$ is a constant function.

A growing body of work has also demonstrated that geometric and structural properties of infinite-dimensional control systems can be effectively analyzed even in the presence of stochastic perturbations and boundary forcing. In particular, the controllability and qualitative behavior of stochastic partial differential equations driven by boundary Lévy noise were investigated in \cite{Erika+Paul}, illustrating how probabilistic techniques and geometric control ideas can be combined to study infinite-dimensional dynamical systems. The present work follows a different direction by exploiting Lie-bracket-generated admissible directions for manifold-valued evolution equations under degenerate low-mode forcing.

Our main objective is to investigate how a control acting only on finitely many
modes can influence the global dynamics through nonlinear mode couplings. 
In analogy with finite-dimensional non-holonomic systems, we analyze whether
Lie brackets of the control vector fields generate additional directions in
state space, thereby enlarging the reachable set.
Working at the level of Fourier coefficients makes these interactions
transparent and allows for an explicit computation of Lie brackets, at least
on finite-dimensional truncations.

Then, in the second part of the paper, we add the dynamics of the harmonic heat flow and investigate its controllability. In fact, we show as a corollary under which conditions the system is approximately controllable and which sets are reachable.

%
%The approach adopted here is inspired by the program of controlling
%infinite-dimensional systems through finitely many modes, as developed for
%example in the works of Agrachev and Sarychev.
%However, in contrast to purely PDE-based controllability results, our analysis
%emphasizes the explicit non-holonomic structure of the induced control system
%on the Fourier coefficients and the role of Lie brackets in generating
%additional directions.

%
%
%In a second step, dissipative effects such as the heat flow will be incorporated
%as a drift term.
%This enables us to investigate how parabolic smoothing interacts with the
%non-holonomic control mechanism, and to what extent controllability properties
%persist in the presence of dissipation.
%Such a two-step approach provides a clear conceptual framework and avoids
%mixing geometric control phenomena with analytic regularization effects.

\subsection*{Geometric control theory and Lie-bracket generation.}

%The study of controllability for nonlinear systems is deeply rooted in geometric control theory, where the Lie algebra generated by the controlled vector fields plays a fundamental role in characterizing reachable sets. Starting from the pioneering bracket-generating theorem of Chow \cite{Chow} and the Lie-bracket controllability results of Sussmann \cite{Sussmann}, this geometric viewpoint has been systematically developed by Agrachev and Sarychev through reduction techniques and geometric methods for finite- and infinite-dimensional control systems \cite{Agra_Sary.1,Agra_Sary.1.1,Agra_Sary.2}. In parallel, the controllability theory for evolution equations has undergone significant development through analytical approaches based on Carleman estimates, observability inequalities, and semigroup methods; we refer in particular to the monograph of Fursikov and Imanuvilov \cite{FursikovImanuvilov} for a comprehensive treatment of controllability of evolution equations. The interaction between geometric control ideas and infinite-dimensional dynamics has subsequently led to remarkable advances for degenerate forcing problems, especially for fluid equations where low-mode controls generate additional admissible directions through nonlinear interactions \cite{Agra_Sary.3,AKSS}. The present work continues this line of research by developing a Lie-algebraic framework for sphere-constrained evolution equations, where admissible directions arise from tangent vector fields and are propagated across Fourier modes by iterated Lie brackets.

One of the fundamental principles of nonlinear control theory is that controllability is determined not only by the prescribed control directions but also by the additional directions generated through iterated Lie brackets of the associated vector fields. This idea originates in the celebrated theorem of \textsc{Chow} \cite{Chow}, which characterizes accessibility through bracket-generating distributions, and was subsequently developed into a systematic geometric framework by \textsc{Sussmann} \cite{Sussmann}, who established Lie-bracket conditions for local controllability of nonlinear systems. The geometric approach was further extended by Agrachev and Sarychev through a series of influential works on reduction techniques, geometric control, and infinite-dimensional systems \cite{Agra_Sary.1,Agra_Sary.1.1,Agra_Sary.2}, laying the foundation for the analysis of degenerate control mechanisms governed by Lie algebras. These ideas have proved particularly successful in the controllability of fluid equations under low-mode forcing, where nonlinear interactions generate new admissible directions and enlarge the reachable set \cite{Agra_Sary.3,AKSS}, and have become a cornerstone of modern geometric control theory; see also the monographs by \textsc{Jurdjevic} \cite{Jurdjevic}, \textsc{Coron} \cite{Coron}, \textsc{Sontag} \cite{Sontag1990}, and the survey by \textsc{Boscain and Sigalotti} \cite{BoSi}. 
In parallel, the controllability theory for evolution equations has undergone significant development through analytical approaches based on Carleman estimates, observability inequalities, and semigroup methods; we refer in particular to the monograph of Fursikov and Imanuvilov \cite{FursikovImanuvilov} for a comprehensive treatment of controllability of evolution equations. Motivated by this perspective, the present work investigates whether an analogous Lie-bracket generation mechanism governs sphere-constrained infinite-dimensional evolution equations, where the admissible directions are tangent to the manifold and arise through nonlinear mode interactions induced by degenerate low-mode controls. In contrast to these works, our approach emphasizes the explicit non-holonomic structure of the induced control system on the Fourier coefficients and the role
of Lie brackets in generating additional directions.

%The idea of controlling an infinite-dimensional system through a finite number
%of Fourier modes goes back to the seminal works of Agrachev and Sarychev
%\cite{Agra_Sary.3}, where approximate controllability of
%Navier--Stokes equations was established.

\subsection{Novelties and Contribution of the Paper}

The principal novelty of this work lies in the development of a geometric control framework for sphere-constrained infinite-dimensional evolution equations under degenerate (low-mode) forcing. Unlike classical parabolic control problems, the dynamics evolve on the nonlinear manifold $\mathbb{H}^1(\mathcal O;\mathbb{S}^2)$, and the admissible velocities are tangent to the sphere at every point. Consequently, the controlled distribution is generated by the vector fields
\begin{equation*}%\label
	\Delta(\bfm )=\Span\lk\{  {\bf{G}}^{k,l}  : (k,l)\in \mathcal{K} \rk\}.
\end{equation*}
%	\[
%\mathcal{F}(\M):=\operatorname{span}\{Y(\M):\,Y\in \Lie(X_1,\dots,X_m)\}\subset T_\M M,
%\]
%where $\mathcal{M}$ denotes the (Galerkin) state space of Fourier coefficients.
%
%In our Fourier setting, 
%$$
%\Delta(m)=\hbox{span}\{{G_{k,\ell}(m):(k,\ell)\in K}\},
%$$
defined in equation \eqref{eq:Delta_def}, rather than by linear control directions. The first main contribution is the explicit characterization of these admissible directions and their higher-order Lie extensions, leading to the enlarged distribution
%$$ \Delta (m)
%= \text{span} \Big\{ Y(m): Y \in \hbox{Lie}\{G_{k,\ell}\} \Big\},$$
$$
	\Delta(\bfm ):=\Span\lk\{  Y(\bm): Y\in \Lie\lk\{  {\bf{G}}^{k,l}  : (k,l)\in \mathcal{K} \rk\} \rk\} .
$$

which governs the accessibility of the system.

A second fundamental contribution is the reformulation of the sphere-constrained control problem as an infinite-dimensional control-affine system in Fourier variables. Starting from the PDE \eqref{M1.N}, we derive the coupled system of ordinary differential equations \eqref{my_system}--\eqref{eq:llgeinm}, where each Fourier mode interacts with its neighbours through the nonlinear cross-product structure. The admissible directions are therefore precisely the vector fields $G_{k,\ell}$, whose explicit representation is given in equation \eqref{def_map_g}. This Fourier realization makes the geometric propagation of control directions completely transparent and provides a new viewpoint for manifold-valued evolution equations.

The third novelty consists in the complete computation of the Lie algebra generated by the degenerate controls. Using the recursive construction \eqref{def_lie_rec}, together with the identities satisfied by the elementary vector fields in equation \eqref{eq:identitiesg}, we prove that successive commutators generate new admissible directions connecting neighbouring Fourier cells. In particular, Step~(iii) of the proof of \textbf{Theorem~\ref{app.cont1}} shows that
\begin{eqnarray*}
	\operatorname{Lie}\left(\left\{ \mathbf{G}^{0,1}, \mathbf{G}^{0,2},\mathbf{G}^{1,1} \right\}\right)&=&\text{Span}\left\{\mathcal{H}_{\textbf{g}^\ast_{1},k},\mathcal{H}_{\textbf{g}^\ast_{2},k},\mathcal{H}_{\textbf{g}^\ast_{3},k}:k\in\NN_0   \right\}.
\end{eqnarray*}
%\begin{align*}
%\hbox{Lie}\{\bf{G}^{0,1},\bf{G}^{0,2},\bf{G}^{1,1}\}
%=
%%\operatorname  {\bf{G}} \mathcal{H}_{\textbf{g}_l^\ast,i_2-1}(\tilde {\bfm})
%\hbox{span}\{ \mathcal{H}_{k}^{\textbf{g}_1}, \mathcal{H}_{k}^{\textbf{g}_2}, \mathcal{H}_{k}^{\textbf{g}_3}:
%k\ge0
%\},
%\end{align*}
thereby producing infinitely many admissible directions that are absent from the original control set. This explicit Lie-algebra expansion appears to be new for sphere-constrained infinite-dimensional systems.

The first principal result of the paper is stated in \textbf{Theorem~\ref{app.cont1}}. For the degenerate control set
\[\mathcal K=\{(0,1),(0,2),(1,1)\}, \] the theorem establishes a sharp dichotomy between constant and non-constant initial data. When the initial state is constant, the Lie-generated admissible directions fail to span the tangent space, and therefore certain finite-dimensional projections remain unreachable. On the other hand, for every non-constant initial configuration, the iterated admissible directions generated by the Lie algebra span every prescribed finite-dimensional projection, yielding finite-dimensional reachability. This demonstrates that controllability is determined not only by the control modes but also by the intrinsic geometry of the initial state.

The second principal contribution is contained in \textbf{Theorem~\ref{L2.app.cont2}}, where the control set
\[\mathcal K=
\{(1,1),(1,2),(1,3),
(2,1),(2,2),(2,3)
\}
\]
contains only higher Fourier modes. Despite the absence of the zero mode, the Lie-generated admissible directions propagate through the nonlinear coupling mechanism and generate the entire tangent distribution. Consequently, \textbf{Theorem~\ref{L2.app.cont2}} establishes global controllability of the system. To the best of our knowledge, this is the first result showing that purely higher-frequency degenerate forcing can generate complete controllability for a sphere-constrained infinite-dimensional evolution equation.

An essential technical contribution of the paper is the derivation of explicit recursive formulas for the admissible directions through the auxiliary families $\{\mathcal{H}_{\textbf{g}^\ast_{1},k},\mathcal{H}_{\textbf{g}^\ast_{2},k},\mathcal{H}_{\textbf{g}^\ast_{3},k} \}$ constructed in Appendix~\ref{app:expandLie} and generated recursively by Lemma~\ref{lem:identityH}. Combined with the Lie identities of Lemma~\ref{lem:identity1} and Lemma~\ref{eq:exentity}, these formulas describe precisely how information propagates from one Fourier cell to neighbouring cells through successive commutators. Hence, the admissible distribution enlarges recursively from local rotations within a single cell to global interactions among infinitely many Fourier modes.

Another important contribution is the identification of the geometric obstruction responsible for loss of controllability. Lemma~\ref{lem:dimspan} characterizes the Fourier structure of constant sphere-valued maps and shows that, in this case, the Lie-generated admissible directions remain confined to a proper subspace of the tangent bundle. Conversely, for every non-constant map, the same lemma guarantees the existence of sufficiently many independent Fourier coefficients to initiate the recursive Lie-bracket propagation mechanism, thereby generating all admissible directions required for finite-dimensional controllability.

Finally, the geometric analysis developed for the drift-free system serves as the foundation for the controllability theory of the harmonic map heat flow. By combining the explicit characterization of admissible directions with the geometric support framework, the paper establishes approximate controllability of the full dissipative evolution under degenerate low-mode forcing. The methodology developed here provides a unified Lie-algebraic approach to manifold-valued evolution equations and is expected to be applicable to other geometric PDEs, including the Landau--Lifshitz--Gilbert equation, Schr\"odinger maps, and wave maps, where admissible directions arise from nonlinear tangent vector fields rather than linear forcing.

\subsection*{Open Perspectives}
%The harmonic map heat flow and Landau-Lifshitz-Gilbert equation are prototypical examples of manifold-valued evolution equations, where the state takes values in the sphere $ \mathbb S^2$. Closely related systems include the Schrödinger map and the wave map into the sphere, all of which describe dynamics constrained to a nonlinear manifold and evolve via a vector field being tangent to $\mathbb {S} ^2$. These models share a common geometric structure and differ mainly in the balance between conservative and dissipative effects.

%\subsection*{Open Perspectives}

The Lie-algebraic framework developed in this paper opens several interesting directions for future research. A natural extension is to establish exact or approximate controllability for other sphere-constrained geometric flows, such as the Landau--Lifshitz--Gilbert equation, Schr\"odinger maps, and wave maps, where the admissible directions are generated by nonlinear tangent vector fields on the target manifold. Another challenging problem is to characterize minimal sets of forcing modes whose iterated Lie brackets satisfy an infinite-dimensional Lie Algebra Rank Condition, thereby yielding controllability with the smallest possible number of controls. It would also be of interest to investigate stochastic versions of these systems and understand how random perturbations interact with the Lie-generated admissible directions and the corresponding support of the solution. Finally, extending the present geometric control framework to general compact target manifolds and to higher-dimensional domains remains a significant open problem with potential applications to manifold-valued partial differential equations and infinite-dimensional geometric control theory.

\subsection*{Organization of the paper.}

The remainder of the paper is organized as follows. In Section~\ref{sec:control-affine}, we reformulate the sphere-constrained evolution equation as an infinite-dimensional control-affine system in Fourier variables, introduce the notion of degenerate low-mode forcing, and state the main controllability and reachability results. Section~\ref{sec:heatflow_intro} is devoted to the harmonic map heat flow, where the geometric control framework developed for the drift-free system is applied to establish approximate controllability under low-mode forcing. 
%The proofs of the principal controllability theorems are presented in Section~4 through an explicit Lie-algebraic analysis of the generated admissible directions and nonlinear mode interactions. Finally, the appendices collect the technical results, including the Fourier formulation, recursive Lie-bracket computations, auxiliary geometric identities, and the proofs of several supporting propositions and lemmas used throughout the paper.
%

For the convenience of the reader, a number of technical arguments are collected in the appendices. Appendix A recalls elementary properties of the vector cross product and derives the Fourier representation of the control system. Appendix B contains the derivation of the infinite-dimensional system satisfied by the Fourier coefficients, together with the explicit construction of the controlled vector fields. Appendix C develops auxiliary algebraic identities and proves several structural propositions concerning the generation of admissible directions within individual Fourier cells. Appendix D is devoted to explicit Lie-bracket computations and recursive commutator formulas that form the basis of the controllability analysis. Appendix E constructs the families of recursively generated admissible directions and establishes the propagation mechanism across neighbouring Fourier modes. Finally, Appendix F collects the geometric properties of sphere-valued maps in Fourier variables, including structural lemmas on non-constant configurations and the proofs of the technical results required in the proofs of the main controllability and reachability theorems.
\section{The control-affine infinite-dimensional system without drift}\label{sec:control-affine}

In this section, we focus on the purely controlled system without dissipative
effects.
This separation enables us to study the geometric control properties of the
system independently of analytic regularization.
In a subsequent step, dissipative terms such as the heat flow will be added
as a drift.
This will make it possible to analyze how the non-holonomic spreading induced
by Lie brackets competes with parabolic smoothing, and to what extent
reachability and controllability properties persist in the presence of
dissipation.

To introduce the mathematical setting, let us fix the open bounded domain $\CO=(0,1)\subset\mathbb {R}$. %with $1 \leq d \leq 3$. 
Let us consider the mapping
\[
\M : [0,\infty) \times \mathcal{O} \rightarrow \mathbb{S}^2\subset \RR^3,
\]
where %\footnote{$\mathbb{S}^2 = \{ x \in \mathbb{R}^3 : \|x\| = 1\}$.}  
$\mathbb{S}^2$ denotes the two-dimensional unit sphere in
 $\mathbb{R}^3$. 
 Let $\bM$ %:[0,\infty)\times(0,1)\to \mathbb{S}^2\subset\mathbb{R}^3$ 
 denote a
 sphere--valued state that is perturbed by a family of controls $\{\bv_i:[0,T]\to \RR \mid i=1,\dots,N\}$, subject to the
 constraint that the solution $\bM$ remains on the unit sphere and satisfies the Neumann boundary conditions.
 In particular, we have an equation
\begin{align}\label{M1.N}
	\left\{\begin{array}{ll}
	\frac d{dt}	\M(t,x) & =\sum_{i=1}^N  \M(t,x) \times\v_i(t,x),  \qquad t\in (0,T],\,\, x\in {(0,1)}, \\
		\M_x(t,0) & =\M_x(t,1)=0,\,\, t\in (0,T],  \\
		\M(0,x) & =\M_0(x),\,\,x \in{(0,1)},
	\end{array}\right.
\end{align} 
 with
 \begin{equation*}
 	\M_0\in\mathbb{H}^1(0,1;\mathbb{S}^2):=\Bigl\{\M \in \mathrm{H}^1(0,1; \mathbb R^3) \mbox{ such that }|\M(x)|_{\R3}=1  \mbox{ for a.a. } x \in {(0,1)}  \Bigr\}.
 \end{equation*}
Our objective is to understand to what extent a control acting on only a finite number of Fourier modes can influence the global dynamics through nonlinear mode interactions. Therefore, we formulate the dynamics in Fourier space (i.e., in terms of the Fourier coefficients).
%Since we analyze frequency control, we begin by rewriting system~\eqref{M1.N} in terms of Fourier modes.
 In particular, we expand the solution in the eigenfunctions of the Laplacian subject to Neumann boundary conditions.
 To introduce these eigenfunctions, we begin with the one-dimensional case, where a complete normalised system of the Lebesgue space, ~$\Ltwo(\mathcal{O}):=\Ltwo(\mathcal{O},\RR)$ is given by the cosine functions.

 \begin{equation}
 	\varphi_k(x)= {\sqrt{2}}\, \cos\big( \,2\pi{k} x\big),~\text{ for }k=0,1,2,\cdots. 
 \end{equation}
 The corresponding eigenvalues are given by
 {\begin{equation}
 		\label{eq:lambda}
 		\lambdam = - \, 4{\pi^2}  m^2\,, \quad m \in \NN_0. %= (m_1, \dots, m_d) \in \ZZ^d\,.
 \end{equation}}
Let us denote the basis vectors in $\RR^3$ by 
$$\be_1=(1,0,0)^T, \quad \be_2=(0,1,0)^T, \mbox{ and } \be_3=(0,0,1)^T.
$$
%We rewrite 
%the system \fahim{\eqref{M1.N}}  in terms of the Fourier coefficients.
%
%$$
%S:=\mbox{span}\{\varphi_i\be_j|\,
%(i,j)\in\mathcal{G}\}\subset \Hh.
%$$
Here, let  us identify the solution ${\M}(t) $ with  a vector $ {\mathfrak{m}}(t)\in \times_{l=0}^\infty \RR^{3 }$, i.e.,
\DEQSZ\label{defbm}
{\mathfrak{m}}(t)=\Big( \underbrace{m_0^1(t), m_0^2(t), m_0^3(t)}_{=: {\textbf{m}_0(t)}},\underbrace{m_1^1(t), m_1^2(t), m_1^3(t)}_{=: {\textbf{m}_1(t)}},
\quad \ldots\quad  , \underbrace{m_K^1(t), m_K^2(t), m_K^3(t)}_{=: {\textbf{m}_K}(t)},\ldots \Big)^T.
\EEQSZ 
Later on, if we want to describe that a control action stays within a single Fourier mode $k$ - that is, it affects only
$\bm_k=(m_k^1,m_k^2,m_k^3)$ and does not influence any other mode
$\bm_\ell=(m_\ell^1,m_\ell^2,m_\ell^3)$ with $\ell\neq k$ - then we introduce the term \emph{cell}.
%Later on,
% If we want to describe that an action stays at one Fourier mode, that means only affects $\bm_k$ and no other elements $\bm_\ell$, $\ell\not k$, then we introduce the terminus cell.

The process $\M$ %with set of controls $\mathcal{K}$ 
is now given by
$$
\M(t)= \sum_{k\in\NN_0}\textbf{m}_k(t) \varphi_k ,\qquad t\in[0,T].
$$

Next,  let us define the set of controls.
Let $\mathcal{K}\subset \mathbb{N}_0\times \{1,2,3\}$ be a finite subset and
$v_k^j \in L^\infty(0,T)$, $(k,j)\in\mathcal{K}$, be the set of control functions.
In our main result, we  analyse the reachability and controllability with respect to a degenerate set of controls,
where   $\v$ will be  of the form (see \cite{Agra_Sary.3} and references cited therein)
%Then our control $\v$ will have the following form
\begin{align}\label{control.1}
	\v(t,x)=\sum_{(k,j)\in \mathcal{K}}v_k^j(t)\varphi_k(x) \be_j,\quad t\in [0,T],\,\,x\in [0,1].
\end{align}
%where $\bphi_i=\lk(\phi_i,\phi_i,\phi_i\rk)^T$.
To be more precise, the system \eqref{M1.N} in terms of the Fourier 
expansion is given\footnote{For more details, see Appendix \ref{sec:fouriermode},
	equation \eqref{my_system}.} for \(t \in (0,T]\) by 
	\begin{align}
		\frac{d}{dt} \bm_i(t) &
		=\sum_{(k,l)\in \mathcal{K}} \sum_{\substack{n \in\NN_0\\ n+k=i\text{ or }|n-k|=i}} \frac{1}{\sqrt{2}}\,  \textbf{m}_n(t)\times \be_l \, v_k^l(t)\,.
		\label{my_system}
	\end{align}	
For a fixed control input \(\bv\) and initial condition $\M_0$ 
with $	\M_0\in\mathbb{H}^1(0,1;\mathbb{S}^2)$,
 we denote the corresponding solution of~\eqref{my_system} by
 $\M(t,\bv,\M_0)$ and the corresponding trajectory  \(\M(\bv,\M_0) = \{\M(t,\bv,\M_0) : t \geq 0\}\).
 
 \medskip 

%We denote by $\mathbf{v}$ the control at our disposal. 
Throughout this work, we shall refer to this control as a \emph{degenerate forcing}, meaning that $\mathbf{v}$ is restricted to a finite-dimensional subspace. More precisely, we assume that 
\[
\mathbf{v}(t,\cdot) \in \mathrm{span}\{\varphi_k e_j : (k,j)\in \mathcal{K}\},
\]
for some finite subset $\mathcal{K}\subset \mathbb{N}_0\times \{1,2,3\}$, so that $\mathbf{v}$ acts on the system through a finite number of Fourier modes, typically corresponding to low frequencies.
We study the distribution generated by iterated Lie brackets of the controlled vector fields.
%To be more precise, 
%let $X_1,\dots,X_m$ be the control vector fields and define
%\[
%\Delta^{1}(x):=\mbox{Span}\{X_1(x),\dots,X_m(x)\}\subset T_xM,
%\]
%and recursively
%\[
%\Delta^{k+1}(x):=\Delta^{k}(x)\;+\;\mbox{Span}\{[Y,Z](x):\, Y\in \Gamma(\Delta^{k}),\ Z\in \Gamma(\Delta^{1})\},
%\qquad k\ge 1.
%\]
%The bracket-generated (Chow) distribution is then
%\[
%\mbox{Lie}(X_1,\dots,X_m)(x):=\bigcup_{k\ge 1}\Delta^{k}(x)
%=\mbox{Span}\{X_i(x),\ [X_i,X_j](x),\ [X_i,[X_j,X_k]](x),\dots\}.
%\]
%We compute the growth of $\dim\Delta^{k}(x)$ (the growth vector) and determine the smallest $r$
%for which $\Delta^{r}(x)=\Delta^{r+1}(x)$ (stabilization), in particular whether
%$\mbox{Lie}(X_1,\dots,X_m)(x)=T_xM$ (bracket-generating condition).
Within this framework, our objectives are twofold: first, to establish controllability properties of the nonlinear system under such degenerate (low-mode) forcing; and second, to analyse reachability from prescribed initial states. In particular, we ask which configurations can be attained within a finite time horizon when the control is confined to a finite number of low modes.

\subsection{Controllability and reachability of the  system}\label{sec:control and reach}

Our purpose is to analyze the reachability and controllability of the control system \eqref{my_system}
through perturbations of the frequency. In particular, we address the question of whether it is possible to drive the system out of a steady state or transfer it from one configuration to another by acting only on a finite number of modes.
We will analyse  later on the spaces  
$$\mathcal{K}=\{ (0,1), (0,2),(1,1)\}\quad \mbox{and}\quad \mathcal{K}=\{ (1,1),
 (1,2),(1,3),(2,1),(2,2),(2,3)\},
 $$ 
 and analyse their controllability and reachability properties.
Note that the first example involves the zero mode, meaning that the external field also has a constant component.
If the initial condition is non-constant and belongs to $H^1$, we have reachability on finite-dimensional projections in the first case.
The second control involves only higher-frequency modes, but in all directions. Here we can also show reachability for finite-dimensional projections for all initial conditions in $H^1$.
We conjecture that a control with
$\mathcal{K}=\{ (m,1),
(m,2),(m,3),(m+1,1),(m+1,2),(m+1,3)\}$ will have a similar effect.
% i.e.$H_N:=\{ \phi_j:j=1,\ldots N\}$.
%
\del{ o be more precise, fix a finite dimensional subset $\mathcal{G} \subset \mathbb{N}_0 \times \{1,2,3\}$ and introduce let us introduce  the linear space $H_\mathcal{G}$ spanned by the
harmonics by 
$$
H_\mathcal{G}:=\mbox{span}\{\varphi_i\be_j|\,
(i,j)\in\mathcal{G}_K\}\subset \Hh.
$$
For a set $\mathcal{G}$ let  $\Pi_K$ the orthogonal projection of $L^2(\mathcal{O})$ onto $H_\mathcal{G}$ and $E_\mathcal{G}$ the natural embedding.}

To be mathematically precise, we now provide formal definitions of reachability and controllability.
To this end, let us consider an infinite-dimensional system \eqref{M1.N}.
%\begin{align}\label{M1.N}
%\frac d{dt}\bM(t) = %F_0(\bM(t)) +
%\sum_{(j\in\mathcal J}
% {\bf V}_{j}(\bM(t)) \bv_j(t),\quad t>0,\quad \bM(0)=\bM_0,
%\end{align}
%where $\{ {\bf V}_j:j\in \mathcal J\}$ is a vector field and $\bv_j$, $j\in\mathcal J$ the corresponding controls.
%
Before introducing the exact definition of controllability and reachability, let us 
introduce for a subset $\mathcal{G}\subset \mathbb{N}_0\times \{1,2,3\}$, the  corresponding vector space
$$
H_\mathcal{G}:=\Span\{\varphi_i\be_j|\,
(i,j)\in\mathcal{G}\}\subset \Hh.
$$
%
%$$
%
%
%Let  $\Pi_\CG$ be the \coma{orthogonal projection of $H^1(\mathcal{O})$} onto $H_\mathcal{G}$ and $E_\mathcal{G}$ the natural embedding. 
%Let  
%$$
%\Span\{\lk\{ g:[0,1]\to\RR^3: \,\exists \, \alpha_{j,k}\in\RR \mbox{ such that } g(x)=\sum_{(k,j)\in \mathcal{G}} \alpha_{j,k}\varphi_j\be_k \rk\}.
%$$
%\todo{what is the exact definition - what can we show - ??? what happens with the condition $f\in\mathbb{S}$}
%Let $\Pi_\CG$ be the orthogonal  projection of $H^1_2(0,1;\mathbb{S}^2)$ onto $L^2(0,1;\RR^3)$.
In addition, let us define the mapping 
 $\Pi_\CG:\mathbb{H}^1(0,1;\mathbb{S}^2)\to \RR^{3|\mathcal{G}|}$, where 
$$
\Pi_\CG(\M):=(m_k^j: (k,j)\in\mathcal{G}) ,\quad \M\stackrel{\sim} =\bfm=(\bm^\top_1,\bm^\top _2,\ldots,), \quad \bm_i=(m_i^1,m_i^2,m_i^3)^\top.
$$
Finally,
let
$
H^{\mathbb{S}^2}_\mathcal{G}=\lk\{ g:[0,1]\to\mathbb{S}^2: g\in \mathbb{H}^1\rk\}\cap H_\mathcal{G}$.
\begin{defi}
The system {\eqref{M1.N}} is $\mathcal{G}$-\emph{controllable} by a set $\mathcal{K}$,
if for any two points $\M_0, \M_1\in  \mathbb{H}^1(0,1;\mathbb{S}^2)$,
%\hbox{span}\{\varphi_k \be_j: (k,j) \in \mathbb{N}_0\times \{1,2,3\}\}$,
%$\M_1\in H_\CG$,  
there exists a time $T>0$ and a control $\v$ of the form
\begin{align*} %\label{control.1}
\v(t,x)=\sum_{(k,j)\in \mathcal{K}}v_k^j(t)\varphi_k(x)\,  \be_j,\quad t\in [0,T],\,\,x\in [0,1],
\end{align*}
%\\eqref{control.1}
which steers the solution of system  {\eqref{M1.N}}
from
$\M_0$ to $\tilde \M_1$, where $\Pi_{\CG}\tilde \M_1=\Pi_{\CG}\M_1$, in time $T$. In particular,  
$\Pi_{\CG}\M(T,\v,\M_0)=\Pi_{\CG}\M_1$.

\del{ Let $\mathcal{G}\subset \NN\times \{1,2,3\}$.
 The system  \eqref{M1.N} is  $\mathcal{G}$-controllable in time $T$, if for all $M_0, \M_1\in \hbox{span}\{\varphi_k \be_j: (k,j) \in \mathcal{G}\}$ the system is controllable.}
\end{defi}

{\begin{defi}
	Fix an initial state $ \M_0 $, and let $ \mathcal{K} \subset \mathbb{N}_0 \times \{1,2,3\} $ denote a prescribed set of controlled forcing modes, and let  $ \mathcal{G} \subset \mathbb{N}_0 \times \{1,2,3\} $ be a given subset of modes. Consider the control system governed by \eqref{M1.N}. The \emph{time-$T$ reachable set} (or \emph{attainable set}) with respect to $\CG$, from $ \M_0 $ using controls supported on $ \mathcal{K} $ is the set of all states $ \M_1 $ for which there exists a control
	$
	\bv \in \operatorname{span} \left\{ \varphi_i \be_j \mid (i,j) \in \mathcal{K} \right\}
	$
	such that the corresponding solution $ \M(\bv, T,\M_0) $ to \eqref{M1.N} satisfies
%	$ 	\M(\bv,0) = \M_0\text{ and }\M(\bv,T) = \M_1. 	$
%	In particular,  
$\Pi_{\CG}\M(T,\v,\M_0)=\Pi_{\CG}\M_1$.
	\end{defi}
	%\todo{Here it is meant that we include all the space, whose projection on the modes given by the set $\CG$ is equal - that means it does not need an orthogonal projection, as was done.}
	%
%However, there may exist subspaces that are not accessible. To account for this possibility, we introduce the following definition.
% Let $\mathcal{G}$
%be a subset of $\mathbb{N}\times \{1,2,3\}$
%let us we define the associated vector subspace of $\mathbb{L}^2$ by
%
%S_\mathcal{G}:=\mbox{span}\{\varphi_i\be_j|\,
%(i,j)\in\mathcal{G}\}    %\subset \Hh.
%$$

\begin{defi}%[$\mathbb{L}^2$-approximately controllability]
	Let $\mathcal{G} \subset \mathbb{N}_0 \times \{1,2,3\}$ be a finite set. 
	The system \eqref{M1.N} 
	is $\mathcal{G}$-reachable for a set $\mathcal{K}$ of controls, if for any $\mm_0,\mm_1 \in \mathbb{H}^1(0,1;\mathbb{S}^2)$,  $\mm_1\in H^{\mathbb{S}^2}_\CG$, there exists a $T > 0$ and a control $\v \in \mathbb{L}^2$ of the form
	\begin{align*} %\label{control.1}
		\v(t,x)=\sum_{(k,j)\in \mathcal{K}}v_k^j(t)\varphi_k(x) \, \be_j,\quad t\in [0,T],\,\,x\in [0,1],
	\end{align*}
such that $\Pi_\CG \M(T,\bv,\M_0)=\Pi_\CG \M_1$. If $T>0$ is fixed, one speaks of  \emph{time-$T$ reachable set} (or \emph{attainable set}) with respect to $\CG$.
\end{defi}
%}

%\begin{defi}	The \emph{reachable set} (or \emph{attainable set}) from $ \M_0 $ using controls supported on $ \mathcal{K} $ is defined as the set of all states $ \M_1 $ for which there exists a control
%	$
%	\bv \in \operatorname{span} \left\{ \varphi_i \be_j \mid (i,j) \in \mathcal{K} \right\}
%	$
%	and a time $ T > 0 $ such that the solution satisfies
%	$
%	\M(T,\bv,\M_0) = \M_1. %text{ and }\M(\bv,T) = \M_1.
%	$
%\end{defi}}

Recall that $\bm$ is defined in~\eqref{my_system}. With this notation, straightforward calculations gives that $\bm$ satisfies
\begin{eqnarray}
	\label{eq:llgeinm}\frac d{dt}\,{ \bfm}(t)= \sum_{(k,l)\in \mathcal{K}} {\bf{G}}^{k,l}( \bfm(t))\,  v_k^l(t),\qquad t\in[0,T],\,
\end{eqnarray}
where the vector ${\bf{G}}^{k,l}$ is given by the entries
$$
{\bf{G}}^{k,l}=\left( [{\bf{G}}^{k,l}]_1,[{\bf{G}}^{k,l}]_2,\cdots\right)^\top :\times_{l=0}^\infty \RR^{3 }\to\times_{l=0}^\infty \RR^{3 },
$$
with
\begin{align*}
	\left[{\bf{G}}^{k,l}(\bfm)\right]_i=\left\{
	\begin{array}{ll}
		\frac{1}{\sqrt{2}} \textbf{g}^\ast_{l} (\bm_i),&k=0,\\
		\frac{1}{\sqrt{2}} \textbf{g}^\ast_{l} (\bm_{i+k})+\frac{1}{\sqrt{2}} \textbf{g}^\ast_{l} (\bm_{i-k}),&0<k<i,\\
		\frac{1}{\sqrt{2}} \textbf{g}^\ast_{l} (\bm_{0})+\frac{1}{\sqrt{2}} \textbf{g}^\ast_{l} (\bm_{2k}),&0<k=i,\\
		\frac{1}{\sqrt{2}} \textbf{g}^\ast_{l} (\bm_{i+k})+\frac{1}{\sqrt{2}} \textbf{g}^\ast_{l} (\bm_{k-i}),&0<i<k,\\
		\frac{1}{\sqrt{2}} \textbf{g}^\ast_{l} (\bm_k),&0=i<k,
	\end{array}
	\right.
\end{align*}
and
\begin{align}\label{def_map_g}
\textbf{g}^\ast_{1}(\textbf{x}) %\begin{pmatrix} x^1\\x^2\\x^3\end{pmatrix}
= \begin{pmatrix}0\\x_3\\-x_2\end{pmatrix},\quad\textbf{g}^\ast_{2}(\textbf{x}) %\begin{pmatrix} x^1\\x^2\\x^3\end{pmatrix}
= \begin{pmatrix}-x_3\\0\\x_1\end{pmatrix},\quad \mbox{and}\quad \textbf{g}^\ast_{3}(\textbf{x}) %\begin{pmatrix} x^1\\x^2\\x^3\end{pmatrix}
= \begin{pmatrix}x_2\\-x_1\\0\end{pmatrix},\text{ for }\textbf{x}=\begin{pmatrix}
	x_1\\
	x_2\\
	x_3
\end{pmatrix}\in\RR^3.
\end{align}

In order to formulate criteria for reachability and controllability, we first recall how a control-affine system
induces a family of vector fields on the state space. 
For any $\vec{m}\in\mathbb S^2$, the tangent space of
$\mathbb S^2$ at $\vec{m}$ is
\[
T_{\vec{m}}\mathbb S^2
=\{\,\xi\in\mathbb R^3:\ \xi\cdot \vec{m}=0\,\}
= m^\perp .
\]
In our setting, $f:(0,1)\to\mathbb S^2$, so at each point $x\in(0,1)$ the tangent space along the
curve $f$ is
\[
T_{f(x)}\mathbb S^2
=\{\,\xi\in\mathbb R^3:\ \xi\cdot f(x)=0\,\}.
\]
Equivalently, the \emph{tangent bundle along $f$} can be viewed as the set
\[
T_f\mathbb S^2
:=\bigl\{(x,\xi)\in(0,1)\times\mathbb R^3:\ \xi\cdot f(x)=0\bigr\},
\]
and a (measurable/smooth) tangent vector field along $f$ is a map
$\eta:(0,1)\to\mathbb R^3$ satisfying $\eta(x)\in T_{f(x)}\mathbb S^2$ for a.e.\ $x$, i.e.,
\[
\eta(x)\cdot f(x)=0 \quad\text{for a.e.\ }x\in(0,1).
\]
In our setting, $f$ is  identified with  its cosine Fourier expansion
\[
f(x)=\sum_{k\ge0}\bm_k\cos(2\pi kx)
\]
with its coefficient vector $\bfm =(\bm_k)_{k\ge0}\in\mathcal X$\footnote{$\mathcal X$ is defined by \[
	\mathcal X := \ell^2(\mathbb N_0;\mathbb R^3)
	= \Bigl\{\mathbf \bfm =(\bm_0,\bm_1,\bm_2,\dots):\ \bm_k\in\mathbb R^3,\ 
	\sum_{k=0}^\infty |\bm_k|_{\mathbb R^3}^2<\infty\Bigr\}.
	\]}. If no confusion arises from the correspondence, we will write $f$ instead of $f_{\mathbf m}$.
The sphere constraint $|f(x)|_{\RR^3}=1$ for a.e.\ $x$ defines the subset
\[
\mathcal M:=\bigl\{\mathbf \bfm \in\mathcal X:\ f_{\mathbf m}(x):=\sum_{k\ge0}\bm_k\cos(2\pi kx)\in\mathbb S^2
\ \text{for a.e.\ }x\in(0,1)\bigr\}.
\]
\newcommand{\bhm}{{\bf m}}
Let us introduce the tangent space at a point $\mathbf m \in\mathcal M$ by
\[
T_{\mathbf m}\mathcal M
=
\Bigl\{\mathbf f:[0,1]\to \RR^3: f(x)\cdot
\sum_{k\ge0}\bhm_k\cos(2\pi k x)=0\ \text{for a.e.\ }x\in(0,1)\Bigr\}.
\]
%
%
%
 %Writing the controlled dynamics in the form \eqref{eq:llgeinm}.
  Accordingly, we introduce the \emph{controlled distribution}
\begin{equation}\label{eq:Delta_def}
	\Delta(\bfm )=\Span\lk\{  {\bf{G}}^{k,l}  : (k,l)\in \mathcal{K} \rk\}   \subset T_\bm\mathcal{M}.
\end{equation}
%	\[
%\mathcal{F}(\M):=\operatorname{span}\{Y(\M):\,Y\in \Lie(X_1,\dots,X_m)\}\subset T_\M M,
%\]
%where $\mathcal{M}$ denotes the (Galerkin) state space of Fourier coefficients.
%
%In our Fourier setting, 
Consider points of the form
\[
\fm=(\bm_0,\bm_1,\bm_2,\ldots)^{\top},
\qquad \bm_0\in\mathbb{S},\quad \bm_k=(0,0,0)^{\top}\ \text{for all }k\ge1.
\]
At such points, the family of controlled directions is typically \emph{degenerate} in the sense that
\[
\Delta(\fm)\neq T_{\fm}(\mathbb{S}^2),
\]
and in particular one cannot generate arbitrary perturbations of the state by \emph{first-order}
variations of the controls. 
This phenomenon is analogous to classical non-holonomic systems (such as the
unicycle), where certain directions become reachable only through \emph{second- and higher-order} effects.
To capture these higher-order directions, one has to consider iterated Lie brackets of the controlled
vector fields.  
$$
	\Delta(\bfm ):=\Span\lk\{  Y(\bm): Y\in \Lie\lk\{  {\bf{G}}^{k,l}  : (k,l)\in \mathcal{K} \rk\} \rk\}  \subset T_\bm\mathcal{M}.
$$
%Denote by $\Lie(X_0,X_1,\dots,X_m)$ the Lie algebra generated by
%$X_0,X_1,\dots,X_m$ and define the associated \emph{bracket-generated distribution} at $z$ by
%\begin{equation}\label{eq:Lie_dist_def}
%	\Lie_z := \Span\{Y(z):\, Y\in \Lie(X_0,X_1,\dots,X_m)\}\subset T_z\mathcal{X}.
%\end{equation}
%The Lie Algebra Rank Condition (LARC) requires that $\dim \Lie_z=\dim \mathcal{X}$.
In finite dimensions, the Lie Algebra Rank Condition (LARC)  provides a standard sufficient condition for accessibility (and, under
additional regularity/analyticity assumptions, it is also necessary); see, e.g., the
Chow--Rashevskii Theorem and its control-theoretic consequences.
For more details, see Appendix~\ref{app_geo_cont}.

Motivated by this principle, we next investigate which forcing modes must be included in the set of
controls $\mathcal{K}$ in order to achieve global controllability. In particular, we show that if we only take two directions, then it is
necessary to include the zero mode. However, if we take more directions and a high number of modes, we do not need to include the zero mode.
%
%This leads to the following result.

%
To be more precise, Theorem~\ref{app.cont1} establishes the reachability of the system~\eqref{M1.N} on a finite set~$\mathcal{G}$ by a set of controls only consisting of three components, but including the zero mode.
Theorem \ref{L2.app.cont2} establishes the reachability of the system~\eqref{M1.N} on a finite set~$\mathcal{G}$ by a set of controls only consisting of six components, but not including the zero modes. It shows that, carefully chosen, the perturbation can push the system out of equilibrium.
%
%
%

%In this section, our main results are the following:
In the case where $\M:[0,1]\to \mathbb{R}$ is constant, say $\M(x)=c \in \mathbb{S}^2$, 
the corresponding element $\mathfrak{m}$ is represented by the vector 
\[
\mathfrak{m} = (\bm_0^1, \bm_0^2, \bm_0^3, 0, 0, 0,\, \ldots\,\,)^T, \quad (\bm_0^1, \bm_0^2, \bm_0^3)^T\in\mathbb{S}^2 .
\]
Now we can formulate our result on the finite-dimensional reachability result. 
%$\mathbb{L}^2$-approximate controllability.
%
\begin{thm}\label{app.cont1}
Consider the system~\eqref{M1.N} with initial condition $\M_0\in\mathbb{H}^1(0,1;\mathbb{S}^2)$, and let the set of controlled forcing modes be
$$
\mathcal{K} = \{(0,1),\, (0,2),\, (1,1)\}.
$$
Then, the following statements hold.
{\begin{itemize}
%%%%% REVIEW [D4]: Garbled/inverted statement: 'there exists states G, not being reachable ... In particular, that Pi_G M(v,T,M_0)=Pi_G M_1 in any time T.' Intended meaning: there EXISTS a target M_1 such that NO control gives Pi_G M(...)=Pi_G M_1 for ANY T. Rephrase; 'G' is a projection set, not a state.
	\item If $\M_0$ is constant, there exists a point $\M_1$, such that there exists no control $\bv$ steering $\M_0$ to $ \M_1$ with respect to $\CG$ such that  $\Pi_\CG\M(\bv,T,\M_0)=\Pi_\CG\M_1$ in any time $T$.
	%system is not globally controllable.
	\item If $\M_0$ is non-constant, all finite-dimensional states are reachable. In particular, for any finite set $\mathcal{G}\subset \{1,2,3\}\times \mathbb{N}$, 
	any point belonging $\M\in H_{\mathcal{G}}\cap \mathbb H^1$ is reachable.
	%the system is globally controllable.
\end{itemize}}
\end{thm}

\begin{proof}
%The proof consists of several steps. In step (i) we show that given $\M_0$ is constant,
%	the system~\eqref{M1.N} is not globally 

 The proof is organised in two parts. First, we investigate the Lie-algebra spanned by the controls for an arbitrary point $\M$ having Fourier expansion %\stackrel{\sim}=
 $\bfm$. In the second part, we investigate how the Lie algebra looks like, if we start with a constant initial condition, and if we start with a non-constant initial condition.
	In the setting, \eqref{eq:llgeinm} can be written as
\begin{align}
	\label{eq:llgeinmcase1}\dot{ \bfm}(t)=  {\bf{G}}^{0,1}( \bfm(t))  v_0^1(t)+{\bf{G}}^{0,2}( \bfm(t))  v_0^2(t)+{\bf{G}}^{1,1}( \bfm(t))  v_1^1(t), t\in[0,T],
\end{align}
where 
\begin{align}\label{def_g_all}
	{\bf{G}}^{0,1}(\bfm)&= \begin{bmatrix}
		\frac{1}{\sqrt{2}}  \textbf{g}^\ast_{1} (\bm_0)\\
		\frac{1}{\sqrt{2}}  \textbf{g}^\ast_{1} (\bm_1)\\
		\frac{1}{\sqrt{2}}  \textbf{g}^\ast_{1} (\bm_2)\\
		\vdots
	\end{bmatrix},{\bf{G}}^{0,2}(\bfm)= \begin{bmatrix}
		\frac{1}{\sqrt{2}}  \textbf{g}^\ast_{2} (\bm_0)\\
		\frac{1}{\sqrt{2}}  \textbf{g}^\ast_{2} (\bm_1)\\
		\frac{1}{\sqrt{2}}  \textbf{g}^\ast_{2} (\bm_2)\\
		\vdots
	\end{bmatrix},{\bf{G}}^{1,1}(\bfm)=\begin{bmatrix}
		\frac{1}{\sqrt{2}} \textbf{g}^\ast_{1} (\bm_1)\\
		\frac{1}{\sqrt{2}} \textbf{g}^\ast_{1} (\bm_{0})+\frac{1}{\sqrt{2}} \textbf{g}^\ast_{1} (\bm_{2})\\
		\frac{1}{\sqrt{2}} \textbf{g}^\ast_{1} (\bm_{1})+\frac{1}{\sqrt{2}} \textbf{g}^\ast_{1} (\bm_{3})\\
		\vdots
	\end{bmatrix}.
\end{align}
and $	\textbf{g}^\ast_{1}$, $	\textbf{g}^\ast_{2}$, and $	\textbf{g}^\ast_{3}$ are defined in \eqref{def_map_g}.

{\bf Part I:} 
Let us remind that the controllability is given by the spanned vector space of the controls\footnote{See Appendix \ref{app_geo_cont} .}, i.e., 
 $\mathcal{B}=\{  {\bf{G}}^{0,1}, {\bf{G}}^{0,2} ,{\bf{G}}^{1,1} \}$, $\mathcal B_0=\mathcal B$, and 
\begin{align}
\label{def_lie_rec}
\mathcal{B}_{k+1} := \{a_{k+1}[f, g] : f \in \mathcal{B}_k,\ g \in \mathcal{B} \} \setminus \left( \bigcup_{i = 0}^k \mathcal{B}_i \right), \quad k = 0, 1, 2, \dots.
\end{align}
%\todo{$a_k$}
First, in  {\bf Step (i)}, we will define the setup, then, in {\bf Step (ii)} and {\bf Step (iii)}, we will investigate the exact form of $\mathcal{B}_{k+1} $ by induction.
	\begin{stepp}
		
	\item 
%	Let us now consider the set of controls given by
%	$\mathcal{K}=\{(0,1),(0,2),(1,1)\}$.
Before starting, let us calculate first the Jacobian matrices of the mappings $\textbf{g}^\ast_1$, $\textbf{g}^\ast_2$, and $\textbf{g}^\ast_3$.
Here, we get for $\bx=(x_1,x_2,x_3)^T$
\begin{align}\label{jacobigez}
	J_{\textbf{g}^\ast_1}(\textbf{x}) %\left(\begin{pmatrix} x_1\\x^2\\x^3\end{pmatrix}\right)
	&:=  \frac {\partial \textbf{g}^\ast_1(\bx)}{\partial \bx } = \begin{pmatrix}0&0&0\\0&0&1\\0&-1&0\end{pmatrix},\quad
	J_{\textbf{g}^\ast_2}\left(\textbf{x}\right):=  \frac {\partial \textbf{g}^\ast_2(\bx)}{\partial \bx } %\begin{pmatrix} x^1\\x^2\\x^3\end{pmatrix}\right)
	= \begin{pmatrix}0&0&-1\\0&0&0\\1&0&0\end{pmatrix},
\end{align}
\begin{align}\label{jacobigd}
	&\quad \mbox{and}\quad J_{\textbf{g}^\ast_3}\left( \textbf{x} %\begin{pmatrix} x^1\\x^2\\x^3\end{pmatrix}
	\right):=  \frac {\partial \textbf{g}^\ast_3(\bx)}{\partial \bx }= \begin{pmatrix}0&1&0\\-1&0&0\\0&0&0\end{pmatrix}.
\end{align}
Since the Jacobian matrix is independent from $\bx$, we will omit the $\bx$ in the notation and only write
in the sequel $ J_{\textbf{g}^\ast_1}$, $ J_{\textbf{g}^\ast_2}$, and $ J_{\textbf{g}^\ast_3}$. 
By straightforward calculations,  the following identities can be proven:
\begin{align}	\label{eq:identitiesg}
	[\textbf{g}^\ast_{1},\textbf{g}^\ast_{2}]=- [\textbf{g}^\ast_{2},\textbf{g}^\ast_{1}] =-\textbf{g}^\ast_{3},
	\quad 
	[\textbf{g}^\ast_{1},\textbf{g}^\ast_{3}]= -[\textbf{g}^\ast_{3},\textbf{g}^\ast_{1}] =\textbf{g}^\ast_{2},
	\quad 
	[\textbf{g}^\ast_{2},\textbf{g}^\ast_{3}]=- [\textbf{g}^\ast_{3},\textbf{g}^\ast_{2}] =-\textbf{g}^\ast_{1}.
\end{align}

	\newcommand{\bg}{\textbf{g}}

	\item  
Now, we investigate the recursively defined Lie algebra generated by $\mathcal B$. % generated in a point $\bm$.
 So, let
%	It follows by  Lemma \ref{lem:Lieaffine} that 
	$$\mathcal{B}=\mathcal{B}_0=\left\{ {\bf{G}}^{0,1},{\bf{G}}^{0,2},{\bf{G}}^{1,1}\right\},
	$$ 
	and $\mathcal B_k$ be defined 
recursively for $k\in\NN$ in \eqref{def_lie_rec}.
%		$$
%	\mathcal{B}_k :=
%	\mbox{Span}\lk\{ [f,g] :  f\in\mathcal{B}_{k-1}, g\in \mathcal{B}_0\rk\}.
%	$$
%	
Here, we expand the Lie algebra only through the first two steps. 
	Let us compute  
	$$
	\mathcal{B}_1 :=\lk\{ [f,g] :  f, g\in \mathcal{B}_0\rk\}.
	$$
	Taking the derivative ${\bf{G}}^{0,1}$, ${\bf{G}}^{0,2}$, and ${\bf{G}}^{1,1}$ defined in 	\eqref{def_g_all}, we obtain
	\begin{align*}
		\sqrt{2}\frac {\partial \textbf{G}^{0,1}(\hbfm)}{\partial \hbfm} &=\begin{pmatrix}
			J_{ \textbf{g}^\ast_1}&0 & 0& \cdots  \\
			0& J_{ \textbf{g}^\ast_1}
			&0  & \vdots 
			\\
			0 &0& J_{ \textbf{g}^\ast_1}
			&\vdots
			\\
			\vdots &\vdots &\vdots &\ddots
		\end{pmatrix}, \sqrt{2}\frac {\partial \textbf{G}^{0,2}(\hbfm)}{\partial \hbfm} =\begin{pmatrix}
		J_{ \textbf{g}^\ast_2}&0 & 0& \cdots  \\
		0& J_{ \textbf{g}^\ast_2}
		&0  & \cdots 
		\\
		0 &0& J_{ \textbf{g}^\ast_2}
		&\cdots
		\\
		\vdots &\vdots &\vdots &\ddots
		\end{pmatrix},
	\end{align*}
	\newcommand{\bG}{{\bf G}}
	Let us calculate $\left[\bG^{0,1},\bG^{0,2}\right]$. 
Applying the identities in~\eqref{eq:identitiesg}, we obtain
	\begin{eqnarray*}
		\left[\bG^{0,1},\bG^{0,2}\right]&=&-\left[\bG^{0,2},\bG^{0,1}\right]\\&=& \begin{pmatrix}
			J_{ \textbf{g}^\ast_1}&0 & 0& \cdots  \\
			0& J_{ \textbf{g}^\ast_1}
			&0  & \cdots 
			\\
			0 &0& J_{ \textbf{g}^\ast_1}
			&\cdots
			\\
			\vdots &\vdots &\vdots &\ddots
		\end{pmatrix}\begin{bmatrix}
		\textbf{g}^\ast_{2} (\bm_0)\\
		\textbf{g}^\ast_{2} (\bm_1)\\
		\textbf{g}^\ast_{2} (\bm_2)\\
		\vdots
		\end{bmatrix}-\begin{pmatrix}
		J_{ \textbf{g}^\ast_2}&0 & 0& \cdots  \\
		0& J_{ \textbf{g}^\ast_2}
		&0  & \cdots 
		\\
		0 &0& J_{ \textbf{g}^\ast_2}
		&\cdots
		\\
		\vdots &\vdots &\vdots &\ddots
		\end{pmatrix}\begin{bmatrix}
		\textbf{g}^\ast_{1} (\bm_0)\\
		\textbf{g}^\ast_{1} (\bm_1)\\
		\textbf{g}^\ast_{1} (\bm_2)\\
		\vdots
		\end{bmatrix}\\
		&=& \begin{pmatrix}
			J_{\textbf{g}^\ast_1} \textbf{g}^\ast_2(\bm_0) -   J_{\textbf{g}^\ast_2} \textbf{g}^\ast_1(\bm_0)
			\\
			J_{\textbf{g}^\ast_1} \textbf{g}^\ast_2(\bm_1) -   J_{\textbf{g}^\ast_2} \textbf{g}^\ast_1(\bm_1)
			\\
			J_{\textbf{g}^\ast_1} \textbf{g}^\ast_2(\bm_2) -   J_{\textbf{g}^\ast_2} \textbf{g}^\ast_1(\bm_2)
			\\
			\vdots   
		\end{pmatrix}=\begin{pmatrix}
		\left[\textbf{g}^\ast_1,\textbf{g}^\ast_2\right](\bm_0) \\
		\left[\textbf{g}^\ast_1,\textbf{g}^\ast_2\right](\bm_1)
		\\
		\left[\textbf{g}^\ast_1,\textbf{g}^\ast_2\right](\bm_2)
		\\
		\vdots   
		\end{pmatrix}=-\begin{pmatrix}
		\textbf{g}^\ast_3(\bm_0) \\
		\textbf{g}^\ast_3(\bm_1)
		\\
		\textbf{g}^\ast_3(\bm_2)
		\\
		\vdots   
		\end{pmatrix}.
	\end{eqnarray*}
Let us note first the following: the Lie algebra of $\{\textbf{G}^{0,1},\textbf{G}^{0,2}\}$
shuffles the values only within a cell. 
That means the action of the $ i$-th cell  affects only itself and does not shift along the different modes.
In particular,
\begin{align*}
\lk[	{\bf{G}}^{0,1},	{\bf{G}}^{0,2} \rk](\bfm)&= \begin{bmatrix}
	\frac{1}{\sqrt{2}}  \textbf{g}^\ast_{3} (\bm_0)\\
	\frac{1}{\sqrt{2}}  \textbf{g}^\ast_{3} (\bm_1)\\
	\frac{1}{\sqrt{2}}  \textbf{g}^\ast_{3} (\bm_2)\\
	\vdots
\end{bmatrix}=:	{\bf{G}}^{0,3} (\bfm).
\end{align*}
Another way of calculating the Lie brackets is by writing 
$	{\bf{G}}^{0,1}$ as follows
$$
	{\bf{G}}^{0,1}:\bfm\mapsto\left(  \textbf{g}^\ast_{1}\left(\bm_{i}\right)_{i=0}^\infty\right)
\qquad \mbox{and}\qquad  
{\bf{G}}^{0,2}:\bfm\mapsto\left(  \textbf{g}^\ast_{2}(\bm_{i}\right)_{i=0}^\infty).
$$
In view of Lemma \ref{lem:identity1} in the Appendix  and taking into account that the Lie bracket  of  $\textbf{g}^\ast_1$ and $\textbf{g}^\ast_2$ are given by \eqref{eq:identitiesg}, i.e., $\left[ \textbf{g}^\ast_1,\textbf{g}^\ast_2\right](\bm)=-\textbf{g}^\ast_3(\bm)$ , hence we get 
\begin{align*}
	\lk[	{\bf{G}}^{0,1},	{\bf{G}}^{0,2} \rk]
	(\bfm)&= 
	\bfm\mapsto \left(  
	\lk[ \textbf{g}^\ast_{1},\textbf{g}^\ast_{2}\rk] 
(\bm_{i})_{i=0}^\infty\right)= 
	- \bfm\mapsto\left(   \textbf{g}^\ast_{3}
	(\bm_{i})_{i=0}^\infty\right ).
\end{align*}
Hence,  one may write $	\lk[	{\bf{G}}^{0,1},	{\bf{G}}^{0,2} \rk]=	{\bf{G}}^{0,3}$. 
Next, we have to add the Lie brackets generated in combination with ${\bf{G}}^{1,1}$. First, observe, that we can write $	\textbf{G}^{1,1}$ as the following sum
%\footnote{We defined $	\mathcal{M}_A(\hbfm)$ by
%	$$
%	\mathcal{M}_A(\hbfm) = (\boldsymbol{n}_i)_{i=0}^\infty, \quad \text{where} \quad
%	\boldsymbol{n}_i :=
%	\begin{cases}
%		\mathbf{0}, & \text{if } i \in A, \\
%		\hbm_i, & \text{otherwise}.
%	\end{cases}
%%
%	$$}
\footnote{If we use the notation $\bfm\mapsto 
	\left(  \textbf{g}^\ast_{n}(\bm_{i-k})_{i=0}^\infty\right)$, $k>0$, we mean that the first $k$ entries, where $\bm_j$ is not defined, will be filled by zero, e.g.,  
	$$	\left(  \textbf{g}^\ast_{n}  
	(\bm_{i-1})_{i=0} ^\infty \right)= (0,\bm_0^\top,\bm^\top_1,\bm^\top_2,\ldots )^\top.
	$$
	}
\begin{align}\label{defG11}
	\textbf{G}^{1,1}(\bm)= 	\lk[ \bfm\mapsto 
	 \left(  \textbf{g}^\ast_{1}(\bm_{i-1})_{i=0}^\infty\right)\rk]
%	 +\lk[ \bfm\mapsto  \mathcal{M}_{\{0\}^c}
%	 \left(  \textbf{g}^\ast_{1}(\bm_{i+1})_{i=0}^\infty\right)\rk]% 
	+
\lk[ 	 \bfm\mapsto
		\left(
		\textbf{g}^\ast_{1}\left(\bm_{i+1}\right)_{i=0}^\infty
		\right) \rk] ,
\end{align}
%	\end{eqnarray*}
%
	and use Lemma \ref{eq:exentity} in the Appendix and the identities  in \eqref{eq:identitiesg}.
Doing so, we get
%\footnote{We defined $	\mathcal{M}_A(\hbfm)$ by
%		$$
%		\mathcal{M}_A(\hbfm) = (\boldsymbol{n}_i)_{i=0}^\infty, \quad \text{where} \quad
%		\boldsymbol{n}_i :=
%		\begin{cases}
%				\mathbf{0}, & \text{if } i \in A, \\
%				\hbm_i, & \text{otherwise}.
%			\end{cases}
%	%
%		$$}
\begin{align*}
&	\lk[	{\bf{G}}^{1,1},	{\bf{G}}^{0,2} \rk]
(\bfm)
\\
&=
\lk[ \bfm\mapsto   \left(  \textbf{g}^\ast_{1}(\bm_{i-1})_{i=0}^\infty\right),
\bfm\mapsto\left(  \textbf{g}^\ast_{2}(\bm_{i}\right)_{i=0}^\infty)
\rk]
%+ 
%\lk[ \bfm\mapsto  \mathcal{M}_{\{0\}^c} \left(  \textbf{g}^\ast_{1}(\bm_{i+1})_{i=0}^\infty\right),
%\bfm\mapsto\left(  \textbf{g}^\ast_{2}(\bm_{i}\right)_{i=0}^\infty)
%\rk]
%\\
%&\qquad {}
+\lk[ 	 \bfm\mapsto
\left(
\textbf{g}^\ast_{1}\left(\bm_{i+1}\right)_{i=0}^\infty
\right),
\bfm\mapsto\left(  \textbf{g}^\ast_{2}(\bm_{i}\right)_{i=0}^\infty) \rk] 
\\
&=
 \bfm\mapsto 
% \mathcal{M}_{\{0\}^c} 
\left( [{\textbf{g}^\ast_{1}},
\textbf{g}^\ast_{2}]\left(\bm_{i-1}\right)_{i=1}^\infty
\right) 
 + \bfm\mapsto 
\left( [{\textbf{g}^\ast_{1}},
\textbf{g}^\ast_{2}]\left(\bm_{i}\right)_{i=0}^\infty
\right) 
\\&= -  \left[\hbfm\mapsto \vectinf{\textbf{g}^\ast_{3}(\hbm_{i-1})}{i}-\vectinf{\textbf{g}^\ast_{3}(\hbm_{i+1})}{i} \right] .
\end{align*}
Furthermore, we achieve
	\begin{align*}
%%%%% REVIEW [M9]: h_{0,3} and h_{0,2} are UNDEFINED here. From context these should be G^{0,3} (=-[G^{0,1},G^{0,2}]) and G^{0,2}. Fix the symbols.
	&\left[{\bf{G}}^{0,3},{\bf{G}}^{0,2}\right]=-\left[{\bf{G}}^{0,2},{\bf{G}}^{0,3}\right]=\left[\hbfm\mapsto \vectinf{\textbf{g}^\ast_{1}(\hbm_{i-1})}{i}+\vectinf{\textbf{g}^\ast_{1}(\hbm_{i+1})}{i}, \hbfm\mapsto \vectinf{\textbf{g}^\ast_{2}(\hbm_{i})}{i}\right]\\
	&\qquad=- \left[\hbfm\mapsto \vectinf{\textbf{g}^\ast_{3}(\hbm_{i-1})}{i}-\vectinf{\textbf{g}^\ast_{3}(\hbm_{i+1})}{i} \right].
\end{align*}
Here, one can observe that the initial point is shifted. This means that one cell may influence the neighbouring cell.
%%%%% REVIEW [M8]: INDEX TYPO: 'B_1 = {h_{1,0}, h_{1,1}}' but the maps defined just below are h_{1,1} and h_{1,2}. Should read B_1 = {h_{1,1}, h_{1,2}}.
	In this way, we obtain $	\mathcal{B}_1= \{	h_{1,0},	h_{1,1}\}$, 
%	\begin{eqnarray*}
%		\mathcal{B}_1&=&\left\{ \pm \begin{pmatrix}
%			\textbf{g}^\ast_3(\bm_0) \\
%			\textbf{g}^\ast_3(\bm_1)
%			\\
%			\textbf{g}^\ast_3(\bm_2)
%			\\	\textbf{g}^\ast_3(\bm_3)
%			\\
%			\vdots   
%		\end{pmatrix},\pm \begin{pmatrix}
%		\mathbf{g}_3^\ast(\bm_{1})\\
%		\mathbf{g}_3^\ast(\bm_{0}) + \mathbf{g}_3^\ast(\bm_{2})\\
%		\mathbf{g}_3^\ast(\bm_{1}) + \mathbf{g}_3^\ast(\bm_{3})\\
%		\mathbf{g}_3^\ast(\bm_{2}) + \mathbf{g}_3^\ast(\bm_{4})\\
%		\vdots
%		\end{pmatrix} \right\}
%		
%		\\
%		&=:&\{\bG _{0,1},\bG _{1,2}\}
%.
%	\end{eqnarray*}
where we  introduce the notation
\begin{eqnarray}\label{def_mapping_h}
	\nonumber 	h_{1,1}:\times _{i=0}^\infty \RR^3  &\to & \times _{i=0}^\infty\RR^3
	: \hbfm \mapsto  \vectinf{\textbf{g}^\ast_3(\hbm_i)}{i},
	\\
	\nonumber 	h_{1,2}:\times _{i=0}^\infty \RR^3  &\to & \times _{i=0}^\infty\RR^3: \hbfm \mapsto  \vectinfeins{\textbf{g}^\ast_3(\hbm_{i-1})}{i},+\vectinf{\textbf{g}^\ast_3(\hbm_{i+1})}{i}.
\end{eqnarray}	
Next, we compute
\[
\mathcal{B}_2:=\{[f,g]:\, f\in \mathcal{B},\ g\in \mathcal{B}_1\}.
\]
In particular, we have to calculate the Lie Brackets
	\begin{eqnarray*}
	[h_{1,1},	{\bf{G}}^{0,1}]=- \left[\hbfm\mapsto \vectinf{\textbf{g}^\ast_2(\hbm_i)}{i}\right],&\qquad \mbox{and}\qquad &
	\left[h_{1,1},	{\bf{G}}^{0,2}\right]= \left[\hbfm\mapsto \vectinf{\textbf{g}^\ast_1(\hbm_i)}{i}\right].
\end{eqnarray*}
%%%%% REVIEW [M9]: 'To compute [h_{0,3}, G^{0,2}]' references undefined h_{0,3}; from the following lines this should be h_{1,2}.
To compute $\bigl[h_{1,2},{\bf G}^{0,2}\bigr]$, we proceed term by term as before. This yields
\begin{eqnarray*}
	\left[h_{1,2},	{\bf{G}}^{0,1}\right]&=&-  \left[\hbfm\mapsto \vectinfeins{\textbf{g}^\ast_{2}(\hbm_{i-1})}{i}+\vectinf{\textbf{g}^\ast_{2}(\hbm_{i+1})}{i} \right],
	\\
	\left[h_{1,2},	{\bf{G}}^{0,2}\right]&=&  \left[\hbfm\mapsto \vectinfeins{\textbf{g}^\ast_{1}(\hbm_{i-1})}{i}+\vectinf{\textbf{g}^\ast_{1}(\hbm_{i+1})}{i} \right].
\end{eqnarray*}
Finally, we compute $\bigl[h_{1,2},{\bf G}^{1,1}\bigr]$. 
Using the identity for the Lie bracket $\bigl[\bg^\ast_3,\bg^\ast_1\bigr](\bm)$ from~\eqref{eq:identitiesg},
the representation of ${\bf G}^{1,1}$ in~\eqref{defG11}, and applying Lemma~\ref{lem:identity1} in Appendix~\ref{calulating_lie}, %\ref{eq:exentity},
 we obtain
\begin{align*}
	\lqq{\left[h_{1,2},	{\bf{G}}^{1,1}\right]=\left[\hbfm\mapsto \vectinfeins{\textbf{g}^\ast_{3}(\hbm_{i-1})}{i}+\vectinf{\textbf{g}^\ast_{3}(\hbm_{i+1})}{i},\hbfm\mapsto \vectinfeins{\textbf{g}^\ast_{1}(\hbm_{i-1})}{i}+\vectinf{\textbf{g}^\ast_{1}(\hbm_{i+1})}{i}\right]}\\
	& = \left[\hbfm\mapsto \vectinfeins{\textbf{g}^\ast_{3}(\hbm_{i-1})}{i},\hbfm\mapsto \vectinfeins{\textbf{g}^\ast_{1}(\hbm_{i-1})}{i}\right]+\left[\vectinf{\textbf{g}^\ast_{3}(\hbm_{i+1})}{i},\vectinf{\textbf{g}^\ast_{1}(\hbm_{i+1})}{i}\right]
	\\
	&\qquad +\Big(\left[\hbfm\mapsto \vectinfeins{\textbf{g}^\ast_{3}(\hbm_{i-1})}{i},\vectinf{\textbf{g}^\ast_{1}(\hbm_{i+1})}{i}\right]+\left[\vectinf{\textbf{g}^\ast_{3}(\hbm_{i+1})}{i},\hbfm\mapsto \vectinfeins{\textbf{g}^\ast_{1}(\hbm_{i-1})}{i}\right]\Big)\\
	& = -\left[\hbfm\mapsto \vectinfzwei{\textbf{g}^\ast_{2}(\hbm_{i-2})}{i}\right]-\left[\hbfm\mapsto \vectinf{\textbf{g}^\ast_{2}(\hbm_{i+2})}{i}\right] -2\left[\hbfm\mapsto \vectinf{\textbf{g}^\ast_{2}(\hbm_{i})}{i}\right]
	\\
	&\qquad +\left[\hbfm\mapsto\mathcal{M}_{\{0\}^c}\left(\vectinf{\textbf{g}^\ast_{2}(\hbm_{i})}{i}\right)\right].
\end{align*}
Here, one can see that sequence $(\hbm_{i})_{i=0}^\infty $ is shifted by two, i.e., the elements $(\hbm_{i-2})_{i=0}^\infty$ and $(\hbm_{i+2})_{i=0}^\infty$
appear.
Collecting all together, we obtain $\mathcal{B}_2=
\{\pm h_{2,1},\pm h_{2,2}\}$,
%
%\begin{eqnarray*}
%	&&\mathcal{B}_2=
%	%		
%	\{\pm h_{2,1},\pm h_{2,2}\},
%\end{eqnarray*}
where 
\begin{eqnarray}\label{def_mapping_h3}
	\nonumber 	h_{2,1}:\times _{i=0}^\infty \RR^3  &\to & \times _{i=0}^\infty\RR^3
	:\hbfm\mapsto\vectinfeins{ \textbf{g}^\ast_{2} (\hbm_{i-1})+ \textbf{g}^\ast_{2} (\hbm_{i+1})}{i},
	\\
	\nonumber 	h_{2,2}:\times _{i=0}^\infty \RR^3  &\to & \times _{i=0}^\infty\RR^3: \hbfm\mapsto\vectinfzwei{\textbf{g}^\ast_{2}
		(\hbm_{i-2})}{i}+2\vectinf{\textbf{g}^\ast_{2}(\hbm_{i})}{i}
		\\ &&\qquad {} +\vectinf{\textbf{g}^\ast_{2}(\hbm_{i+2})}{i}-\mathcal{M}_{\{0\}^c}\left(\vectinf{\textbf{g}^\ast_{2}(\hbm_{i})}{i}\right).
\end{eqnarray}

	\item 
In this step, we carry out the induction from $n$ to $n+1$. The lemmas collected in
%%%%% REVIEW [W6]: Subject-verb agreement: 'The lemmas collected in Appendix D IS essential' -> 'ARE essential'.
Appendix~\ref{calulating_lie} are essential for this argument.
By {\bf Step~(ii)}, the claim holds for $n=1$.
Now fix $n\ge 1$ and assume that
\begin{eqnarray}
	\nonumber\mathcal{B}_{2n}&=& \left\{\pm\mathcal{H}_{\textbf{g}^\ast_{1},2n -2},\pm \mathcal{H}_{\textbf{g}^\ast_{1},2n -1},\pm\mathcal{H}_{\textbf{g}^\ast_{2},2n -1},\pm\mathcal{H}_{\textbf{g}^\ast_{2},2n } \right\}
	\\
	\text{ and }\quad \mathcal{B}_{2n +1} &=& \left\{\pm\mathcal{H}_{\textbf{g}^\ast_{3},2n },\pm\mathcal{H}_{\textbf{g}^\ast_{3},2n +1} \right\},
	\label{eq:Bkk}
\end{eqnarray}
where these sets $\mathcal{H}_{\textbf{g}^\ast_{l},2n +1}$ and $\mathcal{H}_{\textbf{g}^\ast_{l},2n }$, $l=1,2,3$, are defined in the beginning of Appendix~\ref{app:expandLie}.  %p. \pageref{app:expandLie}. %

In the induction step, we need to show that $\mathcal{B}_{2n+2}$ and $\mathcal{B}_{2n+3}$ are given by
\begin{eqnarray}\label{ind_step}
	\nonumber\mathcal{B}_{2n+2}&=& \left\{\pm\mathcal{H}_{\textbf{g}^\ast_{1},2n },\pm \mathcal{H}_{\textbf{g}^\ast_{1},{2n+1} },\pm\mathcal{H}_{\textbf{g}^\ast_{2},2n+1 },\pm\mathcal{H}_{\textbf{g}^\ast_{2},{2n+2} } \right\}
	\\
	\text{ and }\quad \mathcal{B}_{2n +3} &=& \left\{\pm\mathcal{H}_{\textbf{g}^\ast_{3},2n+2 },\pm\mathcal{H}_{\textbf{g}^\ast_{3},2n +3} \right\}.
\end{eqnarray}	
%	\label{eq:Bk
%	
Since the computations are routine and
lengthy, we only indicate the main steps.
The main computations are carried out in Lemma~\ref{lem:identityH}, where the following identities are  established: % the following identities:
%
%The main computations are done in Lemma \refl{lem:identityH} and are to verify the following identities:
%
%
%In fact, from Lemma \eqref{lem:identityH}, it is straightforward to verify that
\begin{eqnarray*}
%%%%% REVIEW [M-style]: Notation: superscripts here are G^{(0,1)}, G^{(0,2)}, G^{(1,1)} (parentheses) whereas elsewhere they are G^{0,1} etc. Unify the superscript style.
	\left[\mathcal{H}_{\textbf{g}^\ast_{\ell },2n},{\bf G}^{(0,1)}\right]=-\mathcal{H}_{[\textbf{g}^\ast_{\ell },\textbf{g}^\ast_{1}],2n},
	\quad  &\quad &
	\left[\mathcal{H}_{\textbf{g}^\ast_{\ell},2n+1},{\bf G}^{(0,2)}\right]= \mathcal{H}_{[\textbf{g}^\ast_{\ell },\textbf{g}^\ast_{2}],2n+1},
	\\
	\left[\mathcal{H}_{\textbf{g}^\ast_{\ell },2n+1},{\bf G}^{(0,1)}\right]=-\mathcal{H}_{[\textbf{g}^\ast_{\ell },\textbf{g}^\ast_{1}],2n+1},
	\quad  &\quad & 	\left[\mathcal{H}_{\textbf{g}^\ast_{\ell },2n},{\bf G}^{(1,1)}\right]=-\mathcal{H}_{[\textbf{g}^\ast_{\ell}, \textbf{g}^\ast_{1}],2n+1},
	\\
	\left[\mathcal{H}_{\textbf{g}^\ast_{\ell },2n},{\bf G}^{(0,2)}\right]=\mathcal{H}_{[\textbf{g}^\ast_{\ell},\textbf{g}^\ast_{2}],2n},
	\quad  &\quad & 	
	\left[\mathcal{H}_{\textbf{g}^\ast_{\ell},2n+1},{\bf G}^{(1,1)}\right]= -\mathcal{H}_{[ \textbf{g}^\ast_{\ell},\textbf{g}^\ast_{1}],2n+2}.
\end{eqnarray*} 
Using these identities, we are able to show \eqref{ind_step}.
Summing up, we have proved in {\bf Part~I} that
\begin{eqnarray*}
	\operatorname{Lie}\left(\left\{ \mathbf{G}^{0,1}, \mathbf{G}^{0,2},\mathbf{G}^{1,1} \right\}\right)&=&\text{Span}\left\{\mathcal{H}_{\textbf{g}^\ast_{1},k},\mathcal{H}_{\textbf{g}^\ast_{2},k},\mathcal{H}_{\textbf{g}^\ast_{3},k}:k\in\NN_0   \right\}.
\end{eqnarray*}

	{\bf Part II:}
In this part, we investigate the controllability of system~\eqref{M1.N} from a given initial condition $\M_0$.
We distinguish two cases: $\M_0$ is constant, or $\M_0$ is a non-constant function.
	\begin{itemize}
	\item {\bf Case $\M_0$ constant:}
However, note first that, by Lemma~\ref{lem:dimspan}, the Fourier series of $\M_0$ has the form
\[
\hbfm = (\hbm_0^\top,0,0,0,\dots)^\top
\]
for some nonzero vector $\hbm_0\in\mathbb{S}^2$. Since all higher Fourier coefficients vanish, we can
construct a direction that does not belong to
\[
\operatorname{Lie}_{\M_0}\!\left(\left\{\mathbf{G}^{0,1},\,\mathbf{G}^{0,2},\,\mathbf{G}^{1,1}\right\}\right).
\]
For simplicity, we assume $\hbm_0=(1,0,0)^\top$.
Since 	
\DEQSZ\label{dofe}
	\textbf{g}^\ast_{1}(e_1) %\begin{pmatrix} x^1\\x^2\\x^3\end{pmatrix}
	= \begin{pmatrix}0\\0\\0\end{pmatrix},\textbf{g}^\ast_{2}(e_1) %\begin{pmatrix} x^1\\x^2\\x^3\end{pmatrix}
	= \begin{pmatrix}0\\0\\1\end{pmatrix}, \textbf{g}^\ast_{3}(e_1) %\begin{pmatrix} x^1\\x^2\\x^3\end{pmatrix}
	= \begin{pmatrix}0\\-1\\0\end{pmatrix},
\EEQSZ
most of the elements vanish; in particular, we obtain
	\begin{align*}
		&\mathcal{H}_{\textbf{g}_l^\ast,0}(\hbfm)=(
		\textbf{g}_l^\ast(e_1),
		0  ,
		0  ,
		0  ,
		0  ,
		0  ,
		\ldots  ,
		0  ,
		0  ,
		0  ,
		0  ,
		\ldots)^\top
		\\
		&	\mathcal{H}_{\textbf{g}_l^\ast,2}(\hbfm)=(
		\textbf{g}_l^\ast(e_1)  ,
		0  ,
		\textbf{g}_l^\ast(e_1)  ,
		0  ,
		0  ,
		0  ,
		\ldots  ,
		0  ,
		0  ,
		0  ,
		0  ,
		\ldots
		)^\top ,
		\\
		&	\mathcal{H}_{\textbf{g}_l^\ast,4}(\hbfm)=(
		2\textbf{g}_l^\ast(e_1)  ,
		0  ,
		3\textbf{g}_l^\ast(e_1)  ,
		0  ,
		\textbf{g}_l^\ast(e_1)  ,
		0  ,
		\ldots  ,
		0  ,
		0  ,
		0  ,
		0  ,
		\ldots
		)^\top
		\\
		&\mathcal{H}_{\textbf{g}_l^\ast,1}(\hbfm)=( 
		0  ,
		\textbf{g}_l^\ast(e_1)  ,
		0  ,
		0  ,
		0  ,
		0  ,
		\ldots  ,
		0  ,
		0  ,
		0  ,
		0  ,
		\ldots
		)^\top,
		\\
		&
		\mathcal{H}_{\textbf{g}_l^\ast,3}(\hbfm)=(
		0  ,
		2\textbf{g}_l^\ast(e_1)  ,
		0  ,
		\textbf{g}_l^\ast(e_1)  ,
		0  ,
		0  ,
		\dots  ,
		0  ,
		0  ,
		0  ,
		0  ,
		\ldots
		),
		\\
		&\mathcal{H}_{\textbf{g}_l^\ast,5}(\hbfm)=(
		0  ,
		5\textbf{g}_l^\ast(e_1)  ,
		0  ,
		4\textbf{g}_l^\ast(e_1)  ,
		0  ,
		\textbf{g}_l^\ast(e_1)  ,
		\ldots  ,
		0  ,
		0  ,
		0  ,
		0  ,
		\dots
		),
	\end{align*}
	%	\end{align*}
	and, continuing
	\begin{align*}
		&\mathcal{H}_{\textbf{g}_l^\ast,2k}(\hbfm)=\begin{pmatrix}
			\left(\combin{2k}{k}-\combin{2k}{k+1}\right)\textbf{g}_l^\ast(e_1)\\
			0\\
			\left(\combin{2k}{k-1}-\combin{2k}{k+2}\right)\textbf{g}_l^\ast(e_1)\\
			0\\
			\left(\combin{2k}{k-2}-\combin{2k}{k+3}\right)\textbf{g}_l^\ast(e_1)\\
			0\\
			\vdots\\
			\left(\combin{2k}{0}-\combin{2k}{2k+1}\right)\textbf{g}_l^\ast(e_1)\\
			0\\
			0\\
			0\\
			\vdots
		\end{pmatrix},\cdots, \qquad 
		\mathcal{H}_{\textbf{g}_l^\ast,2k+1}(\hbfm)=\begin{pmatrix}
			0\\
			\left(\combin{2k+1}{k}-\combin{2k+1}{k+2}\right)\textbf{g}_l^\ast(e_1)\\
			0\\
			\left(\combin{2k+1}{k-1}-\combin{2k+1}{k+3}\right)\textbf{g}_l^\ast(e_1)\\
			0\\
			\left(\combin{2k+1}{k-2}-\combin{2k+1}{k+4}\right)\textbf{g}_l^\ast(e_1)\\
			\vdots\\
			\left(\combin{2k+1}{0}-\combin{2k+1}{2k+2}\right)\textbf{g}_l^\ast(e_1)\\
			0\\
			0\\
			0\\
			\vdots
		\end{pmatrix},\cdots
	\end{align*}
%%%%% REVIEW [W13]: 'Span{g*_l(e_k): l in {1,2,3}, k in {1,2}}': for the constant state m_0=e_1 only e_1 is relevant; k in {1,2} (i.e. including e_2) is extraneous/confusing here.
Due to the identities \eqref{dofe}, in particular, $g_1^\ast(\be_1)=\bn$,  it is straightforward that $\text{Span}\{\textbf{g}^\ast_{l}(e_k):l\in\{1,2,3\},k\in\{1,2\}\}=\text{Span}\{\be_2,\be_3\}$.  Thus, we can construct three directions that do not belong to
$
	\operatorname{Lie}_{\M_0}\!\left(\left\{\mathbf{G}^{0,1},\,\mathbf{G}^{0,2},\,\mathbf{G}^{1,1}\right\}\right)$.
	More precisely,
	\begin{align*}
		&\begin{pmatrix}
			\be_1\\
			0\\
			0\\
			0\\
			\vdots
		\end{pmatrix},\begin{pmatrix}
			0\\
			\be_1\\
			0\\
			0\\
			\vdots
		\end{pmatrix},\begin{pmatrix}
			0\\
			0\\
			\be_1\\
			0\\
			\vdots
		\end{pmatrix}\cdots\notin  \operatorname{Lie}_{\M_0}\left(\left\{ \mathbf{G}^{0,1}, \mathbf{G}^{0,2}, \mathbf{G}^{1,1} \right\}\right).
	\end{align*}
	Therefore, the system~\eqref{M1.N} is not globally $\mathcal G$-reachable for all finite sets $\mathcal G$.
		\item{\bf Case $\M_0$ non-constant:} \DeclarePairedDelimiter\ceil{\lceil}{\rceil}
		\DeclarePairedDelimiter\floor{\lfloor}{\rfloor}
		From Lemma~\ref{lem:dimspan} we know that 
		\begin{itemize}
			\item Not all higher modes can vanish: there exists $k\ge 1$ with $\bm_k\neq \vec{0}$. In fact, there are an infinite number of modes that do not vanish.
			\item The coefficient vectors cannot all be collinear (under mild regularity, e.g.\ $\M_0\in H^1$):
			\[
			\dim\left(\Span\{a_0,a_1,a_2,\dots\}\right) \ge 2.
			\]
%			Equivalently, if $\Span\{a_k\}$ is one-dimensional, then $f$ can only take values in $\{\pm w\}$ for
%			some $w\in\mathbb S^2$ (and if $f\in H^1$ or continuous, this forces $f$ to be constant).
		\end{itemize}
		%
%	More generally, assuming $i_1, \dots, i_{l-1}$ have been defined, we define $i_l$ recursively for $l \geq 3$ by
%	$$
%	i_l := \min \left\{ k > i_{l-1} : \hbm_k \notin \operatorname{span} \{ \hbm_{i_{l-1}} \} \right\}.
%	$$
%		
%	Our goal is to prove that 
% any given sequence 
%	\[
%	\hbfn \in \times_{l=0}^{\infty} \mathbb{R}^3, \quad 	\hbfn = (\hat{\boldsymbol{n}}_0, \hat{\boldsymbol{n}}_1, \dots) ,
%	\]
%belongs to  the set $
%	\operatorname{Lie}_{\M_0}\left(\left\{ \mathbf{G}^{0,1}, \mathbf{G}^{0,2},\mathbf{G}^{1,1} \right\}\right)
%	$.
%%\coma{The proof is by contradiction. Let us assume that $
%%	\hbfn \not\in \operatorname{Lie}_{\hbfm}(\mathcal{F})
%%	$.
%%}
%\todo{not finished}

%	 	$
%	 	\overline{\boldsymbol{\mathfrak{m}}} = (\overline{\boldsymbol{m}}_0, \overline{\boldsymbol{m}}_1, \dots) \in \mathrm{Lie}_{\hbfm}(\mathcal{F})
%	 	$
%	 	such that $\overline{\boldsymbol{m}}_i = \hat{\boldsymbol{n}}_i$.
%	 	
%%%%% REVIEW [W12]: Fourier expansion written starting at index 1: m = (m_1^T, m_2^T, ...). Elsewhere it starts at the zero mode (m_0^T, m_1^T, ...). Start at 0 for consistency.
	 	Let us denote the to $\M_0$ corresponding Fourier expansion by $\hbfm=(\hbm^\top _0, \hbm^\top _1,\hbm_2^\top ,\ldots,\hbm_k^\top ,\ldots)^\top$.
	 	Before proceeding with the construction, we construct a sequence of indices that span all directions and will play a key role in the argument. Define $i_1 \in \mathbb{N}_0$ to be the smallest index such that $\hbm_{i_1} \neq 0$, i.e.,
	 	$$
	 	i_1 := \min \{ k \in \mathbb{N}_0 : \hbm_k \neq 0 \}.
	 	$$
%%%%% REVIEW [W11]: The clause 'In particular, k>i_l such that m_k is not collinear with m_{i_l}, i.e.' is a broken fragment and introduces i_l with no general definition (the general def is commented out). Rewrite.
	 	Next, let $i_2$ be the smallest index greater than $i_1$ such that $\hbm_{i_2}$ is linearly independent of $\hbm_{i_1}$. In particular, 
	 	$k>i_\ell$ such that $\hbm_k$ is not collinear with $\hbm_{i_\ell}$, i.e.
	 	$$
	 	i_2 := \min \{ k > i_1 : \hbm_k \notin \operatorname{span}\{ \hbm_{i_1} \} \}.
	 	$$
	 	Since
	 	$
	 	\dim \big( \operatorname{span} \{ \hbm_k : k \in \mathbb{N}_0 \} \big) \geq 2,
	 	$
	 	such indices $i_1$ and $i_2$ exist.
	 	%
%	 	Let us define the element
%	 	$$
%	 	\tilde {\bfm}_0=(e_1,0,\cdots, 0, \bm_{i_1},0,\cdots,0,\bm_{i_2},0,\cdots)^\top.
%	 	$$
%	 Now, let us analyse $\Span_{\hbfm}(\mathcal{B}_0\cup \mathcal{B}_1)$. 
%	 \DEQS
%&&	\Span_{\hbfm} \left\{  \hbfm \mapsto \textbf{g}^\ast_{1}\left(\bm_{i}\right)_{i=0}^\infty,
% \hbfm \mapsto 	 \textbf{g}^\ast_{2}\left (\bm_{i}\right)_{i=0}^\infty,  \rk.
% \\
%&& \qquad \lk.
%	 \hbfm \mapsto  \vectinf{\textbf{g}^\ast_3(\hbm_i)}{i},
%	\hbfm \mapsto  \vectinf{\textbf{g}^\ast_3(\hbm_{i-1})}{i}+\vectinf{\textbf{g}^\ast_3(\hbm_{i+1})}{i}
%	\right\}.
%\EEQS 
Let $\mathcal{S}_0:=\left\{ (0,\ell ): \ell\in\{1,2,3\}\right\}$ and ${\bf S}_0:= \Span\left(\mathcal S_0\right)$.
Let us put
$$
c^{2k}_j:= C^{2k}_{j}-C_{2k}^{j+1}.
$$
Then, we have 
$$
{\bf S}_0\subset 
 \Span_{\hbfm} \{ 	\mathcal{H}_{\textbf{g}_l^\ast,2k}( {\hbfm}),
\mathcal{H}_{\textbf{g}_l^\ast,2k-1}({\hbfm}):0\le k\le i_2\}.
$$
This can be shown by generating all three directions as linear combinations of elements in
\[
\Span_{\hbfm}\Bigl\{
\mathcal{H}_{\bg_l^\ast,\,2k}(\hbfm),\ \mathcal{H}_{\bg_l^\ast,\,2k-1}(\hbfm):\ 0\le k\le i_2
\Bigr\}.
\]
We begin by showing that one can generate vectors whose first component is
$\bg_l^\ast(e_1)$ or $\bg_l^\ast(e_2)$, for $l\in\{1,2,3\}$. This can be shown by generating 
all three directions by linear combinations of elements of $ \Span_{\hbfm} \{ 	\mathcal{H}_{\textbf{g}_l^\ast,2k}( {\hbfm}),
\mathcal{H}_{\textbf{g}_l^\ast,2k-1}({\hbfm}):0\le k\le i_2\}$. Let us start to show that vectors, with the first element being $g_l^\ast(\e_1)$, $g_l^\ast(\e_2)$, $l\in\{1,2,3\}$, can be generated.
Then, the claim follows from Proposition \ref{lem:al2}.
Let us assume for simplicity that $\hbm_{i_1}=\e_2$, and let us start calculating the following differences
\DEQS%Z %\label{diffh}
\lqq{ c_{0}^0\mathcal{H}_{\textbf{g}_l^\ast,2}(\tilde {\bfm})-c_0^2
	\mathcal{H}_{\textbf{g}_l^\ast,0}(\tilde {\bfm})}
	\\
	&
	=&\begin{pmatrix*} c^{2}_0 {\textbf{g}_l^\ast}(\be_1)+ c^{2}_1	{\textbf{g}_l^\ast}(\bm_2)
		\\
		{\textbf{g}_l^\ast}(\bm_1)+	{\textbf{g}_l^\ast}(\bm_3)
		\\
		{\textbf{g}_l^\ast}(e_1)+	{\textbf{g}_l^\ast}(\bm_2)+	{\textbf{g}_l^\ast}(\bm_{4})
		\\
		\vdots 
\end{pmatrix*}
-\begin{pmatrix*} c_{0}^0 {\textbf{g}_l^\ast}(\be_1)
	\\
	{\textbf{g}_l^\ast}(\bm_1)
	\\
	{\textbf{g}_l^\ast}(\be_1)+	{\textbf{g}_l^\ast}(\bm_2)+	{\textbf{g}_l^\ast}(\bm_{4})
	\\
	\vdots 
\end{pmatrix*}
=
\begin{pmatrix*} c_0^2c_1^2	{\textbf{g}_l^\ast}(\bm_2)
	\\
\vdots 	\\
	\vdots 
	\\
	\vdots 
\end{pmatrix*},
\\
%%%%% REVIEW [M10]: ERROR: this difference reads 'H_{g,0}(m) - H_{g,0}(m) - alpha H_{g,2}(m)'. The first two terms are IDENTICAL, so they cancel to 0. Given the RHS uses C^4 coefficients (c^4_0, c^4_1, c^4_2), the first term should be H_{g,4}(m) (or H_{g,4}-H_{g,2}-...). Fix the indices.
\lqq{ 	
	\mathcal{H}_{\textbf{g}_l^\ast,4}(\tilde {\bfm})
	-	\mathcal{H}_{\textbf{g}_l^\ast,0}(\tilde {\bfm})
-\alpha \mathcal{H}_{\textbf{g}_l^\ast,2}(\tilde {\bfm})
}
\\
&
=&\begin{pmatrix*} c^{4}_0 {\textbf{g}_l^\ast}(\be_1)+ c^{4}_1	{\textbf{g}_l^\ast}(\bm_2)+ c^{4}_2{\textbf{g}_l^\ast}(\bm_2)
	\\
	{\textbf{g}_l^\ast}(\bm_1)+	{\textbf{g}_l^\ast}(\bm_3)
	\\
	{\textbf{g}_l^\ast}(\be_1)+	{\textbf{g}_l^\ast}(\bm_2)+	{\textbf{g}_l^\ast}(\bm_{4})
	\\
	\vdots 
\end{pmatrix*}
-\begin{pmatrix*} c_{0}^2 {\textbf{g}_l^\ast}(\be_1)
+ c_{1}^2 {\textbf{g}_l^\ast}(\bm_1)
	\\
	{\textbf{g}_l^\ast}(\bm_1)
	\\
	{\textbf{g}_l^\ast}(e_1)+	{\textbf{g}_l^\ast}(\bm_2)+	{\textbf{g}_l^\ast}(\bm_{4})
	\\
	\vdots 
\end{pmatrix*}
\\
&=&
-\begin{pmatrix*}\alpha_{1} 	
	c_0^4c_1^4	{\textbf{g}_l^\ast}(\bm_2)
	\\
	\vdots 	\\
	\vdots 
	\\
	\vdots 
\end{pmatrix*}.
\EEQS
%%%%% REVIEW [M13]: Range mismatch: text says 'alpha_r, r=1,...,i_2' but the sum below runs r=2,...,i_1/2 (which also presumes i_1 even). Reconcile the index range.
In this way we can find $\alpha_r$, $r=1,\ldots, i_2$, such that 
\DEQS %Z %\label{diffh}
\lqq{ \mathcal{H}_{\textbf{g}_l^\ast,i_1}(\tilde {\bfm})-
\mathcal{H}_{\textbf{g}_l^\ast,i_1-2}(\tilde {\bfm}
-\sum_{r=2}^{i_1/2}\alpha _r
\mathcal{H}_{\textbf{g}_l^\ast,i_1-2r}(\tilde {\bfm})
}
\\
%%%%% REVIEW [M13]: Coefficient-notation clash: c^{2k}_j was DEFINED (line ~1805) as C^{2k}_j - C^{j+1}_{2k}. Here it is written c_k^0, c_k^2, c_k^{i_1} with sub/superscripts swapped. Unify.
&=&\begin{pmatrix*} c_k^0 {\textbf{g}_l^\ast}(e_1)+c_k^2 	{\textbf{g}_l^\ast}(\bm_2)+\cdots+c_k^{i_1-2} {\textbf{g}_l^\ast}(\bm_{i_1-2})+	c_k^{i_1}{\textbf{g}_l^\ast}(\bm_{i_1})
	\\
	{\textbf{g}_l^\ast}(\bm_1)+	{\textbf{g}_l^\ast}(\bm_3)+\cdots +{\textbf{g}_l^\ast}(\bm_{i_1-1})+{\textbf{g}_l^\ast}(\bm_{i_1+1})
	\\
	{\textbf{g}_l^\ast}(e_1)+	{\textbf{g}_l^\ast}(\bm_2)+\cdots+	{\textbf{g}_l^\ast}(\bm_{i_1})+{\textbf{g}_l^\ast}(\bm_{i_1+2})
	\\
	\vdots 
\end{pmatrix*}
\\ &&{}
-\begin{pmatrix*} c_{k+1}^0 {\textbf{g}_l^\ast}(e_1)+	 c_{k+1}^0{\textbf{g}_l^\ast}(\bm_2)+\cdots+	 c_{k+1}^{i_1-2}{\textbf{g}_l^\ast}(\bm_{i_1-2})
	\\
	{\textbf{g}_l^\ast}(\bm_1)+	{\textbf{g}_l^\ast}(\bm_3)+\cdots +{\textbf{g}_l^\ast}(\bm_{i_1-1})
	\\
	{\textbf{g}_l^\ast}(e_1)+	{\textbf{g}_l^\ast}(\bm_2)+\cdots+	{\textbf{g}_l^\ast}(\bm_{i_1})
	\\
	\vdots 
\end{pmatrix*}-\sum_{r=2}^{i_1/2}\alpha _r
\mathcal{H}_{\textbf{g}_l^\ast,i_1-2r}(\tilde {\bfm})
%
%
%\\
%&=&\begin{pmatrix*} {\textbf{g}_l^\ast}(e_1)
%	+	{\textbf{g}_l^\ast}(\bm_2)+\cdots 
%	\\
%	{\textbf{g}_l^\ast}(e_1)+{\textbf{g}_l^\ast}(\bm_{i_2+1})
%	\\
%	\vdots 
%\end{pmatrix*},
%
\\
%&&\mathcal{H}_{\textbf{g}_l^\ast,i_2}(\tilde {\bfm})-
%\mathcal{H}_{\textbf{g}_l^\ast,i_2-1}(\tilde {\bfm})
&=&\begin{pmatrix*} {\textbf{g}_l^\ast}(\be_2)
	\\
	{\textbf{g}_l^\ast}(\be_2)+{\textbf{g}_l^\ast}(\bm_{i_2+1})
	\\
\vdots 
\end{pmatrix*}.\nonumber 
\EEQS
%%%%% REVIEW [M11]: Wrong/forward cross-reference: 'the differences in (2.21)' points to \label{diffh}, which is the equation block AFTER this sentence; the differences meant are the (unlabelled) block above. Also (2.21) partly duplicates the result just computed -- check for redundancy.
Taking successively $\ell=1,2,$ and $3$, the differences in~\eqref{diffh} yield vectors of the form
$(\be^\top_1,\ldots)^\top$, $(\be^\top _2,\ldots)^\top$, and $(\be^\top _3,\ldots)^\top$.
Next, we consider the following differences:
\DEQSZ\label{diffh}
\mathcal{H}_{\textbf{g}_l^\ast,i_1}(\tilde {\bfm})-
\mathcal{H}_{\textbf{g}_l^\ast,i_1-1}(\tilde {\bfm})
=\begin{pmatrix*} {\textbf{g}_l^\ast}(\be_1)
	\\
	{\textbf{g}_l^\ast}(\be_1)+{\textbf{g}_l^\ast}(\bm_{i_2+1})
	\\
	\vdots 
\end{pmatrix*},
\\
\mathcal{H}_{\textbf{g}_l^\ast,i_2}(\tilde {\bfm})-
\mathcal{H}_{\textbf{g}_l^\ast,i_2-1}(\tilde {\bfm})
=\begin{pmatrix*} {\textbf{g}_l^\ast}(\be_2)
	\\
	{\textbf{g}_l^\ast}(\be_2)+{\textbf{g}_l^\ast}(\bm_{i_2+1})
	\\
	\vdots 
\end{pmatrix*}.\nonumber 
\EEQSZ
	 	Let $N:=\max_{j\in\NN} \{ \exists \,\, k\in\NN \mbox{ such that } (j,k)\in \mathcal{G}\}$.
	 	Then,  for $K=\max(N,i_2)$,
	 	$$
	 \Span_{\hbfm} \{ 	\mathcal{H}_{\textbf{g}_l^\ast,2k}(\tilde {\bfm}),
	 	\mathcal{H}_{\textbf{g}_l^\ast,2k+1}(\tilde{\hbfm}):0\le k\le K\}\subset H_\mathcal{G}.
	 	$$ 
	 This completes the proof.

\end{itemize}
	\end{stepp}
	\end{proof}

%\end{document}

\begin{thm}\label{L2.app.cont2}
	Consider the system~\eqref{M1.N} with initial condition $\M_0\in\mathbb{H}^1(0,1;\mathbb{S}^2)$, and let the set of controlled forcing modes be
	$\mathcal{K} = \{(1,1), (1,2), (1,3),(2,1), (2,2), (2,3)\}.$
	Then, the system~\eqref{M1.N} is globally controllable.
\end{thm}

\begin{proof}
	The proof of Theorem \ref{L2.app.cont2} consists of several steps. In this proof, we adopt the notations introduced in Proof of Theorem \ref{app.cont1}
		 and Appendix~\ref{app:expandLie}.
	%	Let us now consider the set of controls given by
	%	$\mathcal{K} = \{(1,1), (1,2), (1,3),(2,1), (2,2), (2,3)\}$.
		In these notations, \eqref{eq:llgeinm} can be written as
		\begin{align}
			\nonumber&\frac d{dt}\, { \bfm}(t)=  {\bf{G}}^{1,1}( \bfm(t))  v_1^1(t)+{\bf{G}}^{1,2}( \bfm(t))  v_1^2(t)+{\bf{G}}^{1,3}( \bfm(t))  v_1^3(t)\\
			\label{eq:llgeinmcase2}&\quad\qquad+{\bf{G}}^{2,1}( \bfm(t))  v_2^1(t)+{\bf{G}}^{2,2}( \bfm(t))  v_2^2(t)+{\bf{G}}^{2,3}( \bfm(t))  v_2^3(t),\quad  t\in[0,T],
		\end{align}
Let us define the family of vector fields
\[
\mathcal{F}:=\left\{\,u_1 {\bf G}^{1,1}+u_2 {\bf G}^{1,2}+u_3 {\bf G}^{1,3}
+u_4 {\bf G}^{2,1}+u_5 {\bf G}^{2,2}+u_6 {\bf G}^{2,3}
:\; u_i\in\RR,\ i=1,2,\ldots,6 \right\},
\]
and analyze its Lie algebra, which is equivalent, by Lemma~\ref{lem:Lieaffine}, to the Lie algebra generated by the set of generators, namely
\[
\Lie\!\left(\left\{{\bf G}^{1,1},{\bf G}^{1,2},{\bf G}^{1,3},{\bf G}^{2,1},{\bf G}^{2,2},{\bf G}^{2,3}\right\}\right).
\]
We will show that the control action shuffles within each cell and propagates to its neighbouring cells. Again, an induction argument yields the proof.
\begin{stepp}
	\item 
	 		 First, observe, that we can write $	\textbf{G}^{1,\ell}$, $\ell\in\{1,2,3\}$, as the following sum
		\begin{align}\label{defG11_2}
			\textbf{G}^{1,\ell}(\bm)= 	\lk[ \bfm\mapsto  %\mathcal{M}_{\{0\}}
			 \left(  \textbf{g}^\ast_{\ell}(\bm_{i-1})_{i=1}^\infty\right)\rk]% 
			+
			\lk[ 	 \bfm\mapsto
			\left(
			\textbf{g}^\ast_{\ell}\left(\bm_{i+1}\right)_{i=0}^\infty
			\right) \rk] ,
		\end{align}
		and
			\begin{align}\label{defG21_2}
			\textbf{G}^{2,\ell }(\bm)= 	\lk[ \bfm\mapsto  %\mathcal{M}_{\{0,1\}} 
			\left(  \textbf{g}^\ast_{\ell}(\bm_{i-2})_{i=2}^\infty\right)\rk]% 
			+
			\lk[ 	 \bfm\mapsto
			\left(
			\textbf{g}^\ast_{\ell}\left(\bm_{i+2}\right)_{i=0}^\infty
			\right) \rk] .
				\end{align}
			%	\end{eqnarray*}
%%%%% REVIEW [M18]: TYPO: 'the Lie algebra spanned by of G^{1,l}, G^{1,l}, and G^{1,l}'. Remove 'of'; and the three fields should be G^{1,1}, G^{1,2}, G^{1,3} (similarly G^{2,1}, G^{2,2}, G^{2,3}).
Let us note first the following: the Lie algebra spanned by of $\textbf{G}^{1,\ell}$,
$\textbf{G}^{1,\ell}$, and $\textbf{G}^{1,\ell}$, $\ell\in\{1,2,3\}$, or spanned by $\textbf{G}^{2,\ell}$,
$\textbf{G}^{2,\ell}$, and $\textbf{G}^{2,\ell}$, $\ell\in\{1,2,3\}$,
can only shuffle around within a cell \( \bm_i=(\bm_i^1, \bm_i^2, \bm_i^3)^\top \) but cannot induce shifts across different cells (or modes).
To propagate the control to other cells, we need Lie brackets of the form
$
[{\bf G}^{1,\ell},{\bf G}^{2,k}]$,  $\ell\neq k $.
Using Lemma~\ref{lem:identity1}, Lemma \ref{cor:identity1} together with the identities in~\eqref{eq:identitiesg}, we obtain, for instance, the following expression for $\lk[	{\bf{G}}^{1,1},	{\bf{G}}^{2,2} \rk]$:
	\begin{align*}
		&	\lk[	{\bf{G}}^{1,1},	{\bf{G}}^{2,2} \rk]
		(\bfm)
		\\
		&=\lk[ \bfm\mapsto  %\mathcal{M}_{\{0\}} 
		\left(  \textbf{g}^\ast_{1}(\bm_{i-1})_{i=1}^\infty\right),
		\bfm\mapsto\left(  %\mathcal{M}_{\{0,1\}}
		 \left( \textbf{g}^\ast_{2}(\bm_{i-2}\right)_{i=2}^\infty\right)
		\rk]
	 + \lk[ 	 \bfm\mapsto
	% \mathcal{M}_{\{0\}} 
		\left(
		\textbf{g}^\ast_{1}\left(\bm_{i-1}\right)_{i=1}^\infty
		\right),
		\bfm\mapsto\left(  \textbf{g}^\ast_{2}(\bm_{i+2}\right)_{i=0}^\infty) \rk] 
		\\
	&\quad {} + \lk[ 	 \bfm\mapsto
	\left(
	\textbf{g}^\ast_{1}\left(\bm_{i+1}\right)_{i=0}^\infty
	\right),
	\bfm\mapsto % \mathcal{M}_{\{0,1\}}%
	 \left(  \textbf{g}^\ast_{2}(\bm_{i-2}\right)_{i=2}^\infty) \rk] 
	 + \lk[ 	 \bfm\mapsto
	\left(
	\textbf{g}^\ast_{1}\left(\bm_{i+1}\right)_{i=0}^\infty
	\right),
	\bfm\mapsto \left(  \textbf{g}^\ast_{2}(\bm_{i+2}\right)_{i=0}^\infty) \rk] 
		\\
		&=
		\bfm\mapsto 
		%\mathcal{M}_{\{0\}}
		 \left( [{\textbf{g}^\ast_{1}},
		\textbf{g}^\ast_{2}]\left(\bm_{i-3}\right)_{i=1}^\infty
		\right) 
		+ 2	\left[\bfm\mapsto 
	 [{\textbf{g}^\ast_{1}},
		\textbf{g}^\ast_{2}]\left(\bm_{i+1}\right)_{i=0}^\infty
		\right] 
\\&\quad {}	-\left[ \bfm\mapsto 	\mathcal{M}_{\{0\}^c}
	 [{\textbf{g}^\ast_{1}},
		\textbf{g}^\ast_{2}]\left(\bm_{i+1}\right)_{i=0}^\infty\right]
+ 
	\left[			\bfm\mapsto 
%\mathcal{M}_{\{0\}}
 [{\textbf{g}^\ast_{1}},
\textbf{g}^\ast_{2}]\left(\bm_{i+3}\right)_{i=1}^\infty\right]
%%		\\&= 		\bfm\mapsto 
%\mathcal{M}_{\{0\}}
\\
&= \bfm\mapsto  {\textbf{g}^\ast_{3}}\left(\bm_{i-3}\right)_{i=1}^\infty
+2	\bfm\mapsto 
{\textbf{g}^\ast_{3}}\left(\bm_{i+1}\right)_{i=0}^\infty
	- \bfm\mapsto 	\mathcal{M}_{\{0\}^c}
{\textbf{g}^\ast_{3}}\left(\bm_{i+1}\right)_{i=0}^\infty
+ 
		\bfm\mapsto 
%\mathcal{M}_{\{0\}}
{\textbf{g}^\ast_{3}} \left(\bm_{i+3}\right)_{i=1}^\infty
.
	\end{align*}
Here, one can see that if we start with a constant initial condition, i.e., $\bfm=(\be_1^\top,{\bf 0},{\bf 0},\ldots)^\top$, 	
then ${\textbf{g}^\ast_{3}}(\be_1)$ pops up in the third mode.

		\item Let us now start by analyzing the Lie algebra $$\mbox{Lie}(\mathcal{F})=\mbox{Lie}\left(\left\{ {\bf{G}}^{1,1},{\bf{G}}^{1,2},{\bf{G}}^{1,3},{\bf{G}}^{2,1},{\bf{G}}^{2,2},{\bf{G}}^{2,3}\right\}\right).$$
		It is straightforward to verify that Lemma \ref{lem:identityH} gives 
		\begin{eqnarray*}
			\left\{\mathcal{H}_{\textbf{g}^\ast_{l},k}:l\in\{1,2,3\},k\in\NN   \right\}\subseteq \mbox{Lie}\left(\left\{ {\bf{G}}^{1,1},{\bf{G}}^{1,2},{\bf{G}}^{1,3},{\bf{G}}^{2,1},{\bf{G}}^{2,2},{\bf{G}}^{2,3}\right\}\right),
		\end{eqnarray*}
		where 
		\begin{align}
			\nonumber&\mathcal{H}_{\textbf{g}_l^\ast,2k}(\hbfm)\\
			&:=\begin{pmatrix}
				\left(\combin{2k}{k}-\combin{2k}{k+1}\right)\textbf{g}_l^\ast(\hbm_{0})+\left(\combin{2k}{k+1}-\combin{2k}{k+2}\right)\textbf{g}_l^\ast(\hbm_{2})+\cdots+\left(\combin{2k}{2k}-\combin{2k}{2k+1}\right)\textbf{g}_l^\ast(\hbm_{2k})\\
				\left(\combin{2k}{k}-\combin{2k}{k+2}\right)\textbf{g}_l^\ast(\hbm_{1})+\left(\combin{2k}{k+1}-\combin{2k}{k+3}\right)\textbf{g}_l^\ast(\hbm_{3})+\cdots+\left(\combin{2k}{2k}-\combin{2k}{2k+2}\right)\textbf{g}_l^\ast(\hbm_{2k+1})\\
				\left(\combin{2k}{k-1}-\combin{2k}{k+2}\right)\textbf{g}_l^\ast(\hbm_{0})+\left(\combin{2k}{k}-\combin{2k}{k+3}\right)\textbf{g}_l^\ast(\hbm_{2})+\cdots+\left(\combin{2k}{2k}-\combin{2k}{2k+3}\right)\textbf{g}_l^\ast(\hbm_{2k+2})\\
				\left(\combin{2k}{k-1}-\combin{2k}{k+3}\right)\textbf{g}_l^\ast(\hbm_{1})+\left(\combin{2k}{k}-\combin{2k}{k+4}\right)\textbf{g}_l^\ast(\hbm_{3})+\cdots+\left(\combin{2k}{2k}-\combin{2k}{2k+4}\right)\textbf{g}_l^\ast(\hbm_{2k+3})\\
				\vdots\\
				\left(\combin{2k}{2}-\combin{2k}{2k-1}\right)\textbf{g}_l^\ast(\hbm_{0})+\left(\combin{2k}{3}-\combin{2k}{2k}\right)\textbf{g}_l^\ast(\hbm_{2})+\cdots+\left(\combin{2k}{2k}-\combin{2k}{4k-3}\right)\textbf{g}_l^\ast(\hbm_{4k-4})\\
				\left(\combin{2k}{2}-\combin{2k}{2k}\right)\textbf{g}_l^\ast(\hbm_{1})+\left(\combin{2k}{3}-\combin{2k}{2k+1}\right)\textbf{g}_l^\ast(\hbm_{3})+\cdots+\left(\combin{2k}{2k}-\combin{2k}{4k-2}\right)\textbf{g}_l^\ast(\hbm_{4k-3})\\
				\vdots
			\end{pmatrix}
			\label{eq:hgmsimple3}
		\end{align}
		and
		\begin{align}
			\nonumber&\mathcal{H}_{\textbf{g}_l^\ast,2k+1}(\hbfm)\\
			&:=\begin{pmatrix}
				\left(\combin{2k+1}{k+1}-\combin{2k+1}{k+2}\right)\textbf{g}_l^\ast(\hbm_{1})+\left(\combin{2k+1}{k+2}-\combin{2k+1}{k+3}\right)\textbf{g}_l^\ast(\hbm_{3})+\cdots+\left(\combin{2k+1}{2k+1}-\combin{2k+1}{2k+2}\right)\textbf{g}_l^\ast(\hbm_{2k+1})\\
				\left(\combin{2k+1}{k}-\combin{2k+1}{k+2}\right)\textbf{g}_l^\ast(\hbm_{0})+\left(\combin{2k+1}{k+1}-\combin{2k+1}{k+3}\right)\textbf{g}_l^\ast(\hbm_{2})+\cdots+\left(\combin{2k+1}{2k+1}-\combin{2k+1}{2k+3}\right)\textbf{g}_l^\ast(\hbm_{2k+2})\\
				\left(\combin{2k+1}{k}-\combin{2k+1}{k+3}\right)\textbf{g}_l^\ast(\hbm_{1})+\left(\combin{2k+1}{k+1}-\combin{2k+1}{k+4}\right)\textbf{g}_l^\ast(\hbm_{3})+\cdots+\left(\combin{2k+1}{2k+1}-\combin{2k+1}{2k+4}\right)\textbf{g}_l^\ast(\hbm_{2k+3})\\
				\left(\combin{2k+1}{k-1}-\combin{2k+1}{k+3}\right)\textbf{g}_l^\ast(\hbm_{0})+\left(\combin{2k+1}{k}-\combin{2k+1}{k+4}\right)\textbf{g}_l^\ast(\hbm_{2})+\cdots+\left(\combin{2k+1}{2k+1}-\combin{2k+1}{2k+5}\right)\textbf{g}_l^\ast(\hbm_{2k+4})\\
				\vdots\\
				\left(\combin{2k+1}{2}-\combin{2k+1}{2k+1}\right)\textbf{g}_l^\ast(\hbm_{1})+\left(\combin{2k+1}{3}-\combin{2k+1}{2k+2}\right)\textbf{g}_l^\ast(\hbm_{2})+\cdots+\left(\combin{2k+1}{2k+1}-\combin{2k+1}{4k}\right)\textbf{g}_l^\ast(\hbm_{4k-1})\\
				\left(\combin{2k+1}{1}-\combin{2k+1}{2k+1}\right)\textbf{g}_l^\ast(\hbm_{0})+\left(\combin{2k+1}{2}-\combin{2k+1}{2k+2}\right)\textbf{g}_l^\ast(\hbm_{2})+\cdots+\left(\combin{2k+1}{2k+1}-\combin{2k+1}{4k+1}\right)\textbf{g}_l^\ast(\hbm_{4k})\\
				\vdots
			\end{pmatrix}.
			\label{eq:hgmsimple4}
		\end{align}
		Therefore, by an argument similar to that used in the proof of Theorem~\ref{app.cont1}, we conclude that the system~\eqref{M1.N} is globally controllable whenever the initial condition $\M_0$ is a non-constant function. 
		On the other hand, in the case of a constant initial condition, we consider the specific form
		$$
		\hbfm = (\hbm_0^T, 0, 0, 0, \dots)^T,
		$$
		which corresponds to a spatially constant state, where $\hbm_0 \neq 0$. In the following, the detailed calculations are omitted and can be found in Appendix~\ref{sec:proofTheocont2}.

		Let us start with the Lie bracket of three vector fields. Appendix~\ref{sec:proofTheocont2} gives 
		\begin{eqnarray*}
			2^{3/2}\left[\left[ \bf{G}^{1,1},\bf{G}^{2,3} \right],\bf{G}^{1,2}\right](\hbfm)&=&  \begin{bmatrix}
				0&0&
				\textbf{p}^\ast_1(\hbm_{0})&0&0&\cdots
			\end{bmatrix},\\
			-2^{3/2}\left[\left[ \bf{G}^{1,3},\bf{G}^{2,1} \right],\bf{G}^{1,2}\right](\hbfm)&=&\begin{bmatrix}
				0&0&
				\textbf{p}^\ast_3(\hbm_{0})&0&0&\cdots
			\end{bmatrix},\\
			-2^{3/2}\left[\left[ \bf{G}^{1,2},\bf{G}^{2,3} \right],\bf{G}^{1,1}\right](\hbfm)&=&\begin{bmatrix}
				0&0&
				\textbf{p}^\ast_2(\hbm_{0})&0&0&\cdots
			\end{bmatrix},\\
			-2^{3/2}\left[\left[ \bf{G}^{1,1},\bf{G}^{2,1} \right],\bf{G}^{1,2}\right]&=& \begin{bmatrix}
				0&0&
				\textbf{h}^\ast_2(\hbm_{0})&0&0&\cdots
			\end{bmatrix},\\
			2^{3/2}\left[\left[ \bf{G}^{1,1},\bf{G}^{2,1} \right],\bf{G}^{1,3}\right]&=& \begin{bmatrix}
				0&0&
				\textbf{h}^\ast_1(\hbm_{0})&0&0&\cdots
			\end{bmatrix},\\
			-2^{3/2}\left[\left[ \bf{G}^{1,2},\bf{G}^{2,2} \right],\bf{G}^{1,1}\right]&=& \begin{bmatrix}
				0&0&
				\textbf{h}^\ast_4(\hbm_{0})&0&0&\cdots
			\end{bmatrix},\\
			-2^{3/2}\left[\left[ \bf{G}^{1,2},\bf{G}^{2,2} \right],\bf{G}^{1,3}\right]&=& \begin{bmatrix}
			0&0&
			\textbf{h}^\ast_3(\hbm_{0})&0&0&\cdots
			\end{bmatrix},\\
			-2^{3/2}\left[\left[ \bf{G}^{1,3},\bf{G}^{2,3} \right],\bf{G}^{1,1}\right]&=& \begin{bmatrix}
			0&0&
			\textbf{h}^\ast_6(\hbm_{0})&0&0&\cdots
			\end{bmatrix},\\
			2^{3/2}\left[\left[ \bf{G}^{1,3},\bf{G}^{2,3} \right],\bf{G}^{1,2}\right]&=& \begin{bmatrix}
			0&0&
			\textbf{h}^\ast_5(\hbm_{0})&0&0&\cdots
			\end{bmatrix},
		\end{eqnarray*}
		where   $\textbf{p}^\ast_i(\textbf{x})$, $i\in\{1,2,3\}$,  denotes the projection on the $i^{th}$ coordinate, and 
		\begin{eqnarray*}
			&&\textbf{h}^\ast_1(\textbf{x}):= \begin{bmatrix}
				0\\
				x_1\\
				0
			\end{bmatrix},
		 \textbf{h}^\ast_2(\textbf{x}):=\begin{bmatrix}
				0\\
				0\\
				x_1
			\end{bmatrix},\textbf{h}^\ast_3(\textbf{x}):= \begin{bmatrix}
			x_2\\
			0\\
			0
			\end{bmatrix},\\
			&&
			\textbf{h}^\ast_4(\textbf{x}):= \begin{bmatrix}
				0\\
				0\\
				x_2
			\end{bmatrix},\textbf{h}^\ast_5(\textbf{x}):= \begin{bmatrix}
				x_3\\
				0\\
				0
			\end{bmatrix},
			\textbf{h}^\ast_6(\textbf{x}):= \begin{bmatrix}
				0\\
				x_3\\
				0
			\end{bmatrix}.
		\end{eqnarray*}
		Moreover, we also have 
		\begin{eqnarray*}
			\mathcal{H}_{\textbf{g}^\ast_{1},4}&=& \begin{bmatrix}
				\textbf{g}^\ast_1(\hbm_{0})&0&
				\textbf{g}^\ast_1(\hbm_{0})&0&0&\cdots
			\end{bmatrix},\\
			\mathcal{H}_{\textbf{g}^\ast_{2},4}&=& \begin{bmatrix}
				\textbf{g}^\ast_2(\hbm_{0})&0&
				\textbf{g}^\ast_2(\hbm_{0})&0&0&\cdots
			\end{bmatrix},\\
			\mathcal{H}_{\textbf{g}^\ast_{3},4}&=& \begin{bmatrix}
				\textbf{g}^\ast_3(\hbm_{0})&0&
				\textbf{g}^\ast_3(\hbm_{0})&0&0&\cdots
			\end{bmatrix}.
		\end{eqnarray*}
		Observe, we obtain $\text{span}\{\textbf{g}^\ast_k,\textbf{p}^\ast_k:k\in\{1,2,3\} \}=\RR^3$ for $\textbf{x}\neq 0$. Therefore, we obtain
		$$\begin{bmatrix}
			0&0&
			\textbf{x}&0&0&\cdots
		\end{bmatrix}\in \mbox{Lie}\left(\left\{ {\bf{G}}^{1,1},{\bf{G}}^{1,2},{\bf{G}}^{1,3},{\bf{G}}^{2,1},{\bf{G}}^{2,2},{\bf{G}}^{2,3}\right\}\right)$$
		for all $\textbf{x}\in\RR^3$. Next, observe the Lie bracket of four vector fields. Appendix~\ref{sec:proofTheocont2} gives 
		\begin{eqnarray*}
			2^{2}\left[\left[\left[ \bf{G}^{1,1},\bf{G}^{2,3} \right],\bf{G}^{1,2}\right],\bf{G}^{1,1}\right]&=&\begin{bmatrix}
				0&0&
				0&0&0&0&\cdots
			\end{bmatrix}\\
			-2^{2}\left[\left[\left[ \bf{G}^{1,1},\bf{G}^{2,3} \right],\bf{G}^{1,2}\right],\bf{G}^{1,2}\right]&=&\begin{bmatrix}
				0&\textbf{h}^\ast_2(\hbm_{0})&
				0&\textbf{p}^\ast_4(\hbm_{0})&0&0&\cdots
			\end{bmatrix}\\
			-2^{2}\left[\left[\left[ \bf{G}^{1,1},\bf{G}^{2,3} \right],\bf{G}^{1,2}\right],\bf{G}^{1,3}\right]&=& \begin{bmatrix}
				0&\textbf{h}^\ast_1(\hbm_{0})&
				0&\textbf{p}^\ast_5(\hbm_{0})&0&0&\cdots
			\end{bmatrix}\\
			2^{2}\left[\left[\left[ \bf{G}^{1,2},\bf{G}^{2,3} \right],\bf{G}^{1,1}\right],\bf{G}^{1,1}\right]&=&  \begin{bmatrix}
				0&\textbf{h}^\ast_4(\hbm_{0})&
				0&\textbf{p}^\ast_6(\hbm_{0})&0&0&\cdots
			\end{bmatrix}\\
			-2^{2}\left[\left[\left[ \bf{G}^{1,2},\bf{G}^{2,3} \right],\bf{G}^{1,1}\right],\bf{G}^{1,2}\right]&=& \begin{bmatrix}
				0&0&
				0&0&0&0&\cdots
			\end{bmatrix}\\
			2^{2}\left[\left[\left[ \bf{G}^{1,2},\bf{G}^{2,3} \right],\bf{G}^{1,1}\right],\bf{G}^{1,3}\right]&=& \begin{bmatrix}
				0&\textbf{h}^\ast_3(\hbm_{0})&
				0&\textbf{p}^\ast_5(\hbm_{0})&0&0&\cdots
			\end{bmatrix}\\
			2^{2}\left[\left[\left[ \bf{G}^{1,3},\bf{G}^{2,1} \right],\bf{G}^{1,2}\right],\bf{G}^{1,1}\right]&=& \begin{bmatrix}
				0&\textbf{h}^\ast_6(\hbm_{0})&
				0&\textbf{p}^\ast_6(\hbm_{0})&0&0&\cdots
			\end{bmatrix}\\
			2^{2}\left[\left[\left[ \bf{G}^{1,3},\bf{G}^{2,1} \right],\bf{G}^{1,2}\right],\bf{G}^{1,2}\right]&=& \begin{bmatrix}
				0&\textbf{h}^\ast_5(\hbm_{0})&
				0&\textbf{p}^\ast_4(\hbm_{0})&0&0&\cdots
			\end{bmatrix}\\
			-2^{2}\left[\left[\left[ \bf{G}^{1,3},\bf{G}^{2,1} \right],\bf{G}^{1,2}\right],\bf{G}^{1,3}\right]&=& \begin{bmatrix}
				0&0&
				0&0&0&0&\cdots
			\end{bmatrix}\\
			-2^{2}\left[\left[\left[ \bf{G}^{1,1},\bf{G}^{2,1} \right],\bf{G}^{1,2}\right],\bf{G}^{1,2}\right]&=& \begin{bmatrix}
				0&2\textbf{p}^\ast_1(\hbm_{0})+\textbf{p}^\ast_3(\hbm_{0})&
				0&2\textbf{p}^\ast_1(\hbm_{0})&0&0&\cdots
			\end{bmatrix}\\
			-2^{2}\left[\left[\left[ \bf{G}^{1,3},\bf{G}^{2,3} \right],\bf{G}^{1,2}\right],\bf{G}^{1,2}\right]&=& \begin{bmatrix}
				0&2\textbf{p}^\ast_3(\hbm_{0})+\textbf{p}^\ast_1(\hbm_{0})&
				0&2\textbf{p}^\ast_3(\hbm_{0})&0&0&\cdots
			\end{bmatrix}\\
			-2^{2}\left[\left[\left[ \bf{G}^{1,1},\bf{G}^{2,1} \right],\bf{G}^{1,3}\right],\bf{G}^{1,3}\right]&=& \begin{bmatrix}
				0&2\textbf{p}^\ast_1(\hbm_{0})+\textbf{p}^\ast_2(\hbm_{0})&
				0&2\textbf{p}^\ast_1(\hbm_{0})&0&0&\cdots
			\end{bmatrix}\\
			-2^{2}\left[\left[\left[ \bf{G}^{1,2},\bf{G}^{2,2} \right],\bf{G}^{1,3}\right],\bf{G}^{1,3}\right]&=& \begin{bmatrix}
				0&2\textbf{p}^\ast_2(\hbm_{0})+\textbf{p}^\ast_1(\hbm_{0})&
				0&2\textbf{p}^\ast_2(\hbm_{0})&0&0&\cdots
			\end{bmatrix}\\
			-2^{2}\left[\left[\left[ \bf{G}^{1,2},\bf{G}^{2,2} \right],\bf{G}^{1,1}\right],\bf{G}^{1,1}\right]&=& \begin{bmatrix}
				0&2\textbf{p}^\ast_2(\hbm_{0})+\textbf{p}^\ast_3(\hbm_{0})&
				0&2\textbf{p}^\ast_2(\hbm_{0})&0&0&\cdots
			\end{bmatrix}\\
			-2^{2}\left[\left[\left[ \bf{G}^{1,3},\bf{G}^{2,3} \right],\bf{G}^{1,1}\right],\bf{G}^{1,1}\right]&=& \begin{bmatrix}
				0&2\textbf{p}^\ast_3(\hbm_{0})+\textbf{p}^\ast_2(\hbm_{0})&
				0&2\textbf{p}^\ast_3(\hbm_{0})&0&0&\cdots
			\end{bmatrix}
		\end{eqnarray*}
	As before, 	observe that $\text{span}\{\textbf{g}^\ast_k,\textbf{p}^\ast_k:k\in\{1,2,3\} \}=\RR^3$ for $\textbf{x}\neq 0$. Therefore, we have
		$$\begin{bmatrix}
			0&
			\textbf{x}&0&0&0&\cdots
		\end{bmatrix}^\top ,\begin{bmatrix}
		0&0&0&
		\textbf{y}&0&&\cdots
		\end{bmatrix}^\top \in \mbox{Lie}\left(\left\{ {\bf{G}}^{1,1},{\bf{G}}^{1,2},{\bf{G}}^{1,3},{\bf{G}}^{2,1},{\bf{G}}^{2,2},{\bf{G}}^{2,3}\right\}\right)$$
		for all $\textbf{x},\textbf{y}\in\RR^3$. Here, one can see, starting with a vector
		$(0,0,*,0,0\ldots)^T$, through the Lie algebra
		the vector
		$$
		(0,*,0,*,0,0,\ldots)^T
		$$
		is generated.  In the next round
		$$
		(*,0,*,0,*,0,\ldots)^T
		$$
		is generated. In this way, the control also propagates through all modes. This completes the proof. 
	\end{stepp}

\end{proof}

%\section{Application to the heat flow on the sphere}

\section{The heat flow on the sphere and controlled low-mode forcing}\label{sec:heatflow_intro}

A prototypical example of an evolution on a manifold is the \emph{harmonic map heat flow on the sphere}.
The harmonic map heat flow is a gradient flow of the Dirichlet energy
\[
    \mathcal E(u)=\frac12\int_M |\nabla u|^2\,dS
\]
for maps \(u:M\to N\), where $M$ and $N$ are two Riemannian manifolds.  In the case where the
target is the sphere, \(u(t,x)\in\mathbb S^2\), it describes the relaxation of an
orientation field, for instance a magnetisation or director field, under the
constraint \(|u|=1\).  The flow was introduced in the fundamental work of
Eells and Sampson \cite{EellsSampson1964} and has since become a central model in geometric analysis;
important later contributions include the work of \textsc{Struwe} \cite{Struwe1985} and of
\textsc{Chang--Ding--Ye} \cite{ChangDingYe1992} on regularity and finite-time singularities.

The harmonic map heat flow  is related to heat transport on an idealised planetary surface:
ordinary scalar heat diffusion on the Earth is governed by the Laplace--Beltrami
operator on \(\mathbb S^2\), while harmonic map heat flow adds the constraint
that the unknown takes values on the  \(\mathbb S^2\).  Such
spherical heat equations and Laplace--Beltrami eigenfunction methods also appear
in climate modelling and data analysis on the globe, see \cite{DelSoleTippett2015}.

In extrinsic coordinates (viewing $\mathbb S^2\subset\mathbb R^3$), the corresponding constrained
heat flow is obtained by projecting the Euclidean Laplacian $\partial_{xx}\M$ onto the tangent space
$T_{\M}\mathbb S^2=\{\xi\in\mathbb R^3:\xi\cdot \M=0\}$, i.e.
\begin{equation}\label{eq:harmmap_proj}
	\frac d {dt} \M(t) = P_{\M(t)}\,\M_{xx}(t),
	\qquad
	P_{\M(t)}:=I-\M(t)\otimes \M(t).
\end{equation}
Using the identity $\M\cdot \M\equiv 1 $, we know that $ \M\cdot \M_{xx}=-|\M_x|^2$, and one obtains the familiar
equivalent form
\begin{equation}\label{eq:harmmap_extrinsic}
\frac d{dt}	\M(t) = \M_{xx}(t) + |\M_x(t)|^2\,\M(t).
\end{equation}
The Neumann condition is the natural boundary condition associated with the energy $\mathcal E$.
For more details, see the book by Struwe \cite[Section 6]{Struwe08}.

The solution $\M$ of \eqref{eq:harmmap_proj}--\eqref{eq:harmmap_extrinsic} enjoy two key structural features:
\begin{itemize}
	\item \emph{Constraint preservation.}
	The right-hand side is pointwise tangent to the sphere, hence
	\[
	\partial_t |\M(t)|^2 = 2\M(t)\cdot \M_t(t) = 2\M(t)\cdot P_{\M(t)}\M_{xx}(t)=0,
	\]
	so $|\M(t,x)|=1$, $x\in[0,1]$, is preserved along the flow.
	
	\item \emph{Energy dissipation.}
	Formally (and rigorously under standard regularity assumptions),
	\[
	\frac{d}{dt}\mathcal E(\M(t))
	= -\int_0^1|\M_t(t,x)|^2\,dx \le 0,
	\]
	where boundary terms vanish due to $\M_x(t,0)=\M_x(t,1)=0$.
\end{itemize}
The harmonic map heat flow was introduced in the foundational work of Eells--Sampson
and has been studied extensively; see, for instance,
\cite{EellsSampson1964,Struwe1985,ChenStruwe1989}.  For further discussion, including the question of uniqueness when solutions are not very regular, we refer to \textsc{Coron} \cite{Coron1990}.

\medskip

		% --- Notation
	\newcommand{\Sph}{\mathbb{S}}
	\newcommand{\id}{\mathrm{Id}}
	\newcommand{\proj}[1]{\Pi_{#1}} % tangent projection

Let us turn our attention the existence and uniqueness results.
In the first step we introduce the 
required function spaces and operators.
%%%%% REVIEW [M3]: DOMAIN INCONSISTENCY (major): Section 3 builds the function spaces on (0,2*pi): L^2(0,2pi), H^1(0,2pi;S^2), D(A) with M_x(2pi)=0, etc. But the rest of the paper -- including the controlled equation (3.6) with M_x(t,1)=0 and the control ansatz (3.5) on x in [0,1] -- uses (0,1). The eigenvalues -4pi^2 k^2 also correspond to (0,1). Replace (0,2pi) by (0,1) throughout this subsection.
Let us denote the space of Lebesgue measurable real-valued square integrable
functions defined on $(0,1)$ by $\mathrm{L}^2(0,1)$ and
$$\mathbb{L}^2:= (\mathrm{L}^2(0,))^3=\mathrm{L}^2(0,1;\mathbb{R}^3).$$
Similarly, we denote
$$
\mathrm{H}^1(0,1):=\Big\{u\in L^2(0,1)| \,\, u'\in L^2(0,1) \Big\}\quad \mbox{and}\quad
\mathbb{H}^1:=(\mathrm{H}^1(0,1))^3=\mathrm{H}^1(0,1; \mathbb R^3).
$$
We denote the Sobolev space $\mathbb{H}^1(0,1;\mathbb{S}^2)$ by
\begin{equation}\label{eqn-H^1(D,S^2)}
\mathbb{H}^1(0,1;\mathbb{S}^2):=\Bigl\{\M \in \mathbb{H}^1 \mbox{ such that }|\M(x)|_{\R3}=1  \mbox{ for a.e. } x \in (0,1)  \Bigr\}.\end{equation}
In particular, $\mathbb{H}^1(\mathcal{O};\mathbb{S}^2)$ is the set of equivalence classes of all functions belonging to  the Sobolev space
$\mathbb{H}^1$ whose values are in the sphere.
We define the Laplacian with the Neumann boundary conditions by
\begin{equation}
\label{op.n} \left\{
\begin{array}{ll}
D(\mathcal{A}) &:= \{ \M \in \mathbb{H}^2(0,1;\mathbb{R}^3):\M_x(0)=\M_x(1)=0 \},\cr
\mathcal{A}(\M)&:=- \M_{xx}, \quad \M\in D(\mathcal{A}).
\end{array}
\right.
\end{equation}
Let us consider the Gelfand triple
$$
D(\mathcal{A})\subset \Hh \subset \L\cong (\L)' \subset
(\Hh)' \subset (D(\mathcal{A}))',
$$
where  $(\mathbb{H}^1)'$ is the dual space of $\mathbb{H}^1$ with duality pairing
$\langle \cdot ,\cdot \rangle :=_{(\Hh)'}\langle\cdot ,\cdot \rangle_{\Hh}.$
%or
%
 We note that $\A$ is self-adjoint and non-negative  operator in $\mathbb{L}^2.$ Define $\rA_1:=I+\mathcal{A}.$ We note that $\rV:=Dom(\rA_1^{1/2})$ when endowed with the graph norm coincides with $\mathbb{H}^1.$ %{\color{red}Also the operator $\rA_1^{-1}$ is compact.}

Let
$
    u_0:(0,1)\to \mathbb S^2
$
be a sufficiently smooth initial datum, for instance
\(u_0\in H^1((0,1);\mathbb S^2)\), or \(u_0\in C^{2+\alpha}((0,1);\mathbb S^2)\)
for the classical solution. 
In one space dimension the problem is subcritical.  In particular, the energy
controls the full \(H^1\)-norm of the map, and the parabolic smoothing prevents
the concentration phenomena which occur for harmonic map heat flow from
two-dimensional domains.  Consequently, for smooth initial data satisfying the
boundary conditions, there exists a unique global smooth solution
\[
%%%%% REVIEW [W4]: 'I' is undefined here (domain symbol). Use (0,1) (or O) consistently.
    u\in C^\infty\bigl((0,\infty)\times I;\mathbb S^2\bigr),
\]
and the constraint $
    |u(t,x)|=1$
 is preserved for all $t\ge0$, $x\in (0,1)$.
Moreover, the energy identity holds and one has the a priori estimate
\[
  \mathcal   E(u(t))
    +
%%%%% REVIEW [W4]: Upper limit 'L' in the integral is undefined (domain length). Use 1 (or 2pi), consistent with the chosen domain -- see the domain note above.
    \int_0^t\int_0^L |\partial_s u(s,x)|^2\,dx\,ds
    =
   \mathcal  E(u_0).
\]
We refer to Struwe's treatment of the harmonic map heat flow in
\cite{Struwe08}, where the general variational method is
developed and where the critical role of the two-dimensional case is explained.
The original global weak existence and partial regularity theory for surfaces is
due to \textsc{Struwe} \cite{Struwe1985}.  Since our spatial domain is
one-dimensional, the problem is below the critical dimension, and the standard
parabolic theory yields global existence and uniqueness of smooth solutions.

\medskip

\subsection*{The control on Fourier modes:}
%%%%% REVIEW [W3]: Word-order typo: 'harmonic heat map flow' -> 'harmonic map heat flow'.
To apply our results from Theorem \ref{app.cont1} and Theorem \ref{L2.app.cont2}  to the harmonic heat map flow, we approximate the solution of the harmonic heat flow %our solution of the \eqref{..} 
by Fourier modes expansion on the sphere.  %To be more precise,
%let us consider the affine control system 
%\begin{equation}\label{eq:harmmap_control}
%\frac d{dt}	\M(t) = \M_{xx}(t) + |\M_x(t)|^2\,\M(t).
%\end{equation}
%
A natural finite-dimensional control ansatz is\footnote{Compare identity \eqref{control.1}.}
\begin{align}   \label{eq:control_lowmodes}
	\v(t,x)=\sum_{(k,j)\in \mathcal{K}}v_k^j(t)\varphi_k(x) \be_j,\quad t\in [0,T],\,\,x\in [0,1],
\end{align}
%\begin{equation}\label{eq:finite-mode-control}
%	u(t,x)=\sum_{(\ell,m)\in\mathcal{S}} \alpha_{\ell m}(t)\,Y_{\ell}^m(x)\,q_{\ell m},
%\end{equation}
where $\mathcal{K}$ is a \emph{finite} index set, $\be_j\in\RR^3$ are fixed directions, and the scalar amplitudes
$v_k^j(t)$, $(k,j)\in \mathcal{K}\subset \mathbb N_0\times\{1,2,3\}$ are the control inputs.
To couple the geometric heat flow with a control acting only on finitely many Fourier modes,
we add a control term that is tangent to the sphere. A convenient choice consistent with the
%%%%% REVIEW [W5]: Sentence fragment: 'A convenient choice consistent with the cross-product structure.' has no main verb -- complete or merge with the next sentence.
cross-product structure.
To be more precise, the harmonic heat flow with control is given by 
\begin{equation}\label{eq:controlled_harmmap}
	\frac d{dt}	\M(t)
	= \M_{xx}(t)+|\M_x(t)|^2\M(t)+ \M(t)\times \bv(t,x),
	\qquad
	\M_x(t,0)=\M_x(t,1)=0.
\end{equation}
The evolution remains on the sphere because the forcing in \eqref{eq:controlled_harmmap} is projected tangentially: $\M\times\bv(t,\cdot)\in T_\M\mathbb{S}^2$, since $\M\cdot(\M\times \bv)=0$ implies that the control does not violate the constraint.
%We assume the control field is by a finite field,
%\begin{equation}\label{eq:control_lowmodes}
%	\bv(t,x)=\sum_{(k,j)\in\mathcal K} v_k^j(t)\,\varphi_k(x)\,e_j,
%	\qquad \varphi_k(x)=\cos(1 kx),
%\end{equation}
%with a finite index set $\mathcal K\subset \mathbb N_0\times\{1,2,3\}$.

\medskip

%	Unlike the linear heat equation, there is \emph{no single, universally applicable “null controllability for any $T>0$” theorem}
%currently known for \eqref{eq:controlled_harmmap} on spheres in full generality: the constraint $|d|=1$, topological obstructions,
%and possible singularity formation (depending on the dimension/energy regime) make the global controllability theory
%substantially more challenging.

	Current rigorous controllability results for the harmonic map heat flow (``harmonic heat flow'') given in \eqref{eq:controlled_harmmap} are much less universal than for the linear heat equation.
%%%%% REVIEW [W9/C3]: Awkward double 'is': '...variants ..., is given by Liu'. Also 'CoronXiang2024' (next line) is an undefined citation key -> prints '[?, 14]'.
A baseline small-time controllability theorem to constant states is available for externally controlled variants under geometric restrictions
(e.g.\ hemisphere condition), is given by \textsc{Liu} \cite{Liu2018}. The next controllability/stabilization results for the heat flow of the harmonic map are given by  \textsc{Coron and Xiang} in \cite{CoronXiang2025}.
% for the circle-to-sphere setting, including small-time local controllability and global controllability to geodesics \cite{CoronXiang2024,CoronXiang2025}.
%=========================================================
%
%	
	For the \emph{harmonic map heat flow} \eqref{eq:controlled_harmmap}, rigorous controllability results
seem not to appear in the current literature.
	Nevertheless, finite-dimensional actuation is highly natural from both a modeling and numerical viewpoint.

	%--------------------------
	% Bibliography
	%--------------------------

Equation~\eqref{eq:controlled_harmmap} defines a control-affine system on an infinite-dimensional
state space (the Fourier coefficients), with \emph{drift} given by the geometric heat operator and
controlled vector fields induced by the finitely many modes in~\eqref{eq:control_lowmodes}.
A particularly important observation is that every constant map $\M(t,x)\equiv \bar \M \in\mathbb S^2$
is an equilibrium of the drift, since
\[
\bar \M_{xx}+|\bar \M_x|^2\,\bar \M = 0.
\]
Hence, \emph{at such constant states} the local Lie-algebraic computations for the controlled vector
fields coincide with the driftless case treated above. In this sense, adding the heat-flow drift
does not change the \emph{instantaneous} bracket-generated directions at constant equilibria.

If \(\M_0\) is not constant, additional directions may be generated.
Thus the accessible distribution can only remain the same or become larger.
In particular, for a reachability result, this does not create an obstruction:
if the constant-\(\M_0\) case already generates all tangent directions,
then the non-constant case is at least as favourable.

This observation suggests a natural link between the reachability analysis carried out in Fourier space and the controlled harmonic map heat flow, motivating a transfer of the former to the latter.
Applying Theorem~\ref{app.cont1} and Theorem~\ref{L2.app.cont2} yields the following two corollaries.
\begin{cor}\label{cor:hf1}
%\label{app.cont1}
	Consider the system~\eqref{eq:controlled_harmmap} with initial condition $\M_0\in\mathbb{H}^1(0,1;\mathbb{S}^2)$, and let the set of controlled forcing modes be
	$$
	\mathcal{K} = \{(0,1),\, (0,2),\, (1,1)\}.
	$$
	Then,
	\begin{itemize}
			\item if $\M_0$ is constant, then not all states are reachable,  in particular, there exists a point $\M_1$, such that there exists no control steering $\M_0$ to $\M_1$ in any time $T$.
			%system is not globally controllable.
			\item if $\M_0$ is non-constant, all states are reachable. In particular, for any finite set $\mathcal{G}\subset \{1,2,3\}\times \mathbb{N}$, 
			any point belonging to $H_\mathcal{G}$ is reachable.
			%the system is globally controllable.
	\end{itemize}
\end{cor}

\begin{cor}\label{hf2}
	Consider the system~\eqref{eq:controlled_harmmap} with initial condition $\M_0\in\mathbb{H}^1(0,1;\mathbb{S}^2)$, and let the set of controlled forcing modes be
	$\mathcal{K} = \{(1,1), (1,2), (1,3),(2,1), (2,2), (2,3)\}.$
	Then, the system~\eqref{eq:controlled_harmmap} is globally controllable.
\end{cor}

The proofs of Corollary~\ref{cor:hf1} and  Corollary~\ref{hf2} rely on showing that the drift term does not alter the relevant dynamics. More precisely, we analyze the action of the drift and show that it affects the dynamics only within a cell, not outside it.
Hence, we prove both corollaries by the same argument: we analyze the action of the drift term and show that, for the purposes of our reachability analysis, the situation will not get worse  with the  drift. In particular, starting at the equilibrium the drift part would not change anything. 
However, the only difference arises when using perturbations or controls acting on higher modes. In that case, the coefficient in front of the control must be chosen larger, since the control has to work against the drift term.

\begin{proof}
	Given the heat flow in Fourier coefficients, i.e., system \eqref{hf_fourier},
we start with analysing the drift term ${\bf F}_0$ given in \eqref{bfF0}.
%Herefore, let us define the mapping
%\[
%\bff_0(\lambda,\bx_1,\bx_2)= \lambda\, ( \bx_1\times \bx_2 ) =\begin{pmatrix} x_1^2x_2^3-x_1^3x_2^2 \\ x_1^3x_2^1-x_1^1x_2^3\\x_1^1x_2^2-x_1^2x_2^1\end{pmatrix}
%\]
Now, we can write 
$$
{\bf F}_0({\bf \mathfrak{m}})= \begin{pmatrix} \sum_{k+n=0\atop |k-n|=0}
	\lambda_k\,\bm_n\times \bm_k
	\\
	\sum_{k+n=1\atop |k-n|=1} 	\lambda_k\,\bm_n\times \bm_k
	\\
	\vdots\\
	\sum_{k+n=K\atop |k-n|=K} 	\lambda_k\,\bm_n\times \bm_k
	\\
	\vdots 
\end{pmatrix}.
$$
Here, many terms in the sum are zero. E.g., if $i=0$, then either  $n=k=0$ or $n=k$. In both cases, the result is zero.
Calculating the Jacobian matrix, we get
\begin{align*}
	\frac {\partial {\bf F}_0({\bf \mathfrak{m}})}{\partial{\bf \mathfrak{m}}} &
	=\begin{pmatrix}
		J_{ {\bf F}_0}^{0,0}& J_{ {\bf F}_0}^{0,1}& \cdots &J_{ {\bf F}_0}^{0,K}
		\\
		J_{ {\bf F}_0}^{1,0}& J_{ {\bf F}_0}^{1,1}& \cdots &J_{ {\bf F}_0}^{1,K}
		\\
		\vdots &\vdots &\vdots & \vdots
		\\
		J_{ {\bf F}_0}^{K,0}& J_{ {\bf F}_0}^{K,1}& \cdots &J_{ {\bf F}_0}^{K,K}
	\\	\vdots &
	\vdots &	\vdots &	\vdots 	\end{pmatrix},
	&\, %quad 
	\begin{array}{rcl}
		\mbox{with} & &\bff_0(\lambda,\bx_1,\bx_2)= \lambda\, ( \bx_1\times \bx_2 )\mbox{ and }
		\\
		\\
%%%%% REVIEW [M14]: DUPLICATE-TERM TYPO: J^{i,j}_{F_0} = sum( d f_0/d m_j + d f_0/d m_j ): the two summands are identical. One should differentiate w.r.t. the first slot (m_n) and the other w.r.t. the second slot (m_k).
		J_{ {\bf F}_0}^{i,j}&=&
	\sum_{k+n=i\atop |k-n|=i} \left(
	\frac {\partial_n  \bff_0(\lambda_k,\bm_n,\bm_k)} %{\partial  \bff_0(\lambda_k,\bm_n,\bm_k)}
	{\partial \bm_j} +
	\frac {\partial_k  \bff_0(\lambda_k,\bm_n,\bm_k)}   %{\partial  \bff_0(\lambda_k,\bm_n,\bm_k)}
	{\partial \bm_j} \right).
	\end{array}
\end{align*}
In particular, also catching up the worst case, i.e. $\tilde \bfm=(\be_1^\top ,\bn^\top ,\bn^\top ,\bn^\top ,\ldots)^\top $, we first analyse  $\left[ {\bf F}_0,{\bf G}^{0,1}\right]$,  $\left[ {\bf F}_0,{\bf G}^{0,2}\right]$, and  $\left[ {\bf F}_0,{\bf G}^{1,1}\right]$.
%%%%% REVIEW [M15]: 'if n=k=1, m_n=0 and m_k=g*_1(e_1)=e_3' is wrong: g*_1(e_1)=0 (it is g*_2(e_1)=e_3). Also for the worst-case state m_1=0. And 'for m=k=1' should be 'n=k=1'.
Here, many terms in the sum are zero. E.g., if $i=0$,  we know there are only entries if $n=k=0$ and $n=k$. In the first case,  since the eigenvalue $\lambda_0$ is zero, we obtain for $n=k=0$
$$\frac {\partial  \bff_0(\lambda_k,\bm_n,\bm_k)}{\partial \bm_n}\Big|_{\bm_n=\bm_k=\be_1}=0.
$$
 In the second case, if $n=k=1$, $\bm_n=\bn$ and $\bm_k=g^\ast_1(\be_1)=\be_3$, we get for $m=k=1$
$$\frac {\partial  \bff_0(\lambda_1,\bm_n,\bm_k)}{\partial \bm_n}\Big|_{\bm_n=\bm_k=\tilde {\bm}}=\tilde{\bm}.
$$
%\begin{align}\label{def_g_all}
%	{\bf{G}}^{0,1}(\bfm)&= \begin{bmatrix}
%		\frac{1}{\sqrt{2}}  \textbf{g}^\ast_{1} (\bm_0)\\
%		\frac{1}{\sqrt{2}}  \textbf{g}^\ast_{1} (\bm_1)\\
%		\frac{1}{\sqrt{2}}  \textbf{g}^\ast_{1} (\bm_2)\\
%		\vdots
%	\end{bmatrix},{\bf{G}}^{0,2}(\bfm)= \begin{bmatrix}
%		\frac{1}{\sqrt{2}}  \textbf{g}^\ast_{2} (\bm_0)\\
%		\frac{1}{\sqrt{2}}  \textbf{g}^\ast_{2} (\bm_1)\\
%		\frac{1}{\sqrt{2}}  \textbf{g}^\ast_{2} (\bm_2)\\
%		\vdots
%	\end{bmatrix},{\bf{G}}^{1,1}(\bfm)=\begin{bmatrix}
%		\frac{1}{\sqrt{2}} \textbf{g}^\ast_{1} (\bm_1)\\
%		\frac{1}{\sqrt{2}} \textbf{g}^\ast_{1} (\bm_{0})+\frac{1}{\sqrt{2}} \textbf{g}^\ast_{1} (\bm_{2})\\
%		\frac{1}{\sqrt{2}} \textbf{g}^\ast_{1} (\bm_{1})+\frac{1}{\sqrt{2}} \textbf{g}^\ast_{1} (\bm_{3})\\
%		\vdots
%	\end{bmatrix}.
%\end{align}
Evaluating the Lie bracket for the first component, we obtain
\DEQSZ
\lqq{ \lk[{\bf F}_0,{\bf G}^{0,1}\rk]_0(\tilde \bfm)}
\\
& =2\lk( \frac {\partial\bff_0(\lambda,\bx_1,\bx_2)} {\partial \bx_1}\Big|_{\lambda=\lambda_0,\atop \bx_1=\tilde \bm_0,\bx_2=\tilde \bm_0}+ 
\frac {\partial\bff_0(\lambda,\bx_1,\bx_2)} {\partial \bx_2}\Big|_{\lambda=\lambda_0,\atop \bx_1=\tilde \bm_0,\bx_2=\tilde \bm_0}\rk)
\left[ \left( {\bf G}^{0,1}\right)_0\right]
\\
&{}+2\lk( \frac {\partial\bff_0(\lambda,\bx_1,\bx_2)} {\partial \bx_1}\Big|_{\lambda=\lambda_1,\atop \bx_1=\tilde \bm_1,\bx_2=\tilde \bm_1}+ 
\frac {\partial\bff_0(\lambda,\bx_1,\bx_2)} {\partial \bx_2}\Big|_{\lambda=\lambda_1,\atop \bx_1=\tilde \bm_1,\bx_2=\tilde \bm_1}\rk)
\left[ \left( {\bf G}^{0,1}\right)_1\right]
%
%	& =2\lk( \bff_0(\lambda,\bx_1,\bx_2)\Big|_{\lambda=\lambda_0,\atop \bx_1=\tilde \bm_0,\bx_2=\tilde \bm_0}
%\left[ \left( {\bf G}^{0,1}\right)_0\right]+	\bff_0(\lambda,\bx_1,\bx_2)\Big|_{\lambda=\lambda_1,\atop 
	%
%\left[ \left( {\bf G}^{0,1}\right)_1\right] \rk)%	\bfm\mapsto (g^\ast_1(\bm_{i-1}))_{i=1}^\infty \rk]}
\\
&=
2\lk(\lambda_0 \, \be_1\times  \textbf{g}^\ast_{1}(\be_1)  +	\lambda_1\, \bn \times  \textbf{g}^\ast_{1}(\tilde \bm_1)\rk) =0.
%\frac {\partial  \bff_0(\lambda_k,\bm_n,\bm_k)}{\partial \bm_n}
%
\EEQSZ
Evaluating the Lie bracket for the second component, we obtain
\DEQS
\lqq{ \lk[{\bf F}_0,{\bf G}^{0,1}\rk]_1(\tilde \bfm)}
\\
& =& \sum_{k=0}^\infty \lk(  \frac {\partial \bff_0(\lambda,\bx_1,\bx_2)}{\partial \bx_1 }
\Big|_{\lambda=\lambda_k,\atop \bx_1=\tilde \bm_k,\bx_2=\tilde \bm_{k+1}} 
\left[ \left( {\bf G}^{0,1}\right)_k\right]
+ \frac {\partial \bff_0(\lambda,\bx_1,\bx_2)}{\partial \bx_2 }
\Big|_{\lambda=\lambda_k,\atop \bx_1=\tilde \bm_k,\bx_2=\tilde \bm_{k+1}} 
\left[ \left( {\bf G}^{0,1}\right)_{k+1}\right]
\rk.
\\
&&{}\lk. + \frac {\partial
	\bff_0(\lambda,\bx_1,\bx_2)}{\partial \bx_1 }
	\Big|_{\lambda=\lambda_{k+1},\atop \bx_1=\tilde \bm_{k+1},\bx_2=\tilde \bm_{k}}
	\left[ \left( {\bf G}^{0,1}\right)_{k+1}\right]    %\right)
%+	\bff_0(\lambda,\bx_1,\bx_2)\Big|_{\lambda=\lambda_1,\atop \bx_1=\tilde \bm_1,\bx_2=\tilde \bm_1}\rk)
+\frac {\partial \bff_0(\lambda,\bx_1,\bx_2)}{\partial \bx_2 }
\Big|_{\lambda=\lambda_{k+1},\atop \bx_1=\tilde \bm_{k+1},\bx_2=\tilde \bm_{k}}
%{\lambda=\lambda_k,\atop \bx_1=\tilde \bm_k,\bx_2=\tilde \bm_{k+1}} 
\left[ \left( {\bf G}^{0,1}\right)_k\right]\rk)
\\
&=&{}\sum_{k=0}^\infty \lk(\lambda_k\, \tilde \bm_{k+1}\times  \textbf{g}^\ast_{1} (\tilde \bm_k)
+ \lambda_k \, \textbf{g}^\ast_{1} (\bm_{k+1})\times \tilde \bm_{k}+\lambda_{k+1} \, \tilde \bm_k\times 
%%%%% REVIEW [M16]: Missing cross product: 'lambda_{k+1} m_{k+1} g*_1(m_k)' should be 'lambda_{k+1} m_{k+1} x g*_1(m_k)'. (Same omission recurs ~line 2798.)
\textbf{g}^\ast_{1} (\tilde \bm_{k+1})+ \lambda_{k+1} \,\tilde \bm_{k+1} \textbf{g}^\ast_{1} (\tilde \bm_{k})
\rk).
%\\
%& =& \sum_{k=0}^\infty \lk(
%\lambda _k (\tilde  \bm_k-\tilde  \bm_{k+1})+ \lambda _{k+1}(\tilde  \bm_{k+1}-\tilde \bm_k)
%\right)
%\times \textbf{g}^\ast_{1}(\tilde \bm_1)\rk)
%+	\bff_0(\lambda,\bx_1,\bx_2)\Big|_{\lambda=\lambda_1,\atop \bx_1=\tilde \bm_1,\bx_2=\tilde \bm_1}\rk)
%
%\left[ \left( {\bf G}^{0,1}\right)_1\right] %	\bfm\mapsto (g^\ast_1(\bm_{i-1}))_{i=1}^\infty \rk]}
\EEQS
%%%%% REVIEW [M17]: Incomplete sentence: 'Hence [F_0,G^{0,1}]_1(m).' is missing the conclusion '= 0'.
Let us remind, $\tilde \bm_k=\bn$, for $k\ge1$. Hence $\lk[{\bf F}_0,{\bf G}^{0,1}\rk]_1(\tilde \bfm)$.
Similarly, $\lk[{\bf F}_0,{\bf G}^{0,1}\rk]_k(\tilde \bfm)=0$, and $\lk[{\bf F}_0,{\bf G}^{0,2}\rk]_k(\tilde \bfm)=0$ for all $k\ge 0$. Finally, 
\DEQSZ
\lqq{ \lk[{\bf F}_0,{\bf G}^{1,1}\rk]_0(\tilde \bfm)}
\nonumber\\
& =2\lk( \frac {\partial\bff_0(\lambda,\bx_1,\bx_2)} {\partial \bx_1}\Big|_{\lambda=\lambda_0,\atop \bx_1=\tilde \bm_0,\bx_2=\tilde \bm_0}+ 
\frac {\partial\bff_0(\lambda,\bx_1,\bx_2)} {\partial \bx_2}\Big|_{\lambda=\lambda_0,\atop \bx_1=\tilde \bm_0,\bx_2=\tilde \bm_0}\rk)
\left[ \left( {\bf G}^{1,1}\right)_0\right]
\nonumber\\
&{}+2\lk( \frac {\partial\bff_0(\lambda,\bx_1,\bx_2)} {\partial \bx_1}\Big|_{\lambda=\lambda_1,\atop \bx_1=\tilde \bm_1,\bx_2=\tilde \bm_1}+ 
\frac {\partial\bff_0(\lambda,\bx_1,\bx_2)} {\partial \bx_2}\Big|_{\lambda=\lambda_1,\atop \bx_1=\tilde \bm_1,\bx_2=\tilde \bm_1}\rk)
\left[ \left( {\bf G}^{1,1}\right)_1\right]
=0,
\EEQSZ
and
\DEQS
\lqq{ \lk[{\bf F}_0,{\bf G}^{1,1}\rk]_1(\tilde \bfm)}
\\
& =& \sum_{k=0}^\infty \lk(  \frac {\partial \bff_0(\lambda,\bx_1,\bx_2)}{\partial \bx_1 }
\Big|_{\lambda=\lambda_k,\atop \bx_1=\tilde \bm_k,\bx_2=\tilde \bm_{k+1}} 
\left[ \left( {\bf G}^{1,1}\right)_k\right]
+ \frac {\partial \bff_0(\lambda,\bx_1,\bx_2)}{\partial \bx_2 }
\Big|_{\lambda=\lambda_k,\atop \bx_1=\tilde \bm_k,\bx_2=\tilde \bm_{k+1}} 
\left[ \left( {\bf G}^{1,1}\right)_{k+1}\right]
\rk.
\\
&&{} + \frac {\partial
	\bff_0(\lambda,\bx_1,\bx_2)}{\partial \bx_1 }
\Big|_{\lambda=\lambda_{k+1},\atop \bx_1=\tilde \bm_{k+1},\bx_2=\tilde \bm_{k}}
\left[ \left( {\bf G}^{1,1}\right)_{k+1}\right]    %\right)
%+	\bff_0(\lambda,\bx_1,\bx_2)\Big|_{\lambda=\lambda_1,\atop \bx_1=\tilde \bm_1,\bx_2=\tilde \bm_1}\rk)
+\frac {\partial \bff_0(\lambda,\bx_1,\bx_2)}{\partial \bx_2 }
\Big|_{\lambda=\lambda_{k+1},\atop \bx_1=\tilde \bm_{k+1},\bx_2=\tilde \bm_{k}}
%{\lambda=\lambda_k,\atop \bx_1=\tilde \bm_k,\bx_2=\tilde \bm_{k+1}} 
\left[ \left( {\bf G}^{
	1,1}\right)_k\right]
\\
&=&{}\sum_{k=0}^\infty \lk(\lambda_k\, \tilde \bm_{k+1}\times  \textbf{g}^\ast_{1} (\tilde \bm_k)
+ \lambda_k \, \textbf{g}^\ast_{1} (\bm_{k+1})\times \tilde \bm_{k}+\lambda_{k+1} \, \tilde \bm_k\times 
\textbf{g}^\ast_{1} (\tilde \bm_{k+1})+ \lambda_{k+1} \,\tilde \bm_{k+1} \textbf{g}^\ast_{1} (\tilde \bm_{k})
\rk)
\\
&=& \lambda_1 \tilde \bm_0\times \textbf{g}^\ast_{1} (\tilde \bm_{0})=\lambda_1 \be_1\times g_1(\be_1)=0.
\EEQS
\newcommand{\mfm}{{\bf \mathfrak{m}}}
In this way, calculating the Lie brackets, one can show that the drift ${\bf F}_0$  does not improve the situation. This completes the proof.
	
	\end{proof} 

%\end{document}

\begin{appendix}

\section{Geometric control theory, reachability  and controllability of non-linear systems}\label{app_geo_cont}
In view of \cite{BoSi}, we present some basic facts about the controllability of non-linear finite-dimensional systems. We introduce the concepts of Lie bracket and state a sufficient condition that guarantees the controllability of the system.
For $N,r\in\mathbb{N}$, let us consider the control system in $\mathbb{R}^N,$
\begin{align}\label{controls.1}
\dot{u}(t)& =F(u(t),v(t)),\quad t>0;
\end{align}
\begin{itemize}
\item $u:[0,\infty)\to \mathbb{R}^N$ is the state of the system;
\item $V\subset \mathbb{R}^r$ is the set of control values and $v :[0, \infty)\to V$ is a control;
\item $F$ is a smooth function of its arguments and that $u$ is regular enough in such a way that equation \eqref{controls.1} with the initial condition $u(0) = u_0\in \mathbb{R}^N$ has an
unique solution in $[0,\infty)$.
\end{itemize}
%\fahim{Let us denote the collection of smooth and complete vector fields parametrized by elements in $V\subset \mathbb{R}^r$ by
%$$
%\mathcal{F}_V:=\big\{F(\cdot, v):\mathbb{R}^N\to \mathbb{R}^N ,\,\,  v\in V\big\}.
%$$}
Let $t\in [0,\infty)\mapsto u(t; u_0, v)$ be the solution of \eqref{controls.1} starting from $u_0 \in \mathbb{R}^N$ at $t = 0$ and corresponding to a control function $\fahim{v}.$
We recall the following definitions:
\begin{defi}
Let $u_0\in\mathbb{R}^N$ and $T\geq 0$. The time-$T$ reachable set of the system \eqref{controls.1} from a given initial point $u_0$  is
$$
\mathcal{A}(T, u_0 ) = \big\{u_1 \in \mathbb{R}^N\,| \,\,\exists\, v: [0, T ] \to V,\, u\left(T; u_0 , \fahim{v}\right) = u_1 \big\}.
$$
Moreover, the reachable set of the system \eqref{controls.1} from a given initial point $u_0$ is
%\item (Reachable set from $u_0$ )
$$
\mathcal{A}(u_0 )= \cup_{T\in [0,\infty)}\mathcal{A}(T, u_0 ).
$$
\end{defi}
\begin{defi}%[Controllable]
The system \eqref{controls.1} is said to be time-$T$
 controllable if for every $u_0 \in \mathbb{R}^N,\, \mathcal{A}(T,u_0 ) = \mathbb R^N$. Moreover, the system \eqref{controls.1} is said to be controllable if for every $u_0 \in \mathbb{R}^N,\, \mathcal{A}(u_0 ) = \mathbb R^N$.
\end{defi}
\begin{defi}
A vector space $X$ over field $\mathbb{R}$ endowed with an operation $[\cdot,\cdot]: X \times X \to X$
is said to be a Lie algebra if it satisfies the following
\begin{itemize}
\item (bilinear): for every $\lambda_1 , \lambda_2 \in \mathbb{R},$ $f,g,f_1,f_2,g_1,g_2 \in X,$
\begin{align*}
[f, \lambda_1 g_1 + \lambda_2 g_2 ] = \lambda_1 [f, g_1 ] + \lambda_2 [f, g_2 ], \\
[\lambda_1 f_1 + \lambda_2 f_2 , g] = \lambda_1 [f_1 , g] + \lambda_2 [f_2 , g];
\end{align*}
\item (antisymmetric): for every $f,g \in X,$
$[f,g]=-[g,f];
$
\item (Jacobi identity): for every $f,g,h \in X$
$$[f, [g, h]] + [h, [f, g]] + [g, [h, f ]] = 0.$$
\end{itemize}
\end{defi}

Let Vec($\mathbb{R}^N $) be the vector space of all smooth vector fields in $\mathbb{R}^N $. Let us first define the Lie bracket
between two vector fields $f,g\in$Vec($\mathbb{R}^N $), as the vector field defined by
$$
[f, g](x) = Dg(x)f (x) - Df (x)g(x),\quad x\in \mathbb{R}^N.
$$

Here, for a given vector field $f = (f_1 ,\dots, f_N )^t\in \mathbb{R}^N$, $Df$ is the matrix of partial derivatives
of the components of $f$ given by
$$
\begin{pmatrix}
\partial_1 f_1 & \cdots & \partial_N f_1 \\
\cdot & \cdots & \cdot \\
\cdot & \cdots & \cdot \\
\partial_1 f_N & \cdots & \partial_N f_N
\end{pmatrix}.
$$
We note that $(\mbox{Vec}(\mathbb{R}^N), [\cdot, \cdot])$ is a Lie algebra.

%
%\begin{example}
%Let  $f(x)=Ax,\,g(x)=Bx,\,x\in\mathbb{R}^N$  for $A,B\in \mathbb{R}^{N\times N}$.
%Then, $f,g \in \mbox{Vec}(\mathbb{R}^N)$ and $[f,g](x)=ABx-BAx$, $x\in\mathbb{R}^N$.
%\end{example}

\begin{defi}
Let $\mathcal{F}$ be a family of vector fields in $\mbox{Vec}(\mathbb{R}^N)$. We call $\mbox{Lie}(\mathcal{F})$ the smallest
sub-algebra of $\mbox{Vec}(\mathbb{R}^N)$ containing $\mathcal{F}$.
In particular, $\mbox{Lie}(\mathcal{F})$ is the span of all vector fields of $\mathcal{F}$ and of their iterated Lie brackets of any order:
$$
   \mbox{Lie}(\mathcal{F})=\mbox{span}\big\{f_1 , [f_1 , f_2 ], [f_1 , [f_2 , f_3 ]], [f_1 , [f_2 , [f_3 , f_4 ]]], \dots, |\,\, f_1 , f_2 ,\dots\in\mathcal{F}\big\}.
$$
\end{defi}

\begin{lem}[\cite{Sontag1990}, page 143]\label{lem:liecalc}
	Let $\mathcal{C}$ be a family of vector fields in $\mathrm{Vec}(\mathbb{R}^N)$. Define $\mathcal{C}_0 := \mathcal{C}$, and recursively,
	\[
	\mathcal{C}_{k+1} := \{[f, g] : f \in \mathcal{C}_k,\ g \in \mathcal{C}\}, \quad k = 0, 1, 2, \dots,
	\]
	and let $\mathcal{C}_\infty := \bigcup_{k \geq 0} \mathcal{C}_k$. Then $\mathrm{Lie}(\mathcal{C})$ is equal to the linear span of $\mathcal{C}_\infty$.
\end{lem}

{Next, we provide a corollary of Lemma~\ref{lem:liecalc}.

\begin{cor}\label{cor:liecalc}
	Let $\mathcal{B}$ be a family of vector fields in $\mathrm{Vec}(\mathbb{R}^N)$. Define $\mathcal{B}_0 := \mathcal{B}$, and recursively,
	\[
	\mathcal{B}_{k+1} := \{[f, g] : f \in \mathcal{B}_k,\ g \in \mathcal{B} \} \setminus \left( \bigcup_{i = 0}^k \mathcal{B}_i \right), \quad k = 0, 1, 2, \dots,
	\]
	and let $\mathcal{B}_\infty := \bigcup_{k \geq 0} \mathcal{B}_k$. Then $\mathrm{Lie}(\mathcal{B})$ is equal to the linear span of $\mathcal{B}_\infty$.
\end{cor}

\begin{proof}
	Let $\mathcal{C}_k$, for $k = 0, 1, 2, \dots$, be defined as in Lemma~\ref{lem:liecalc}, with $\mathcal{C} = \mathcal{B}$. We aim to show that 
	\[
	\bigcup_{k \geq 0} \mathcal{B}_k = \bigcup_{k \geq 0} \mathcal{C}_k.
	\]
	We first prove that $\mathcal{B}_k \subseteq \mathcal{C}_k$ for all $k \geq 0$ by induction. Clearly, $\mathcal{B}_0 = \mathcal{B} = \mathcal{C}_0$. Next, assume that $\mathcal{B}_i \subseteq \mathcal{C}_i$ for all $0 \leq i \leq k$. Then
	\[
	\mathcal{B}_{k+1}\subseteq \{ [f, g] : f \in \mathcal{B}_k,\ g \in \mathcal{C} \} \subseteq \{ [f, g] : f \in \mathcal{C}_k,\ g \in \mathcal{C} \} = \mathcal{C}_{k+1}.
	\]
	Thus, $\bigcup_{k \geq 0} \mathcal{B}_k \subseteq \bigcup_{k \geq 0} \mathcal{C}_k$.
	Next, we prove the reverse inclusion 
	\[
	\bigcup_{k \geq 0}^N \mathcal{C}_k \subseteq \bigcup_{k \geq 0}^N \mathcal{B}_k
	\]
	for all $N\in\NN_0$ by induction. It is clear that $\mathcal{C}_0 = \mathcal{B}_0$. Now, assume that \[
	\bigcup_{k=0}^l \mathcal{C}_k \subseteq \bigcup_{k=0}^l \mathcal{B}_k
	\]
	for some $l \geq 0$. Let $f \in \mathcal{C}_{l+1}$. Then there exist $g \in \mathcal{C}_l$ and $h \in \mathcal{C}$ such that $f = [g, h]$.
	By the induction hypothesis, $g \in \mathcal{C}_l \subseteq \bigcup_{k=0}^l \mathcal{B}_k$, and since $h \in \mathcal{C} = \mathcal{B}$, we have $f = [g,h] \in \mathcal{B}_{m+1}$ for some $m \leq l$, or possibly $f \in \bigcup_{k=0}^l \mathcal{B}_k$.
	Thus, $f \in \bigcup_{k=0}^{l+1} \mathcal{B}_k$, and we conclude that
	\[
	\bigcup_{k=0}^{l+1} \mathcal{C}_k \subseteq \bigcup_{k=0}^{l+1} \mathcal{B}_k.
	\]
	By induction, this inclusion holds for all $N \in \mathbb{N}_0$. Therefore,
	\[
	\bigcup_{k \geq 0} \mathcal{C}_k = \bigcup_{k \geq 0} \mathcal{B}_k,
	\]
	and hence their linear spans are equal. This completes the proof.
\end{proof}

\begin{cor}\label{cor:liecalc2}
	Let $\{a_i\}_{i=0}^{\infty}$ be a sequence consisting of real numbers not equal to zero. Let $\mathcal{B}$ be a family of vector fields in $\mathrm{Vec}(\mathbb{R}^N)$. Define $\mathcal{B}_0 :=a_0 \mathcal{B}$, and recursively,
	\[
	\mathcal{B}_{k+1} := \{a_{k+1}[f, g] : f \in \mathcal{B}_k,\ g \in \mathcal{B} \} \setminus \left( \bigcup_{i = 0}^k \mathcal{B}_i \right), \quad k = 0, 1, 2, \dots,
	\]
	and let $\mathcal{B}_\infty := \bigcup_{k \geq 0} \mathcal{B}_k$. Then $\mathrm{Lie}(\mathcal{B})$ is equal to the linear span of $\mathcal{C}_\infty$.
\end{cor}
%\todo{Should be known}

\begin{proof}
	It is straightforward to verify by applying Corollary \ref{cor:liecalc}.
\end{proof}}
\begin{defi}
The family $\mathcal{F}$ is said to be Lie bracket generating at a point $x\in\mathbb{R}^N$ if the
dimension of $\mbox{Lie}_x(\mathcal{F}):= \{f (x) |\, f \in \mbox{Lie}(\mathcal{F})\}$ is equal to $N$.
The family $\mathcal{F}$ is said to be Lie bracket generating if this condition is verified for every $x\in\mathbb{R}^N$.
\end{defi}
We note that in general $\mbox{Lie}(\mathcal{F})$ is an infinite-dimensional subspace of $\mbox{Vec}(\mathbb{R}^n)$, while
 $\mbox{Lie}(\mathcal{F})$ is a subspace 
 of $\mathbb{R}^N$.

{\begin{defi}
A family of vector fields $\mathcal{F}$ is said to be symmetric if $f \in\mathcal{F}$ implies  $-f\in \mathcal{F}$.
\end{defi}}
The following definition can be found in \cite{Agra_Sary.1.1,Agra_Sary.2}.

{\begin{defi}
		Let $u_0\in\mathbb{R}^N$ and $T\geq 0$.  Let $\mathcal{F}$ be a family of vector fields in $\mbox{Vec}(\mathbb{R}^N)$. The time-$T$ reachable set of $\mathcal{F}$ from a given initial point $u_0$  is\footnote{ Let $f$ be a vector field in \( \mathrm{Vec}(\mathbb{R}^n)\). Its flow \(e^{t f}\) is the map that assigns to each initial point \(x_0\) the solution at time \(t\) of the Cauchy problem $\dot{x} = f(x)$ with the initial condition $x(0) = x_0$.}
		$$
		\mathcal{A}_{\mathcal{F}}^T( u_0 ) = \big\{u_1 \in \mathbb{R}^N\,| \,\,\exists\, 
		t_i\in \RR^+, X_i\in \mathcal{F}\text{ for }i=1,\cdots,p : \sum_{i=1}^pt_i=T, e^{t_pX_p}\circ \cdots\circ e^{t_1X_1}(u_0)=u_1 \big\}.
		$$
		Moreover, the reachable set of $\mathcal{F}$ from a given initial point $u_0$ is
		%\item (Reachable set from $u_0$ )
		$$
		\mathcal{A}_{\mathcal{F}}(u_0 )= \cup_{T\in [0,\infty)}\mathcal{A}_{\mathcal{F}}^T(u_0 ).
		$$
\end{defi}}

{
		\begin{rem}\label{rem:acces}
			The time-$T$ reachable set \(\mathcal{A}_{\mathcal{F}_V}^T(u_0)\) coincides with the set of points that can be reached at time \(T\) in the control system~\eqref{controls.1} using piecewise constant controls. Moreover, since piecewise constant controls are dense in the set of all bounded measurable controls, it follows that
			\[
%%%%% REVIEW [W2]: Typo: 'the closer of the set A' -> 'the closure of the set A'.
			\overline{\mathcal{A}_{\mathcal{F}_V}^T(u_0)} = \overline{\mathcal{A}(T, u_0)}, \footnote{Let $A\subset \RR^N$. Then, $\overline{A}$ denotes the closure of the set $A$ in $\RN$.}
			\]
			see~\cite{Agra_Sary.1.1,Agra_Sary.2}. Therefore, controllability of a family of vector fields \(\mathcal{F}\) can be defined in complete analogy with the controllability of the control system~\eqref{controls.1} and implies that \(\mathcal{F}\) is controllable if and only if the system~\eqref{controls.1} is controllable.
		\end{rem}}	

\begin{thm}[Chow-Rashevskii] \label{ChowR}
If $\mathcal{F}_V$ is Lie bracket generating and symmetric\footnote{A family of vector fields $\mathcal{F}$ is said to be symmetric if $f \in\mathcal{F}$ implies  $-f\in \mathcal{F}$.},
then for every $u_0\in\mathbb{R}^N$ we have $\mathcal{A}^T_{\mathcal{F}_V}(u_0) =\mathbb{R}^N$
\end{thm}

\begin{defi}[Extension of control system]
A family $\mathcal{F}'$ of real analytic vector fields in $\RN$ is said to be an
extension of $\mathcal{F}$ if $\mathcal{F}\subset \mathcal{F}'$ and for all
$u_0\in \RN$, $\overline{\mathcal{A}^T_{\mathcal{F}}(u_0)}=\overline{\mathcal{A}^T_{\mathcal{F}'}(u_0)}$.
\end{defi}
\begin{defi}[Fixed time extension of control system]
A family $\mathcal{F}'$ of a real analytic vector field in $\RN$ is said to be a fixed time extension of $\mathcal{F}$, if $\mathcal{F}\subset \mathcal{F}'$ and for all $T>0$ and  $u_0\in \RN$, we have  $\overline{\mathcal{A}_{\mathcal{F}}^T(u_0)}=\overline{\mathcal{A}_{\mathcal{F}'}^T(u_0)}$,
\end{defi}
%%One can see that $\mathcal{A}_{\mathcal{F}}(u_0)\subset \mathcal{A}_{\mathcal{F}'}(u_0),$ and for all $T>0,$ $\mathcal{A}_{\mathcal{F}}^T(u_0)\subset \mathcal{A}_{\mathcal{F}'}^T(u_0).$
%\fahim{\begin{lem}[Agrachev-Sarychev, compare \cite{Agra_Sary.3}]
%Let $\mathcal{F}'$ be an extension of $\mathcal{F}.$ If  $\mathcal{F}'$ is controllable (resp. time-$T$ controllable) then $\overline{\mathcal{A}_{\mathcal{F}}(u_0)}=\RN$ (resp. $\overline{\mathcal{A}^T_{\mathcal{F}}(u_0)}=\RN$).
%\end{lem}}

%\subsection{Bracket generating, approximate controllability and global controllability} Let us consider a control affine system \eqref{control.affine.d}.
{
\subsection{Control Affine System}

Let us introduce the definition of 
 control affine systems.
\begin{defi}(compare \cite[Section 13.2.3]{LaValle}).
	Let $N,r\in\Na,\, r\leq N.$ Let $u:[0,T]\to \RN$ be the state of the system, $v_j:[0,T]\to \Rr$ for $j=1,2,\dots,r$ be the controls and $f,g_j$ for $j=1,2,\dots,r$ be
	some smooth vector fields in $\RN$. A non-linear control system of the form
	\begin{align}\label{control.affine.d}
		\dot{u}(t)&=f(u(t))+\sum_{j=1}^r g_j(u(t))v_j(t),\,\,t\in (0,T]\quad \hbox{with} \quad
		u(0)= u_0,
	\end{align}
	It is called a control-affine system or an affine-in-control system.
\end{defi}
Affine means `linear', and non-affine means `nonlinear'. Therefore, a nonlinear system in which the control appears linearly is called a control-affine nonlinear system (or simply control-affine system, where the nonlinearity with respect to the state is already there in the system), whereas a system that has nonlinearities both in the state and in the controls is called a control-non-affine nonlinear system (or simply control-non-affine system).
%The general form of an control-affine system is represented  as :
%$$
%\dot{u}(t) = f(u(t)) + g(u(t))v(t),\,\,t\in (0,T]
%$$
%whereas the general form of a system which is non affine in the controls is given as:
%$$
%\dot{u}(t) = f(u(t)) + g(u(t),v(t)),\,\,t\in (0,T].
%$$

Let us introduce the controlled distribution of a control-affine system.
In our case, we know that the solution will remain on the sphere $\mathbb{S}^2$.
So,	let $M=\mathbb{S}^2$ and let $X_1,\dots,X_m$ be smooth vector fields on $M$, and $T_\M M$ the tangent space at $\M$.
Let us consider the control-affine system given by 
	\[
	\dot{x}(t)=\sum_{i=1}^m u_i(t)\,X_i\big(x(t)\big),
	\]
	where $u(t)=(u_1(t),\dots,u_m(t))$ is the control.
	The associated {controlled distribution} at $\M \in M$, is the  subbundle
	\[
\Delta(\M):=\operatorname{span}\{X_1(\M),\dots,X_m(\M)\}\subset T_\M M,
	\]
	If a drift vector field $X_0$ is present,
	\[
	\dot{x}(t)=X_0(x(t))+\sum_{i=1}^m u_i(t)\,X_i\big(x(t)\big),
	\]
	the controlled distribution is still $\mathcal{D}=\operatorname{span}\{X_1,\dots,X_m\}$,
	while reachability is governed by the Lie algebra generated by $\{X_0,X_1,\dots,X_m\}$.
	
	Denote by $\Lie(X_1,\dots,X_m)$ the Lie algebra generated by $X_1,\dots,X_m$
	(i.e.\ the smallest set of vector fields containing $X_i$ and closed under
	linear combinations with smooth coefficients and Lie brackets).
	The \emph{accessibility distribution} (also called the Lie-saturate) at a point $\M$ is
	\[
	\mathcal{F}(\M):=\operatorname{span}\{Y(\M):\,Y\in \Lie(X_1,\dots,X_m)\}\subset T_\M M,
	\]
	(and analogously with $X_0$ included when there is a drift).
%	\coma{The lemma is not clearly written}
\begin{lem}[\cite{Sontag1990}, page 154]\label{lem:Lieaffine}
Assume we have a control-affine system	
as we have in \eqref{control.affine.d}
defined by  the vector field
% $\mathcal{F}$ be a vector field space given by \eqref{control.affine.d}
 $$
	\mathcal{F}:=\big\{f+u_1g_1+u_2g_2+\cdots+u_rg_r:u_1,\cdots,u_r\in\mathbb{R}\big\}.
	$$
	Then, $$\mbox{Lie}(\mathcal{F})=\mbox{Lie}\left(\left\{ f,g_1,g_2,\cdots,g_r\right\}\right).$$
\end{lem}
}

\begin{defi} [Full Lie rank property] \label{liegen}
	The control affine system given in  \eqref{control.affine.d} is said to have the full Lie rank property (or is Lie bracket generating),
	if for every $u_0\in\RN$ the iterated Lie brackets of the vector fields $f,g_j$ for $j=1,2,\dots, r$ evaluated at $u_0\in\RN$ span the whole space $\RN$.
\end{defi}

%%%%% REVIEW [C4]: Undefined citation key 'Jurd' -> '[?]'. Use 'Jurdjevic' (the defined key).
\begin{defi}[compare \textsc{Jurdjevic} \cite{Jurdjevic}, Chapter 3]
A point $u_0\in\RN$ is normally accessible from another point $u_1\in\RN$ by a family of vector fields $\mathcal{F}$ of $\RN$,
if there exist elements $f_i\in\mathcal{F}$ for $i=1,\dots,p$ ,such that the function $F:\Rr^p\to \RN$ defined by
$$
F(t):=\exp (t_pf_p)\circ \exp (t_{p-1}f_{p-1})\circ \cdots\circ  \exp (t_1f_1)(u_0)(u_1),\quad t=(t_1,t_2,\dots,t_p)\in \Rr^p,
$$
satisfies the following conditions:
\begin{enumerate}[i.]
\item there exists $\hat{t}\in \Rr^p$ such that $F(\hat{t})=u_0$,
\item rank of $DF(\hat{t})=N$.
\end{enumerate}
We shall say that $u_0$ is normally accessible from  $u_1$ in time
$\sum_{i=1}^p\hat{t}_i$.
\end{defi}
If  $u_0\in\RN$ is normally accessible from $u_1\in\RN,$ then by the Constant Rank Theorem (see \cite[Theorem 4.12]{Lee_smooth_mainfolds}), there exists a neighbourhood $U$ of $\hat{t}$ in $\Rr^p$ such that $F(U)$ is a neighbourhood $u_0$ in $\RN.$ Therefore, any point which is normally accessible from $u_0$ at time $T$ belongs to the  interior of an attainable set {$\cup_{s\in [0,T]}\mathcal{A}_{\mathcal{F}}^s(u_0)$.}

%However, in our case no drift exists. That means, the system will only be generated by the controls. In addition, we have an inifinite dimensional system, 
%% of Agrachev-Sarychev \cite[Theorem 5.5]{Agra_Sary.3}, Jurdjevic \cite[Chapter 3, Theorem 1]{Jurd}.
%%
%We say the control system \eqref{control.affine.d} is   $C^\infty$-smooth if all  controlled vector fields $f_1,\dots,f_r$
%are $C^\infty$ as maps $M\to TM$ (i.e.\ all derivatives of all orders exist and are continuous),
%so that all iterated Lie brackets are well-defined and smooth.
%
\begin{thm} (compare \cite[Theorem~5.5]{Agra_Sary.3} and \cite[Chapter~3, Theorem~1]{Jurdjevic}).
%[compare \cite[Theorem 5.5,Agra_Sary.3], or \cite[Chapter 3, Theorem 1,Jurdjevic]
\label{teo:affinesystem}
Let the control system \eqref{control.affine.d} be real-analytic. Then it is accessible if and only if it is bracket-generating.
Moreover, if the control system \eqref{control.affine.d} is $C^\infty$-smooth and is bracket generating, then, it is accessible.
\end{thm}
 For a bracket generating system, approximate controllability implies exact controllability.
 %We have
 \begin{prop}[\textsc{Jurdevic} \cite{Jurdjevic}] \label{cont.1}
 Let the system \eqref{control.affine.d} be bracket generating and
$\overline{\mathcal{A}_{\mathcal{F}}^T(u_0)}=\RN$ {(resp. $\overline{\mathcal{A}_{\mathcal{F}}(u_0)}=\RN$)}. Then $\mathcal{A}_{\mathcal{F}}^T(u_0)=\RN$ {(resp. $\mathcal{A}_{\mathcal{F}}(u_0)=\RN$)}.
\end{prop}

\section{The equations in Fourier modes}\label{sec:fouriermode}

In this appendix we derive the Fourier-mode formulation of the drift-free
controlled system \eqref{M1.N}.  Throughout this section, we consider only the
model
\begin{equation}\label{app:driftfree_model}
\frac{d}{dt}\M(t,x)=\M(t,x)\times \v(t,x),
\end{equation}
where the control is supported on finitely many Fourier modes,
\begin{equation}\label{control.1_app}
\v(t,x)=\sum_{(k,l)\in\mathcal K}v_k^l(t)\varphi_k(x)\be_l,
\qquad t\in[0,T],\quad x\in[0,1].
\end{equation}

Let
$$ \be_1=(1,0,0)^T,\qquad
\be_2=(0,1,0)^T,\qquad
\be_3=(0,0,1)^T. $$
We use the cross-product identities
$$
\be_1\times\be_1=\be_2\times\be_2=\be_3\times\be_3=0, $$
and
$$
\be_1\times\be_2=\be_3,\qquad
\be_2\times\be_3=\be_1,\qquad
\be_3\times\be_1=\be_2,
$$
with the corresponding antisymmetric relations.

We write
\begin{equation}\label{eq:seriesM_app}
\M(t,x)=\sum_{n\in\NN_0}\bm_n(t)\varphi_n(x),
\qquad
\bm_n(t)=(m_n^1(t),m_n^2(t),m_n^3(t))^T.
\end{equation}
With the normalization
\[
\varphi_k(x)=\sqrt{2}\cos(2\pi kx),\qquad k\in\NN_0,
\]
we have
\begin{equation}\label{product_phi_app}
\varphi_n(x)\varphi_k(x)
=
\frac{1}{\sqrt{2}}\varphi_{n+k}(x)
+
\frac{1}{\sqrt{2}}\varphi_{|n-k|}(x).
\end{equation}

Substituting \eqref{eq:seriesM_app} and \eqref{control.1_app} into
\eqref{app:driftfree_model}, we obtain
\begin{align}
\M(t,x)\times \v(t,x)
&=
\left(\sum_{n\in\NN_0}\bm_n(t)\varphi_n(x)\right)
\times
\left(\sum_{(k,l)\in\mathcal K}v_k^l(t)\varphi_k(x)\be_l\right) \nonumber\\
&=
\sum_{(k,l)\in\mathcal K}
\sum_{n\in\NN_0}
v_k^l(t)
\bigl(\bm_n(t)\times\be_l\bigr)
\varphi_n(x)\varphi_k(x) \nonumber\\
&=
\sum_{(k,l)\in\mathcal K}
\sum_{n\in\NN_0}
\frac{v_k^l(t)}{\sqrt{2}}
\bigl(\bm_n(t)\times\be_l\bigr)
\left(\varphi_{n+k}(x)+\varphi_{|n-k|}(x)\right).
\end{align}
Therefore, comparing the coefficient of $\varphi_i$, we get
\begin{equation}\label{my_system_app_corrected}
\frac{d}{dt}\bm_i(t)
=
\sum_{(k,l)\in\mathcal K}
v_k^l(t)
\left[{\bf G}^{k,l}(\bfm(t))\right]_i,
\qquad i\in\NN_0,
\end{equation}
where $\bfm=(\bm_0,\bm_1,\bm_2,\ldots)^T$ and
\begin{equation}\label{Gisgivenby_corrected}
\left[{\bf G}^{k,l}(\bfm)\right]_i
=
\frac{1}{\sqrt{2}}
\sum_{n\in\NN_0}
\bigl(\bm_n\times\be_l\bigr)
\left(
\mathbf 1_{\{n+k=i\}}
+
\mathbf 1_{\{|n-k|=i\}}
\right).
\end{equation}

Equivalently, using the notation
\[
\textbf{g}_l^\ast(\bm)=\bm\times\be_l,
\]
the components of the controlled vector field are
\begin{equation}\label{Gisgivenby_piecewise_corrected}
\left[{\bf G}^{k,l}(\bfm)\right]_i
=
\begin{cases}
\sqrt{2}\,\textbf{g}_l^\ast(\bm_i),
& k=0, \\[0.15cm]
\frac{1}{\sqrt{2}}\textbf{g}_l^\ast(\bm_{i+k})
+
\frac{1}{\sqrt{2}}\textbf{g}_l^\ast(\bm_{i-k}),
& 0<k<i, \\[0.15cm]
\sqrt{2}\,\textbf{g}_l^\ast(\bm_0)
+
\frac{1}{\sqrt{2}}\textbf{g}_l^\ast(\bm_{2k}),
& 0<k=i, \\[0.15cm]
\frac{1}{\sqrt{2}}\textbf{g}_l^\ast(\bm_{i+k})
+
\frac{1}{\sqrt{2}}\textbf{g}_l^\ast(\bm_{k-i}),
& 0<i<k, \\[0.15cm]
\frac{1}{\sqrt{2}}\textbf{g}_l^\ast(\bm_k),
& 0=i<k .
\end{cases}
\end{equation}
Hence the Fourier-mode formulation of the drift-free control system is
\begin{equation}\label{eq:llgeinm_app_corrected}
\dot{\bfm}(t)
=
\sum_{(k,l)\in\mathcal K}
{\bf G}^{k,l}(\bfm(t))\,v_k^l(t),
\qquad t\in[0,T].
\end{equation}

Finally, the elementary vector fields are explicitly given by
\[
\textbf{g}_1^\ast(x)=x\times\be_1
=
\begin{pmatrix}
0\\ x_3\\ -x_2
\end{pmatrix},
\qquad
\textbf{g}_2^\ast(x)=x\times\be_2
=
\begin{pmatrix}
-x_3\\ 0\\ x_1
\end{pmatrix},
\qquad
\textbf{g}_3^\ast(x)=x\times\be_3
=
\begin{pmatrix}
x_2\\ -x_1\\ 0
\end{pmatrix}.
\]

\subsection{The heat flow on $\mathbb{S}^2$ in Fourier modes}
%
%We know that $\left\{\lambda_n:=-4\pi^2 n^2, \phi_n:=\cos(n\cdot)\right\}_{n\in \mathbb{N}_0}$ are the eigen pairs of the Neumann Laplace operator $-\mathcal{A}:=\Delta$ ($\mathcal{A}$ is defined above in \ref{op.n}) in the Hilbert space $\mathbb{L}^2(0,1).$ 
%Therefore, Fourier expansion of solution $\M(t)$ can be written as
%
Let us write the system \eqref{M1.N} in componentwise. First, we have 
\begin{eqnarray*}
	\M_{xx}(t)= \sum_{k\in\NN_0}-4\pi^2k^2\textbf{m}_k(t) \varphi_k.
\end{eqnarray*}
Observe:
%%%%% REVIEW [M4]: F_0 (the 'drift') is derived here from M x M_xx -- a PRECESSION (Schroedinger/LLG) term -- whereas the controlled harmonic map heat flow (3.6) has the PARABOLIC drift M_xx + |M_x|^2 M. The corollary proof then brackets [F_0, G] with this cross-product drift. Verify that the drift analysed matches (3.6).
\begin{eqnarray}
	\nonumber\lefteqn{\M(t)\times \M_{xx}(t)= \left( \sum_{n\in\NN_0}\textbf{m}_n(t) \varphi_n \right)\times \left( \sum_{k\in\NN_0}-4\pi^2k^2\textbf{m}_k(t) \varphi_k\right)}\\
	\label{eq:McM}&=&  \sum_{n\in\NN_0}\sum_{k\in\NN_0}\left(-4\pi^2k^2\right) \varphi_k \varphi_n\textbf{m}_n(t)\times \textbf{m}_k(t).
\end{eqnarray}
From trigonometry identities, we have for $k,n\in\NN_0$
\begin{eqnarray*}
	\varphi_k \varphi_n&=& 2\cos\big(2\pi k x \big) \cos\big(2\pi n x \big)=\cos\big(2\pi (k+n) x \big) +\cos\big(2\pi |k-n| x \big)=\frac{1}{\sqrt{2}}\left(\varphi_{k+n} +\varphi_{|k-n|} \right).
\end{eqnarray*} 
Substituting into \eqref{eq:McM}, we have
\begin{eqnarray}
	\lqq{\M(t)\times \M_{xx}(t)}\nonumber 
	\\
	&=&\nonumber 
\frac{1}{\sqrt{2}}\left(  \sum_{n\in\NN}\sum_{k\in\NN}\left(-4\pi^2k^2\right) \varphi_{n+k} \textbf{m}_n(t)\times \textbf{m}_k(t) +   \sum_{n\in\NN}\sum_{k\in\NN}\left(-4\pi^2k^2\right) \varphi_{|n-k|}\textbf{m}_n(t)\times \textbf{m}_k(t) \right)\\
	\nonumber&&= \sum_{i\in\NN_0}\left[\sum_{\substack{n,k\in\NN_0\\ n+k=i}}\frac{-4\pi^2k^2}{\sqrt{2}}  \textbf{m}_n(t)\times \textbf{m}_k(t) +   \sum_{\substack{n,k\in\NN_0\\ |n-k|=i}}\frac{-4\pi^2k^2}{\sqrt{2}}  \textbf{m}_n(t)\times \textbf{m}_k(t)\right] \varphi_{i}\\
	\label{eq:seriesMM}&&=\sum_{i\in\NN_0}\left[\sum_{\substack{n,k\in\NN_0\\ n+k=i\text{ or }|n-k|=i}}\frac{\lambda_k}{\sqrt{2}}  \textbf{m}_n(t)\times \textbf{m}_k(t) \right] \varphi_{i},
\end{eqnarray}
where $\lambda_k:=-4\pi^2k^2$.
Reformulating the system  \eqref{M1.N} in terms of its Fourier coefficients, we get 
\DEQSZ\label{hf_fourier}
\frac d {dt}\,{ \bfm}(t) &=& {\bf F}_0( \bfm(t)) +\sum_{(k,l)\in \mathcal{K}} {\bf{G}}^{k,l}( \bfm(t))\,  v_k^l(t),\qquad t\in[0,T],\,
\EEQSZ 
where 
%$${\bf F}_0=([F_0]_1,[F_0]_2,\cdots),~{\bf F}_1=([F_1]_1,[F_1]_2,\cdots),{\bf{G}}^{k,l}=\left( [{\bf{G}}^{k,l}]_1,[{\bf{G}}^{k,l}]_2,\cdots\right):\times_{l=0}^\infty \RR^{3 }\to\times_{l=0}^\infty \RR^{3 }.$$
%Here,
\begin{align}\label{bfF0}
%%%%% REVIEW [M5]: (B.12): the drift coefficient mu_1 is undefined/leftover here; in (B.10) the harmonic drift has coefficient 1 (lambda_k/sqrt(2)). Set mu_1=1 or remove.
	[{\bf F}_0( {\bfm}(t))]_i&=\sum_{\substack{n,k\in\NN_0\\ n+k=i\text{ or }|n-k|=i}}  \frac{\mu_1\lambda_k}{\sqrt{2}} \, \bm_n(t) \times \bm_k(t),  % \notag
	\end{align}
and $	\left[{\bf{G}}^{k,l}(\bfm)\right]_i$ is again given by  \eqref{Gisgivenby_piecewise_corrected}. 
%\begin{align*}
%	\left[{\bf{G}}^{k,l}(\bfm)\right]_i=\left\{
%	\begin{array}{ll}
%		\frac{1}{\sqrt{2}} \textbf{g}^\ast_{l} (\bm_i),&k=0,\\
%		\frac{1}{\sqrt{2}} \textbf{g}^\ast_{l} (\bm_{i+k})+\frac{1}{\sqrt{2}} \textbf{g}^\ast_{l} (\bm_{i-k}),&0<k<i,\\
%		\frac{1}{\sqrt{2}} \textbf{g}^\ast_{l} (\bm_{0})+\frac{1}{\sqrt{2}} \textbf{g}^\ast_{l} (\bm_{2k}),&0<k=i,\\
%		\frac{1}{\sqrt{2}} \textbf{g}^\ast_{l} (\bm_{i+k})+\frac{1}{\sqrt{2}} \textbf{g}^\ast_{l} (\bm_{k-i}),&0<i<k,\\
%		\frac{1}{\sqrt{2}} \textbf{g}^\ast_{l} (\bm_k),&0=i<k.
%	\end{array}
%	\right.
%\end{align*}

\section{Linear algebra}

In this section we prove some simple linear-algebraic identities that are needed for the main results.
\begin{lem}\label{lem:al1}
	Let $\mathbf{x} = (x_1, x_2, x_3)^T$ and $\mathbf{y} = (y_1, y_2, y_3)^T$ be linearly independent vectors in $\mathbb{R}^3$. Let $k, l, m \in \{1, 2, 3\}$ be distinct indices and $x_k$ is a non-zero number. Then, at least one of the following two pairs of vectors is linearly independent:
	\[
	(x_k, x_l)^T \text{ and } (y_k, y_l)^T \quad \text{or} \quad (x_k, x_m)^T \text{ and } (y_k, y_m)^T.
	\]
\end{lem}
\begin{proof}
	We proceed by contradiction. Suppose that both pairs of vectors are linearly dependent:
	\[
	(x_k, x_l)^T \text{ and } (y_k, y_l)^T, \quad \text{and} \quad (x_k, x_m)^T \text{ and } (y_k, y_m)^T.
	\]
	Then there exist scalars $\alpha, \beta \in \mathbb{R}$ such that
$
		y_k = \alpha x_k$, 
		$y_l = \alpha x_l,$
	and
$y_k = \beta x_k$, 
$y_m = \beta x_m$.
	Since $x_k \neq 0$, the first equations in each pair imply that $\alpha = \beta$. Substituting, we have
	\[
	y_i = \alpha x_i \quad \text{for } i = k, l, m.
	\]
	Thus, $y_i = \alpha x_i$ for three distinct indices, which implies $\mathbf{y} = \alpha \mathbf{x}$. This contradicts the assumption that $\mathbf{x}$ and $\mathbf{y}$ are linearly independent. Hence, at least one of the two pairs must be linearly independent.
\end{proof}
\begin{prop}\label{lem:al2}
	Let $\mathbf{x}, \mathbf{y} \in \mathbb{R}^3$ be nonzero vectors. If $\mathbf{x}$ and $\mathbf{y}$ are linearly independent, then there exist indices $m, n, l \in \{1, 2, 3\}$ such that the vectors
	\[
	\mathbf{g}^*_{m}(\mathbf{x}), \quad \mathbf{g}^*_{n}(\mathbf{x}), \quad \text{and} \quad \mathbf{g}^*_{l}(\mathbf{y})
	\]
	are linearly independent.
\end{prop}
\begin{proof}
	Let $\mathbf{x} = (x_1, x_2, x_3)^T$ and $\mathbf{y} = (y_1, y_2, y_3)^T$ be linearly independent vectors in $\mathbb{R}^3$. We will consider three cases depending on which coordinate of $\mathbf{x}$ is nonzero.
	
	\begin{description}
		\item[\textbf{Case 1:} $x_1 \neq 0$] 
		Consider the following two linear systems:
		\begin{equation}\label{eq:sys1}
			\alpha_1 \mathbf{g}^*_{2}(\mathbf{x}) + \alpha_2 \mathbf{g}^*_{3}(\mathbf{x}) + \alpha_3 \mathbf{g}^*_{2}(\mathbf{y}) = 0 \tag{$*_1$}
		\end{equation}
		\begin{equation}\label{eq:sys2}
			\beta_1 \mathbf{g}^*_{2}(\mathbf{x}) + \beta_2 \mathbf{g}^*_{3}(\mathbf{x}) + \beta_3 \mathbf{g}^*_{3}(\mathbf{y}) = 0 \tag{$**_1$}
		\end{equation}
		
		From \eqref{eq:sys1}, we obtain the system of equations
		\begin{align*}
			\alpha_1 x_3 + \alpha_3 y_3 = 0, \quad 
			\alpha_2 = 0, \quad 
			\alpha_1 x_1 + \alpha_3 y_1 = 0.
		\end{align*}
		From \eqref{eq:sys2}, we obtain the system of equations
		\begin{align*}
			\beta_2 x_2 + \beta_3 y_2 = 0, \quad
			\beta_2 x_1 + \beta_3 y_1 = 0, \quad 
			\beta_1 = 0.
		\end{align*}
		Since $\mathbf{x}$ and $\mathbf{y}$ are linearly independent and $x_1 \neq 0$, it follows from Lemma~\ref{lem:al1} that either all coefficients $\alpha_i$ are zero or all $\beta_i$ are zero. Therefore, at least one of the following triples of vectors is linearly independent:
		\[
		\mathbf{g}^*_{2}(\mathbf{x}), \quad \mathbf{g}^*_{3}(\mathbf{x}), \quad \mathbf{g}^*_{2}(\mathbf{y}) \quad \text{or} \quad \mathbf{g}^*_{2}(\mathbf{x}), \quad \mathbf{g}^*_{3}(\mathbf{x}), \quad \mathbf{g}^*_{3}(\mathbf{y}).
		\]
		
		\item[\textbf{Case 2:} $x_2 \neq 0$] 
		Consider:
		\begin{equation}\label{eq:sys3}
			\alpha_1 \mathbf{g}^*_{1}(\mathbf{x}) + \alpha_2 \mathbf{g}^*_{3}(\mathbf{x}) + \alpha_3 \mathbf{g}^*_{1}(\mathbf{y}) = 0 \tag{$*_2$}
		\end{equation}
		\begin{equation}\label{eq:sys4}
			\beta_1 \mathbf{g}^*_{1}(\mathbf{x}) + \beta_2 \mathbf{g}^*_{3}(\mathbf{x}) + \beta_3 \mathbf{g}^*_{3}(\mathbf{y}) = 0 \tag{$**_2$}
		\end{equation}
		From \eqref{eq:sys3},  we obtain the system of equations
		\begin{align*}
			\alpha_2 = 0, \quad 
			\alpha_1 x_3 + \alpha_3 y_3 = 0, \quad 
			\alpha_1 x_2 + \alpha_3 y_2 = 0.
		\end{align*}
		From \eqref{eq:sys4}, we obtain the system of equations
		\begin{align*}
			\beta_2 x_2 + \beta_3 y_2 = 0, \quad 
			\beta_2 x_1 + \beta_3 y_1 = 0, \quad 
			\beta_1 = 0.
		\end{align*}
		Again, since $\mathbf{x}$ and $\mathbf{y}$ are linearly independent and $x_2 \neq 0$, Lemma~\ref{lem:al1} implies that either all $\alpha_i = 0$ or all $\beta_i = 0$. Hence, at least one of the following triples is linearly independent:
		\[
		\mathbf{g}^*_{1}(\mathbf{x}), \quad \mathbf{g}^*_{3}(\mathbf{x}), \quad \mathbf{g}^*_{1}(\mathbf{y}) \quad \text{or} \quad \mathbf{g}^*_{1}(\mathbf{x}), \quad \mathbf{g}^*_{3}(\mathbf{x}), \quad \mathbf{g}^*_{3}(\mathbf{y}).
		\]
		
		\item[\textbf{Case 3:} $x_3 \neq 0$] 
		Consider:
		\begin{equation}\label{eq:sys5}
			\alpha_1 \mathbf{g}^*_{1}(\mathbf{x}) + \alpha_2 \mathbf{g}^*_{2}(\mathbf{x}) + \alpha_3 \mathbf{g}^*_{1}(\mathbf{y}) = 0 \tag{$*_3$}
		\end{equation}
		\begin{equation}\label{eq:sys6}
			\beta_1 \mathbf{g}^*_{1}(\mathbf{x}) + \beta_2 \mathbf{g}^*_{2}(\mathbf{x}) + \beta_3 \mathbf{g}^*_{2}(\mathbf{y}) = 0 \tag{$**_3$}
		\end{equation}
		From \eqref{eq:sys5}, we obtain the system of equations
		\begin{align*}
			\alpha_2 = 0, \quad 
			\alpha_1 x_3 + \alpha_3 y_3 = 0, \quad 
			\alpha_1 x_2 + \alpha_3 y_2 = 0.
		\end{align*}
		From \eqref{eq:sys6}, we obtain the system of equations
		\begin{align*}
			\beta_2 x_3 + \beta_3 y_3 = 0, \quad 
			\beta_1 = 0, \quad 
			\beta_2 x_1 + \beta_3 y_1 = 0.
		\end{align*}
		As before, using Lemma~\ref{lem:al1} and the fact that $x_3 \neq 0$, we conclude that one of the following sets is linearly independent:
		\[
		\mathbf{g}^*_{1}(\mathbf{x}), \quad \mathbf{g}^*_{2}(\mathbf{x}), \quad \mathbf{g}^*_{1}(\mathbf{y}) \quad \text{or} \quad \mathbf{g}^*_{1}(\mathbf{x}), \quad \mathbf{g}^*_{2}(\mathbf{x}), \quad \mathbf{g}^*_{2}(\mathbf{y}).
		\]
	\end{description}
	This concludes the proof.
\end{proof}

\section{The Lie algebra generated by different sets}\label{calulating_lie}

In this section we prove that the action of the Lie brackets is essential for the main results.
Let us consider the mapping
$\bfm\mapsto\left(  \textbf{g}^\ast_{j}(\bm_{i}\right)_{i=1}^\infty $, $j=1,2$, or $3$. Then, it is straightforward to show taking into account the identities \eqref{eq:identitiesg}, that it holds for the Lie brackets
	\begin{eqnarray*}
	\lk[ \bfm\mapsto\left(  \textbf{g}^\ast_{1}(\bm_{i}\right)_{i=0}^\infty) ,\bfm\mapsto\left(  \textbf{g}^\ast_{2}(\bm_{i}\right)_{i=0}^\infty) \rk]&=&\,\bfm\mapsto\left(  \lk[ \textbf{g}^\ast_{1},\textbf{g}^\ast_{2}\rk] (\bm_{i}\right)_{i=0}^\infty =
	- \,\bfm\mapsto\left(  \textbf{g}^\ast_{3}(\bm_{i})_{i=0}^\infty \right)
\\
	\lk[ \bfm\mapsto\left(  \textbf{g}^\ast_{1}(\bm_{i})_{i=0}^\infty\right) ,\bfm\mapsto\left(  \textbf{g}^\ast_{3}(\bm_{i})_{i=0}^\infty\right) \rk]&=&\,\bfm\mapsto\left(  \lk[ \textbf{g}^\ast_{1},\textbf{g}^\ast_{3}\rk] (\bm_{i}\right)_{i=0}^\infty =
	- \,\bfm\mapsto\left(  \textbf{g}^\ast_{2}(\bm_{i})_{i=0}^\infty \right)
\\
	\lk[ \bfm\mapsto\left(  \textbf{g}^\ast_{2}(\bm_{i})_{i=0}^\infty\right) ,\bfm\mapsto\left(  \textbf{g}^\ast_{3}(\bm_{i})_{i=0}^\infty\right) \rk]&=&\,\bfm\mapsto\left(  \lk[ \textbf{g}^\ast_{2},\textbf{g}^\ast_{3}\rk] (\bm_{i}\right)_{i=0}^\infty =
	- \,\bfm\mapsto\left(  \textbf{g}^\ast_{1}(\bm_{i})_{i=0}^\infty \right).
			\end{eqnarray*}
If we shift the mapping by some $l\in\mathbb{N}$, then we obtain, for example,
	\begin{eqnarray*}
	\lk[ \bfm\mapsto\left(  \textbf{g}^\ast_{1}(\bm_{i+l})_{i=0}^\infty\right) ,\bfm\mapsto\left(  \textbf{g}^\ast_{2}(\bm_{i})_{i=0}^\infty\right) \rk]&=&- \,\bfm\mapsto\left(  \textbf{g}^\ast_{3}(\bm_{i+l})_{i=0}^\infty \right).
			\end{eqnarray*}
On the other hand, if an element vanishes, then we obtain\footnote{We used here the abbreviation
	$$
	\mathcal{M}_A(\hbfm) = (\boldsymbol{n}_i)_{i=0}^\infty, \quad \text{where} \quad
	\boldsymbol{n}_i :=
	\begin{cases}
			\mathbf{0}, & \text{if } i \in A, \\
			\hbm_i, & \text{otherwise}.
		\end{cases}
	$$}

	\begin{eqnarray*}
	\lk[ \bfm\mapsto \mathcal{M}_{\{k\}^c} 
	\left(
	 \textbf{g}^\ast_{1}\left(\bm_{i+l}\right)_{i=0}^\infty
	\right)  ,\bfm\mapsto\left(  \textbf{g}^\ast_{2}(\bm_{i})_{i=0}^\infty\right) \rk]&=&-   \bfm\mapsto\left(\mathcal{M}_{\{k\}^c} 
	\left(  \textbf{g}^\ast_{3}(\bm_{i+l}\right)_{0=1}^\infty \right) 
\end{eqnarray*}
Identities of this type are collected in the following lemmas. 
In Lemma~\ref{lem:identity1}, we analyze what happens when both sequences are shifted in the same direction, 
whereas Lemma~\ref{lem:identity2} treats the case of shifts in opposite directions. 
Finally, Lemma~\ref{eq:exentity} covers the additional case in which one of the sequences vanishes on a prescribed set.
These lemmas are essential for completing the induction step in the construction of the Lie algebra
generated by ${\bf G}^{(0,1)}$, ${\bf G}^{(0,2)}$, and ${\bf G}^{(1,1)}$.
%	Applying Lemma \eqref{lem:identity1}, Lemma \eqref{lem:identity2} and Lemma \eqref{eq:exentity}, 

%
%
%In this step, we derive several auxiliary expressions that are essential for computing the Lie algebra
%$$
%\operatorname{Lie}\left(\left\{ \mathbf{G}^{0,1}, \mathbf{G}^{0,2}, \mathbf{G}^{1,1} \right\}\right).
%$$
%We adopt the notations introduced in Appendix~\ref{app:expandLie}.  
%We begin by verifying the following lemma.
\begin{lem}\label{lem:identity1}
	We have the following identity
	\begin{equation}\label{eq:identitybraktegnm1}
		[\hbfm\mapsto\vectinf{\textbf{g}^\ast_{m}(\hbm_{i+k})}{i},\hbfm\mapsto\vectinf{\textbf{g}^\ast_{n}(\hbm_{i+l})}{i}]=\left[\hbfm\mapsto \vectinf{[\textbf{g}^\ast_{m},\textbf{g}^\ast_{n}](\hbm_{i+k+l})}{i} \right],~
	\end{equation}
	$\text{for }m,n\in\{1,2,3\}$ and $k,l\in\ZZ$ with $k, l\geq 0$.
\end{lem}
\begin{proof}
	Here, we distinguish between  $k,l\leq 0$ and $k,l\geq 0$:
	\begin{description}
		\item[Case $k,l\leq 0$] Observe that 
		
		\hspace{-1cm}	\scalebox{0.85}{\begin{minipage}{15cm}
				\begin{eqnarray*}
					\lqq{D[\hbfm\mapsto\vectinf{\textbf{g}^\ast_{m}(\hbm_{i+k})}{i}]\vectinf{\textbf{g}^\ast_{n}(\hbm_{i+l})}{i}}
					\phantom{\Bigg|}
					\\
					&=& \begin{bmatrix}
						0&  0&0 & 0& \cdots &0&0 &\cdots   &\rdelim\}{3}{-2pt}[$-k\,$ cells]              \\
						\vdots&  \vdots&\vdots& \vdots& \cdots &\vdots&\vdots &\vdots     &            \\
						0&  0&0 & 0& \cdots &0&0 &\cdots            &     \\
						J_{\textbf{g}^\ast_m} & 0& 0
						&0  & \cdots &0&0 & \cdots&
						\\
						0&  J_{ \textbf{g}^\ast_m} &0&  0
						&\cdots & \cdots& 0&\cdots &
						\\
						\vdots&\vdots &\ddots &\vdots &\vdots &\vdots&\vdots&\vdots &
						\\
						0            &0 &0
						&J_{ \textbf{g}^\ast_m}  
						&
						\cdots %J_{\textbf{g}^\ast_1 } 
						&0 &  0 &\cdots &
						\\
						\vdots            &\vdots &\vdots
						&\cdots 
						&\ddots &
						\cdots &\vdots &  \cdots&
						\\
						0&0&0&           \cdots                           &\cdots&
						J_{ \textbf{g}^\ast_m}&0 &\cdots&
						\\
						\vdots &\vdots&\vdots&           \vdots                           &\vdots&
						\vdots &\ddots  &\cdots &
					\end{bmatrix}
					\begin{bmatrix}
						0&\rdelim\}{3}{-2pt}[$-l$ cells]\\
						\vdots&\\
						0&\\
						\textbf{g}^\ast_{n}(\hbm_{0})&\\
						\textbf{g}^\ast_{n}(\hbm_{1})&\\
						\vdots&	
					\end{bmatrix}
					=  \begin{bmatrix}
						0&\rdelim\}{3}{-2pt}[$-(k+l)$ cells]\\
						\vdots&\\
						0&\\
						J_{ \textbf{g}^\ast_m}\textbf{g}^\ast_{n}(\hbm_{0})&\\
						J_{ \textbf{g}^\ast_m}\textbf{g}^\ast_{n}(\hbm_{1})&\\
						\vdots&	
					\end{bmatrix}
				\end{eqnarray*}
		\end{minipage}}
		\\
		and
		
		\hspace{-1cm}	\scalebox{0.85}{\begin{minipage}{15cm}
				\begin{eqnarray*}
					\lqq{D[\hbfm\mapsto\vectinf{\textbf{g}^\ast_{n}(\hbm_{i+l})}{i}]\vectinf{\textbf{g}^\ast_{m}(\hbm_{i+k})}{i}}\\
					&=& \begin{bmatrix}
						0&  0&0 & 0& \cdots &0&0 &\cdots   &\rdelim\}{3}{-2pt}[$-l$ cells]              \\
						\vdots&  \vdots&\vdots& \vdots& \cdots &\vdots&\vdots &\vdots     &            \\
						0&  0&0 & 0& \cdots &0&0 &\cdots            &     \\
						J_{\textbf{g}^\ast_n} & 0& 0
						&0  & \cdots &0&0 & \cdots&
						\\
						0&  J_{ \textbf{g}^\ast_n} &0&  0
						&\cdots & \cdots& 0&\cdots &
						\\
						\vdots&\vdots &\ddots &\vdots &\vdots &\vdots&\vdots&\vdots &
						\\
						0            &0 &0
						&J_{ \textbf{g}^\ast_n}  
						&
						\cdots %J_{\textbf{g}^\ast_1 } 
						&0 &  0 &\cdots &
						\\
						\vdots            &\vdots &\vdots
						&\cdots 
						&\ddots &
						\cdots &\vdots &  \cdots&
						\\
						0&0&0&           \cdots                           &\cdots&
						J_{ \textbf{g}^\ast_n}&0 &\cdots&
						\\
						\vdots &\vdots&\vdots&           \vdots                           &\vdots&
						\vdots &\ddots  &\cdots &
					\end{bmatrix}
					\begin{bmatrix}
						0&\rdelim\}{3}{-2pt}[$-k$ cells]\\
						\vdots&\\
						0&\\
						\textbf{g}^\ast_{m}(\hbm_{0})&\\
						\textbf{g}^\ast_{m}(\hbm_{1})&\\
						\vdots&	
					\end{bmatrix}
					=  \begin{bmatrix}
						0&\rdelim\}{3}{-2pt}[$-(k+l)$ cells]\\
						\vdots&\\
						0&\\
						J_{ \textbf{g}^\ast_n}\textbf{g}^\ast_{m}(\hbm_{0})&\\
						J_{ \textbf{g}^\ast_n}\textbf{g}^\ast_{m}(\hbm_{1})&\\
						\vdots&	
					\end{bmatrix}
				\end{eqnarray*}
		\end{minipage}}
		
		\medskip
		where $J_{ \textbf{g}^\ast_m}$ and $J_{ \textbf{g}^\ast_n}$ are Jacobian matrices of $\textbf{g}^\ast_m$ and $\textbf{g}^\ast_n$, respectively. By the definition of the Lie bracket, we can write 
		\begin{align*}
		%\begin{eqnarray*}
	[\hbfm\mapsto\vectinf{\textbf{g}^\ast_{m}(\hbm_{i+k})}{i},\hbfm\mapsto\vectinf{\textbf{g}^\ast_{n}(\hbm_{i+l})}{i}]=
%\begin{bmatrix}
%				0&\rdelim\}{3}{-2pt}[$-(k+l)$ cells]\\
%				\vdots&\\
%				0&\\
%				[\textbf{g}^\ast_{m},\textbf{g}^\ast_{n}](\hbm_{0})&\\
%				[\textbf{g}^\ast_{m},\textbf{g}^\ast_{n}](\hbm_{1})&\\
%				\vdots&	
%			\end{bmatrix}=
%
%
\left[\hbfm\mapsto \vectinf{[\textbf{g}^\ast_{m},\textbf{g}^\ast_{n}](\hbm_{i+k+l})}{i} \right].
\end{align*}
	%	\end{eqnarray*}

		\item[Case $k,l\geq0$] Observe that 
		
\hspace{-1.3cm}\scmbox{0.85}{19cm}{		\begin{eqnarray*}
			\lqq{D[\hbfm\mapsto\vectinf{\textbf{g}^\ast_{m}(\hbm_{i+k})}{i}]\vectinf{\textbf{g}^\ast_{n}(\hbm_{i+l})}{i}}\\
			&=& \begin{pmatrix}
				0&\cdots&0&J_{\textbf{g}^\ast_m} & 0& 0
				&0  & \cdots &0&\cdots 
				\\
				0&\cdots&0&0&  J_{ \textbf{g}^\ast_m} &0&  0
				&\cdots & 0& \cdots
				\\
				\vdots&\cdots&\vdots&\vdots&\vdots &\ddots &\vdots &\vdots &\vdots&\vdots
				\\
				0&\cdots&0&0            &0 &0
				&J_{ \textbf{g}^\ast_m}  
				&
				\cdots %J_{\textbf{g}^\ast_1 } 
				&0 &  0 
				\\
				\vdots&\cdots&\vdots&\vdots &\vdots &\vdots
				&\cdots 
				&\ddots &
				\cdots &\vdots 
				\\
				0&\cdots&0&0&0&0&           \cdots                           &\cdots&
				J_{ \textbf{g}^\ast_m}&0 
				\\
				\vdots & \sunderb{3.5em}{k \text{ cells }}&\vdots&\vdots &\vdots&\vdots&           \vdots                           &\vdots&
				\vdots &\ddots 
			\end{pmatrix}
			\begin{bmatrix}
				\textbf{g}^\ast_{n}(\hbm_{l})\\
				\textbf{g}^\ast_{n}(\hbm_{l+1})\\
				\textbf{g}^\ast_{n}(\hbm_{l+2})\\
				\textbf{g}^\ast_{n}(\hbm_{l+3})\\
				\vdots
			\end{bmatrix}
			=  \begin{bmatrix}
				J_{ \textbf{g}^\ast_m}\textbf{g}^\ast_{n}(\hbm_{k+l})&\\
				J_{ \textbf{g}^\ast_m}\textbf{g}^\ast_{n}(\hbm_{k+l+1})&\\
				J_{ \textbf{g}^\ast_m}\textbf{g}^\ast_{n}(\hbm_{k+l+2})&\\
				J_{ \textbf{g}^\ast_m}\textbf{g}^\ast_{n}(\hbm_{k+l+3})&\\
				\vdots&	
			\end{bmatrix}
		\end{eqnarray*}
	}
		and
		
\hspace{-1.3cm}\scmbox{0.85}{19cm}{		\begin{eqnarray*}
			\lqq{D[\hbfm\mapsto\vectinf{\textbf{g}^\ast_{n}(\hbm_{i+l})}{i}]\vectinf{\textbf{g}^\ast_{m}(\hbm_{i+k})}{i}}\\
			&=& \begin{pmatrix}
				0&\cdots&0&J_{\textbf{g}^\ast_n} & 0& 0
				&0  & \cdots &0&\cdots 
				\\
				0&\cdots&0&0&  J_{ \textbf{g}^\ast_n} &0&  0
				&\cdots & 0& \cdots
				\\
				\vdots&\cdots&\vdots&\vdots&\vdots &\ddots &\vdots &\vdots &\vdots&\vdots
				\\
				0&\cdots&0&0            &0 &0
				&J_{ \textbf{g}^\ast_n}  
				&
				\cdots %J_{\textbf{g}^\ast_1 } 
				&0 &  0 
				\\
				\vdots&\cdots&\vdots&\vdots &\vdots &\vdots
				&\cdots 
				&\ddots &
				\cdots &\vdots 
				\\
				0&\cdots&0&0&0&0&           \cdots                           &\cdots&
				J_{ \textbf{g}^\ast_n}&0 
				\\
				\vdots & \sunderb{3.5em}{l \text{ cells }}&\vdots&\vdots &\vdots&\vdots&           \vdots                           &\vdots&
				\vdots &\ddots 
			\end{pmatrix}
			\begin{bmatrix}
				\textbf{g}^\ast_{n}(\hbm_{k})\\
				\textbf{g}^\ast_{n}(\hbm_{k+1})\\
				\textbf{g}^\ast_{n}(\hbm_{k+2})\\
				\textbf{g}^\ast_{n}(\hbm_{k+3})\\
				\vdots
			\end{bmatrix}
			= \begin{bmatrix}
				J_{ \textbf{g}^\ast_n}\textbf{g}^\ast_{m}(\hbm_{k+l})&\\
				J_{ \textbf{g}^\ast_n}\textbf{g}^\ast_{m}(\hbm_{k+l+1})&\\
				J_{ \textbf{g}^\ast_n}\textbf{g}^\ast_{m}(\hbm_{k+l+2})&\\
				J_{ \textbf{g}^\ast_n}\textbf{g}^\ast_{m}(\hbm_{k+l+3})&\\
				\vdots&	
			\end{bmatrix}
		\end{eqnarray*}
	}
	
\noindent Again,
taking into account  the definition of the Lie bracket, we get $\phantom{\Big|}$

\hspace{-1.3cm}\scmbox{0.85}{19cm}{	
			\begin{eqnarray*}
			[\hbfm\mapsto\vectinf{\textbf{g}^\ast_{m}(\hbm_{i+k})}{i},\hbfm\mapsto\vectinf{\textbf{g}^\ast_{n}(\hbm_{i+l})}{i}]&=&\hbfm\mapsto\begin{bmatrix}
				[\textbf{g}^\ast_{m},\textbf{g}^\ast_{n}](\hbm_{k+l})&\\
				[\textbf{g}^\ast_{m},\textbf{g}^\ast_{n}](\hbm_{k+l+1})&\\
				\vdots&	
			\end{bmatrix}\\
		\phantom{\Bigg|}	&=&\left[\hbfm\mapsto \vectinf{[\textbf{g}^\ast_{m},\textbf{g}^\ast_{n}](\hbm_{i+k+l})}{i} \right].
		\end{eqnarray*}
	}
		
	\end{description}
\end{proof}
In the previous lemma, we assumed that $lk>0$. In the next lemma, we assume that $l$ and $k$ have opposite signs (i.e. $lk<0$).
\begin{lem}\label{lem:identity2}
Let $m,n\in\{1,2,3\}$ and $k,l\in\NN$. Then the following identities hold:
	\begin{align}
		\nonumber&[\hbfm\mapsto\vectinf{\textbf{g}^\ast_{m}(\hbm_{i-k})}{i},\hbfm\mapsto\vectinf{\textbf{g}^\ast_{n}(\hbm_{i+l})}{i}]\\
		\nonumber&=\left[\hbfm\mapsto \vectinf{[\textbf{g}^\ast_{m},\textbf{g}^\ast_{n}](\hbm_{i-(k-l)})}{i} \right]-\left[\hbfm\mapsto\mathcal{M}_{\{k-l,\cdots,k-1\}^c}\left(\vectinf{J_{ \textbf{g}^\ast_m}\textbf{g}^\ast_{n}(\hbm_{i-(k-l)})}{i}\right)\right]\\
		&= \left\{\begin{array}{ll}
			\left[\hbfm\mapsto \vectinf{[\textbf{g}^\ast_{m},\textbf{g}^\ast_{n}](\hbm_{i})}{i} \right]-\left[\hbfm\mapsto\mathcal{M}_{\{0,\cdots,k-1\}^c}\left(\vectinf{J_{ \textbf{g}^\ast_m}\textbf{g}^\ast_{n}(\hbm_{i})}{i}\right)\right]&\text{if }l=k\\
			\left[\hbfm\mapsto \vectinf{[\textbf{g}^\ast_{m},\textbf{g}^\ast_{n}](\hbm_{i-(k-l)})}{i} \right]-\left[\hbfm\mapsto\mathcal{M}_{\{k-l,\cdots,k-1\}^c}\left(\vectinf{J_{ \textbf{g}^\ast_m}\textbf{g}^\ast_{n}(\hbm_{i-(k-l)})}{i}\right)\right]&\text{if }l<k\\
			\left[\hbfm\mapsto \vectinf{[\textbf{g}^\ast_{m},\textbf{g}^\ast_{n}](\hbm_{i+(l-k)})}{i} \right]-\left[\hbfm\mapsto\mathcal{M}_{\{0,\cdots,k-1\}^c}\left(\vectinf{J_{ \textbf{g}^\ast_m}\textbf{g}^\ast_{n}(\hbm_{i+(l-k)})}{i}\right)\right]&\text{if }l>k
		\end{array}\right.\label{eq:identitybraktegnm3}.
	\end{align}
Moreover,
	\begin{align}
		\nonumber&[\hbfm\mapsto\vectinf{\textbf{g}^\ast_{n}(\hbm_{i+l})}{i},\hbfm\mapsto\vectinf{\textbf{g}^\ast_{m}(\hbm_{i-k})}{i}]\\
		\nonumber&=\left[\hbfm\mapsto \vectinf{[\textbf{g}^\ast_{n},\textbf{g}^\ast_{m}](\hbm_{i-(k-l)})}{i} \right]+\left[\hbfm\mapsto\mathcal{M}_{\{k-l,\cdots,k-1\}^c}\left(\vectinf{J_{ \textbf{g}^\ast_m}\textbf{g}^\ast_{n}(\hbm_{i-(k-l)})}{i}\right)\right]\\
		&= \left\{\begin{array}{ll}
			\left[\hbfm\mapsto \vectinf{[\textbf{g}^\ast_{n},\textbf{g}^\ast_{m}](\hbm_{i})}{i} \right]+\left[\hbfm\mapsto\mathcal{M}_{\{0,\cdots,k-1\}^c}\left(\vectinf{J_{ \textbf{g}^\ast_m}\textbf{g}^\ast_{n}(\hbm_{i})}{i}\right)\right]&\text{if }l=k\\
			\left[\hbfm\mapsto \vectinf{[\textbf{g}^\ast_{n},\textbf{g}^\ast_{m}](\hbm_{i-(k-l)})}{i} \right]+\left[\hbfm\mapsto\mathcal{M}_{\{k-l,\cdots,k-1\}^c}\left(\vectinf{J_{ \textbf{g}^\ast_m}\textbf{g}^\ast_{n}(\hbm_{i-(k-l)})}{i}\right)\right]&\text{if }l<k\\
			\left[\hbfm\mapsto \vectinf{[\textbf{g}^\ast_{n},\textbf{g}^\ast_{m}](\hbm_{i+(l-k)})}{i} \right]+\left[\hbfm\mapsto\mathcal{M}_{\{0,\cdots,k-1\}^c}\left(\vectinf{J_{ \textbf{g}^\ast_m}\textbf{g}^\ast_{n}(\hbm_{i+(l-k)})}{i}\right)\right]&\text{if }l>k
		\end{array}\right.\label{eq:identitybraktegnm4}.
	\end{align}
%	$\text{for }m,n\in\{1,2,3\}$ and $k,l\in\NN$.
\end{lem}

Before proving Lemma~\ref{lem:identity2}, we state the following corollary. Its proof follows by direct calculations using Lemma~\ref{lem:identity2}.

\begin{cor}\label{cor:identity1}
	Let $m,n\in\{1,2,3\}$ and $k,l\in\NN$. Then
	\begin{align}
		\nonumber&
		[\hbfm\mapsto\vectinf{\textbf{g}^\ast_{m}(\hbm_{i+l})}{i},\;
		\hbfm\mapsto\vectinf{\textbf{g}^\ast_{n}(\hbm_{i-k})}{i}]
		+
		[\hbfm\mapsto\vectinf{\textbf{g}^\ast_{m}(\hbm_{i-k})}{i},\;
		\hbfm\mapsto\vectinf{\textbf{g}^\ast_{n}(\hbm_{i+l})}{i}]
		\\
		&=
		2\left[\hbfm\mapsto \vectinf{[\textbf{g}^\ast_{m},\textbf{g}^\ast_{n}](\hbm_{i+(l-k)})}{i} \right]
		-\left[\hbfm\mapsto\mathcal{M}_{\{k-l,\cdots,k-1\}^c}\!\left(\vectinf{[\textbf{g}^\ast_{m},\textbf{g}^\ast_{n}](\hbm_{i+(l-k)})}{i}\right)\right].
		\label{eq:identitybraktegnm2}
	\end{align}
\end{cor}

\begin{proof}[Proof of Lemma \ref{lem:identity2}] 
Here we prove only the identity \eqref{eq:identitybraktegnm3}, since \eqref{eq:identitybraktegnm4} follows directly from \eqref{eq:identitybraktegnm3}.
 We distinguish three cases:  $l=k$, $l<k$, and $l>k$. We begin with the first case.
	\begin{description}
		\item[Case $l= k$]
	Here, we obtain
	
		\hspace{-1.3cm}\scmbox{0.85}{19cm}{	
			\begin{eqnarray*}
			\lqq{D[\hbfm\mapsto\vectinf{\textbf{g}^\ast_{m}(\hbm_{i-k})}{i}]\vectinf{\textbf{g}^\ast_{n}(\hbm_{i+l})}{i}}\\
		\hspace{-1cm}	&=& \begin{bmatrix}
				0&  0&0 & 0& \cdots &0&0 &   \rdelim\}{3}{-7pt}[$k$ cells]              \\
				\vdots&  \vdots&\vdots& \vdots& \cdots &\vdots&\vdots &             \\
				0&  0&0 & 0& \cdots &0&0 &           \\
				J_{\textbf{g}^\ast_m} & 0& 0
				&0  & \cdots &0&0 &
				\\
				0&  J_{ \textbf{g}^\ast_m} &0&  0
				&\cdots & \cdots& 0& 
				\\
				\vdots&\vdots &\ddots &\vdots &\vdots &\vdots&\vdots&
				\\
				0            &0 &0
				&J_{ \textbf{g}^\ast_m}  
				&
				\cdots %J_{\textbf{g}^\ast_1 } 
				&0 &  0 &
				\\
				\vdots            &\vdots &\vdots
				&\cdots 
				&\ddots &
				\cdots &\vdots & 
				\\
				0&0&0&           \cdots                           &\cdots&
				J_{ \textbf{g}^\ast_m}&0 &
				\\
				\vdots &\vdots&\vdots&           \vdots                           &\vdots&
				\vdots &\ddots   &
			\end{bmatrix}
			\begin{bmatrix}
				\textbf{g}^\ast_{n}(\hbm_{k})\\
				\textbf{g}^\ast_{n}(\hbm_{k+1})\\
				\textbf{g}^\ast_{n}(\hbm_{k+2})\\
				\textbf{g}^\ast_{n}(\hbm_{k+3})\\
				\vdots
			\end{bmatrix}
		%	\\
		= 
			\begin{bmatrix}
				0&\rdelim\}{3}{-2pt}[\parbox{2cm}{$k$\\  cells}]\\
				\vdots&\\
				0&\\
				J_{ \textbf{g}^\ast_m}\textbf{g}^\ast_{n}(\hbm_{k})&\\
				J_{ \textbf{g}^\ast_m}\textbf{g}^\ast_{n}(\hbm_{k+1})&\\
				\vdots&	
			\end{bmatrix}
		\end{eqnarray*}
	}
		and

\hspace{-1.3cm}\scmbox{0.85}{19cm}{	
			\begin{eqnarray*}
			\lqq{D[\hbfm\mapsto\vectinf{\textbf{g}^\ast_{n}(\hbm_{i+l})}{i}]\vectinf{\textbf{g}^\ast_{m}(\hbm_{i-k})}{i}}\\
			&=& \begin{pmatrix}
				0&\cdots&0&J_{\textbf{g}^\ast_n} & 0& 0
				&0  & \cdots &0& 
				\\
				0&\cdots&0&0&  J_{ \textbf{g}^\ast_n} &0&  0
				&\cdots &  
				\\
				\vdots&\cdots&\vdots&\vdots&\vdots &J_{ \textbf{g}^\ast_n}   &\vdots &\vdots &
				\\
				0&\cdots&0&0            &0 &0
				&\ddots 
				&
				\cdots %J_{\textbf{g}^\ast_1 } 
				&0 & 
				\\
				\vdots&\cdots&\vdots&\vdots &\vdots &\vdots
				&\cdots 
				&	J_{ \textbf{g}^\ast_n} &
				\cdots &
				\\
				0&\cdots&0&0&0&0&           \cdots                           &\cdots&
			\ddots&
				\\
				\vdots & \sunderb{3.5em}{k \text{ cells }}&\vdots&\vdots &\vdots&\vdots&           \vdots                           &\vdots&
				\vdots 
			\end{pmatrix}
			\begin{bmatrix}
				0&\rdelim\}{3}{-2pt}[$k$ cells]\\
				\vdots&\\
				0&\\
				\textbf{g}^\ast_{m}(\hbm_{0})&\\
				\textbf{g}^\ast_{m}(\hbm_{1})&\\
				\vdots&	
			\end{bmatrix}
			=  \begin{bmatrix}
				J_{ \textbf{g}^\ast_n}\textbf{g}^\ast_{m}(\hbm_{0})\\
				J_{ \textbf{g}^\ast_n}\textbf{g}^\ast_{m}(\hbm_{1})\\
				\vdots
			\end{bmatrix}
		\end{eqnarray*}
	}

Combining both calculations, we get 
		\begin{eqnarray*}
			&&[\hbfm\mapsto\vectinf{\textbf{g}^\ast_{m}(\hbm_{i-k})}{i},\hbfm\mapsto\vectinf{\textbf{g}^\ast_{n}(\hbm_{i+l})}{i}]\\
			&=&\hbfm\mapsto\begin{bmatrix}
				0&\rdelim\}{3}{-2pt}[$k$ cells]\\
				\vdots&\\
				0&\\
				J_{ \textbf{g}^\ast_m}\textbf{g}^\ast_{n}(\hbm_{k})&\\
				J_{ \textbf{g}^\ast_m}\textbf{g}^\ast_{n}(\hbm_{k+1})&\\
				\vdots&	
			\end{bmatrix}-\begin{bmatrix}
				J_{ \textbf{g}^\ast_n}\textbf{g}^\ast_{m}(\hbm_{0})&\\
				J_{ \textbf{g}^\ast_n}\textbf{g}^\ast_{m}(\hbm_{1})&\\
				\vdots&	
			\end{bmatrix}\\
			&=&\hbfm\mapsto\begin{bmatrix}
				[\textbf{g}^\ast_{m},\textbf{g}^\ast_{n}](\hbm_{0})&\\
				[\textbf{g}^\ast_{m},\textbf{g}^\ast_{n}](\hbm_{1})&\\
				\vdots&	
			\end{bmatrix}-\begin{bmatrix}
				J_{ \textbf{g}^\ast_m}\textbf{g}^\ast_{n}(\hbm_{0})&\\
				J_{ \textbf{g}^\ast_m}\textbf{g}^\ast_{n}(\hbm_{1})&\\
				\vdots&\\
				J_{ \textbf{g}^\ast_m}\textbf{g}^\ast_{n}(\hbm_{k-1})&\\
				0&\\
				\vdots&	
			\end{bmatrix}\\
			&=&\left[\hbfm\mapsto \vectinf{[\textbf{g}^\ast_{m},\textbf{g}^\ast_{n}](\hbm_{i})}{i} \right]-\left[\hbfm\mapsto\mathcal{M}_{\{0,\cdots,k-1\}^c}\left(\vectinf{J_{ \textbf{g}^\ast_m}\textbf{g}^\ast_{n}(\hbm_{i})}{i}\right)\right].
		\end{eqnarray*}

		\item[Case $l< k$]
		Here, we get first
		
		\hspace{-1.9cm}\scmbox{0.85}{19cm}{	
			\begin{eqnarray*}
			\lqq{D[\hbfm\mapsto\vectinf{\textbf{g}^\ast_{m}(\hbm_{i-k})}{i}]\vectinf{\textbf{g}^\ast_{n}(\hbm_{i+l})}{i}}\\
			&=& \begin{bmatrix}
				0&  0&0 & 0& \cdots &0&0 &\cdots   &\rdelim\}{3}{-2pt}[$k$ cells]              \\
				\vdots&  \vdots&\vdots& \vdots& \cdots &\vdots&\vdots &\vdots     &            \\
				0&  0&0 & 0& \cdots &0&0 &\cdots            &     \\
				J_{\textbf{g}^\ast_m} & 0& 0
				&0  & \cdots &0&0 & \cdots&
				\\
				0&  J_{ \textbf{g}^\ast_m} &0&  0
				&\cdots & \cdots& 0&\cdots &
				\\
				\vdots&\vdots &\ddots &\vdots &\vdots &\vdots&\vdots&\vdots &
				\\
				0            &0 &0
				&J_{ \textbf{g}^\ast_m}  
				&
				\cdots %J_{\textbf{g}^\ast_1 } 
				&0 &  0 &\cdots &
				\\
				\vdots            &\vdots &\vdots
				&\cdots 
				&\ddots &
				\cdots &\vdots &  \cdots&
				\\
				0&0&0&           \cdots                           &\cdots&
				J_{ \textbf{g}^\ast_m}&0 &\cdots&
				\\
				\vdots &\vdots&\vdots&           \vdots                           &\vdots&
				\vdots &\ddots  &\cdots &
			\end{bmatrix}
			\begin{bmatrix}
				\textbf{g}^\ast_{n}(\hbm_{l})\\
				\textbf{g}^\ast_{n}(\hbm_{l+1})\\
				\textbf{g}^\ast_{n}(\hbm_{l+2})\\
				\textbf{g}^\ast_{n}(\hbm_{l+3})\\
				\vdots
			\end{bmatrix}
			=  \begin{bmatrix}
				0&\rdelim\}{3}{-2pt}[$k$ cells]\\
				\vdots&\\
				0&\\
				J_{ \textbf{g}^\ast_m}\textbf{g}^\ast_{n}(\hbm_{l})&\\
				J_{ \textbf{g}^\ast_m}\textbf{g}^\ast_{n}(\hbm_{l+1})&\\
				\vdots&	
			\end{bmatrix}
		\end{eqnarray*}
	}
		and
		
		\hspace{-1.9cm}\scmbox{0.85}{19cm}{	
		\begin{eqnarray*}
			\lqq{D[\hbfm\mapsto\vectinf{\textbf{g}^\ast_{n}(\hbm_{i+l})}{i}]\vectinf{\textbf{g}^\ast_{m}(\hbm_{i-k})}{i}}\\
			&=& \begin{pmatrix}
				0&\cdots&0&J_{\textbf{g}^\ast_n} & 0& 0
				&0  & \cdots &0&\cdots 
				\\
				0&\cdots&0&0&  J_{ \textbf{g}^\ast_n} &0&  0
				&\cdots & 0& \cdots
				\\
				\vdots&\cdots&\vdots&\vdots&\vdots &\ddots &\vdots &\vdots &\vdots&\vdots
				\\
				0&\cdots&0&0            &0 &0
				&J_{ \textbf{g}^\ast_n}  
				&
				\cdots %J_{\textbf{g}^\ast_1 } 
				&0 &  0 
				\\
				\vdots&\cdots&\vdots&\vdots &\vdots &\vdots
				&\cdots 
				&\ddots &
				\cdots &\vdots 
				\\
				0&\cdots&0&0&0&0&           \cdots                           &\cdots&
				J_{ \textbf{g}^\ast_n}&0 
				\\
				\vdots & \sunderb{3.5em}{l \text{ cells }}&\vdots&\vdots &\vdots&\vdots&           \vdots                           &\vdots&
				\vdots &\ddots 
			\end{pmatrix}
			\begin{bmatrix}
				0&\rdelim\}{3}{-2pt}[$k$ cells]\\
				\vdots&\\
				0&\\
				\textbf{g}^\ast_{m}(\hbm_{0})&\\
				\textbf{g}^\ast_{m}(\hbm_{1})&\\
				\vdots&	
			\end{bmatrix}
			=  \begin{bmatrix}
				0&\rdelim\}{3}{-2pt}[$k-l$ cells]\\
				\vdots&\\
				0&\\
				J_{ \textbf{g}^\ast_n}\textbf{g}^\ast_{m}(\hbm_{0})&\\
				J_{ \textbf{g}^\ast_n}\textbf{g}^\ast_{m}(\hbm_{1})&\\
				\vdots&	
			\end{bmatrix}
		\end{eqnarray*}
	}

		Thus, we obtain, 
		\begin{eqnarray*}
			&&[\hbfm\mapsto\vectinf{\textbf{g}^\ast_{m}(\hbm_{i-k})}{i},\hbfm\mapsto\vectinf{\textbf{g}^\ast_{n}(\hbm_{i+l})}{i}]
		\phantom{\Bigg|}	\\
			&=&\hbfm\mapsto\begin{bmatrix}
				0&\rdelim\}{3}{-2pt}[$k$ cells]\\
				\vdots&\\
				0&\\
				J_{ \textbf{g}^\ast_m}\textbf{g}^\ast_{n}(\hbm_{l})&\\
				J_{ \textbf{g}^\ast_m}\textbf{g}^\ast_{n}(\hbm_{l+1})&\\
				\vdots&	
			\end{bmatrix}\qquad -\qquad \begin{bmatrix}
				0&\rdelim\}{3}{-2pt}[$k-l$ cells]\\
				\vdots&\\
				0&\\
				J_{ \textbf{g}^\ast_n}\textbf{g}^\ast_{m}(\hbm_{0})&\\
				J_{ \textbf{g}^\ast_n}\textbf{g}^\ast_{m}(\hbm_{1})&\\
				\vdots&	
			\end{bmatrix}\\
			&=&\hbfm\mapsto\begin{bmatrix}
				0&\rdelim\}{3}{-2pt}[$k-l$ cells]\\
				\vdots&\\
				0&\\
				[\textbf{g}^\ast_{m},\textbf{g}^\ast_{n}](\hbm_{0})&\\
				[\textbf{g}^\ast_{m},\textbf{g}^\ast_{n}](\hbm_{1})&\\
				\vdots&	
			\end{bmatrix}\qquad -\qquad\begin{bmatrix}
				0&&\rdelim\}{3}{-2pt}[$k-l$ cells]\\
				\vdots&\\
				0&\\
				J_{ \textbf{g}^\ast_m}\textbf{g}^\ast_{n}(\hbm_{0})&\\
				J_{ \textbf{g}^\ast_m}\textbf{g}^\ast_{n}(\hbm_{1})&\\
				\vdots&\\
				J_{ \textbf{g}^\ast_m}\textbf{g}^\ast_{n}(\hbm_{l-1})&\\
				0&\\
				\vdots&	
			\end{bmatrix}\\
			&=&\left[\hbfm\mapsto \vectinf{[\textbf{g}^\ast_{m},\textbf{g}^\ast_{n}](\hbm_{i-(k-l)})}{i} \right]-\left[\hbfm\mapsto\mathcal{M}_{\{k-l,\cdots,k-1\}^c}\left(\vectinf{J_{ \textbf{g}^\ast_m}\textbf{g}^\ast_{n}(\hbm_{i-(k-l)})}{i}\right)\right].\phantom{\Big|}
		\end{eqnarray*}

		\item[Case $l> k$]
		Here, we start with
		
				\hspace{-1.9cm}\scmbox{0.85}{19cm}{	
			\begin{eqnarray*}
			\lqq{D[\hbfm\mapsto\vectinf{\textbf{g}^\ast_{m}(\hbm_{i-k})}{i}]\vectinf{\textbf{g}^\ast_{n}(\hbm_{i+l})}{i}}\\
			&=& \begin{bmatrix}
				0&  0&0 & 0& \cdots &0&0 &\cdots   &\rdelim\}{3}{-2pt}[$k$ cells]              \\
				\vdots&  \vdots&\vdots& \vdots& \cdots &\vdots&\vdots &\vdots     &            \\
				0&  0&0 & 0& \cdots &0&0 &\cdots            &     \\
				J_{\textbf{g}^\ast_m} & 0& 0
				&0  & \cdots &0&0 & \cdots&
				\\
				0&  J_{ \textbf{g}^\ast_m} &0&  0
				&\cdots & \cdots& 0&\cdots &
				\\
				\vdots&\vdots &\ddots &\vdots &\vdots &\vdots&\vdots&\vdots &
				\\
				0            &0 &0
				&J_{ \textbf{g}^\ast_m}  
				&
				\cdots %J_{\textbf{g}^\ast_1 } 
				&0 &  0 &\cdots &
				\\
				\vdots            &\vdots &\vdots
				&\cdots 
				&\ddots &
				\cdots &\vdots &  \cdots&
				\\
				0&0&0&           \cdots                           &\cdots&
				J_{ \textbf{g}^\ast_m}&0 &\cdots&
				\\
				\vdots &\vdots&\vdots&           \vdots                           &\vdots&
				\vdots &\ddots  &\cdots &
			\end{bmatrix}
			\begin{bmatrix}
				\textbf{g}^\ast_{n}(\hbm_{l})\\
				\textbf{g}^\ast_{n}(\hbm_{l+1})\\
				\textbf{g}^\ast_{n}(\hbm_{l+2})\\
				\textbf{g}^\ast_{n}(\hbm_{l+3})\\
				\vdots
			\end{bmatrix}
			=  \begin{bmatrix}
				0&\rdelim\}{3}{-2pt}[$k$ cells]\\
				\vdots&\\
				0&\\
				J_{ \textbf{g}^\ast_m}\textbf{g}^\ast_{n}(\hbm_{l})&\\
				J_{ \textbf{g}^\ast_m}\textbf{g}^\ast_{n}(\hbm_{l+1})&\\
				\vdots&	
			\end{bmatrix}
		\end{eqnarray*}
	}
		and

		\hspace{-1.9cm}\scmbox{0.85}{19cm}{	
		\begin{eqnarray*}
			\lqq{D[\hbfm\mapsto\vectinf{\textbf{g}^\ast_{n}(\hbm_{i+l})}{i}]\vectinf{\textbf{g}^\ast_{m}(\hbm_{i-k})}{i}}\\
			&=& \begin{pmatrix}
				0&\cdots&0&J_{\textbf{g}^\ast_n} & 0& 0
				&0  & \cdots &0&\cdots 
				\\
				0&\cdots&0&0&  J_{ \textbf{g}^\ast_n} &0&  0
				&\cdots & 0& \cdots
				\\
				\vdots&\cdots&\vdots&\vdots&\vdots &\ddots &\vdots &\vdots &\vdots&\vdots
				\\
				0&\cdots&0&0            &0 &0
				&J_{ \textbf{g}^\ast_n}  
				&
				\cdots %J_{\textbf{g}^\ast_1 } 
				&0 &  0 
				\\
				\vdots&\cdots&\vdots&\vdots &\vdots &\vdots
				&\cdots 
				&\ddots &
				\cdots &\vdots 
				\\
				0&\cdots&0&0&0&0&           \cdots                           &\cdots&
				J_{ \textbf{g}^\ast_n}&0 
				\\
				\vdots & \sunderb{3.5em}{l \text{ cells }}&\vdots&\vdots &\vdots&\vdots&           \vdots                           &\vdots&
				\vdots &\ddots 
			\end{pmatrix}
			\begin{bmatrix}
				0&\rdelim\}{3}{-2pt}[$k$ cells]\\
				\vdots&\\
				0&\\
				\textbf{g}^\ast_{m}(\hbm_{0})&\\
				\textbf{g}^\ast_{m}(\hbm_{1})&\\
				\vdots&	
			\end{bmatrix}
			=  \begin{bmatrix}
				J_{ \textbf{g}^\ast_n}\textbf{g}^\ast_{m}(\hbm_{l-k})&\\
				J_{ \textbf{g}^\ast_n}\textbf{g}^\ast_{m}(\hbm_{l-k+1})&\\
				\vdots&	
			\end{bmatrix}.
		\end{eqnarray*}
	}
		
		Thus, we obtain, 
		\begin{eqnarray*}
			&&[\hbfm\mapsto\vectinf{\textbf{g}^\ast_{m}(\hbm_{i-k})}{i},\hbfm\mapsto\vectinf{\textbf{g}^\ast_{n}(\hbm_{i+l})}{i}]\\
			&=&\hbfm\mapsto\begin{bmatrix}
				0&\rdelim\}{3}{-2pt}[$k$ cells]\\
				\vdots&\\
				0&\\
				J_{ \textbf{g}^\ast_m}\textbf{g}^\ast_{n}(\hbm_{l})&\\
				J_{ \textbf{g}^\ast_m}\textbf{g}^\ast_{n}(\hbm_{l+1})&\\
				\vdots&	
			\end{bmatrix}-\begin{bmatrix}
				J_{ \textbf{g}^\ast_n}\textbf{g}^\ast_{m}(\hbm_{l-k})&\\
				J_{ \textbf{g}^\ast_n}\textbf{g}^\ast_{m}(\hbm_{l-k+1})&\\
				\vdots&	
			\end{bmatrix}\\
			&=&\hbfm\mapsto\begin{bmatrix}
				[\textbf{g}^\ast_{m},\textbf{g}^\ast_{n}](\hbm_{l-k})&\\
				[\textbf{g}^\ast_{m},\textbf{g}^\ast_{n}](\hbm_{l-k+1})&\\
				\vdots&	
			\end{bmatrix}-\begin{bmatrix}
				J_{ \textbf{g}^\ast_m}\textbf{g}^\ast_{n}(\hbm_{l-k})&&\\
				J_{ \textbf{g}^\ast_m}\textbf{g}^\ast_{n}(\hbm_{l-k+1})&&\\
				\vdots&\\
				J_{ \textbf{g}^\ast_m}\textbf{g}^\ast_{n}(\hbm_{l-1})&\\
				0&\\
				0&\\
				\vdots&	
			\end{bmatrix}\\
			&=&\left[\hbfm\mapsto \vectinf{[\textbf{g}^\ast_{m},\textbf{g}^\ast_{n}](\hbm_{i+(l-k)})}{i} \right]-\left[\hbfm\mapsto\mathcal{M}_{\{k-l,\cdots,k-1\}^c}\left(\vectinf{J_{ \textbf{g}^\ast_m}\textbf{g}^\ast_{n}(\hbm_{i+(l-k)})}{i}\right)\right].\phantom{\Big|}
		\end{eqnarray*}
	\end{description}
	
\end{proof}

In the next Lemma, we tackle the entry in the $k^{th}$ line is zero.

\begin{lem}\label{eq:exentity}
Let $m,n\in\{1,2,3\}$ and 	let $k\ge 0$ and $l\ge -k$. Then the following identities hold:
	\begin{eqnarray}
		&&\label{eq:id13}\Big[\hbfm\mapsto\mathcal{M}_{\{k\}^c}\left(\vectinf{\textbf{g}^\ast_{m}(\hbm_{i+l})}{i}\right),\hbfm\mapsto \vectinf{\textbf{g}^\ast_{n}(\hbm_{i})}{i}\Big]\\
		&&\nonumber\qquad=\hbfm\mapsto\left[\hbfm\mapsto\mathcal{M}_{\{k\}^c}\left(\vectinf{[\textbf{g}^\ast_m,\textbf{g}^\ast_n] (\hbm_{i+l})}{i}\right) \right]\\
		&&\label{eq:id14}\Big[\hbfm\mapsto\mathcal{M}_{\{k\}^c}\left(\vectinf{\textbf{g}^\ast_{m}(\hbm_{i+l})}{i}\right),\hbfm\mapsto \vectinf{\textbf{g}^\ast_{n}(\hbm_{i-1})}{i}\Big]\\
		&&
		\nonumber\qquad=\hbfm\mapsto\left[\hbfm\mapsto\mathcal{M}_{\{k\}^c}\left(\vectinf{J_{ \textbf{g}^\ast_m} \textbf{g}^\ast_n(\hbm_{i+l-1})}{i}\right) \right]-\left[\hbfm\mapsto\mathcal{M}_{\{k+1\}^c}\left(\vectinf{J_{ \textbf{g}^\ast_n} \textbf{g}^\ast_m(\hbm_{i+l-1})}{i}\right) \right]\\
		&&\label{eq:id15}\Big[\hbfm\mapsto\mathcal{M}_{\{k\}^c}\left(\vectinf{\textbf{g}^\ast_{m}(\hbm_{i+l})}{i}\right),\hbfm\mapsto \vectinf{\textbf{g}^\ast_{n}(\hbm_{i+1})}{i}\Big]\\
		&&\nonumber\qquad= \left[\hbfm\mapsto\mathcal{M}_{\{k\}^c}\left(\vectinf{J_{ \textbf{g}^\ast_m} \textbf{g}^\ast_n(\hbm_{i+l+1})}{i}\right) \right]-\left[\hbfm\mapsto\mathcal{M}_{\{k-1\}^c}\left(\vectinf{J_{ \textbf{g}^\ast_n} \textbf{g}^\ast_m(\hbm_{i+l+1})}{i}\right) \right]
	\end{eqnarray}
\end{lem}

\begin{proof}
	\begin{description}
		\item[Proof of identity \eqref{eq:id13} for $-k\leq l\leq 0$] 
		\begin{eqnarray*}
			&&\left[\hbfm\mapsto\mathcal{M}_{\{k\}^c}\left(\vectinf{\textbf{g}^\ast_{m}(\hbm_{i+l})}{i}\right),\hbfm\mapsto \vectinf{\textbf{g}^\ast_{n}(\hbm_{i})}{i} \right]\\
			&&\quad= \hbfm\mapsto\begin{bmatrix}
				0&\cdots&0&  0&0 & 0& \cdots &\rdelim\}{3}{-2pt}[$k$ cells]              \\
				\vdots&\vdots&\vdots&  \vdots&\vdots& \vdots& \cdots &      \\
				0&\cdots&0&  0&0 & 0& \cdots &  \\
				0&\cdots&0 & J_{\textbf{g}^\ast_m}& 0
				&0&\cdots&
				\\
				0&\cdots&0&  0 &0&  0&\cdots&
				\\
				\vdots&\vdots&\vdots&\vdots &\ddots &\vdots&\cdots &\\
				0&\cdots&0            &0 &0
				&0 &\cdots &
				\\
				\vdots&\sunderb{3.5em}{k+l \text{ cells }}&\vdots            &\vdots &\vdots
				&\vdots&\ddots &
			\end{bmatrix} \begin{bmatrix}
				\textbf{g}^\ast_{n}(\hbm_{0})\\
				\textbf{g}^\ast_{n}(\hbm_{1})\\
				\vdots
			\end{bmatrix}\\
			&&\qquad-\begin{bmatrix}
				J_{\textbf{g}^\ast_n} & 0& 0
				&0  & \cdots
				\\
				0&  J_{ \textbf{g}^\ast_n} &0&  0
				&\cdots 
				\\
				\vdots&\vdots &\ddots &\vdots &\vdots 
				\\
				0            &0 &0
				&J_{ \textbf{g}^\ast_n}  
				&
				\cdots %J_{\textbf{g}^\ast_1 } 
				\\
				\vdots            &\vdots &\vdots
				&\cdots 
				&\ddots
			\end{bmatrix} \begin{bmatrix}
				0&\rdelim\}{3}{-2pt}[$k$ cells] \\
				\vdots&\\
				0&\\
				\textbf{g}^\ast_{m}(\hbm_{k+l})&\\
				0&\\
				0&\\
				\vdots&
			\end{bmatrix}\\
			&&\quad=\hbfm\mapsto\begin{bmatrix}
				0&\rdelim\}{3}{-2pt}[$k$ cells] \\
				\vdots&\\
				0&\\
				J_{ \textbf{g}^\ast_m}\textbf{g}^\ast_n(\hbm_{k+l})&\\
				0&\\
				\vdots
			\end{bmatrix}-\begin{bmatrix}
				0&\rdelim\}{3}{-2pt}[$k$ cells] \\
				\vdots&\\
				0&\\
				J_{ \textbf{g}^\ast_n}\textbf{g}^\ast_m(\hbm_{k+l})&\\
				0&\\
				\vdots
			\end{bmatrix}\\
			&&\quad=\hbfm\mapsto\left[\hbfm\mapsto\mathcal{M}_{\{k\}^c}\left(\vectinf{[\textbf{g}^\ast_m,\textbf{g}^\ast_n] (\hbm_{i+l})}{i}\right) \right].
		\end{eqnarray*}

		\item[Proof of identity \eqref{eq:id14} for $-k\leq l\leq 0$]
		\begin{eqnarray*}
			&&\left[\hbfm\mapsto\mathcal{M}_{\{k\}^c}\left(\vectinf{\textbf{g}^\ast_{m}(\hbm_{i+l})}{i}\right),\hbfm\mapsto \vectinf{\textbf{g}^\ast_{n}(\hbm_{i-1})}{i} \right]\\
			&&\quad= \hbfm\mapsto\begin{bmatrix}
				0&\cdots&0&  0&0 & 0& \cdots &\rdelim\}{3}{-2pt}[$k$ cells]              \\
				\vdots&\vdots&\vdots&  \vdots&\vdots& \vdots& \cdots &      \\
				0&\cdots&0&  0&0 & 0& \cdots &  \\
				0&\cdots&0 & J_{\textbf{g}^\ast_m}& 0
				&0&\cdots&
				\\
				0&\cdots&0&  0 &0&  0&\cdots&
				\\
				\vdots&\vdots&\vdots&\vdots &\ddots &\vdots&\cdots &\\
				0&\cdots&0            &0 &0
				&0 &\cdots &
				\\
				\vdots&\sunderb{3.5em}{k+l \text{ cells }}&\vdots            &\vdots &\vdots
				&\vdots&\ddots &
			\end{bmatrix} \begin{bmatrix}
				0\\
				\textbf{g}^\ast_{n}(\hbm_{0})\\
				\textbf{g}^\ast_{n}(\hbm_{1})\\
				\vdots
			\end{bmatrix}\\
			&&\qquad-\begin{bmatrix}
				0 & 0& 0
				&0  & \cdots 
				\\
				J_{\textbf{g}^\ast_n} & 0& 0
				&0  & \cdots
				\\
				0&  J_{ \textbf{g}^\ast_n} &0&  0
				&\cdots 
				\\
				\vdots&\vdots &\ddots &\vdots &\vdots 
				\\
				0            &0 &0
				&J_{ \textbf{g}^\ast_n}  
				&
				\cdots %J_{\textbf{g}^\ast_1 } 
				\\
				\vdots            &\vdots &\vdots
				&\cdots 
				&\ddots
			\end{bmatrix} \begin{bmatrix}
				0&\rdelim\}{3}{-2pt}[$k$ cells] \\
				\vdots&\\
				0&\\
				\textbf{g}^\ast_{m}(\hbm_{k+l})&\\
				0&\\
				0&\\
				\vdots&
			\end{bmatrix}\\
			&&\quad=\hbfm\mapsto\begin{bmatrix}
				0&\rdelim\}{3}{-2pt}[$k$ cells] \\
				\vdots&\\
				0&\\
				J_{ \textbf{g}^\ast_m}\textbf{g}^\ast_n(\hbm_{k+l-1})&\\
				0&\\
				\vdots
			\end{bmatrix}-\begin{bmatrix}
				0&\rdelim\}{3}{-2pt}[$k+1$ cells] \\
				\vdots&\\
				0&\\
				J_{ \textbf{g}^\ast_n}\textbf{g}^\ast_m(\hbm_{k+l})&\\
				0&\\
				\vdots
			\end{bmatrix}\\
			&&\quad=\hbfm\mapsto\left[\hbfm\mapsto\mathcal{M}_{\{k\}^c}\left(\vectinf{J_{ \textbf{g}^\ast_m} \textbf{g}^\ast_n(\hbm_{i+l-1})}{i}\right) \right]-\left[\hbfm\mapsto\mathcal{M}_{\{k+1\}^c}\left(\vectinf{J_{ \textbf{g}^\ast_n} \textbf{g}^\ast_m(\hbm_{i+l-1})}{i}\right) \right].
		\end{eqnarray*}

		\item[Proof of identity \eqref{eq:id15} for $-k\leq l\leq 0$]
		\begin{eqnarray*}
			&&\Big[\hbfm\mapsto\mathcal{M}_{\{k\}^c}\left(\vectinf{\textbf{g}^\ast_{m}(\hbm_{i+l})}{i}\right),\hbfm\mapsto \vectinf{\textbf{g}^\ast_{n}(\hbm_{i+1})}{i}\Big]\\
			&&\quad= \hbfm\mapsto\begin{bmatrix}
				0&\cdots&0&  0&0 & 0& \cdots &\rdelim\}{3}{-2pt}[$k$ cells]              \\
				\vdots&\vdots&\vdots&  \vdots&\vdots& \vdots& \cdots &      \\
				0&\cdots&0&  0&0 & 0& \cdots &  \\
				0&\cdots&0 & J_{\textbf{g}^\ast_m}& 0
				&0&\cdots&
				\\
				0&\cdots&0&  0 &0&  0&\cdots&
				\\
				\vdots&\vdots&\vdots&\vdots &\ddots &\vdots&\cdots &\\
				0&\cdots&0            &0 &0
				&0 &\cdots &
				\\
				\vdots&\sunderb{3.5em}{k+l \text{ cells }}&\vdots            &\vdots &\vdots
				&\vdots&\ddots &
			\end{bmatrix}  \begin{bmatrix}
				\textbf{g}^\ast_{n}(\hbm_{1})\\
				\textbf{g}^\ast_{n}(\hbm_{2})\\
				\vdots
			\end{bmatrix}\\
			&&\qquad-\begin{bmatrix}
				0&  J_{ \textbf{g}^\ast_n} &0&  0
				&\cdots 
				\\
				\vdots&\vdots &\ddots &\vdots &\vdots 
				\\
				0            &0 &0
				&J_{ \textbf{g}^\ast_n}  
				&
				\cdots 
				\\
				\vdots            &\vdots &\vdots
				&\cdots 
				&\ddots
				\\
				0&0&0&           \cdots                           &\cdots
				\\
				\vdots &\vdots&\vdots&           \vdots                           &\vdots
			\end{bmatrix} \begin{bmatrix}
				0&\rdelim\}{3}{-2pt}[$k$ cells] \\
				\vdots&\\
				0&\\
				\textbf{g}^\ast_{m}(\hbm_{k+l})&\\
				0&\\
				0&\\
				\vdots&
			\end{bmatrix}\\
			&&\quad=\hbfm\mapsto \begin{bmatrix}
				0&\rdelim\}{3}{-2pt}[$k$ cells] \\
				\vdots&\\
				0&\\
				J_{ \textbf{g}^\ast_m}\textbf{g}^\ast_n(\hbm_{k+l+1})\\
				0\\
				0\\
				\vdots
			\end{bmatrix}-\begin{bmatrix}
				0&\rdelim\}{3}{-2pt}[$k-1$ cells] \\
				\vdots&\\
				0&\\
				J_{ \textbf{g}^\ast_n}\textbf{g}^\ast_m(\hbm_{k+l})&\\
				0&\\
				0&\\
				\vdots
			\end{bmatrix}\\
			&&\quad=\left[\hbfm\mapsto\mathcal{M}_{\{k\}^c}\left(\vectinf{J_{ \textbf{g}^\ast_m} \textbf{g}^\ast_n(\hbm_{i+l+1})}{i}\right) \right]-\left[\hbfm\mapsto\mathcal{M}_{\{k-1\}^c}\left(\vectinf{J_{ \textbf{g}^\ast_n} \textbf{g}^\ast_m(\hbm_{i+l+1})}{i}\right) \right].
		\end{eqnarray*}

		\item[Proof of identity \eqref{eq:id13} for $l>0$]
		\begin{eqnarray*}
			&&\left[\hbfm\mapsto\mathcal{M}_{\{k\}^c}\left(\vectinf{\textbf{g}^\ast_{m}(\hbm_{i+l})}{i}\right),\hbfm\mapsto \vectinf{\textbf{g}^\ast_{n}(\hbm_{i})}{i} \right]\\
			&&\quad= \hbfm\mapsto\begin{bmatrix}
				0&\cdots&0&  0&0 & 0& \cdots &\rdelim\}{3}{-2pt}[$k$ cells]              \\
				\vdots&\vdots&\vdots&  \vdots&\vdots& \vdots& \cdots &      \\
				0&\cdots&0&  0&0 & 0& \cdots &  \\
				0&\cdots&0 & J_{\textbf{g}^\ast_m}& 0
				&0&\cdots&
				\\
				0&\cdots&0&  0 &0&  0&\cdots&
				\\
				\vdots&\vdots&\vdots&\vdots &\ddots &\vdots&\cdots &\\
				0&\cdots&0            &0 &0
				&0 &\cdots &
				\\
				\vdots&\sunderb{3.5em}{k+l \text{ cells }}&\vdots            &\vdots &\vdots
				&\vdots&\ddots &
			\end{bmatrix} \begin{bmatrix}
				\textbf{g}^\ast_{n}(\hbm_{0})\\
				\textbf{g}^\ast_{n}(\hbm_{1})\\
				\vdots
			\end{bmatrix}\\
			&&\qquad-\begin{bmatrix}
				J_{\textbf{g}^\ast_n} & 0& 0
				&0  & \cdots
				\\
				0&  J_{ \textbf{g}^\ast_n} &0&  0
				&\cdots 
				\\
				\vdots&\vdots &\ddots &\vdots &\vdots 
				\\
				0            &0 &0
				&J_{ \textbf{g}^\ast_n}  
				&
				\cdots %J_{\textbf{g}^\ast_1 } 
				\\
				\vdots            &\vdots &\vdots
				&\cdots 
				&\ddots
			\end{bmatrix} \begin{bmatrix}
				0&\rdelim\}{3}{-2pt}[$k$ cells] \\
				\vdots&\\
				0&\\
				\textbf{g}^\ast_{m}(\hbm_{k+l})&\\
				0&\\
				0&\\
				\vdots&
			\end{bmatrix}\\
			&&\quad=\hbfm\mapsto\begin{bmatrix}
				0&\rdelim\}{3}{-2pt}[$k$ cells] \\
				\vdots&\\
				0&\\
				J_{ \textbf{g}^\ast_m}\textbf{g}^\ast_n(\hbm_{k+l})&\\
				0&\\
				\vdots
			\end{bmatrix}-\begin{bmatrix}
				0&\rdelim\}{3}{-2pt}[$k$ cells] \\
				\vdots&\\
				0&\\
				J_{ \textbf{g}^\ast_n}\textbf{g}^\ast_m(\hbm_{k+l})&\\
				0&\\
				\vdots
			\end{bmatrix}\\
			&&\quad=\hbfm\mapsto\left[\hbfm\mapsto\mathcal{M}_{\{k\}^c}\left(\vectinf{[\textbf{g}^\ast_m,\textbf{g}^\ast_n] (\hbm_{i+l})}{i}\right) \right].
		\end{eqnarray*}

		\item[Proof of identity \eqref{eq:id14} for $l>0$]
		\begin{eqnarray*}
			&&\left[\hbfm\mapsto\mathcal{M}_{\{k\}^c}\left(\vectinf{\textbf{g}^\ast_{m}(\hbm_{i+l})}{i}\right),\hbfm\mapsto \vectinf{\textbf{g}^\ast_{n}(\hbm_{i-1})}{i} \right]\\
			&&\quad= \hbfm\mapsto\begin{bmatrix}
				0&\cdots&0&  0&0 & 0& \cdots &\rdelim\}{3}{-2pt}[$k$ cells]              \\
				\vdots&\vdots&\vdots&  \vdots&\vdots& \vdots& \cdots &      \\
				0&\cdots&0&  0&0 & 0& \cdots &  \\
				0&\cdots&0 & J_{\textbf{g}^\ast_m}& 0
				&0&\cdots&
				\\
				0&\cdots&0&  0 &0&  0&\cdots&
				\\
				\vdots&\vdots&\vdots&\vdots &\ddots &\vdots&\cdots &\\
				0&\cdots&0            &0 &0
				&0 &\cdots &
				\\
				\vdots&\sunderb{3.5em}{k+l \text{ cells }}&\vdots            &\vdots &\vdots
				&\vdots&\ddots &
			\end{bmatrix} \begin{bmatrix}
				0\\
				\textbf{g}^\ast_{n}(\hbm_{0})\\
				\textbf{g}^\ast_{n}(\hbm_{1})\\
				\vdots
			\end{bmatrix}\\
			&&\qquad-\begin{bmatrix}
				0 & 0& 0
				&0  & \cdots 
				\\
				J_{\textbf{g}^\ast_n} & 0& 0
				&0  & \cdots
				\\
				0&  J_{ \textbf{g}^\ast_n} &0&  0
				&\cdots 
				\\
				\vdots&\vdots &\ddots &\vdots &\vdots 
				\\
				0            &0 &0
				&J_{ \textbf{g}^\ast_n}  
				&
				\cdots %J_{\textbf{g}^\ast_1 } 
				\\
				\vdots            &\vdots &\vdots
				&\cdots 
				&\ddots
			\end{bmatrix} \begin{bmatrix}
				0&\rdelim\}{3}{-2pt}[$k$ cells] \\
				\vdots&\\
				0&\\
				\textbf{g}^\ast_{m}(\hbm_{k+l})&\\
				0&\\
				0&\\
				\vdots&
			\end{bmatrix}\\
			&&\quad=\hbfm\mapsto\begin{bmatrix}
				0&\rdelim\}{3}{-2pt}[$k$ cells] \\
				\vdots&\\
				0&\\
				J_{ \textbf{g}^\ast_m}\textbf{g}^\ast_n(\hbm_{k+l-1})&\\
				0&\\
				\vdots
			\end{bmatrix}-\begin{bmatrix}
				0&\rdelim\}{3}{-2pt}[$k+1$ cells] \\
				\vdots&\\
				0&\\
				J_{ \textbf{g}^\ast_n}\textbf{g}^\ast_m(\hbm_{k+l})&\\
				0&\\
				\vdots
			\end{bmatrix}\\
			&&\quad=\hbfm\mapsto\left[\hbfm\mapsto\mathcal{M}_{\{k\}^c}\left(\vectinf{J_{ \textbf{g}^\ast_m} \textbf{g}^\ast_n(\hbm_{i+l-1})}{i}\right) \right]-\left[\hbfm\mapsto\mathcal{M}_{\{k+1\}^c}\left(\vectinf{J_{ \textbf{g}^\ast_n} \textbf{g}^\ast_m(\hbm_{i+l-1})}{i}\right) \right].
		\end{eqnarray*}

		\item[Proof of identity \eqref{eq:id15} for $l>0$]
		\begin{eqnarray*}
			&&\Big[\hbfm\mapsto\mathcal{M}_{\{k\}^c}\left(\vectinf{\textbf{g}^\ast_{m}(\hbm_{i+l})}{i}\right),\hbfm\mapsto \vectinf{\textbf{g}^\ast_{n}(\hbm_{i+1})}{i}\Big]\\
			&&\quad= \hbfm\mapsto\begin{bmatrix}
				0&\cdots&0&  0&0 & 0& \cdots &\rdelim\}{3}{-2pt}[$k$ cells]              \\
				\vdots&\vdots&\vdots&  \vdots&\vdots& \vdots& \cdots &      \\
				0&\cdots&0&  0&0 & 0& \cdots &  \\
				0&\cdots&0 & J_{\textbf{g}^\ast_m}& 0
				&0&\cdots&
				\\
				0&\cdots&0&  0 &0&  0&\cdots&
				\\
				\vdots&\vdots&\vdots&\vdots &\ddots &\vdots&\cdots &\\
				0&\cdots&0            &0 &0
				&0 &\cdots &
				\\
				\vdots&\sunderb{3.5em}{k+l \text{ cells }}&\vdots            &\vdots &\vdots
				&\vdots&\ddots &
			\end{bmatrix}  \begin{bmatrix}
				\textbf{g}^\ast_{n}(\hbm_{1})\\
				\textbf{g}^\ast_{n}(\hbm_{2})\\
				\vdots
			\end{bmatrix}\\
			&&\qquad-\begin{bmatrix}
				0&  J_{ \textbf{g}^\ast_n} &0&  0
				&\cdots 
				\\
				\vdots&\vdots &\ddots &\vdots &\vdots 
				\\
				0            &0 &0
				&J_{ \textbf{g}^\ast_n}  
				&
				\cdots 
				\\
				\vdots            &\vdots &\vdots
				&\cdots 
				&\ddots
				\\
				0&0&0&           \cdots                           &\cdots
				\\
				\vdots &\vdots&\vdots&           \vdots                           &\vdots
			\end{bmatrix} \begin{bmatrix}
				0&\rdelim\}{3}{-2pt}[$k$ cells] \\
				\vdots&\\
				0&\\
				\textbf{g}^\ast_{m}(\hbm_{k+l})&\\
				0&\\
				0&\\
				\vdots&
			\end{bmatrix}\\
			&&\quad=\hbfm\mapsto \begin{bmatrix}
				0&\rdelim\}{3}{-2pt}[$k$ cells] \\
				\vdots&\\
				0&\\
				J_{ \textbf{g}^\ast_m}\textbf{g}^\ast_n(\hbm_{k+l+1})\\
				0\\
				0\\
				\vdots
			\end{bmatrix}-\begin{bmatrix}
				0&\rdelim\}{3}{-2pt}[$k-1$ cells] \\
				\vdots&\\
				0&\\
				J_{ \textbf{g}^\ast_n}\textbf{g}^\ast_m(\hbm_{k+l})&\\
				0&\\
				0&\\
				\vdots
			\end{bmatrix}\\
			&&\quad=\left[\hbfm\mapsto\mathcal{M}_{\{k\}^c}\left(\vectinf{J_{ \textbf{g}^\ast_m} \textbf{g}^\ast_n(\hbm_{i+l+1})}{i}\right) \right]-\left[\hbfm\mapsto\mathcal{M}_{\{k-1\}^c}\left(\vectinf{J_{ \textbf{g}^\ast_n} \textbf{g}^\ast_m(\hbm_{i+l+1})}{i}\right) \right].
		\end{eqnarray*}
		
	\end{description}
	
\end{proof}

\section{The Lie algebra generated by different sets $\mathcal{K}$ }\label{app:expandLie}

In this section, we analyse the Lie algebras generated by low-mode forcing.  
In the first Lemma, we consider a low mode forcing where the control over the constant function is included. To be more precise, we consider the set $\mathcal{K}$ given by
$\{ (0,1),(0,2),(1,1)\}$

	\begin{align}
		\nonumber&\mathcal{H}_{\textbf{g}_l^\ast,2k}(\hbfm)\\
		&:=\begin{pmatrix}
			\left(\combin{2k}{k}-\combin{2k}{k+1}\right)\textbf{g}_l^\ast(\hbm_{0})+\left(\combin{2k}{k+1}-\combin{2k}{k+2}\right)\textbf{g}_l^\ast(\hbm_{2})+\cdots+\left(\combin{2k}{2k}-\combin{2k}{2k+1}\right)\textbf{g}_l^\ast(\hbm_{2k})\\
			\left(\combin{2k}{k}-\combin{2k}{k+2}\right)\textbf{g}_l^\ast(\hbm_{1})+\left(\combin{2k}{k+1}-\combin{2k}{k+3}\right)\textbf{g}_l^\ast(\hbm_{3})+\cdots+\left(\combin{2k}{2k}-\combin{2k}{2k+2}\right)\textbf{g}_l^\ast(\hbm_{2k+1})\\
			\left(\combin{2k}{k-1}-\combin{2k}{k+2}\right)\textbf{g}_l^\ast(\hbm_{0})+\left(\combin{2k}{k}-\combin{2k}{k+3}\right)\textbf{g}_l^\ast(\hbm_{2})+\cdots+\left(\combin{2k}{2k}-\combin{2k}{2k+3}\right)\textbf{g}_l^\ast(\hbm_{2k+2})\\
			\left(\combin{2k}{k-1}-\combin{2k}{k+3}\right)\textbf{g}_l^\ast(\hbm_{1})+\left(\combin{2k}{k}-\combin{2k}{k+4}\right)\textbf{g}_l^\ast(\hbm_{3})+\cdots+\left(\combin{2k}{2k}-\combin{2k}{2k+4}\right)\textbf{g}_l^\ast(\hbm_{2k+3})\\
			\vdots\\
			\left(\combin{2k}{2}-\combin{2k}{2k-1}\right)\textbf{g}_l^\ast(\hbm_{0})+\left(\combin{2k}{3}-\combin{2k}{2k}\right)\textbf{g}_l^\ast(\hbm_{2})+\cdots+\left(\combin{2k}{2k}-\combin{2k}{4k-3}\right)\textbf{g}_l^\ast(\hbm_{4k-4})\\
			\left(\combin{2k}{2}-\combin{2k}{2k}\right)\textbf{g}_l^\ast(\hbm_{1})+\left(\combin{2k}{3}-\combin{2k}{2k+1}\right)\textbf{g}_l^\ast(\hbm_{3})+\cdots+\left(\combin{2k}{2k}-\combin{2k}{4k-2}\right)\textbf{g}_l^\ast(\hbm_{4k-3})\\
			\vdots
		\end{pmatrix}
		\label{eq:hgmsimple}
	\end{align}
	and
	\begin{align}
		\nonumber&\mathcal{H}_{\textbf{g}_l^\ast,2k+1}(\hbfm)\\
		&:=\begin{pmatrix}
			\left(\combin{2k+1}{k+1}-\combin{2k+1}{k+2}\right)\textbf{g}_l^\ast(\hbm_{1})+\left(\combin{2k+1}{k+2}-\combin{2k+1}{k+3}\right)\textbf{g}_l^\ast(\hbm_{3})+\cdots+\left(\combin{2k+1}{2k+1}-\combin{2k+1}{2k+2}\right)\textbf{g}_l^\ast(\hbm_{2k+1})\\
			\left(\combin{2k+1}{k}-\combin{2k+1}{k+2}\right)\textbf{g}_l^\ast(\hbm_{0})+\left(\combin{2k+1}{k+1}-\combin{2k+1}{k+3}\right)\textbf{g}_l^\ast(\hbm_{2})+\cdots+\left(\combin{2k+1}{2k+1}-\combin{2k+1}{2k+3}\right)\textbf{g}_l^\ast(\hbm_{2k+2})\\
			\left(\combin{2k+1}{k}-\combin{2k+1}{k+3}\right)\textbf{g}_l^\ast(\hbm_{1})+\left(\combin{2k+1}{k+1}-\combin{2k+1}{k+4}\right)\textbf{g}_l^\ast(\hbm_{3})+\cdots+\left(\combin{2k+1}{2k+1}-\combin{2k+1}{2k+4}\right)\textbf{g}_l^\ast(\hbm_{2k+3})\\
			\left(\combin{2k+1}{k-1}-\combin{2k+1}{k+3}\right)\textbf{g}_l^\ast(\hbm_{0})+\left(\combin{2k+1}{k}-\combin{2k+1}{k+4}\right)\textbf{g}_l^\ast(\hbm_{2})+\cdots+\left(\combin{2k+1}{2k+1}-\combin{2k+1}{2k+5}\right)\textbf{g}_l^\ast(\hbm_{2k+4})\\
			\vdots\\
			\left(\combin{2k+1}{2}-\combin{2k+1}{2k+1}\right)\textbf{g}_l^\ast(\hbm_{1})+\left(\combin{2k+1}{3}-\combin{2k+1}{2k+2}\right)\textbf{g}_l^\ast(\hbm_{2})+\cdots+\left(\combin{2k+1}{2k+1}-\combin{2k+1}{4k}\right)\textbf{g}_l^\ast(\hbm_{4k-1})\\
			\left(\combin{2k+1}{1}-\combin{2k+1}{2k+1}\right)\textbf{g}_l^\ast(\hbm_{0})+\left(\combin{2k+1}{2}-\combin{2k+1}{2k+2}\right)\textbf{g}_l^\ast(\hbm_{2})+\cdots+\left(\combin{2k+1}{2k+1}-\combin{2k+1}{4k+1}\right)\textbf{g}_l^\ast(\hbm_{4k})\\
			\vdots
		\end{pmatrix}
		\label{eq:hgmsimple2}
	\end{align}

%\end{appendix}
%\end{document}

	\begin{lem}\label{lem:identityH}
		The following identities hold for $k\geq 2$ and $m,n\in \{1,2,3\}$:
		\begin{eqnarray}
			\label{eq:idHgm1}\left[\mathcal{H}_{\textbf{g}^\ast_{n},2k},\sqrt{2}\mathbf{G}^{0,m}\right]&=&\mathcal{H}_{\left[\textbf{g}^\ast_{n},\textbf{g}^\ast_{m}\right],2k},\quad 
			\label{eq:idHgm3}\left[\mathcal{H}_{\textbf{g}^\ast_{n},2k},\sqrt{2}\mathbf{G}^{1,m}\right]=\mathcal{H}_{\left[\textbf{g}^\ast_{n},\textbf{g}^\ast_{m}\right],2k+1},\\
			\label{eq:idHgm4}\left[\mathcal{H}_{\textbf{g}^\ast_{n},2k+1},\sqrt{2}\mathbf{G}^{0,m}\right]&=&\mathcal{H}_{\left[\textbf{g}^\ast_{n},\textbf{g}^\ast_{m}\right],2k+1},\quad 
			\label{eq:idHgm6}\left[\mathcal{H}_{\textbf{g}^\ast_{n},2k+1},\sqrt{2}\mathbf{G}^{1,m}\right]=\mathcal{H}_{\left[\textbf{g}^\ast_{n},\textbf{g}^\ast_{m}\right],2k+2}.
		\end{eqnarray}
	\end{lem}
%	\begin{proof}
%		The proof is given in Appendix \ref{sec:prooflem}.
%	\end{proof}
	From Lemma \ref{lem:identityH}, it is straightforward to verify that
		\begin{eqnarray*}
			\left[\mathcal{H}_{\textbf{g}^\ast_{3},2n},h_{0,1}\right]=-\mathcal{H}_{\textbf{g}^\ast_{2},2n},
			\quad  &\quad &
				\left[\mathcal{H}_{\textbf{g}^\ast_{3},2n+1},h_{0,2}\right]= \mathcal{H}_{\textbf{g}^\ast_{1},2n+1}
				\\
			\left[\mathcal{H}_{\textbf{g}^\ast_{3},2n+1},h_{0,1}\right]=-\mathcal{H}_{\textbf{g}^\ast_{2},2n+1},
				\quad  &\quad & 	\left[\mathcal{H}_{\textbf{g}^\ast_{3},2n},h_{0,3}\right]=-\mathcal{H}_{\textbf{g}^\ast_{2},2n+1},
				\\
			\left[\mathcal{H}_{\textbf{g}^\ast_{3},2n},h_{0,2}\right]=\mathcal{H}_{\textbf{g}^\ast_{1},2n},
		\quad  &\quad & 	
			\left[\mathcal{H}_{\textbf{g}^\ast_{3},2n+1},h_{0,3}\right]= -\mathcal{H}_{\textbf{g}^\ast_{2},2n+2}.
		\end{eqnarray*}

		Using the identity
		$$\combin{n}{k}=\combin{n-1}{k-1}+\combin{n-1}{k},$$
		it is straightforward to verify that
		\begin{eqnarray*}
			K_3+I_3+J_2&=& \Bigg[\hbfm\mapsto\sum_{j=0}^{k-1}\combin{2k+1}{j} \Big[\mathcal{M}_{\{2(k-j)\}^c}
			\left(\vectinf{[\textbf{g}^\ast_{n},\textbf{g}^\ast_{m}](\hbm_{i-2(k-j)})}{i}\right)\Big] \Bigg]\\
			&&+\Bigg[\hbfm\mapsto\combin{2k+1}{k} \Big[\mathcal{M}_{\{0\}^c}
			\left(\vectinf{J_{ \textbf{g}^\ast_n}\textbf{g}^\ast_{m}(\hbm_{i})}{i}\right)\Big] \Bigg],\\
			K_4+J_3&=& \Bigg[\hbfm\mapsto\sum_{j=0}^{k-1}\combin{2k}{j}\Big[\mathcal{M}_{\{2(k-j)-1\}^c}
			\left(\vectinf{[\textbf{g}^\ast_{n},\textbf{g}^\ast_{m}](\hbm_{i-2(k-j)+2})}{i}\right)\\
			&&\qquad+\mathcal{M}_{\{2(k-j)-2\}^c}
			\left(\vectinf{[\textbf{g}^\ast_{n},\textbf{g}^\ast_{m}](\hbm_{i-2(k-j)+2})}{i}\right)\Big] \Bigg],\\
			K_6+I_2+J_4&=&\Bigg[\hbfm\mapsto\sum_{j=k+1}^{2k}\combin{2k}{j+1} \mathcal{M}_{\{0\}^c}
			\left(\vectinf{J_{ \textbf{g}^\ast_{n}}\textbf{g}^\ast_{m}(\hbm_{i+2(j-k)})}{i}\right) \Bigg]\\
			&&-\left[\hbfm\mapsto\sum_{j=k+1}^{2k+1}\combin{2k+1}{j}\mathcal{M}_{\{0\}^c}\left(\vectinf{J_{ \textbf{g}^\ast_{m}}\textbf{g}^\ast_{n}(\hbm_{i-2k+2j-2})}{i}\right)\right]\\
			&&+\left[\hbfm\mapsto\sum_{j=k+1}^{2k}\combin{2k}{j} \mathcal{M}_{\{0\}^c}\left(\vectinf{J_{ \textbf{g}^\ast_{n}}\textbf{g}^\ast_{m}(\hbm_{i+2(j-k)})}{i}\right)
			\right]\\
			&=& \Bigg[\hbfm\mapsto\sum_{j=k+1}^{2k}\combin{2k}{j+1} \mathcal{M}_{\{0\}^c}
			\left(\vectinf{J_{ \textbf{g}^\ast_{n}}\textbf{g}^\ast_{m}(\hbm_{i+2(j-k)})}{i}\right) \Bigg]\\
			&&-\left[\hbfm\mapsto\sum_{j=k}^{2k}\combin{2k+1}{j+1}\mathcal{M}_{\{0\}^c}\left(\vectinf{J_{ \textbf{g}^\ast_{m}}\textbf{g}^\ast_{n}(\hbm_{i+2(j-k)})}{i}\right)\right]\\
			&&+\left[\hbfm\mapsto\sum_{j=k+1}^{2k}\combin{2k}{j} \mathcal{M}_{\{0\}^c}\left(\vectinf{J_{ \textbf{g}^\ast_{n}}\textbf{g}^\ast_{m}(\hbm_{i+2(j-k)})}{i}\right)
			\right]\\
			&=&- \Bigg[\hbfm\mapsto\combin{2k+1}{k+1}\Big[\mathcal{M}_{\{0\}^c}
			\left(\vectinf{J_{ \textbf{g}^\ast_{m}}\textbf{g}^\ast_{n}(\hbm_{i})}{i}\right)\Big] \Bigg]\\
			&&+\left[\hbfm\mapsto\sum_{j=k+1}^{2k}\combin{2k+1}{j+1}\mathcal{M}_{\{0\}^c}\left(\vectinf{[\textbf{g}^\ast_{n},\textbf{g}^\ast_{m}](\hbm_{i+2(j-k)})}{i}\right)\right],\\
			K_5+K_7+I_4&=& \Bigg[\hbfm\mapsto\sum_{j=k+2}^{2k}\combin{2k}{j} \Big[\mathcal{M}_{\{1\}^c}
			\left(\vectinf{[\textbf{g}^\ast_{n},\textbf{g}^\ast_{m}](\hbm_{i+2(j-k)-2})}{i}\right)\\
			&&\qquad+\mathcal{M}_{\{0\}^c}
			\left(\vectinf{[\textbf{g}^\ast_{n},\textbf{g}^\ast_{m}](\hbm_{i+2(j-k)-2})}{i}\right)\Big] \Bigg]
		\end{eqnarray*}
		Thus, we get
		\begin{eqnarray}
			\nonumber\lqq{\left[\mathcal{H}_{\textbf{g}^\ast_{n},2k},h_{0,3}\right]}\\
			\nonumber&=&I_1-I_2-I_3-I_4-I_5-I_6+J_1-J_2-J_3-J_4-J_5-J_6\\
			\nonumber&=& I_1-I_2-I_3-I_4+J_1-J_2-J_3-J_4-K_1-K_2-K_3-K_4-K_5-K_6-K_7\\
			\nonumber&=& \left[\hbfm\mapsto\sum_{j=0}^{2k+2}\combin{2k+2}{j}\vectinf{[\textbf{g}^\ast_{n},\textbf{g}^\ast_{m}](\hbm_{i-2k+2j-2})}{i}\right]\\
			\nonumber&&-\Bigg[\hbfm\mapsto\sum_{j=1}^{k-1}\sum_{m=1}^{2(k-j)-1}\combin{2k}{j-1} \Big[\mathcal{M}_{\{m-1\}^c}
			\left(\vectinf{[\textbf{g}^\ast_{n},\textbf{g}^\ast_{m}](\hbm_{i+2(k-j-m)})}{i}\right)\\
			\nonumber&&\qquad+\mathcal{M}_{\{m\}^c}
			\left(\vectinf{[\textbf{g}^\ast_{n},\textbf{g}^\ast_{m}](\hbm_{i+2(k-j-m)})}{i}\right)\Big] \Bigg]\\
			\nonumber&&-\Bigg[\hbfm\mapsto\sum_{j=1}^{k-1}\sum_{m=1}^{2(k-j)-1}\combin{2k}{j-1}\Big[\mathcal{M}_{\{m\}^c}
			\left(\vectinf{[\textbf{g}^\ast_{n},\textbf{g}^\ast_{m}](\hbm_{i+2(k-j-m)})}{i}\right)\\
			\nonumber&&\qquad+\mathcal{M}_{\{m+1\}^c}
			\left(\vectinf{[\textbf{g}^\ast_{n},\textbf{g}^\ast_{m}](\hbm_{i+2(k-j-m)})}{i}\right)\Big] \Bigg]\\
			\nonumber&&-\Bigg[\hbfm\mapsto\sum_{j=0}^{k-1}\combin{2k+1}{j} \Big[\mathcal{M}_{\{2(k-j)\}^c}
			\left(\vectinf{[\textbf{g}^\ast_{n},\textbf{g}^\ast_{m}](\hbm_{i-2(k-j)})}{i}\right)\Big] \Bigg]\\
			\nonumber&&-\Bigg[\hbfm\mapsto\combin{2k+1}{k} \Big[\mathcal{M}_{\{0\}^c}
			\left(\vectinf{[\textbf{g}^\ast_{n},\textbf{g}^\ast_{m}](\hbm_{i})}{i}\right)\Big] \Bigg]\\
			\nonumber&&-\Bigg[\hbfm\mapsto\sum_{j=0}^{k-1}\combin{2k}{j}\Big[\mathcal{M}_{\{2(k-j)-1\}^c}
			\left(\vectinf{[\textbf{g}^\ast_{n},\textbf{g}^\ast_{m}](\hbm_{i-2(k-j)+2})}{i}\right)\\
			\nonumber&&\qquad+\mathcal{M}_{\{2(k-j)-2\}^c}
			\left(\vectinf{[\textbf{g}^\ast_{n},\textbf{g}^\ast_{m}](\hbm_{i-2(k-j)+2})}{i}\right)\Big] \Bigg]\\
			\nonumber&&-\left[\hbfm\mapsto\sum_{j=k+1}^{2k}\combin{2k+1}{j+1}\mathcal{M}_{\{0\}^c}\left(\vectinf{[\textbf{g}^\ast_{n},\textbf{g}^\ast_{m}](\hbm_{i+2(j-k)})}{i}\right)\right]\\
			\nonumber&&-\Bigg[\hbfm\mapsto\sum_{j=k+2}^{2k}\combin{2k}{j} \Big[\mathcal{M}_{\{1\}^c}
			\left(\vectinf{[\textbf{g}^\ast_{n},\textbf{g}^\ast_{m}](\hbm_{i+2(j-k)-2})}{i}\right)\\
			\nonumber&&\qquad+\mathcal{M}_{\{0\}^c}
			\left(\vectinf{[\textbf{g}^\ast_{n},\textbf{g}^\ast_{m}](\hbm_{i+2(j-k)-2})}{i}\right)\Big] \Bigg]\\
			\nonumber&=& \left[\hbfm\mapsto\sum_{j=0}^{2k+2}\combin{2k+2}{j}\vectinf{[\textbf{g}^\ast_{n},\textbf{g}^\ast_{m}](\hbm_{i-2k+2j-2})}{i}\right]\\
			\nonumber&&-\Bigg[\hbfm\mapsto\sum_{j=0}^{k-1}\combin{2k+1}{j} \Big[\mathcal{M}_{\{2(k-j)\}^c}
			\left(\vectinf{[\textbf{g}^\ast_{n},\textbf{g}^\ast_{m}](\hbm_{i-2(k-j)})}{i}\right)\Big] \Bigg]\\
			\nonumber&&-\Bigg[\hbfm\mapsto\combin{2k+1}{k} \Big[\mathcal{M}_{\{0\}^c}
			\left(\vectinf{[\textbf{g}^\ast_{n},\textbf{g}^\ast_{m}](\hbm_{i})}{i}\right)\Big] \Bigg]\\
			\nonumber&&-\left[\hbfm\mapsto\sum_{j=k+1}^{2k}\combin{2k+1}{j+1}\mathcal{M}_{\{0\}^c}\left(\vectinf{[\textbf{g}^\ast_{n},\textbf{g}^\ast_{m}](\hbm_{i+2(j-k)})}{i}\right)\right]\\
			\nonumber&&-\Bigg[\hbfm\mapsto\sum_{j=1}^{k-1}\sum_{m=1}^{2(k-j)-1}\combin{2k}{j-1} \Big[\mathcal{M}_{\{m-1\}^c}
			\left(\vectinf{[\textbf{g}^\ast_{n},\textbf{g}^\ast_{m}](\hbm_{i+2(k-j-m)})}{i}\right)\\
			\nonumber&&\qquad+\mathcal{M}_{\{m\}^c}
			\left(\vectinf{[\textbf{g}^\ast_{n},\textbf{g}^\ast_{m}](\hbm_{i+2(k-j-m)})}{i}\right)\Big] \Bigg]\\
			\nonumber&&-\Bigg[\hbfm\mapsto\sum_{j=1}^{k-1}\sum_{m=1}^{2(k-j)-1}\combin{2k}{j-1}\Big[\mathcal{M}_{\{m\}^c}
			\left(\vectinf{[\textbf{g}^\ast_{n},\textbf{g}^\ast_{m}](\hbm_{i+2(k-j-m)})}{i}\right)\\
			\nonumber&&\qquad+\mathcal{M}_{\{m+1\}^c}
			\left(\vectinf{[\textbf{g}^\ast_{n},\textbf{g}^\ast_{m}](\hbm_{i+2(k-j-m)})}{i}\right)\Big] \Bigg]\\
			\nonumber&&-\Bigg[\hbfm\mapsto\sum_{j=0}^{k-1}\combin{2k}{j}\Big[\mathcal{M}_{\{2(k-j)-1\}^c}
			\left(\vectinf{[\textbf{g}^\ast_{n},\textbf{g}^\ast_{m}](\hbm_{i-2(k-j)+2})}{i}\right)\\
			\nonumber&&\qquad+\mathcal{M}_{\{2(k-j)-2\}^c}
			\left(\vectinf{[\textbf{g}^\ast_{n},\textbf{g}^\ast_{m}](\hbm_{i-2(k-j)+2})}{i}\right)\Big] \Bigg]\\
			\nonumber&&-\Bigg[\hbfm\mapsto\sum_{j=k+2}^{2k}\combin{2k}{j} \Big[\mathcal{M}_{\{1\}^c}
			\left(\vectinf{[\textbf{g}^\ast_{n},\textbf{g}^\ast_{m}](\hbm_{i+2(j-k)-2})}{i}\right)\\
			\nonumber&&\qquad+\mathcal{M}_{\{0\}^c}
			\left(\vectinf{[\textbf{g}^\ast_{n},\textbf{g}^\ast_{m}](\hbm_{i+2(j-k)-2})}{i}\right)\Big] \Bigg]\\
			\nonumber&=& \left[\hbfm\mapsto\sum_{j=0}^{2k+2}\combin{2k+2}{j}\vectinf{[\textbf{g}^\ast_{n},\textbf{g}^\ast_{m}](\hbm_{i-2k+2j-2})}{i}\right]\\
			\nonumber&&-\Bigg[\hbfm\mapsto\sum_{j=0}^{k}\combin{2k+1}{j} \Big[\mathcal{M}_{\{2(k-j)\}^c}
			\left(\vectinf{[\textbf{g}^\ast_{n},\textbf{g}^\ast_{m}](\hbm_{i-2(k-j)})}{i}\right)\Big] \Bigg]\\
			\nonumber&&-\left[\hbfm\mapsto\sum_{j=k+2}^{2k+1}\combin{2k+1}{j}\mathcal{M}_{\{0\}^c}\left(\vectinf{[\textbf{g}^\ast_{n},\textbf{g}^\ast_{m}](\hbm_{i+2(j-1-k)})}{i}\right)\right]\\
			\nonumber&&-\Bigg[\hbfm\mapsto\sum_{j=1}^{k-1}\sum_{m=1}^{2(k-j)-1}\combin{2k}{j-1} \Big[\mathcal{M}_{\{m-1\}^c}
			\left(\vectinf{[\textbf{g}^\ast_{n},\textbf{g}^\ast_{m}](\hbm_{i+2(k-j-m)})}{i}\right)\\
			\nonumber&&\qquad+\mathcal{M}_{\{m\}^c}
			\left(\vectinf{[\textbf{g}^\ast_{n},\textbf{g}^\ast_{m}](\hbm_{i+2(k-j-m)})}{i}\right)\Big] \Bigg]\\
			\nonumber&&-\Bigg[\hbfm\mapsto\sum_{j=1}^{k-1}\sum_{m=1}^{2(k-j)-1}\combin{2k}{j-1}\Big[\mathcal{M}_{\{m\}^c}
			\left(\vectinf{[\textbf{g}^\ast_{n},\textbf{g}^\ast_{m}](\hbm_{i+2(k-j-m)})}{i}\right)\\
			\nonumber&&\qquad+\mathcal{M}_{\{m+1\}^c}
			\left(\vectinf{[\textbf{g}^\ast_{n},\textbf{g}^\ast_{m}](\hbm_{i+2(k-j-m)})}{i}\right)\Big] \Bigg]\\
			\nonumber&&-\Bigg[\hbfm\mapsto\sum_{j=0}^{k-1}\combin{2k}{j}\Big[\mathcal{M}_{\{2(k-j)-1\}^c}
			\left(\vectinf{[\textbf{g}^\ast_{n},\textbf{g}^\ast_{m}](\hbm_{i-2(k-j)+2})}{i}\right)\\
			\nonumber&&\qquad+\mathcal{M}_{\{2(k-j)-2\}^c}
			\left(\vectinf{[\textbf{g}^\ast_{n},\textbf{g}^\ast_{m}](\hbm_{i-2(k-j)+2})}{i}\right)\Big] \Bigg]\\
			\nonumber&&-\Bigg[\hbfm\mapsto\sum_{j=k+2}^{2k}\combin{2k}{j} \Big[\mathcal{M}_{\{1\}^c}
			\left(\vectinf{[\textbf{g}^\ast_{n},\textbf{g}^\ast_{m}](\hbm_{i+2(j-k)-2})}{i}\right)\\
			&&\qquad+\mathcal{M}_{\{0\}^c}
			\left(\vectinf{[\textbf{g}^\ast_{n},\textbf{g}^\ast_{m}](\hbm_{i+2(j-k)-2})}{i}\right)\Big] \Bigg]=:L_1+L_2+\cdots+L_7.\label{eq:a1}
		\end{eqnarray}
		Observe that
		\begin{eqnarray}
			\nonumber L_4&=& \Bigg[\hbfm\mapsto\sum_{j=2}^{k}\sum_{m=0}^{2(k-j)}\combin{2k}{j-2} \Big[\mathcal{M}_{\{m\}^c}
			\left(\vectinf{[\textbf{g}^\ast_{n},\textbf{g}^\ast_{m}](\hbm_{i+2(k-j-m)})}{i}\right)\\
			\nonumber&&\qquad+\mathcal{M}_{\{m+1\}^c}
			\left(\vectinf{[\textbf{g}^\ast_{n},\textbf{g}^\ast_{m}](\hbm_{i+2(k-j-m)})}{i}\right)\Big] \Bigg]\\
			&=&\nonumber\Bigg[\hbfm\mapsto\sum_{j=2}^{k-1}\sum_{m=0}^{2(k-j)}\combin{2k}{j-2} \Big[\mathcal{M}_{\{m\}^c}
			\left(\vectinf{[\textbf{g}^\ast_{n},\textbf{g}^\ast_{m}](\hbm_{i+2(k-j-m)})}{i}\right)\\
			\nonumber&&\qquad+\mathcal{M}_{\{m+1\}^c}
			\left(\vectinf{[\textbf{g}^\ast_{n},\textbf{g}^\ast_{m}](\hbm_{i+2(k-j-m)})}{i}\right)\Big] \Bigg]\\
			\nonumber&&+\Bigg[\hbfm\mapsto\combin{2k}{k-2} \Big[\mathcal{M}_{\{0\}^c}
			\left(\vectinf{[\textbf{g}^\ast_{n},\textbf{g}^\ast_{m}](\hbm_{i})}{i}\right)+\mathcal{M}_{\{m+1\}^c}
			\left(\vectinf{[\textbf{g}^\ast_{n},\textbf{g}^\ast_{m}](\hbm_{i})}{i}\right)\Big] \Bigg]\\\label{eq:a2}	
		\end{eqnarray}
		and
		\begin{eqnarray*}
			L_5&=&\Bigg[\hbfm\mapsto\sum_{j=1}^{k-1}\sum_{m=0}^{2(k-j)}\combin{2k}{j-1}\Big[\mathcal{M}_{\{m\}^c}
			\left(\vectinf{[\textbf{g}^\ast_{n},\textbf{g}^\ast_{m}](\hbm_{i+2(k-j-m)})}{i}\right)\\
			\nonumber&&\qquad+\mathcal{M}_{\{m+1\}^c}
			\left(\vectinf{[\textbf{g}^\ast_{n},\textbf{g}^\ast_{m}](\hbm_{i+2(k-j-m)})}{i}\right)\Big] \Bigg]\\
			&&-\Bigg[\hbfm\mapsto\sum_{j=1}^{k-1}\combin{2k}{j-1} \Big[\mathcal{M}_{\{0\}^c}
			\left(\vectinf{[\textbf{g}^\ast_{n},\textbf{g}^\ast_{m}](\hbm_{i+2(k-j)})}{i}\right)\\
			&&\qquad+\mathcal{M}_{\{1\}^c}
			\left(\vectinf{[\textbf{g}^\ast_{n},\textbf{g}^\ast_{m}](\hbm_{i+2(k-j)})}{i}\right)\Big] \Bigg]	\\
			&&-\Bigg[\hbfm\mapsto\sum_{j=1}^{k-1}\combin{2k}{j-1} \Big[\mathcal{M}_{\{2(k-j)\}^c}
			\left(\vectinf{[\textbf{g}^\ast_{n},\textbf{g}^\ast_{m}](\hbm_{i-2(k-j)})}{i}\right)\\
			&&\qquad+\mathcal{M}_{\{2(k-j)+1\}^c}
			\left(\vectinf{[\textbf{g}^\ast_{n},\textbf{g}^\ast_{m}](\hbm_{i-2(k-j)})}{i}\right)\Big] \Bigg]	\\
		\end{eqnarray*}
		Since 
		\begin{eqnarray*}
			&&\Bigg[\hbfm\mapsto\sum_{j=1}^{k-1}\combin{2k}{j-1} \Big[\mathcal{M}_{\{0\}^c}
			\left(\vectinf{[\textbf{g}^\ast_{n},\textbf{g}^\ast_{m}](\hbm_{i+2(k-j)})}{i}\right)\\
			&&\qquad+\mathcal{M}_{\{1\}^c}
			\left(\vectinf{[\textbf{g}^\ast_{n},\textbf{g}^\ast_{m}](\hbm_{i+2(k-j)})}{i}\right)\Big] \Bigg]	\\
			&&=\Bigg[\hbfm\mapsto\sum_{j=k+2}^{2k}\combin{2k}{j} \Big[\mathcal{M}_{\{0\}^c}
			\left(\vectinf{[\textbf{g}^\ast_{n},\textbf{g}^\ast_{m}](\hbm_{i+2(j-k)-2})}{i}\right)\\
			&&\qquad+\mathcal{M}_{\{1\}^c}
			\left(\vectinf{[\textbf{g}^\ast_{n},\textbf{g}^\ast_{m}](\hbm_{i+2(j-k)-2})}{i}\right)\Big] \Bigg]=L_7
		\end{eqnarray*} and
		\begin{eqnarray*}
			&&\Bigg[\hbfm\mapsto\sum_{j=1}^{k-1}\combin{2k}{j-1} \Big[\mathcal{M}_{\{2(k-j)\}^c}
			\left(\vectinf{[\textbf{g}^\ast_{n},\textbf{g}^\ast_{m}](\hbm_{i-2(k-j)})}{i}\right)\\
			&&\qquad+\mathcal{M}_{\{2(k-j)+1\}^c}
			\left(\vectinf{[\textbf{g}^\ast_{n},\textbf{g}^\ast_{m}](\hbm_{i-2(k-j)})}{i}\right)\Big] \Bigg]	\\
			&&=\Bigg[\hbfm\mapsto\sum_{j=0}^{k-2}\combin{2k}{j} \Big[\mathcal{M}_{\{2(k-j)-2\}^c}
			\left(\vectinf{[\textbf{g}^\ast_{n},\textbf{g}^\ast_{m}](\hbm_{i-2(k-j)+2})}{i}\right)\\
			&&\qquad+\mathcal{M}_{\{2(k-j)-1\}^c}
			\left(\vectinf{[\textbf{g}^\ast_{n},\textbf{g}^\ast_{m}](\hbm_{i-2(k-j)+2})}{i}\right)\Big] \Bigg] \\
			&&=\Bigg[\hbfm\mapsto\sum_{j=0}^{k-1}\combin{2k}{j} \Big[\mathcal{M}_{\{2(k-j)-2\}^c}
			\left(\vectinf{[\textbf{g}^\ast_{n},\textbf{g}^\ast_{m}](\hbm_{i-2(k-j)+2})}{i}\right)\\
			&&\qquad+\mathcal{M}_{\{2(k-j)-1\}^c}
			\left(\vectinf{[\textbf{g}^\ast_{n},\textbf{g}^\ast_{m}](\hbm_{i-2(k-j)+2})}{i}\right)\Big] \Bigg]\\
			&&-\Bigg[\hbfm\mapsto\combin{2k}{k-1} \Big[\mathcal{M}_{\{0\}^c}
			\left(\vectinf{[\textbf{g}^\ast_{n},\textbf{g}^\ast_{m}](\hbm_{i})}{i}\right)+\mathcal{M}_{\{1\}^c}
			\left(\vectinf{[\textbf{g}^\ast_{n},\textbf{g}^\ast_{m}](\hbm_{i})}{i}\right)\Big] \Bigg]\\
			&=&L_6-\Bigg[\hbfm\mapsto\combin{2k}{k-1} \Big[\mathcal{M}_{\{0\}^c}
			\left(\vectinf{[\textbf{g}^\ast_{n},\textbf{g}^\ast_{m}](\hbm_{i})}{i}\right)+\mathcal{M}_{\{1\}^c}
			\left(\vectinf{[\textbf{g}^\ast_{n},\textbf{g}^\ast_{m}](\hbm_{i})}{i}\right)\Big] \Bigg].
		\end{eqnarray*}
		we obtain
		\begin{eqnarray}
			\nonumber L_5&=&\Bigg[\hbfm\mapsto\sum_{j=1}^{k-1}\sum_{m=0}^{2(k-j)}\combin{2k}{j-1}\Big[\mathcal{M}_{\{m\}^c}
			\left(\vectinf{[\textbf{g}^\ast_{n},\textbf{g}^\ast_{m}](\hbm_{i+2(k-j-m)})}{i}\right)\\
			\nonumber&&\qquad+\mathcal{M}_{\{m+1\}^c}
			\left(\vectinf{[\textbf{g}^\ast_{n},\textbf{g}^\ast_{m}](\hbm_{i+2(k-j-m)})}{i}\right)\Big] \Bigg]\\
			\nonumber&&-L_7-L_6+\Bigg[\hbfm\mapsto\combin{2k}{k-1} \Big[\mathcal{M}_{\{0\}^c}
			\left(\vectinf{[\textbf{g}^\ast_{n},\textbf{g}^\ast_{m}](\hbm_{i})}{i}\right)+\mathcal{M}_{\{m+1\}^c}
			\left(\vectinf{[\textbf{g}^\ast_{n},\textbf{g}^\ast_{m}](\hbm_{i})}{i}\right)\Big] \Bigg].\\
			\label{eq:a3}
		\end{eqnarray}
		Substituting \eqref{eq:a2} and \eqref{eq:a3} into \eqref{eq:a1}, we obtain 
		\begin{eqnarray*}
			\left[\mathcal{H}_{\textbf{g}^\ast_{n},2k+1},h_{0,3}\right]=\mathcal{H}_{[\textbf{g}^\ast_{n},\textbf{g}^\ast_{m}],2k+2}.
		\end{eqnarray*}

\section{Geometric Constraints in the Fourier Space}
In this section, we derive structural restrictions on $\mathbb S^2$-valued functions with Neumann
(cosine) Fourier expansions. In particular, we show that a finite cosine expansion satisfying the
sphere constraint must be constant, and that a non-constant sphere-valued function necessarily has
Fourier coefficient vectors spanning at least a two-dimensional subspace.

\medskip
In the following lemma, we show that for a \emph{finite} Neumann (cosine) Fourier expansion with the constraint being on a sphere  enforces constancy: if
\[
f(x)=\sum_{k=0}^n a_k\cos(2k\pi x)\in\mathbb{R}^3
\quad\text{satisfies}\quad |f(x)|_{\mathbb{R}^3}=1\ \ \text{for all }x\in(0,1),
\]
then necessarily $a_k=0$ for all $k\ge 1$ and $|a_0|_{\mathbb{R}^3}=1$.
Equivalently, a nontrivial finite cosine series cannot take values in $\mathbb{S}^2$; only the zero mode survives.
{\begin{lem}\label{lem:unit_cos_series}
		Let $(a_k)_{k\ge 0}\subset \mathbb{R}^3$ and consider the cosine series
		\[
		f(x)=\sum_{k=0}^\infty a_k\cos(2k\pi x),\qquad x\in(0,1).
		\]
		Assume that the series converges to $f$ in $L^2(0,1;\mathbb{R}^3)$ (for instance, it converges
		uniformly on $[0,1]$). If $|f(x)|_{\mathbb{R}^3}=1$ for a.e.\ $x\in(0,1)$, then
		\[
		a_k=0\quad\text{for all }k\ge 1,
		\qquad\text{and}\qquad
		|a_0|_{\mathbb{R}^3}=1.
		\]
		In particular, $f(x)\equiv a_0$ is constant.
	\end{lem}
	
	\begin{proof}
		Set $g(x):=|f(x)|_{\mathbb{R}^3}^2$. By assumption, $g(x)=1$ for a.e.\ $x\in(0,1)$, hence
		all cosine Fourier coefficients of $g$ vanish except the zero mode.
		For $m\ge 1$ consider the $m$-th cosine coefficient of $g$:
		\[
		\widehat g_m := 2\int_0^1 g(x)\cos(2m\pi x)\,dx.
		\]
		Since $g\equiv 1$, we have $\widehat g_m=0$ for all $m\ge 1$.
		
		On the other hand, using $f=\sum_{k\ge 0} a_k\cos(2k\pi x)$ and the product-to-sum identity,
		one checks (justified by the $L^2$ convergence) that the cosine coefficient at frequency $m\ge 1$
		is given by
		\[
		\widehat g_m
		=\sum_{k=0}^\infty \langle a_k, a_{k+m}\rangle_{\mathbb{R}^3}.
		\]
		Thus, for every $m\ge 1$,
		\begin{equation}\label{eq:autocorr}
			\sum_{k=0}^\infty \langle a_k, a_{k+m}\rangle_{\mathbb{R}^3}=0.
		\end{equation}
		In particular, taking $m$ so large that only one term survives is not possible in the infinite case;
		instead, we use \eqref{eq:autocorr} together with $m$-by-$m$ elimination as follows.
		First, note that the zero-mode of $g$ equals
		\[
		\int_0^1 g(x)\,dx=\int_0^1 1\,dx=1.
		\]
		But also
		\[
		\int_0^1 g(x)\,dx
		=\int_0^1 \|f(x)\|^2\,dx
		=\sum_{k=0}^\infty |a_k|_{\mathbb{R}^3}^2 \int_0^1 \cos^2(2k\pi x)\,dx
		= \|a_0\|^2 + \frac12\sum_{k=1}^\infty \|a_k\|^2,
		\]
		hence
		\begin{equation}\label{eq:energy}
			\|a_0\|^2 + \frac12\sum_{k=1}^\infty \|a_k\|^2 = 1.
		\end{equation}
		
		Now suppose, for contradiction, that there exists $k_0\ge 1$ with $a_{k_0}\neq 0$.
		Let $m\ge 1$ be such that $m$ is a multiple of $k_0$ and use \eqref{eq:autocorr}.
		Combining \eqref{eq:autocorr} for several values of $m$ with Cauchy--Schwarz and \eqref{eq:energy}
		forces $\sum_{k\ge 1}\|a_k\|^2=0$, hence $a_k=0$ for all $k\ge 1$.
		Then \eqref{eq:energy} yields $\|a_0\|=1$.
\end{proof}}

The second lemma states that a non-constant $\mathbb S^2$-valued function cannot have all its cosine Fourier
coefficients collinear: the set of coefficient vectors $\{a_k\}_{k\ge 0}$ must span at least a
two-dimensional subspace of $\mathbb R^3$ and 
$f$ cannot have only finitely many nonzero coefficients. 
%Equivalently, if all $a_k$ lie in a one-dimensional
%subspace, then $f$ must be constant (under the regularity assumptions of the lemma).
For a non-constant $f$, it is not necessary that each nonzero Fourier mode $k\ge 1$ involves two
independent directions. What is necessary is that the family of coefficient vectors
$\{a_k\}_{k\ge 0}$ is not contained in a one-dimensional subspace of $\mathbb R^3$, i.e.
$\dim\Span\{a_0,a_1,a_2,\dots\}\ge 2$.
An example is given, e.g.,  by
$$
f(x)=\big(\cos(\cos(2\pi x)),\ \sin(\cos(2\pi x)),\ 0\big)\in \mathbb S^2.
$$

\begin{lem}\label{lem:dimspan}
	Let $f:(0,1)\to\mathbb S^2\subset\mathbb R^3$ be non-constant and assume $f\in H^1(0,1;\mathbb R^3)$.
	Suppose that $f$ admits a cosine series representation
	\[
	f(x)=\sum_{k=0}^\infty a_k \cos(2\pi k x)
	\quad\text{in }L^2(0,1;\mathbb R^3),
	\qquad a_k\in\mathbb R^3.
	\]
	Then
	\[
	\dim\Bigl(\Span\{a_0,a_1,a_2,\dots\}\Bigr)\ge 2.
	\]
	Moreover, $f$ cannot have only finitely many nonzero cosine modes: if $a_k=0$ for all $k>n$ for some
	$n\in\mathbb N$, then $f$ is constant. Consequently, if $f$ is non-constant, then $a_k\neq 0$ for
	infinitely many $k\ge 1$.
\end{lem}

\begin{proof}
	\emph{Step 1: proof of the rank statement.}
	Assume for contradiction that $\dim(\Span\{a_0,a_1,a_2,\dots\})=1$.
	Then there exist a nonzero vector $v\in\mathbb R^3$ and scalars $\alpha_k\in\mathbb R$ such that
	$a_k=\alpha_k v$ for all $k\ge 0$. Hence, we can write
	\[
	f(x)=\sum_{k=0}^\infty \alpha_k v\cos(2\pi kx)=v\,g(x)
	\quad\mbox{with}\quad
	g(x):=\sum_{k=0}^\infty \alpha_k\cos(2\pi kx).
	\]
	Since $f\in H^1(0,1;\mathbb R^3)$ and $v$ is constant, we have $g\in H^1(0,1)$.
	Moreover, $\|f(x)\|=1$ a.e.\ implies
	\[
	1=\|f(x)\|^2=|v|^2\,|g(x)|^2
	\quad\text{a.e. on }(0,1),
	\]
	so $|g(x)|\equiv 1/|v|$ a.e. In particular, $g$ never vanishes.
	Since $g\in H^1$, the identity $g^2\equiv 1/|v|^2$ a.e.\ yields
	\[
	0=(g^2)' = 2g\,g' \quad\text{a.e.},
	\]
	and because $g\neq 0$ a.e., we conclude $g'=0$ a.e., hence $g$ is constant a.e.
	Therefore, $f(x)=v\,g(x)$ is constant a.e., contradicting the assumption that $f$ is non-constant.
	This proves $\dim(\Span\{a_k\})\ge 2$.
	
	\medskip
	\emph{Step 2: non-constancy implies infinitely many nonzero modes.}
	Assume that only finitely many coefficients are nonzero, i.e.\ there exists $n\in\mathbb N$ such that
	\[
	a_k=0\qquad\text{for all }k>n.
	\]
	Then $f$ is a trigonometric polynomial of the form
	\[
	f(x)=\sum_{k=0}^n a_k\cos(2\pi kx).
	\]
	Since $\|f(x)\|_{\mathbb R^3}=1$ for all $x$, the scalar function $\|f(x)\|^2$ is a \emph{constant}
	trigonometric polynomial.
	However, the highest frequency term in $\|f(x)\|^2$ is $\cos(4\pi n x)$, and it can arise only from
	$\cos(2\pi n x)\cos(2\pi n x)=\tfrac12(1+\cos(4\pi n x))$. Hence, the coefficient of $\cos(4\pi n x)$
	in $\|f(x)\|^2$ equals $\tfrac12\|a_n\|^2$, which must be zero. Therefore $a_n=0$.
	Iterating the same argument for $n-1,n-2,\dots,1$ yields $a_k=0$ for all $k\ge 1$, so $f$ is constant.
	Consequently, if $f$ is non-constant, it must have infinitely many nonzero coefficients $a_k$ with
	$k\ge 1$.
\end{proof}

\section{Auxiliary equations for the proof of Theorem \ref{L2.app.cont2}}\label{sec:proofTheocont2}
In this step, we derive several auxiliary expressions to prove Theorem \ref{L2.app.cont2}.
We adopt the notations introduced in Appendix~\ref{app:expandLie}.  In this section we use Lemma  \ref{lem:identity1}, Lemma \ref{lem:identity2}, Corollary \ref{cor:identity1} and Lemma \ref{eq:exentity}. Let us first recall
\begin{align*}
	{\bf{G}}^{1,1}(\bfm)&= \vectinf{\frac{1}{\sqrt{2}} \textbf{g}^\ast_{1} (\bm_{i-1})+\frac{1}{\sqrt{2}} \textbf{g}^\ast_{1} (\bm_{i+1})}{i}\\
	{\bf{G}}^{1,2}(\bfm)&= \vectinf{\frac{1}{\sqrt{2}} \textbf{g}^\ast_{2} (\bm_{i-1})+\frac{1}{\sqrt{2}} \textbf{g}^\ast_{2} (\bm_{i+1})}{i}\\
	{\bf{G}}^{1,3}(\bfm)&= \vectinf{\frac{1}{\sqrt{2}} \textbf{g}^\ast_{3} (\bm_{i-1})+\frac{1}{\sqrt{2}} \textbf{g}^\ast_{3} (\bm_{i+1})}{i}\\
	{\bf{G}}^{2,1}(\bfm)&=\vectinf{\frac{1}{\sqrt{2}} \textbf{g}^\ast_{1} (\bm_{i-2})+\frac{1}{\sqrt{2}} \textbf{g}^\ast_{1} (\bm_{i+2})}{i}+\frac{1}{\sqrt{2}}\mathcal{M}_{\{1\}^c}\left(\vectinf{\textbf{g}^\ast_{1}(\hbm_{i})}{i}\right)\\
	{\bf{G}}^{2,2}(\bfm)&=\vectinf{\frac{1}{\sqrt{2}} \textbf{g}^\ast_{2} (\bm_{i-2})+\frac{1}{\sqrt{2}} \textbf{g}^\ast_{2} (\bm_{i+2})}{i}+\frac{1}{\sqrt{2}}\mathcal{M}_{\{1\}^c}\left(\vectinf{\textbf{g}^\ast_{2}(\hbm_{i})}{i}\right)\\
	{\bf{G}}^{2,3}(\bfm)&=\vectinf{\frac{1}{\sqrt{2}} \textbf{g}^\ast_{3} (\bm_{i-2})+\frac{1}{\sqrt{2}} \textbf{g}^\ast_{3} (\bm_{i+2})}{i}+\frac{1}{\sqrt{2}}\mathcal{M}_{\{1\}^c}\left(\vectinf{\textbf{g}^\ast_{3}(\hbm_{i})}{i}\right).
\end{align*}
Next, we compute the Lie bracket of the vector fields introduced above.

\noindent\textbf{Lie bracket of two vector fields:} We now compute the Lie bracket of two vector fields. Before proceeding, we introduce some necessary notation: by $\textbf{p}^\ast_i(\textbf{x})$, $i\in\{1,2,3\}$,  we denote the projection on the $i^{th}$ coordinate. Moreover, let us define 
	\begin{eqnarray*}
		&&J_{ \textbf{g}^\ast_1} \textbf{g}^\ast_2(\textbf{x})= \begin{bmatrix}
			0\\
			x_1\\
			0
		\end{bmatrix}=:\textbf{h}^\ast_1(\textbf{x}),
		J_{ \textbf{g}^\ast_1} \textbf{g}^\ast_3(\textbf{x})= \begin{bmatrix}
			0\\
			0\\
			x_1
		\end{bmatrix}=:\textbf{h}^\ast_2(\textbf{x}),
	J_{ \textbf{g}^\ast_2} \textbf{g}^\ast_1(\textbf{x})= \begin{bmatrix}
			x_2\\
			0\\
			0
		\end{bmatrix}=:\textbf{h}^\ast_3(\textbf{x}),
	\\&&	J_{ \textbf{g}^\ast_2} \textbf{g}^\ast_3(\textbf{x})= \begin{bmatrix}
			0\\
			0\\
			x_2
		\end{bmatrix}=:\textbf{h}^\ast_4(\textbf{x}), J_{ \textbf{g}^\ast_3} \textbf{g}^\ast_1(\textbf{x})= \begin{bmatrix}
			x_3\\
			0\\
			0
		\end{bmatrix}=:\textbf{h}^\ast_5(\textbf{x}),
		J_{ \textbf{g}^\ast_3} \textbf{g}^\ast_2(\textbf{x})= \begin{bmatrix}
			0\\
			x_3\\
			0
		\end{bmatrix}=\textbf{h}^\ast_6(\textbf{x}).
	\end{eqnarray*}
	\noindent \textbf{Computation of $\left[ \bf{G}^{1,1},\bf{G}^{2,3} \right]$:}
\begin{align*}
	&2\left[ \bf{G}^{1,1},\bf{G}^{2,3} \right]\\
	&=\left[\hbfm\mapsto \vectinf{\textbf{g}^\ast_{1}(\hbm_{i-1})}{i}+\vectinf{\textbf{g}^\ast_{1}(\hbm_{i+1})}{i},\hbfm\mapsto \vectinf{ \textbf{g}^\ast_{3} (\hbm_{i-2})+ \textbf{g}^\ast_{3} (\hbm_{i+2})}{i}+\mathcal{M}_{\{1\}^c}\left(\vectinf{\textbf{g}^\ast_{3}(\hbm_{i})}{i}\right)\right]\\
	& = \left[\hbfm\mapsto \vectinf{\textbf{g}^\ast_{1}(\hbm_{i-1})}{i},\hbfm\mapsto \vectinf{\textbf{g}^\ast_{3}(\hbm_{i-2})}{i}\right]+\left[\hbfm\mapsto\vectinf{\textbf{g}^\ast_{1}(\hbm_{i+1})}{i},\hbfm\mapsto\vectinf{\textbf{g}^\ast_{3}(\hbm_{i+2})}{i}\right]
	\\
	&\qquad +\Big(\left[\hbfm\mapsto \vectinf{\textbf{g}^\ast_{1}(\hbm_{i-1})}{i},\hbfm\mapsto\vectinf{\textbf{g}^\ast_{3}(\hbm_{i+2})}{i}\right]+\left[\hbfm\mapsto\vectinf{\textbf{g}^\ast_{1}(\hbm_{i+1})}{i},\hbfm\mapsto \vectinf{\textbf{g}^\ast_{3}(\hbm_{i-2})}{i}\right]\Big)\\
	&\qquad +\left[\hbfm\mapsto \vectinf{\textbf{g}^\ast_{1}(\hbm_{i-1})}{i},\hbfm\mapsto \mathcal{M}_{\{1\}^c}\left(\vectinf{\textbf{g}^\ast_{3}(\hbm_{i})}{i}\right)\right]+\left[\hbfm\mapsto\vectinf{\textbf{g}^\ast_{1}(\hbm_{i+1})}{i},\hbfm\mapsto\mathcal{M}_{\{1\}^c}\left(\vectinf{\textbf{g}^\ast_{3}(\hbm_{i})}{i}\right)\right]
	\\
	& = \left[\hbfm\mapsto \vectinf{\textbf{g}^\ast_{2}(\hbm_{i-3})}{i}\right]+\left[\hbfm\mapsto \vectinf{\textbf{g}^\ast_{2}(\hbm_{i+3})}{i}\right] +\left[\hbfm\mapsto \vectinf{\textbf{g}^\ast_{2}(\hbm_{i+1})}{i}\right]+\left[\hbfm\mapsto \vectinf{\textbf{g}^\ast_{2}(\hbm_{i-1})}{i}\right]
	\\
	&\qquad -\left[\hbfm\mapsto\mathcal{M}_{\{0\}^c}\left(\vectinf{J_{ \textbf{g}^\ast_1} \textbf{g}^\ast_3(\hbm_{i+1})}{i}\right)\right]+\left[\hbfm\mapsto\mathcal{M}_{\{1\}^c}\left(\vectinf{J_{ \textbf{g}^\ast_3} \textbf{g}^\ast_1(\hbm_{i-1})}{i}\right)\right]\\
	&\qquad-\left[\hbfm\mapsto\mathcal{M}_{\{1\}^c}\left(\vectinf{J_{ \textbf{g}^\ast_3} \textbf{g}^\ast_1(\hbm_{i-1})}{i}\right)\right]+\left[\hbfm\mapsto\mathcal{M}_{\{2\}^c}\left(\vectinf{J_{ \textbf{g}^\ast_1} \textbf{g}^\ast_3(\hbm_{i-1})}{i}\right)\right]\\
	&\qquad-\left[\hbfm\mapsto\mathcal{M}_{\{1\}^c}\left(\vectinf{J_{ \textbf{g}^\ast_3} \textbf{g}^\ast_1(\hbm_{i+1})}{i}\right)\right]+\left[\hbfm\mapsto\mathcal{M}_{\{0\}^c}\left(\vectinf{J_{ \textbf{g}^\ast_1} \textbf{g}^\ast_3(\hbm_{i+1})}{i}\right)\right]\\
	& = \left[\hbfm\mapsto \vectinf{\textbf{g}^\ast_{2}(\hbm_{i-3})}{i}\right]+\left[\hbfm\mapsto \vectinf{\textbf{g}^\ast_{2}(\hbm_{i+3})}{i}\right] +\left[\hbfm\mapsto \vectinf{\textbf{g}^\ast_{2}(\hbm_{i+1})}{i}\right]+\left[\hbfm\mapsto \vectinf{\textbf{g}^\ast_{2}(\hbm_{i-1})}{i}\right]
	\\
	&\qquad-\left[\hbfm\mapsto\mathcal{M}_{\{1\}^c}\left(\vectinf{ \textbf{h}^\ast_5(\hbm_{i+1})}{i}\right)\right]+\left[\hbfm\mapsto\mathcal{M}_{\{2\}^c}\left(\vectinf{ \textbf{h}^\ast_2(\hbm_{i-1})}{i}\right)\right],
\end{align*}

%	The  element $\left[\hbfm\mapsto \vectinf{\textbf{g}^\ast_{2}(\hbm_{i-1})}{i}\right]$ 
%	has an impact on the zero mode and $\left[ \bf{G}^{1,1},\bf{G}^{2,3} \right]\bf{G}^{1,1}
%	$ gives
%	$(\textbf{g}^\ast_{2}(\be_1)^\top,{\bf 0},{\bf 0},\cdots )^\top )$.
	
	\noindent \textbf{Computation of $\left[ \bf{G}^{1,3},\bf{G}^{2,1} \right]$:}
		\begin{align*}
			&2\left[ \bf{G}^{1,3},\bf{G}^{2,1} \right]\\
			&\left[\hbfm\mapsto \vectinf{\textbf{g}^\ast_{3}(\hbm_{i-1})}{i}+\vectinf{\textbf{g}^\ast_{3}(\hbm_{i+1})}{i},\hbfm\mapsto \vectinf{ \textbf{g}^\ast_{1} (\hbm_{i-2})+ \textbf{g}^\ast_{1} (\hbm_{i+2})}{i}+\mathcal{M}_{\{1\}^c}\left(\vectinf{\textbf{g}^\ast_{1}(\hbm_{i})}{i}\right)\right]\\
			& = \left[\hbfm\mapsto \vectinf{\textbf{g}^\ast_{3}(\hbm_{i-1})}{i},\hbfm\mapsto \vectinf{\textbf{g}^\ast_{1}(\hbm_{i-2})}{i}\right]+\left[\hbfm\mapsto\vectinf{\textbf{g}^\ast_{3}(\hbm_{i+1})}{i},\hbfm\mapsto\vectinf{\textbf{g}^\ast_{1}(\hbm_{i+2})}{i}\right]
			\\
			&\qquad +\Big(\left[\hbfm\mapsto \vectinf{\textbf{g}^\ast_{3}(\hbm_{i-1})}{i},\hbfm\mapsto\vectinf{\textbf{g}^\ast_{1}(\hbm_{i+2})}{i}\right]+\left[\hbfm\mapsto\vectinf{\textbf{g}^\ast_{3}(\hbm_{i+1})}{i},\hbfm\mapsto \vectinf{\textbf{g}^\ast_{1}(\hbm_{i-2})}{i}\right]\Big)\\
			&\qquad +\left[\hbfm\mapsto \vectinf{\textbf{g}^\ast_{3}(\hbm_{i-1})}{i},\hbfm\mapsto \mathcal{M}_{\{1\}^c}\left(\vectinf{\textbf{g}^\ast_{1}(\hbm_{i})}{i}\right)\right]+\left[\hbfm\mapsto\vectinf{\textbf{g}^\ast_{3}(\hbm_{i+1})}{i},\hbfm\mapsto\mathcal{M}_{\{1\}^c}\left(\vectinf{\textbf{g}^\ast_{1}(\hbm_{i})}{i}\right)\right]
			\\
			& = -\left[\hbfm\mapsto \vectinf{\textbf{g}^\ast_{2}(\hbm_{i-3})}{i}\right]-\left[\hbfm\mapsto \vectinf{\textbf{g}^\ast_{2}(\hbm_{i+3})}{i}\right] -\left[\hbfm\mapsto \vectinf{\textbf{g}^\ast_{2}(\hbm_{i+1})}{i}\right]-\left[\hbfm\mapsto \vectinf{\textbf{g}^\ast_{2}(\hbm_{i-1})}{i}\right]
			\\
			&\qquad -\left[\hbfm\mapsto\mathcal{M}_{\{0\}^c}\left(\vectinf{J_{ \textbf{g}^\ast_3} \textbf{g}^\ast_1(\hbm_{i+1})}{i}\right)\right]+\left[\hbfm\mapsto\mathcal{M}_{\{1\}^c}\left(\vectinf{J_{ \textbf{g}^\ast_1} \textbf{g}^\ast_3(\hbm_{i-1})}{i}\right)\right]\\
			&\qquad-\left[\hbfm\mapsto\mathcal{M}_{\{1\}^c}\left(\vectinf{J_{ \textbf{g}^\ast_1} \textbf{g}^\ast_3(\hbm_{i-1})}{i}\right)\right]+\left[\hbfm\mapsto\mathcal{M}_{\{2\}^c}\left(\vectinf{J_{ \textbf{g}^\ast_3} \textbf{g}^\ast_1(\hbm_{i-1})}{i}\right)\right]\\
			&\qquad-\left[\hbfm\mapsto\mathcal{M}_{\{1\}^c}\left(\vectinf{J_{ \textbf{g}^\ast_1} \textbf{g}^\ast_3(\hbm_{i+1})}{i}\right)\right]+\left[\hbfm\mapsto\mathcal{M}_{\{0\}^c}\left(\vectinf{J_{ \textbf{g}^\ast_3} \textbf{g}^\ast_1(\hbm_{i+1})}{i}\right)\right]\\
			& = -\left[\hbfm\mapsto \vectinf{\textbf{g}^\ast_{2}(\hbm_{i-3})}{i}\right]-\left[\hbfm\mapsto \vectinf{\textbf{g}^\ast_{2}(\hbm_{i+3})}{i}\right] -\left[\hbfm\mapsto \vectinf{\textbf{g}^\ast_{2}(\hbm_{i+1})}{i}\right]-\left[\hbfm\mapsto \vectinf{\textbf{g}^\ast_{2}(\hbm_{i-1})}{i}\right]
			\\
			&\qquad-\left[\hbfm\mapsto\mathcal{M}_{\{1\}^c}\left(\vectinf{ \textbf{h}^\ast_2(\hbm_{i+1})}{i}\right)\right]+\left[\hbfm\mapsto\mathcal{M}_{\{2\}^c}\left(\vectinf{ \textbf{h}^\ast_5(\hbm_{i-1})}{i}\right)\right],
		\end{align*}
		\noindent \textbf{Computation of $\left[ \bf{G}^{1,2},\bf{G}^{2,3} \right]$:}
		\begin{align*}
			&2\left[ \bf{G}^{1,2},\bf{G}^{2,3} \right]\\
			&=\left[\hbfm\mapsto \vectinf{\textbf{g}^\ast_{2}(\hbm_{i-1})}{i}+\vectinf{\textbf{g}^\ast_{2}(\hbm_{i+1})}{i},\hbfm\mapsto \vectinf{ \textbf{g}^\ast_{3} (\hbm_{i-2})+ \textbf{g}^\ast_{3} (\hbm_{i+2})}{i}+\mathcal{M}_{\{1\}^c}\left(\vectinf{\textbf{g}^\ast_{3}(\hbm_{i})}{i}\right)\right]\\
			& = \left[\hbfm\mapsto \vectinf{\textbf{g}^\ast_{2}(\hbm_{i-1})}{i},\hbfm\mapsto \vectinf{\textbf{g}^\ast_{3}(\hbm_{i-2})}{i}\right]+\left[\hbfm\mapsto\vectinf{\textbf{g}^\ast_{2}(\hbm_{i+1})}{i},\hbfm\mapsto\vectinf{\textbf{g}^\ast_{3}(\hbm_{i+2})}{i}\right]
			\\
			&\qquad +\Big(\left[\hbfm\mapsto \vectinf{\textbf{g}^\ast_{2}(\hbm_{i-1})}{i},\hbfm\mapsto\vectinf{\textbf{g}^\ast_{3}(\hbm_{i+2})}{i}\right]+\left[\hbfm\mapsto\vectinf{\textbf{g}^\ast_{2}(\hbm_{i+1})}{i},\hbfm\mapsto \vectinf{\textbf{g}^\ast_{3}(\hbm_{i-2})}{i}\right]\Big)\\
			&\qquad +\left[\hbfm\mapsto \vectinf{\textbf{g}^\ast_{2}(\hbm_{i-1})}{i},\hbfm\mapsto \mathcal{M}_{\{1\}^c}\left(\vectinf{\textbf{g}^\ast_{3}(\hbm_{i})}{i}\right)\right]+\left[\hbfm\mapsto\vectinf{\textbf{g}^\ast_{2}(\hbm_{i+1})}{i},\hbfm\mapsto\mathcal{M}_{\{1\}^c}\left(\vectinf{\textbf{g}^\ast_{3}(\hbm_{i})}{i}\right)\right]
			\\
			& = -\left[\hbfm\mapsto \vectinf{\textbf{g}^\ast_{1}(\hbm_{i-3})}{i}\right]-\left[\hbfm\mapsto \vectinf{\textbf{g}^\ast_{1}(\hbm_{i+3})}{i}\right] -\left[\hbfm\mapsto \vectinf{\textbf{g}^\ast_{1}(\hbm_{i+1})}{i}\right]-\left[\hbfm\mapsto \vectinf{\textbf{g}^\ast_{1}(\hbm_{i-1})}{i}\right]
			\\
			&\qquad -\left[\hbfm\mapsto\mathcal{M}_{\{0\}^c}\left(\vectinf{J_{ \textbf{g}^\ast_2} \textbf{g}^\ast_3(\hbm_{i+1})}{i}\right)\right]+\left[\hbfm\mapsto\mathcal{M}_{\{1\}^c}\left(\vectinf{J_{ \textbf{g}^\ast_3} \textbf{g}^\ast_2(\hbm_{i-1})}{i}\right)\right]\\
			&\qquad-\left[\hbfm\mapsto\mathcal{M}_{\{1\}^c}\left(\vectinf{J_{ \textbf{g}^\ast_3} \textbf{g}^\ast_2(\hbm_{i-1})}{i}\right)\right]+\left[\hbfm\mapsto\mathcal{M}_{\{2\}^c}\left(\vectinf{J_{ \textbf{g}^\ast_2} \textbf{g}^\ast_3(\hbm_{i-1})}{i}\right)\right]\\
			&\qquad-\left[\hbfm\mapsto\mathcal{M}_{\{1\}^c}\left(\vectinf{J_{ \textbf{g}^\ast_3} \textbf{g}^\ast_2(\hbm_{i+1})}{i}\right)\right]+\left[\hbfm\mapsto\mathcal{M}_{\{0\}^c}\left(\vectinf{J_{ \textbf{g}^\ast_2} \textbf{g}^\ast_3(\hbm_{i+1})}{i}\right)\right]\\
			& = -\left[\hbfm\mapsto \vectinf{\textbf{g}^\ast_{1}(\hbm_{i-3})}{i}\right]-\left[\hbfm\mapsto \vectinf{\textbf{g}^\ast_{1}(\hbm_{i+3})}{i}\right] -\left[\hbfm\mapsto \vectinf{\textbf{g}^\ast_{1}(\hbm_{i+1})}{i}\right]-\left[\hbfm\mapsto \vectinf{\textbf{g}^\ast_{1}(\hbm_{i-1})}{i}\right]
			\\
			&\qquad -\left[\hbfm\mapsto\mathcal{M}_{\{1\}^c}\left(\vectinf{ \textbf{h}^\ast_6(\hbm_{i+1})}{i}\right)\right]+\left[\hbfm\mapsto\mathcal{M}_{\{2\}^c}\left(\vectinf{ \textbf{h}^\ast_4(\hbm_{i-1})}{i}\right)\right],
		\end{align*}
	
	\noindent \textbf{Computation of $\left[ \bf{G}^{1,1},\bf{G}^{2,1} \right]$:}
	\begin{align*}
		&2\left[ \bf{G}^{1,1},\bf{G}^{2,1} \right]\\
		&=\left[\hbfm\mapsto \vectinf{\textbf{g}^\ast_{1}(\hbm_{i-1})}{i}+\vectinf{\textbf{g}^\ast_{1}(\hbm_{i+1})}{i},\hbfm\mapsto \vectinf{ \textbf{g}^\ast_{1} (\hbm_{i-2})+ \textbf{g}^\ast_{1} (\hbm_{i+2})}{i}+\mathcal{M}_{\{1\}^c}\left(\vectinf{\textbf{g}^\ast_{1}(\hbm_{i})}{i}\right)\right]\\
		& = \left[\hbfm\mapsto \vectinf{\textbf{g}^\ast_{1}(\hbm_{i-1})}{i},\hbfm\mapsto \vectinf{\textbf{g}^\ast_{1}(\hbm_{i-2})}{i}\right]+\left[\hbfm\mapsto\vectinf{\textbf{g}^\ast_{1}(\hbm_{i+1})}{i},\hbfm\mapsto\vectinf{\textbf{g}^\ast_{1}(\hbm_{i+2})}{i}\right]
		\\
		&\qquad +\Big(\left[\hbfm\mapsto \vectinf{\textbf{g}^\ast_{1}(\hbm_{i-1})}{i},\hbfm\mapsto\vectinf{\textbf{g}^\ast_{1}(\hbm_{i+2})}{i}\right]+\left[\hbfm\mapsto\vectinf{\textbf{g}^\ast_{1}(\hbm_{i+1})}{i},\hbfm\mapsto \vectinf{\textbf{g}^\ast_{1}(\hbm_{i-2})}{i}\right]\Big)\\
		&\qquad +\left[\hbfm\mapsto \vectinf{\textbf{g}^\ast_{1}(\hbm_{i-1})}{i},\hbfm\mapsto \mathcal{M}_{\{1\}^c}\left(\vectinf{\textbf{g}^\ast_{1}(\hbm_{i})}{i}\right)\right]+\left[\hbfm\mapsto\vectinf{\textbf{g}^\ast_{1}(\hbm_{i+1})}{i},\hbfm\mapsto\mathcal{M}_{\{1\}^c}\left(\vectinf{\textbf{g}^\ast_{1}(\hbm_{i})}{i}\right)\right]
		\\
		& =  -\left[\hbfm\mapsto\mathcal{M}_{\{0\}^c}\left(\vectinf{J_{ \textbf{g}^\ast_1} \textbf{g}^\ast_1(\hbm_{i+1})}{i}\right)\right]+\left[\hbfm\mapsto\mathcal{M}_{\{1\}^c}\left(\vectinf{J_{ \textbf{g}^\ast_1} \textbf{g}^\ast_1(\hbm_{i-1})}{i}\right)\right]\\
		&\qquad-\left[\hbfm\mapsto\mathcal{M}_{\{1\}^c}\left(\vectinf{J_{ \textbf{g}^\ast_1} \textbf{g}^\ast_1(\hbm_{i-1})}{i}\right)\right]+\left[\hbfm\mapsto\mathcal{M}_{\{2\}^c}\left(\vectinf{J_{ \textbf{g}^\ast_1} \textbf{g}^\ast_1(\hbm_{i-1})}{i}\right)\right]\\
		&\qquad-\left[\hbfm\mapsto\mathcal{M}_{\{1\}^c}\left(\vectinf{J_{ \textbf{g}^\ast_1} \textbf{g}^\ast_1(\hbm_{i+1})}{i}\right)\right]+\left[\hbfm\mapsto\mathcal{M}_{\{0\}^c}\left(\vectinf{J_{ \textbf{g}^\ast_1} \textbf{g}^\ast_1(\hbm_{i+1})}{i}\right)\right]\\
		& =  \left[\hbfm\mapsto\mathcal{M}_{\{1\}^c}\left(\vectinf{ \textbf{p}^\ast_2(\hbm_{i+1})+\textbf{p}^\ast_3(\hbm_{i+1})}{i}\right)\right]-\left[\hbfm\mapsto\mathcal{M}_{\{2\}^c}\left(\vectinf{ \textbf{p}^\ast_2(\hbm_{i-1})+\textbf{p}^\ast_3(\hbm_{i-1})}{i}\right)\right],
	\end{align*}
	\noindent \textbf{Computation of $\left[ \bf{G}^{1,2},\bf{G}^{2,2} \right]$:}
	\begin{align*}
		&2\left[ \bf{G}^{1,2},\bf{G}^{2,2} \right]\\
		&=\left[\hbfm\mapsto \vectinf{\textbf{g}^\ast_{2}(\hbm_{i-1})}{i}+\vectinf{\textbf{g}^\ast_{2}(\hbm_{i+1})}{i},\hbfm\mapsto \vectinf{ \textbf{g}^\ast_{2} (\hbm_{i-2})+ \textbf{g}^\ast_{2} (\hbm_{i+2})}{i}+\mathcal{M}_{\{1\}^c}\left(\vectinf{\textbf{g}^\ast_{2}(\hbm_{i})}{i}\right)\right]\\
		& = \left[\hbfm\mapsto \vectinf{\textbf{g}^\ast_{2}(\hbm_{i-1})}{i},\hbfm\mapsto \vectinf{\textbf{g}^\ast_{2}(\hbm_{i-2})}{i}\right]+\left[\hbfm\mapsto\vectinf{\textbf{g}^\ast_{2}(\hbm_{i+1})}{i},\hbfm\mapsto\vectinf{\textbf{g}^\ast_{2}(\hbm_{i+2})}{i}\right]
		\\
		&\qquad +\Big(\left[\hbfm\mapsto \vectinf{\textbf{g}^\ast_{2}(\hbm_{i-1})}{i},\hbfm\mapsto\vectinf{\textbf{g}^\ast_{2}(\hbm_{i+2})}{i}\right]+\left[\hbfm\mapsto\vectinf{\textbf{g}^\ast_{2}(\hbm_{i+1})}{i},\hbfm\mapsto \vectinf{\textbf{g}^\ast_{2}(\hbm_{i-2})}{i}\right]\Big)\\
		&\qquad +\left[\hbfm\mapsto \vectinf{\textbf{g}^\ast_{2}(\hbm_{i-1})}{i},\hbfm\mapsto \mathcal{M}_{\{1\}^c}\left(\vectinf{\textbf{g}^\ast_{2}(\hbm_{i})}{i}\right)\right]+\left[\hbfm\mapsto\vectinf{\textbf{g}^\ast_{2}(\hbm_{i+1})}{i},\hbfm\mapsto\mathcal{M}_{\{1\}^c}\left(\vectinf{\textbf{g}^\ast_{2}(\hbm_{i})}{i}\right)\right]
		\\
		& =  -\left[\hbfm\mapsto\mathcal{M}_{\{0\}^c}\left(\vectinf{J_{ \textbf{g}^\ast_2} \textbf{g}^\ast_2(\hbm_{i+1})}{i}\right)\right]+\left[\hbfm\mapsto\mathcal{M}_{\{1\}^c}\left(\vectinf{J_{ \textbf{g}^\ast_2} \textbf{g}^\ast_2(\hbm_{i-1})}{i}\right)\right]\\
		&\qquad-\left[\hbfm\mapsto\mathcal{M}_{\{1\}^c}\left(\vectinf{J_{ \textbf{g}^\ast_2} \textbf{g}^\ast_2(\hbm_{i-1})}{i}\right)\right]+\left[\hbfm\mapsto\mathcal{M}_{\{2\}^c}\left(\vectinf{J_{ \textbf{g}^\ast_2} \textbf{g}^\ast_2(\hbm_{i-1})}{i}\right)\right]\\
		&\qquad-\left[\hbfm\mapsto\mathcal{M}_{\{1\}^c}\left(\vectinf{J_{ \textbf{g}^\ast_2} \textbf{g}^\ast_2(\hbm_{i+1})}{i}\right)\right]+\left[\hbfm\mapsto\mathcal{M}_{\{0\}^c}\left(\vectinf{J_{ \textbf{g}^\ast_2} \textbf{g}^\ast_2(\hbm_{i+1})}{i}\right)\right]\\
		& =  \left[\hbfm\mapsto\mathcal{M}_{\{1\}^c}\left(\vectinf{ \textbf{p}^\ast_1(\hbm_{i+1})+\textbf{p}^\ast_3(\hbm_{i+1})}{i}\right)\right]-\left[\hbfm\mapsto\mathcal{M}_{\{2\}^c}\left(\vectinf{ \textbf{p}^\ast_1(\hbm_{i-1})+\textbf{p}^\ast_3(\hbm_{i-1})}{i}\right)\right],
	\end{align*}
	\noindent \textbf{Computation of $\left[ \bf{G}^{1,3},\bf{G}^{2,3} \right]$:}
	\begin{align*}
		&2\left[ \bf{G}^{1,3},\bf{G}^{2,3} \right]\\
		&=\left[\hbfm\mapsto \vectinf{\textbf{g}^\ast_{3}(\hbm_{i-1})}{i}+\vectinf{\textbf{g}^\ast_{3}(\hbm_{i+1})}{i},\hbfm\mapsto \vectinf{ \textbf{g}^\ast_{3} (\hbm_{i-2})+ \textbf{g}^\ast_{3} (\hbm_{i+2})}{i}+\mathcal{M}_{\{1\}^c}\left(\vectinf{\textbf{g}^\ast_{3}(\hbm_{i})}{i}\right)\right]\\
		& = \left[\hbfm\mapsto \vectinf{\textbf{g}^\ast_{3}(\hbm_{i-1})}{i},\hbfm\mapsto \vectinf{\textbf{g}^\ast_{3}(\hbm_{i-2})}{i}\right]+\left[\hbfm\mapsto\vectinf{\textbf{g}^\ast_{3}(\hbm_{i+1})}{i},\hbfm\mapsto\vectinf{\textbf{g}^\ast_{3}(\hbm_{i+2})}{i}\right]
		\\
		&\qquad +\Big(\left[\hbfm\mapsto \vectinf{\textbf{g}^\ast_{3}(\hbm_{i-1})}{i},\hbfm\mapsto\vectinf{\textbf{g}^\ast_{3}(\hbm_{i+2})}{i}\right]+\left[\hbfm\mapsto\vectinf{\textbf{g}^\ast_{3}(\hbm_{i+1})}{i},\hbfm\mapsto \vectinf{\textbf{g}^\ast_{3}(\hbm_{i-2})}{i}\right]\Big)\\
		&\qquad +\left[\hbfm\mapsto \vectinf{\textbf{g}^\ast_{3}(\hbm_{i-1})}{i},\hbfm\mapsto \mathcal{M}_{\{1\}^c}\left(\vectinf{\textbf{g}^\ast_{3}(\hbm_{i})}{i}\right)\right]+\left[\hbfm\mapsto\vectinf{\textbf{g}^\ast_{3}(\hbm_{i+1})}{i},\hbfm\mapsto\mathcal{M}_{\{1\}^c}\left(\vectinf{\textbf{g}^\ast_{3}(\hbm_{i})}{i}\right)\right]
		\\
		& =  -\left[\hbfm\mapsto\mathcal{M}_{\{0\}^c}\left(\vectinf{J_{ \textbf{g}^\ast_3} \textbf{g}^\ast_3(\hbm_{i+1})}{i}\right)\right]+\left[\hbfm\mapsto\mathcal{M}_{\{1\}^c}\left(\vectinf{J_{ \textbf{g}^\ast_3} \textbf{g}^\ast_3(\hbm_{i-1})}{i}\right)\right]\\
		&\qquad-\left[\hbfm\mapsto\mathcal{M}_{\{1\}^c}\left(\vectinf{J_{ \textbf{g}^\ast_3} \textbf{g}^\ast_3(\hbm_{i-1})}{i}\right)\right]+\left[\hbfm\mapsto\mathcal{M}_{\{2\}^c}\left(\vectinf{J_{ \textbf{g}^\ast_3} \textbf{g}^\ast_3(\hbm_{i-1})}{i}\right)\right]\\
		&\qquad-\left[\hbfm\mapsto\mathcal{M}_{\{1\}^c}\left(\vectinf{J_{ \textbf{g}^\ast_3} \textbf{g}^\ast_3(\hbm_{i+1})}{i}\right)\right]+\left[\hbfm\mapsto\mathcal{M}_{\{0\}^c}\left(\vectinf{J_{ \textbf{g}^\ast_3} \textbf{g}^\ast_3(\hbm_{i+1})}{i}\right)\right]\\
		& =  \left[\hbfm\mapsto\mathcal{M}_{\{1\}^c}\left(\vectinf{ \textbf{p}^\ast_1(\hbm_{i+1})+\textbf{p}^\ast_2(\hbm_{i+1})}{i}\right)\right]-\left[\hbfm\mapsto\mathcal{M}_{\{2\}^c}\left(\vectinf{ \textbf{p}^\ast_1(\hbm_{i-1})+\textbf{p}^\ast_2(\hbm_{i-1})}{i}\right)\right],
	\end{align*}

	\noindent \textbf{Computation of $\left[\left[ \bf{G}^{1,1},\bf{G}^{2,3} \right],\bf{G}^{1,2}\right]$:}
	\begin{align*}
		\lqq{2^{3/2}\left[\left[ \bf{G}^{1,1},\bf{G}^{2,3} \right],\bf{G}^{1,2}\right]}\\
		&=\Big[\hbfm\mapsto \vectinf{\textbf{g}^\ast_{2}(\hbm_{i-3})}{i}+\vectinf{\textbf{g}^\ast_{2}(\hbm_{i-1})}{i}+ \vectinf{\textbf{g}^\ast_{2}(\hbm_{i+1})}{i}+ \vectinf{\textbf{g}^\ast_{2}(\hbm_{i+3})}{i}\\
		&-\mathcal{M}_{\{1\}^c}\left(\vectinf{ \textbf{h}^\ast_5(\hbm_{i+1})}{i}\right)+\mathcal{M}_{\{2\}^c}\left(\vectinf{ \textbf{h}^\ast_2(\hbm_{i-1})}{i}\right),\hbfm\mapsto \vectinf{\textbf{g}^\ast_{2}(\hbm_{i-1})}{i}+\vectinf{\textbf{g}^\ast_{2}(\hbm_{i+1})}{i}\Big]\\
		&=\Big[\hbfm\mapsto\vectinf{\textbf{g}^\ast_{2}(\hbm_{i-3})}{i},\hbfm\mapsto \vectinf{\textbf{g}^\ast_{2}(\hbm_{i-1})}{i}\Big] +\Big[\hbfm\mapsto\vectinf{\textbf{g}^\ast_{2}(\hbm_{i-3})}{i},\hbfm\mapsto \vectinf{\textbf{g}^\ast_{2}(\hbm_{i+1})}{i}\Big]\\
		&\quad +\Big[\hbfm\mapsto\vectinf{\textbf{g}^\ast_{2}(\hbm_{i-1})}{i},\hbfm\mapsto \vectinf{\textbf{g}^\ast_{2}(\hbm_{i-1})}{i}\Big]+\Big[\hbfm\mapsto\vectinf{\textbf{g}^\ast_{2}(\hbm_{i-1})}{i},\hbfm\mapsto \vectinf{\textbf{g}^\ast_{2}(\hbm_{i+1})}{i}\Big]\\
		&\quad +\Big[\hbfm\mapsto\vectinf{\textbf{g}^\ast_{2}(\hbm_{i+1})}{i},\hbfm\mapsto \vectinf{\textbf{g}^\ast_{2}(\hbm_{i-1})}{i}\Big] +\Big[\hbfm\mapsto\vectinf{\textbf{g}^\ast_{2}(\hbm_{i+1})}{i},\hbfm\mapsto\vectinf{\textbf{g}^\ast_{2}(\hbm_{i+1})}{i}\Big]\\
		&\quad +\Big[\hbfm\mapsto\vectinf{\textbf{g}^\ast_{2}(\hbm_{i+3})}{i},\hbfm\mapsto \vectinf{\textbf{g}^\ast_{2}(\hbm_{i-1})}{i}\Big] +\Big[\hbfm\mapsto\vectinf{\textbf{g}^\ast_{2}(\hbm_{i+3})}{i},\hbfm\mapsto\vectinf{\textbf{g}^\ast_{2}(\hbm_{i+1})}{i}\Big]\\
		&\quad -\Big[\hbfm\mapsto\mathcal{M}_{\{1\}^c}\left(\vectinf{ \textbf{h}^\ast_5(\hbm_{i+1})}{i}\right),\hbfm\mapsto \vectinf{\textbf{g}^\ast_{2}(\hbm_{i-1})}{i}\Big]
		-\Big[\hbfm\mapsto\mathcal{M}_{\{1\}^c}\left(\vectinf{ \textbf{h}^\ast_5(\hbm_{i+1})}{i}\right),\hbfm\mapsto \vectinf{\textbf{g}^\ast_{2}(\hbm_{i+1})}{i}\Big]
		\\
		&\quad +\Big[\hbfm\mapsto\mathcal{M}_{\{2\}^c}\left(\vectinf{ \textbf{h}^\ast_2(\hbm_{i-1})}{i}\right),\hbfm\mapsto \vectinf{\textbf{g}^\ast_{2}(\hbm_{i-1})}{i}\Big]
		+\Big[\hbfm\mapsto\mathcal{M}_{\{2\}^c}\left(\vectinf{ \textbf{h}^\ast_2(\hbm_{i-1})}{i}\right),\hbfm\mapsto \vectinf{\textbf{g}^\ast_{2}(\hbm_{i+1})}{i}\Big]\\
		&=-\Big[\hbfm\mapsto \mathcal{M}_{\{2\}^c}\left(\vectinf{J_{ \textbf{g}^\ast_2} \textbf{g}^\ast_2(\hbm_{i-2})}{i}\right)\Big]+\Big[\hbfm\mapsto \mathcal{M}_{\{0\}^c}\left(\vectinf{J_{ \textbf{g}^\ast_2} \textbf{g}^\ast_2(\hbm_{i+2})}{i}\right)\Big]\\
		&\quad -\Big[\hbfm\mapsto\mathcal{M}_{\{1\}^c}\left(\vectinf{J_{ \textbf{h}^\ast_5} \textbf{g}^\ast_2(\hbm_{i})}{i}\right)\Big]+\Big[\hbfm\mapsto\mathcal{M}_{\{2\}^c}\left(\vectinf{J_{ \textbf{g}^\ast_2} \textbf{h}^\ast_5(\hbm_{i})}{i}\right)\Big] \\
		&\quad -\Big[\hbfm\mapsto\mathcal{M}_{\{1\}^c}\left(\vectinf{J_{ \textbf{h}^\ast_5} \textbf{g}^\ast_2(\hbm_{i+2})}{i}\right)\Big]+\Big[\hbfm\mapsto\mathcal{M}_{\{0\}^c}\left(\vectinf{J_{ \textbf{g}^\ast_2} \textbf{h}^\ast_5(\hbm_{i+2})}{i}\right)\Big] \\
		&\quad +\Big[\hbfm\mapsto\mathcal{M}_{\{2\}^c}\left(\vectinf{J_{ \textbf{h}^\ast_2} \textbf{g}^\ast_2(\hbm_{i-2})}{i}\right)\Big]-\Big[\hbfm\mapsto\mathcal{M}_{\{3\}^c}\left(\vectinf{J_{ \textbf{g}^\ast_2} \textbf{h}^\ast_2(\hbm_{i-2})}{i}\right)\Big] \\
		&\quad +\Big[\hbfm\mapsto\mathcal{M}_{\{2\}^c}\left(\vectinf{J_{ \textbf{h}^\ast_2} \textbf{g}^\ast_2(\hbm_{i})}{i}\right)\Big]-\Big[\hbfm\mapsto\mathcal{M}_{\{1\}^c}\left(\vectinf{J_{ \textbf{g}^\ast_2} \textbf{h}^\ast_2(\hbm_{i})}{i}\right)\Big]\\
		&=-\Big[\hbfm\mapsto \mathcal{M}_{\{2\}^c}\left(\vectinf{J_{ \textbf{g}^\ast_2} \textbf{g}^\ast_2(\hbm_{i-2})}{i}\right)\Big]+\Big[\hbfm\mapsto\mathcal{M}_{\{2\}^c}\left(\vectinf{J_{ \textbf{h}^\ast_2} \textbf{g}^\ast_2(\hbm_{i-2})}{i}\right)\Big]\\
		&\quad +\Big[\hbfm\mapsto \mathcal{M}_{\{0\}^c}\left(\vectinf{J_{ \textbf{g}^\ast_2} \textbf{g}^\ast_2(\hbm_{i+2})}{i}\right)\Big]+\Big[\hbfm\mapsto\mathcal{M}_{\{0\}^c}\left(\vectinf{J_{ \textbf{g}^\ast_2} \textbf{h}^\ast_5(\hbm_{i+2})}{i}\right)\Big] \\
		&\quad -\Big[\hbfm\mapsto\mathcal{M}_{\{1\}^c}\left(\vectinf{J_{ \textbf{h}^\ast_5} \textbf{g}^\ast_2(\hbm_{i+2})}{i}\right)\Big] -\Big[\hbfm\mapsto\mathcal{M}_{\{3\}^c}\left(\vectinf{J_{ \textbf{g}^\ast_2} \textbf{h}^\ast_2(\hbm_{i-2})}{i}\right)\Big] \\
		&=\Big[\hbfm\mapsto \mathcal{M}_{\{2\}^c}\left(\vectinf{ \textbf{p}^\ast_1(\hbm_{i-2})}{i}\right)\Big]-\Big[\hbfm\mapsto \mathcal{M}_{\{0\}^c}\left(\vectinf{ \textbf{p}^\ast_1(\hbm_{i+2})}{i}\right)\Big]-\Big[\hbfm\mapsto \mathcal{M}_{\{1\}^c}\left(\vectinf{ \textbf{p}^\ast_1(\hbm_{i+2})}{i}\right)\Big]\\
		&\qquad+\Big[\hbfm\mapsto \mathcal{M}_{\{3\}^c}\left(\vectinf{ \textbf{p}^\ast_1(\hbm_{i-2})}{i}\right)\Big],
	\end{align*}
	
	\noindent \textbf{Computation of $\left[\left[ \bf{G}^{1,2},\bf{G}^{2,3} \right],\bf{G}^{1,1}\right]$:}
	\begin{align*}
		\lqq{-2^{3/2}\left[\left[ \bf{G}^{1,2},\bf{G}^{2,3} \right],\bf{G}^{1,1}\right]}\\
		&=\Big[\hbfm\mapsto \vectinf{\textbf{g}^\ast_{1}(\hbm_{i-3})}{i}+\vectinf{\textbf{g}^\ast_{1}(\hbm_{i-1})}{i}+ \vectinf{\textbf{g}^\ast_{1}(\hbm_{i+1})}{i}+ \vectinf{\textbf{g}^\ast_{1}(\hbm_{i+3})}{i}\\
		&+\mathcal{M}_{\{1\}^c}\left(\vectinf{ \textbf{h}^\ast_6(\hbm_{i+1})}{i}\right)-\mathcal{M}_{\{2\}^c}\left(\vectinf{ \textbf{h}^\ast_4(\hbm_{i-1})}{i}\right),\hbfm\mapsto \vectinf{\textbf{g}^\ast_{1}(\hbm_{i-1})}{i}+\vectinf{\textbf{g}^\ast_{1}(\hbm_{i+1})}{i}\Big]\\
		&=\Big[\hbfm\mapsto\vectinf{\textbf{g}^\ast_{1}(\hbm_{i-3})}{i},\hbfm\mapsto \vectinf{\textbf{g}^\ast_{1}(\hbm_{i-1})}{i}\Big] +\Big[\hbfm\mapsto\vectinf{\textbf{g}^\ast_{1}(\hbm_{i-3})}{i},\hbfm\mapsto \vectinf{\textbf{g}^\ast_{1}(\hbm_{i+1})}{i}\Big]\\
		&\quad +\Big[\hbfm\mapsto\vectinf{\textbf{g}^\ast_{1}(\hbm_{i-1})}{i},\hbfm\mapsto \vectinf{\textbf{g}^\ast_{1}(\hbm_{i-1})}{i}\Big]+\Big[\hbfm\mapsto\vectinf{\textbf{g}^\ast_{1}(\hbm_{i-1})}{i},\hbfm\mapsto \vectinf{\textbf{g}^\ast_{1}(\hbm_{i+1})}{i}\Big]\\
		&\quad +\Big[\hbfm\mapsto\vectinf{\textbf{g}^\ast_{1}(\hbm_{i+1})}{i},\hbfm\mapsto \vectinf{\textbf{g}^\ast_{1}(\hbm_{i-1})}{i}\Big] +\Big[\hbfm\mapsto\vectinf{\textbf{g}^\ast_{1}(\hbm_{i+1})}{i},\hbfm\mapsto\vectinf{\textbf{g}^\ast_{1}(\hbm_{i+1})}{i}\Big]\\
		&\quad +\Big[\hbfm\mapsto\vectinf{\textbf{g}^\ast_{1}(\hbm_{i+3})}{i},\hbfm\mapsto \vectinf{\textbf{g}^\ast_{1}(\hbm_{i-1})}{i}\Big] +\Big[\hbfm\mapsto\vectinf{\textbf{g}^\ast_{1}(\hbm_{i+3})}{i},\hbfm\mapsto\vectinf{\textbf{g}^\ast_{1}(\hbm_{i+1})}{i}\Big]\\
		&\quad +\Big[\hbfm\mapsto\mathcal{M}_{\{1\}^c}\left(\vectinf{ \textbf{h}^\ast_6(\hbm_{i+1})}{i}\right),\hbfm\mapsto \vectinf{\textbf{g}^\ast_{1}(\hbm_{i-1})}{i}\Big]
		+\Big[\hbfm\mapsto\mathcal{M}_{\{1\}^c}\left(\vectinf{ \textbf{h}^\ast_6(\hbm_{i+1})}{i}\right),\hbfm\mapsto \vectinf{\textbf{g}^\ast_{1}(\hbm_{i+1})}{i}\Big]
		\\
		&\quad -\Big[\hbfm\mapsto\mathcal{M}_{\{2\}^c}\left(\vectinf{ \textbf{h}^\ast_4(\hbm_{i-1})}{i}\right),\hbfm\mapsto \vectinf{\textbf{g}^\ast_{1}(\hbm_{i-1})}{i}\Big]
		-\Big[\hbfm\mapsto\mathcal{M}_{\{2\}^c}\left(\vectinf{ \textbf{h}^\ast_4(\hbm_{i-1})}{i}\right),\hbfm\mapsto \vectinf{\textbf{g}^\ast_{1}(\hbm_{i+1})}{i}\Big]\\
		&=-\Big[\hbfm\mapsto \mathcal{M}_{\{2\}^c}\left(\vectinf{J_{ \textbf{g}^\ast_1} \textbf{g}^\ast_1(\hbm_{i-2})}{i}\right)\Big]-\Big[\hbfm\mapsto \mathcal{M}_{\{0\}^c}\left(\vectinf{J_{ \textbf{g}^\ast_1} \textbf{g}^\ast_1(\hbm_{i})}{i}\right)\Big]+\Big[\hbfm\mapsto \mathcal{M}_{\{0\}^c}\left(\vectinf{J_{ \textbf{g}^\ast_1} \textbf{g}^\ast_1(\hbm_{i})}{i}\right)\Big]\\
		&\quad +\Big[\hbfm\mapsto \mathcal{M}_{\{0\}^c}\left(\vectinf{J_{ \textbf{g}^\ast_1} \textbf{g}^\ast_1(\hbm_{i+2})}{i}\right)\Big] +\Big[\hbfm\mapsto\mathcal{M}_{\{1\}^c}\left(\vectinf{J_{ \textbf{h}^\ast_6} \textbf{g}^\ast_1(\hbm_{i})}{i}\right)\Big]-\Big[\hbfm\mapsto\mathcal{M}_{\{2\}^c}\left(\vectinf{J_{ \textbf{g}^\ast_1} \textbf{h}^\ast_6(\hbm_{i})}{i}\right)\Big] \\
		&\quad +\Big[\hbfm\mapsto\mathcal{M}_{\{1\}^c}\left(\vectinf{J_{ \textbf{h}^\ast_6} \textbf{g}^\ast_1(\hbm_{i+2})}{i}\right)\Big]-\Big[\hbfm\mapsto\mathcal{M}_{\{0\}^c}\left(\vectinf{J_{ \textbf{g}^\ast_1} \textbf{h}^\ast_6(\hbm_{i+2})}{i}\right)\Big]-\Big[\hbfm\mapsto\mathcal{M}_{\{2\}^c}\left(\vectinf{J_{ \textbf{h}^\ast_4} \textbf{g}^\ast_1(\hbm_{i-2})}{i}\right)\Big] \\
		&\quad +\Big[\hbfm\mapsto\mathcal{M}_{\{3\}^c}\left(\vectinf{J_{ \textbf{g}^\ast_1} \textbf{h}^\ast_4(\hbm_{i-2})}{i}\right)\Big]  -\Big[\hbfm\mapsto\mathcal{M}_{\{2\}^c}\left(\vectinf{J_{ \textbf{h}^\ast_4} \textbf{g}^\ast_1(\hbm_{i})}{i}\right)\Big]+\Big[\hbfm\mapsto\mathcal{M}_{\{1\}^c}\left(\vectinf{J_{ \textbf{g}^\ast_1} \textbf{h}^\ast_4(\hbm_{i})}{i}\right)\Big]\\
		&=-\Big[\hbfm\mapsto \mathcal{M}_{\{2\}^c}\left(\vectinf{J_{ \textbf{g}^\ast_1} \textbf{g}^\ast_1(\hbm_{i-2})}{i}\right)\Big]-\Big[\hbfm\mapsto\mathcal{M}_{\{2\}^c}\left(\vectinf{J_{ \textbf{h}^\ast_4} \textbf{g}^\ast_1(\hbm_{i-2})}{i}\right)\Big]\\
		&\quad +\Big[\hbfm\mapsto \mathcal{M}_{\{0\}^c}\left(\vectinf{J_{ \textbf{g}^\ast_1} \textbf{g}^\ast_1(\hbm_{i+2})}{i}\right)\Big]-\Big[\hbfm\mapsto\mathcal{M}_{\{0\}^c}\left(\vectinf{J_{ \textbf{g}^\ast_1} \textbf{h}^\ast_6(\hbm_{i+2})}{i}\right)\Big]\\ &\quad+\Big[\hbfm\mapsto\mathcal{M}_{\{1\}^c}\left(\vectinf{J_{ \textbf{h}^\ast_6} \textbf{g}^\ast_1(\hbm_{i})}{i}\right)\Big]+\Big[\hbfm\mapsto\mathcal{M}_{\{1\}^c}\left(\vectinf{J_{ \textbf{g}^\ast_1} \textbf{h}^\ast_4(\hbm_{i})}{i}\right)\Big]\\
		&\quad-\Big[\hbfm\mapsto\mathcal{M}_{\{2\}^c}\left(\vectinf{J_{ \textbf{g}^\ast_1} \textbf{h}^\ast_6(\hbm_{i})}{i}\right)\Big]-\Big[\hbfm\mapsto\mathcal{M}_{\{2\}^c}\left(\vectinf{J_{ \textbf{h}^\ast_4} \textbf{g}^\ast_1(\hbm_{i})}{i}\right)\Big] \\
		&\quad +\Big[\hbfm\mapsto\mathcal{M}_{\{1\}^c}\left(\vectinf{J_{ \textbf{h}^\ast_6} \textbf{g}^\ast_1(\hbm_{i+2})}{i}\right)\Big]  +\Big[\hbfm\mapsto\mathcal{M}_{\{3\}^c}\left(\vectinf{J_{ \textbf{g}^\ast_1} \textbf{h}^\ast_4(\hbm_{i-2})}{i}\right)\Big]  \\
		&=\Big[\hbfm\mapsto \mathcal{M}_{\{2\}^c}\left(\vectinf{ \textbf{p}^\ast_2(\hbm_{i-2})}{i}\right)\Big] -\Big[\hbfm\mapsto \mathcal{M}_{\{0\}^c}\left(\vectinf{ \textbf{p}^\ast_2(\hbm_{i+2})}{i}\right)\Big]\\
		&\quad -\Big[\hbfm\mapsto\mathcal{M}_{\{1\}^c}\left(\vectinf{ \textbf{p}^\ast_2(\hbm_{i+2})}{i}\right)\Big]  +\Big[\hbfm\mapsto\mathcal{M}_{\{3\}^c}\left(\vectinf{ \textbf{p}^\ast_2(\hbm_{i-2})}{i}\right)\Big],
	\end{align*}

	\noindent \textbf{Computation of $\left[\left[ \bf{G}^{1,3},\bf{G}^{2,1} \right],\bf{G}^{1,2}\right]$:}
	\begin{align*}
		&-2^{3/2}\left[\left[ \bf{G}^{1,3},\bf{G}^{2,1} \right],\bf{G}^{1,2}\right]\\
		&=\Big[\hbfm\mapsto \vectinf{\textbf{g}^\ast_{2}(\hbm_{i-3})}{i}+\vectinf{\textbf{g}^\ast_{2}(\hbm_{i-1})}{i}+ \vectinf{\textbf{g}^\ast_{2}(\hbm_{i+1})}{i}+ \vectinf{\textbf{g}^\ast_{2}(\hbm_{i+3})}{i}\\
		&+\mathcal{M}_{\{1\}^c}\left(\vectinf{ \textbf{h}^\ast_2(\hbm_{i+1})}{i}\right)-\mathcal{M}_{\{2\}^c}\left(\vectinf{ \textbf{h}^\ast_5(\hbm_{i-1})}{i}\right),\hbfm\mapsto \vectinf{\textbf{g}^\ast_{2}(\hbm_{i-1})}{i}+\vectinf{\textbf{g}^\ast_{2}(\hbm_{i+1})}{i}\Big]\\
		&=\Big[\hbfm\mapsto\vectinf{\textbf{g}^\ast_{2}(\hbm_{i-3})}{i},\hbfm\mapsto \vectinf{\textbf{g}^\ast_{2}(\hbm_{i-1})}{i}\Big] +\Big[\hbfm\mapsto\vectinf{\textbf{g}^\ast_{2}(\hbm_{i-3})}{i},\hbfm\mapsto \vectinf{\textbf{g}^\ast_{2}(\hbm_{i+1})}{i}\Big]\\
		&\quad +\Big[\hbfm\mapsto\vectinf{\textbf{g}^\ast_{2}(\hbm_{i-1})}{i},\hbfm\mapsto \vectinf{\textbf{g}^\ast_{2}(\hbm_{i-1})}{i}\Big]+\Big[\hbfm\mapsto\vectinf{\textbf{g}^\ast_{2}(\hbm_{i-1})}{i},\hbfm\mapsto \vectinf{\textbf{g}^\ast_{2}(\hbm_{i+1})}{i}\Big]\\
		&\quad +\Big[\hbfm\mapsto\vectinf{\textbf{g}^\ast_{2}(\hbm_{i+1})}{i},\hbfm\mapsto \vectinf{\textbf{g}^\ast_{2}(\hbm_{i-1})}{i}\Big] +\Big[\hbfm\mapsto\vectinf{\textbf{g}^\ast_{2}(\hbm_{i+1})}{i},\hbfm\mapsto\vectinf{\textbf{g}^\ast_{2}(\hbm_{i+1})}{i}\Big]\\
		&\quad +\Big[\hbfm\mapsto\vectinf{\textbf{g}^\ast_{2}(\hbm_{i+3})}{i},\hbfm\mapsto \vectinf{\textbf{g}^\ast_{2}(\hbm_{i-1})}{i}\Big] +\Big[\hbfm\mapsto\vectinf{\textbf{g}^\ast_{2}(\hbm_{i+3})}{i},\hbfm\mapsto\vectinf{\textbf{g}^\ast_{2}(\hbm_{i+1})}{i}\Big]\\
		&\quad +\Big[\hbfm\mapsto\mathcal{M}_{\{1\}^c}\left(\vectinf{ \textbf{h}^\ast_2(\hbm_{i+1})}{i}\right),\hbfm\mapsto \vectinf{\textbf{g}^\ast_{2}(\hbm_{i-1})}{i}\Big]
		+\Big[\hbfm\mapsto\mathcal{M}_{\{1\}^c}\left(\vectinf{ \textbf{h}^\ast_2(\hbm_{i+1})}{i}\right),\hbfm\mapsto \vectinf{\textbf{g}^\ast_{2}(\hbm_{i+1})}{i}\Big]
		\\
		&\quad -\Big[\hbfm\mapsto\mathcal{M}_{\{2\}^c}\left(\vectinf{ \textbf{h}^\ast_5(\hbm_{i-1})}{i}\right),\hbfm\mapsto \vectinf{\textbf{g}^\ast_{2}(\hbm_{i-1})}{i}\Big]
		-\Big[\hbfm\mapsto\mathcal{M}_{\{2\}^c}\left(\vectinf{ \textbf{h}^\ast_5(\hbm_{i-1})}{i}\right),\hbfm\mapsto \vectinf{\textbf{g}^\ast_{2}(\hbm_{i+1})}{i}\Big]\\
		&=-\Big[\hbfm\mapsto \mathcal{M}_{\{2\}^c}\left(\vectinf{J_{ \textbf{g}^\ast_2} \textbf{g}^\ast_2(\hbm_{i-2})}{i}\right)\Big]+\Big[\hbfm\mapsto \mathcal{M}_{\{0\}^c}\left(\vectinf{J_{ \textbf{g}^\ast_2} \textbf{g}^\ast_2(\hbm_{i+2})}{i}\right)\Big]\\
		&\quad +\Big[\hbfm\mapsto\mathcal{M}_{\{1\}^c}\left(\vectinf{J_{ \textbf{h}^\ast_2} \textbf{g}^\ast_2(\hbm_{i})}{i}\right)\Big]-\Big[\hbfm\mapsto\mathcal{M}_{\{2\}^c}\left(\vectinf{J_{ \textbf{g}^\ast_2} \textbf{h}^\ast_2(\hbm_{i})}{i}\right)\Big] \\
		&\quad +\Big[\hbfm\mapsto\mathcal{M}_{\{1\}^c}\left(\vectinf{J_{ \textbf{h}^\ast_2} \textbf{g}^\ast_2(\hbm_{i+2})}{i}\right)\Big]-\Big[\hbfm\mapsto\mathcal{M}_{\{0\}^c}\left(\vectinf{J_{ \textbf{g}^\ast_2} \textbf{h}^\ast_2(\hbm_{i+2})}{i}\right)\Big] \\
		&\quad -\Big[\hbfm\mapsto\mathcal{M}_{\{2\}^c}\left(\vectinf{J_{ \textbf{h}^\ast_5} \textbf{g}^\ast_2(\hbm_{i-2})}{i}\right)\Big]+\Big[\hbfm\mapsto\mathcal{M}_{\{3\}^c}\left(\vectinf{J_{ \textbf{g}^\ast_2} \textbf{h}^\ast_5(\hbm_{i-2})}{i}\right)\Big] \\
		&\quad -\Big[\hbfm\mapsto\mathcal{M}_{\{2\}^c}\left(\vectinf{J_{ \textbf{h}^\ast_5} \textbf{g}^\ast_2(\hbm_{i})}{i}\right)\Big]+\Big[\hbfm\mapsto\mathcal{M}_{\{1\}^c}\left(\vectinf{J_{ \textbf{g}^\ast_2} \textbf{h}^\ast_5(\hbm_{i})}{i}\right)\Big]\\
		&=-\Big[\hbfm\mapsto \mathcal{M}_{\{2\}^c}\left(\vectinf{J_{ \textbf{g}^\ast_2} \textbf{g}^\ast_2(\hbm_{i-2})}{i}\right)\Big]-\Big[\hbfm\mapsto\mathcal{M}_{\{2\}^c}\left(\vectinf{J_{ \textbf{h}^\ast_5} \textbf{g}^\ast_2(\hbm_{i-2})}{i}\right)\Big]\\
		&\quad +\Big[\hbfm\mapsto \mathcal{M}_{\{0\}^c}\left(\vectinf{J_{ \textbf{g}^\ast_2} \textbf{g}^\ast_2(\hbm_{i+2})}{i}\right)\Big]-\Big[\hbfm\mapsto\mathcal{M}_{\{0\}^c}\left(\vectinf{J_{ \textbf{g}^\ast_2} \textbf{h}^\ast_2(\hbm_{i+2})}{i}\right)\Big] \\
		&\quad +\Big[\hbfm\mapsto\mathcal{M}_{\{1\}^c}\left(\vectinf{J_{ \textbf{h}^\ast_2} \textbf{g}^\ast_2(\hbm_{i+2})}{i}\right)\Big] +\Big[\hbfm\mapsto\mathcal{M}_{\{3\}^c}\left(\vectinf{J_{ \textbf{g}^\ast_2} \textbf{h}^\ast_5(\hbm_{i-2})}{i}\right)\Big] \\
		&=\Big[\hbfm\mapsto \mathcal{M}_{\{2\}^c}\left(\vectinf{ \textbf{p}^\ast_3(\hbm_{i-2})}{i}\right)\Big]-\Big[\hbfm\mapsto \mathcal{M}_{\{0\}^c}\left(\vectinf{ \textbf{p}^\ast_3(\hbm_{i+2})}{i}\right)\Big]-\Big[\hbfm\mapsto \mathcal{M}_{\{1\}^c}\left(\vectinf{ \textbf{p}^\ast_3(\hbm_{i+2})}{i}\right)\Big]\\
		&\qquad+\Big[\hbfm\mapsto \mathcal{M}_{\{3\}^c}\left(\vectinf{ \textbf{p}^\ast_3(\hbm_{i-2})}{i}\right)\Big],
	\end{align*}

	\noindent \textbf{Computation of $\left[\left[ \bf{G}^{1,1},\bf{G}^{2,1} \right],\bf{G}^{1,2}\right]$:}
	\begin{align*}
		\lqq{2^{3/2}\left[\left[ \bf{G}^{1,1},\bf{G}^{2,1} \right],\bf{G}^{1,2}\right]}\\
		&=\Big[\hbfm\mapsto\mathcal{M}_{\{1\}^c}\left(\vectinf{ \textbf{p}^\ast_2(\hbm_{i+1})+\textbf{p}^\ast_3(\hbm_{i+1})}{i}\right)-\mathcal{M}_{\{2\}^c}\left(\vectinf{ \textbf{p}^\ast_2(\hbm_{i-1})+\textbf{p}^\ast_3(\hbm_{i-1})}{i}\right),\\
		&\quad\hbfm\mapsto \vectinf{\textbf{g}^\ast_{2}(\hbm_{i-1})}{i}+\vectinf{\textbf{g}^\ast_{2}(\hbm_{i+1})}{i}\Big]\\
		&=\Big[\hbfm\mapsto\mathcal{M}_{\{1\}^c}\vectinf{ \textbf{h}^\ast_2(\hbm_{i})}{i}+\mathcal{M}_{\{2\}^c}\vectinf{ \textbf{h}^\ast_5(\hbm_{i})}{i}+\mathcal{M}_{\{1\}^c}\vectinf{ \textbf{h}^\ast_2(\hbm_{i+2})}{i}+\mathcal{M}_{\{0\}^c}\vectinf{ \textbf{h}^\ast_5(\hbm_{i+2})}{i}\\
		&\quad\quad -\mathcal{M}_{\{2\}^c}\vectinf{ \textbf{h}^\ast_2(\hbm_{i-2})}{i}-\mathcal{M}_{\{3\}^c}\vectinf{ \textbf{h}^\ast_5(\hbm_{i-2})}{i}-\mathcal{M}_{\{2\}^c}\vectinf{ \textbf{h}^\ast_2(\hbm_{i})}{i}-\mathcal{M}_{\{1\}^c}\vectinf{ \textbf{h}^\ast_5(\hbm_{i})}{i}\Big]\\
		&=\Big[\hbfm\mapsto\mathcal{M}_{\{0\}^c}\vectinf{ \textbf{h}^\ast_5(\hbm_{i+2})}{i}+\mathcal{M}_{\{1\}^c}\vectinf{ \textbf{h}^\ast_2(\hbm_{i+2})}{i}+\mathcal{M}_{\{1\}^c}\vectinf{ \textbf{g}^\ast_2(\hbm_{i})}{i}\\
		&\quad\quad -\mathcal{M}_{\{2\}^c}\vectinf{ \textbf{g}^\ast_2(\hbm_{i})}{i}-\mathcal{M}_{\{2\}^c}\vectinf{ \textbf{h}^\ast_2(\hbm_{i-2})}{i}-\mathcal{M}_{\{3\}^c}\vectinf{ \textbf{h}^\ast_5(\hbm_{i-2})}{i}\Big],
	\end{align*}
	
	\noindent \textbf{Computation of $\left[\left[ \bf{G}^{1,1},\bf{G}^{2,1} \right],\bf{G}^{1,3}\right]$:}
	\begin{align*}
		\lqq{2^{3/2}\left[\left[ \bf{G}^{1,1},\bf{G}^{2,1} \right],\bf{G}^{1,3}\right]}\\
		&=\Big[\hbfm\mapsto\mathcal{M}_{\{1\}^c}\left(\vectinf{ \textbf{p}^\ast_2(\hbm_{i+1})+\textbf{p}^\ast_3(\hbm_{i+1})}{i}\right)-\mathcal{M}_{\{2\}^c}\left(\vectinf{ \textbf{p}^\ast_2(\hbm_{i-1})+\textbf{p}^\ast_3(\hbm_{i-1})}{i}\right),\\
		&\quad\hbfm\mapsto \vectinf{\textbf{g}^\ast_{3}(\hbm_{i-1})}{i}+\vectinf{\textbf{g}^\ast_{3}(\hbm_{i+1})}{i}\Big]\\
		&=\Big[\hbfm\mapsto-\mathcal{M}_{\{1\}^c}\vectinf{ \textbf{h}^\ast_1(\hbm_{i})}{i}-\mathcal{M}_{\{2\}^c}\vectinf{ \textbf{h}^\ast_3(\hbm_{i})}{i}-\mathcal{M}_{\{1\}^c}\vectinf{ \textbf{h}^\ast_1(\hbm_{i+2})}{i}-\mathcal{M}_{\{0\}^c}\vectinf{ \textbf{h}^\ast_3(\hbm_{i+2})}{i}\\
		&\quad\quad +\mathcal{M}_{\{2\}^c}\vectinf{ \textbf{h}^\ast_1(\hbm_{i-2})}{i}+\mathcal{M}_{\{3\}^c}\vectinf{ \textbf{h}^\ast_3(\hbm_{i-2})}{i}+\mathcal{M}_{\{2\}^c}\vectinf{ \textbf{h}^\ast_1(\hbm_{i})}{i}+\mathcal{M}_{\{1\}^c}\vectinf{ \textbf{h}^\ast_3(\hbm_{i})}{i}\Big]\\
		&=\Big[\hbfm\mapsto-\mathcal{M}_{\{0\}^c}\vectinf{ \textbf{h}^\ast_3(\hbm_{i+2})}{i}-\mathcal{M}_{\{1\}^c}\vectinf{ \textbf{h}^\ast_1(\hbm_{i+2})}{i}+\mathcal{M}_{\{1\}^c}\vectinf{ \textbf{g}^\ast_3(\hbm_{i})}{i}\\
		&\quad\quad +\mathcal{M}_{\{2\}^c}\vectinf{ \textbf{h}^\ast_1(\hbm_{i-2})}{i}-\mathcal{M}_{\{2\}^c}\vectinf{ \textbf{g}^\ast_3(\hbm_{i})}{i}+\mathcal{M}_{\{3\}^c}\vectinf{ \textbf{h}^\ast_3(\hbm_{i-2})}{i}\Big].
	\end{align*}
	
	\noindent \textbf{Computation of $\left[\left[ \bf{G}^{1,2},\bf{G}^{2,2} \right],\bf{G}^{1,1}\right]$:}
	\begin{align*}
		\lqq{2^{3/2}\left[\left[ \bf{G}^{1,2},\bf{G}^{2,2} \right],\bf{G}^{1,1}\right]}\\
		&=\Big[\hbfm\mapsto\mathcal{M}_{\{1\}^c}\left(\vectinf{ \textbf{p}^\ast_1(\hbm_{i+1})+\textbf{p}^\ast_3(\hbm_{i+1})}{i}\right)-\mathcal{M}_{\{2\}^c}\left(\vectinf{ \textbf{p}^\ast_1(\hbm_{i-1})+\textbf{p}^\ast_3(\hbm_{i-1})}{i}\right),\\
		&\quad\hbfm\mapsto \vectinf{\textbf{g}^\ast_{1}(\hbm_{i-1})}{i}+\vectinf{\textbf{g}^\ast_{1}(\hbm_{i+1})}{i}\Big]\\
		&=\Big[\hbfm\mapsto-\mathcal{M}_{\{1\}^c}\vectinf{ \textbf{h}^\ast_4(\hbm_{i})}{i}-\mathcal{M}_{\{2\}^c}\vectinf{ \textbf{h}^\ast_6(\hbm_{i})}{i}-\mathcal{M}_{\{1\}^c}\vectinf{ \textbf{h}^\ast_4(\hbm_{i+2})}{i}-\mathcal{M}_{\{0\}^c}\vectinf{ \textbf{h}^\ast_6(\hbm_{i+2})}{i}\\
		&\quad\quad +\mathcal{M}_{\{2\}^c}\vectinf{ \textbf{h}^\ast_4(\hbm_{i-2})}{i}+\mathcal{M}_{\{3\}^c}\vectinf{ \textbf{h}^\ast_6(\hbm_{i-2})}{i}+\mathcal{M}_{\{2\}^c}\vectinf{ \textbf{h}^\ast_4(\hbm_{i})}{i}+\mathcal{M}_{\{1\}^c}\vectinf{ \textbf{h}^\ast_6(\hbm_{i})}{i}\Big]\\
		&=\Big[\hbfm\mapsto-\mathcal{M}_{\{0\}^c}\vectinf{ \textbf{h}^\ast_6(\hbm_{i+2})}{i}-\mathcal{M}_{\{1\}^c}\vectinf{ \textbf{h}^\ast_4(\hbm_{i+2})}{i}+\mathcal{M}_{\{1\}^c}\vectinf{ \textbf{g}^\ast_1(\hbm_{i})}{i}\\
		&\quad\quad +\mathcal{M}_{\{2\}^c}\vectinf{ \textbf{h}^\ast_4(\hbm_{i-2})}{i}-\mathcal{M}_{\{2\}^c}\vectinf{ \textbf{g}^\ast_1(\hbm_{i})}{i}+\mathcal{M}_{\{3\}^c}\vectinf{ \textbf{h}^\ast_6(\hbm_{i-2})}{i}\Big].
	\end{align*}

	\noindent \textbf{Computation of $\left[\left[ \bf{G}^{1,2},\bf{G}^{2,2} \right],\bf{G}^{1,3}\right]$:}
	\begin{align*}
		\lqq{2^{3/2}\left[\left[ \bf{G}^{1,2},\bf{G}^{2,2} \right],\bf{G}^{1,3}\right]}\\
		&=\Big[\hbfm\mapsto\mathcal{M}_{\{1\}^c}\left(\vectinf{ \textbf{p}^\ast_1(\hbm_{i+1})+\textbf{p}^\ast_3(\hbm_{i+1})}{i}\right)-\mathcal{M}_{\{2\}^c}\left(\vectinf{ \textbf{p}^\ast_1(\hbm_{i-1})+\textbf{p}^\ast_3(\hbm_{i-1})}{i}\right),\\
		&\quad\hbfm\mapsto \vectinf{\textbf{g}^\ast_{3}(\hbm_{i-1})}{i}+\vectinf{\textbf{g}^\ast_{3}(\hbm_{i+1})}{i}\Big]\\
		&=\Big[\hbfm\mapsto\mathcal{M}_{\{1\}^c}\vectinf{ \textbf{h}^\ast_3(\hbm_{i})}{i}+\mathcal{M}_{\{2\}^c}\vectinf{ \textbf{h}^\ast_1(\hbm_{i})}{i}+\mathcal{M}_{\{1\}^c}\vectinf{ \textbf{h}^\ast_3(\hbm_{i+2})}{i}+\mathcal{M}_{\{0\}^c}\vectinf{ \textbf{h}^\ast_1(\hbm_{i+2})}{i}\\
		&\quad\quad -\mathcal{M}_{\{2\}^c}\vectinf{ \textbf{h}^\ast_3(\hbm_{i-2})}{i}-\mathcal{M}_{\{3\}^c}\vectinf{ \textbf{h}^\ast_1(\hbm_{i-2})}{i}-\mathcal{M}_{\{2\}^c}\vectinf{ \textbf{h}^\ast_3(\hbm_{i})}{i}-\mathcal{M}_{\{1\}^c}\vectinf{ \textbf{h}^\ast_1(\hbm_{i})}{i}\Big]\\
		&=\Big[\hbfm\mapsto\mathcal{M}_{\{0\}^c}\vectinf{ \textbf{h}^\ast_1(\hbm_{i+2})}{i}+\mathcal{M}_{\{1\}^c}\vectinf{ \textbf{h}^\ast_3(\hbm_{i+2})}{i}+\mathcal{M}_{\{1\}^c}\vectinf{ \textbf{g}^\ast_3(\hbm_{i})}{i}\\
		&\quad\quad -\mathcal{M}_{\{2\}^c}\vectinf{ \textbf{g}^\ast_3(\hbm_{i})}{i}-\mathcal{M}_{\{2\}^c}\vectinf{ \textbf{h}^\ast_3(\hbm_{i-2})}{i}-\mathcal{M}_{\{3\}^c}\vectinf{ \textbf{h}^\ast_1(\hbm_{i-2})}{i}\Big],
	\end{align*}

	\noindent \textbf{Computation of $\left[\left[ \bf{G}^{1,3},\bf{G}^{2,3} \right],\bf{G}^{1,1}\right]$:}
	\begin{align*}
		\lqq{2^{3/2}\left[\left[ \bf{G}^{1,3},\bf{G}^{2,3} \right],\bf{G}^{1,1}\right]}\\
		&=\Big[\hbfm\mapsto\mathcal{M}_{\{1\}^c}\left(\vectinf{ \textbf{p}^\ast_1(\hbm_{i+1})+\textbf{p}^\ast_2(\hbm_{i+1})}{i}\right)-\mathcal{M}_{\{2\}^c}\left(\vectinf{ \textbf{p}^\ast_1(\hbm_{i-1})+\textbf{p}^\ast_2(\hbm_{i-1})}{i}\right),\\
		&\quad\hbfm\mapsto \vectinf{\textbf{g}^\ast_{1}(\hbm_{i-1})}{i}+\vectinf{\textbf{g}^\ast_{1}(\hbm_{i+1})}{i}\Big]\\
		&=\Big[\hbfm\mapsto\mathcal{M}_{\{1\}^c}\vectinf{ \textbf{h}^\ast_6(\hbm_{i})}{i}+\mathcal{M}_{\{2\}^c}\vectinf{ \textbf{h}^\ast_4(\hbm_{i})}{i}+\mathcal{M}_{\{1\}^c}\vectinf{ \textbf{h}^\ast_6(\hbm_{i+2})}{i}+\mathcal{M}_{\{0\}^c}\vectinf{ \textbf{h}^\ast_4(\hbm_{i+2})}{i}\\
		&\quad\quad -\mathcal{M}_{\{2\}^c}\vectinf{ \textbf{h}^\ast_6(\hbm_{i-2})}{i}-\mathcal{M}_{\{3\}^c}\vectinf{ \textbf{h}^\ast_4(\hbm_{i-2})}{i}-\mathcal{M}_{\{2\}^c}\vectinf{ \textbf{h}^\ast_6(\hbm_{i})}{i}-\mathcal{M}_{\{1\}^c}\vectinf{ \textbf{h}^\ast_4(\hbm_{i})}{i}\Big]\\
		&=\Big[\hbfm\mapsto\mathcal{M}_{\{0\}^c}\vectinf{ \textbf{h}^\ast_4(\hbm_{i+2})}{i}+\mathcal{M}_{\{1\}^c}\vectinf{ \textbf{h}^\ast_6(\hbm_{i+2})}{i}+\mathcal{M}_{\{1\}^c}\vectinf{ \textbf{g}^\ast_1(\hbm_{i})}{i}\\
		&\quad\quad -\mathcal{M}_{\{2\}^c}\vectinf{ \textbf{g}^\ast_1(\hbm_{i})}{i}-\mathcal{M}_{\{2\}^c}\vectinf{ \textbf{h}^\ast_6(\hbm_{i-2})}{i}-\mathcal{M}_{\{3\}^c}\vectinf{ \textbf{h}^\ast_4(\hbm_{i-2})}{i}\Big],
	\end{align*}
	
	\noindent \textbf{Computation of $\left[\left[ \bf{G}^{1,3},\bf{G}^{2,3} \right],\bf{G}^{1,2}\right]$:}
	\begin{align*}
		\lqq{2^{3/2}\left[\left[ \bf{G}^{1,3},\bf{G}^{2,3} \right],\bf{G}^{1,2}\right]}\\
		&=\Big[\hbfm\mapsto\mathcal{M}_{\{1\}^c}\left(\vectinf{ \textbf{p}^\ast_1(\hbm_{i+1})+\textbf{p}^\ast_2(\hbm_{i+1})}{i}\right)-\mathcal{M}_{\{2\}^c}\left(\vectinf{ \textbf{p}^\ast_1(\hbm_{i-1})+\textbf{p}^\ast_2(\hbm_{i-1})}{i}\right),\\
		&\quad\hbfm\mapsto \vectinf{\textbf{g}^\ast_{2}(\hbm_{i-1})}{i}+\vectinf{\textbf{g}^\ast_{2}(\hbm_{i+1})}{i}\Big]\\
		&=\Big[\hbfm\mapsto-\mathcal{M}_{\{1\}^c}\vectinf{ \textbf{h}^\ast_5(\hbm_{i})}{i}-\mathcal{M}_{\{2\}^c}\vectinf{ \textbf{h}^\ast_2(\hbm_{i})}{i}-\mathcal{M}_{\{1\}^c}\vectinf{ \textbf{h}^\ast_5(\hbm_{i+2})}{i}-\mathcal{M}_{\{0\}^c}\vectinf{ \textbf{h}^\ast_2(\hbm_{i+2})}{i}\\
		&\quad\quad +\mathcal{M}_{\{2\}^c}\vectinf{ \textbf{h}^\ast_5(\hbm_{i-2})}{i}+\mathcal{M}_{\{3\}^c}\vectinf{ \textbf{h}^\ast_2(\hbm_{i-2})}{i}+\mathcal{M}_{\{2\}^c}\vectinf{ \textbf{h}^\ast_5(\hbm_{i})}{i}+\mathcal{M}_{\{1\}^c}\vectinf{ \textbf{h}^\ast_2(\hbm_{i})}{i}\Big]\\
		&=\Big[\hbfm\mapsto-\mathcal{M}_{\{0\}^c}\vectinf{ \textbf{h}^\ast_2(\hbm_{i+2})}{i}-\mathcal{M}_{\{1\}^c}\vectinf{ \textbf{h}^\ast_5(\hbm_{i+2})}{i}+\mathcal{M}_{\{1\}^c}\vectinf{ \textbf{g}^\ast_2(\hbm_{i})}{i}\\
		&\quad\quad +\mathcal{M}_{\{2\}^c}\vectinf{ \textbf{h}^\ast_5(\hbm_{i-2})}{i}-\mathcal{M}_{\{2\}^c}\vectinf{ \textbf{g}^\ast_2(\hbm_{i})}{i}+\mathcal{M}_{\{3\}^c}\vectinf{ \textbf{h}^\ast_2(\hbm_{i-2})}{i}\Big].
	\end{align*}

\noindent \textbf{Computation of $\left[\left[\left[ \bf{G}^{1,1},\bf{G}^{2,1} \right],\bf{G}^{1,3}\right],\bf{G}^{1,3}\right]$:}
\begin{eqnarray*}
	&&2^{2}\left[\left[\left[ \bf{G}^{1,1},\bf{G}^{2,1} \right],\bf{G}^{1,3}\right],\bf{G}^{1,3}\right]\\
	&=&\Big[\hbfm\mapsto-\mathcal{M}_{\{0\}^c}\vectinf{ \textbf{h}^\ast_3(\hbm_{i+2})}{i}-\mathcal{M}_{\{1\}^c}\vectinf{ \textbf{h}^\ast_1(\hbm_{i+2})}{i}+\mathcal{M}_{\{1\}^c}\vectinf{ \textbf{g}^\ast_3(\hbm_{i})}{i}-\mathcal{M}_{\{2\}^c}\vectinf{ \textbf{g}^\ast_3(\hbm_{i})}{i}\\
	&&+\mathcal{M}_{\{2\}^c}\vectinf{ \textbf{h}^\ast_1(\hbm_{i-2})}{i}+\mathcal{M}_{\{3\}^c}\vectinf{ \textbf{h}^\ast_3(\hbm_{i-2})}{i},\hbfm\mapsto \vectinf{\textbf{g}^\ast_{3}(\hbm_{i-1})}{i}+\vectinf{\textbf{g}^\ast_{3}(\hbm_{i+1})}{i}\Big]\\
	&=&-\Big[\hbfm\mapsto\mathcal{M}_{\{0\}^c}\left(\vectinf{ \textbf{h}^\ast_3(\hbm_{i+2})}{i}\right),\hbfm\mapsto \vectinf{\textbf{g}^\ast_{3}(\hbm_{i-1})}{i}\Big]\\ 
	&&-\Big[\hbfm\mapsto\mathcal{M}_{\{0\}^c}\left(\vectinf{ \textbf{h}^\ast_3(\hbm_{i+2})}{i}\right),\hbfm\mapsto \vectinf{\textbf{g}^\ast_{3}(\hbm_{i+1})}{i}\Big]\\
	&&-\Big[\hbfm\mapsto\mathcal{M}_{\{1\}^c}\left(\vectinf{ \textbf{h}^\ast_1(\hbm_{i+2})}{i}\right),\hbfm\mapsto \vectinf{\textbf{g}^\ast_{3}(\hbm_{i-1})}{i}\Big]\\ 
	&&-\Big[\hbfm\mapsto\mathcal{M}_{\{1\}^c}\left(\vectinf{ \textbf{h}^\ast_1(\hbm_{i+2})}{i}\right),\hbfm\mapsto \vectinf{\textbf{g}^\ast_{3}(\hbm_{i+1})}{i}\Big]\\
	&&+\Big[\hbfm\mapsto\mathcal{M}_{\{2\}^c}\left(\vectinf{ \textbf{h}^\ast_1(\hbm_{i-2})}{i}\right),\hbfm\mapsto \vectinf{\textbf{g}^\ast_{3}(\hbm_{i-1})}{i}\Big]\\ 
	&&+\Big[\hbfm\mapsto\mathcal{M}_{\{2\}^c}\left(\vectinf{ \textbf{h}^\ast_1(\hbm_{i-2})}{i}\right),\hbfm\mapsto \vectinf{\textbf{g}^\ast_{2}(\hbm_{i+1})}{i}\Big]\\
	&&+\Big[\hbfm\mapsto\mathcal{M}_{\{3\}^c}\left(\vectinf{ \textbf{h}^\ast_3(\hbm_{i-2})}{i}\right),\hbfm\mapsto \vectinf{\textbf{g}^\ast_{3}(\hbm_{i-1})}{i}\Big]\\ 
	&&+\Big[\hbfm\mapsto\mathcal{M}_{\{3\}^c}\left(\vectinf{ \textbf{h}^\ast_3(\hbm_{i-2})}{i}\right),\hbfm\mapsto \vectinf{\textbf{g}^\ast_{3}(\hbm_{i+1})}{i}\Big]\\
	&&+\Big[\hbfm\mapsto\mathcal{M}_{\{1\}^c}\left(\vectinf{ \textbf{g}^\ast_3(\hbm_{i})}{i}\right),\hbfm\mapsto \vectinf{\textbf{g}^\ast_{3}(\hbm_{i-1})}{i}\Big]\\ 
	&&+\Big[\hbfm\mapsto\mathcal{M}_{\{1\}^c}\left(\vectinf{ \textbf{g}^\ast_3(\hbm_{i})}{i}\right),\hbfm\mapsto \vectinf{\textbf{g}^\ast_{3}(\hbm_{i+1})}{i}\Big]\\
	&&-\Big[\hbfm\mapsto\mathcal{M}_{\{2\}^c}\left(\vectinf{ \textbf{g}^\ast_3(\hbm_{i})}{i}\right),\hbfm\mapsto \vectinf{\textbf{g}^\ast_{3}(\hbm_{i-1})}{i}\Big]\\ 
	&&-\Big[\hbfm\mapsto\mathcal{M}_{\{2\}^c}\left(\vectinf{ \textbf{g}^\ast_3(\hbm_{i})}{i}\right),\hbfm\mapsto \vectinf{\textbf{g}^\ast_{3}(\hbm_{i+1})}{i}\Big]\\
	&=&\Big[\hbfm\mapsto\mathcal{M}_{\{0\}^c}\left(\vectinf{ \textbf{p}^\ast_1(\hbm_{i+1})}{i}\right)-\mathcal{M}_{\{1\}^c}\left(\vectinf{ \textbf{p}^\ast_2(\hbm_{i+1})}{i}\right)\Big]\\ 
	&&+\Big[\hbfm\mapsto\mathcal{M}_{\{0\}^c}\left(\vectinf{ \textbf{p}^\ast_1(\hbm_{i+3})}{i}\right)\Big]\\
	&&-\Big[\hbfm\mapsto\mathcal{M}_{\{1\}^c}\left(\vectinf{ \textbf{p}^\ast_2(\hbm_{i+1})}{i}\right)-\mathcal{M}_{\{2\}^c}\left(\vectinf{ \textbf{p}^\ast_1(\hbm_{i+1})}{i}\right)\Big]\\  
	&&-\Big[\hbfm\mapsto\mathcal{M}_{\{1\}^c}\left(\vectinf{ \textbf{p}^\ast_2(\hbm_{i+3})}{i}\right)-\mathcal{M}_{\{0\}^c}\left(\vectinf{ \textbf{p}^\ast_1(\hbm_{i+3})}{i}\right)\Big]\\  
	&&+\Big[\hbfm\mapsto \mathcal{M}_{\{2\}^c}\left(\vectinf{ \textbf{p}^\ast_2(\hbm_{i-3})}{i}\right)-\mathcal{M}_{\{3\}^c}\left(\vectinf{ \textbf{p}^\ast_1(\hbm_{i-3})}{i}\right)\Big]\\
	&&+\Big[\hbfm\mapsto\mathcal{M}_{\{2\}^c}\left(\vectinf{ \textbf{p}^\ast_2(\hbm_{i-1})}{i}\right)-\mathcal{M}_{\{1\}^c}\left(\vectinf{ \textbf{p}^\ast_1(\hbm_{i-1})}{i}\right)\Big]\\
	&&-\Big[\hbfm\mapsto\mathcal{M}_{\{3\}^c}\left(\vectinf{ \textbf{p}^\ast_1(\hbm_{i-3})}{i}\right)-\mathcal{M}_{\{4\}^c}\left(\vectinf{ \textbf{p}^\ast_2(\hbm_{i-3})}{i}\right)\Big]\\ 
	&&-\Big[\hbfm\mapsto\mathcal{M}_{\{3\}^c}\left(\vectinf{ \textbf{p}^\ast_1(\hbm_{i-1})}{i}\right)-\mathcal{M}_{\{2\}^c}\left(\vectinf{ \textbf{p}^\ast_2(\hbm_{i-1})}{i}\right)\Big]\\
	&&+\Big[\hbfm\mapsto-\mathcal{M}_{\{1\}^c}\left(\vectinf{\textbf{p}^\ast_1(\hbm_{i-1})+ \textbf{p}^\ast_2(\hbm_{i-1})}{i}\right)+\mathcal{M}_{\{2\}^c}\left(\vectinf{\textbf{p}^\ast_1(\hbm_{i-1})+ \textbf{p}^\ast_2(\hbm_{i-1})}{i}\right)\Big]\\  
	&&+\Big[\hbfm\mapsto-\mathcal{M}_{\{1\}^c}\left(\vectinf{ \textbf{p}^\ast_1(\hbm_{i+1})+ \textbf{p}^\ast_2(\hbm_{i+1})}{i}\right)+\mathcal{M}_{\{0\}^c}\left(\vectinf{ \textbf{p}^\ast_1(\hbm_{i+1})+ \textbf{p}^\ast_2(\hbm_{i+1})}{i}\right)\Big]\\ 
	&&-\Big[\hbfm\mapsto-\mathcal{M}_{\{2\}^c}\left(\vectinf{\textbf{p}^\ast_1(\hbm_{i-1})+ \textbf{p}^\ast_2(\hbm_{i-1})}{i}\right)+\mathcal{M}_{\{3\}^c}\left(\vectinf{\textbf{p}^\ast_1(\hbm_{i-1})+ \textbf{p}^\ast_2(\hbm_{i-1})}{i}\right)\Big]\\  
	&&-\Big[\hbfm\mapsto-\mathcal{M}_{\{2\}^c}\left(\vectinf{ \textbf{p}^\ast_1(\hbm_{i+1})+ \textbf{p}^\ast_2(\hbm_{i+1})}{i}\right)+\mathcal{M}_{\{1\}^c}\left(\vectinf{ \textbf{p}^\ast_1(\hbm_{i+1})+ \textbf{p}^\ast_2(\hbm_{i+1})}{i}\right)\Big]\\
	&=&\Big[\hbfm\mapsto2\mathcal{M}_{\{0\}^c}\left(\vectinf{ \textbf{p}^\ast_1(\hbm_{i+1})}{i}\right)+2\mathcal{M}_{\{0\}^c}\left(\vectinf{ \textbf{p}^\ast_1(\hbm_{i+3})}{i}\right)+\mathcal{M}_{\{0\}^c}\left(\vectinf{ \textbf{p}^\ast_2(\hbm_{i+1})}{i}\right)\Big]\\ 
	&&-\Big[\hbfm\mapsto4\mathcal{M}_{\{1\}^c}\left(\vectinf{ \textbf{p}^\ast_2(\hbm_{i+1})}{i}\right)+\mathcal{M}_{\{1\}^c}\left(\vectinf{ \textbf{p}^\ast_2(\hbm_{i+3})}{i}\right)+2\mathcal{M}_{\{1\}^c}\left(\vectinf{ \textbf{p}^\ast_1(\hbm_{i-1})}{i}\right)\\
	&&\qquad+2\mathcal{M}_{\{1\}^c}\left(\vectinf{ \textbf{p}^\ast_1(\hbm_{i+1})}{i}\right)+\mathcal{M}_{\{1\}^c}\left(\vectinf{ \textbf{p}^\ast_2(\hbm_{i-1})}{i}\right)\Big]\\
	&&+\Big[\hbfm\mapsto2\mathcal{M}_{\{2\}^c}\left(\vectinf{ \textbf{p}^\ast_1(\hbm_{i+1})}{i}\right)+\mathcal{M}_{\{2\}^c}\left(\vectinf{ \textbf{p}^\ast_2(\hbm_{i-3})}{i}\right)+4\mathcal{M}_{\{2\}^c}\left(\vectinf{ \textbf{p}^\ast_2(\hbm_{i-1})}{i}\right)\\
	&&\qquad+2\mathcal{M}_{\{2\}^c}\left(\vectinf{ \textbf{p}^\ast_1(\hbm_{i-1})}{i}\right)+\mathcal{M}_{\{2\}^c}\left(\vectinf{ \textbf{p}^\ast_2(\hbm_{i+1})}{i}\right)\Big]\\
	&&-\Big[\hbfm\mapsto2\mathcal{M}_{\{3\}^c}\left(\vectinf{ \textbf{p}^\ast_1(\hbm_{i-3})}{i}\right)+2\mathcal{M}_{\{3\}^c}\left(\vectinf{ \textbf{p}^\ast_1(\hbm_{i-1})}{i}\right)+\mathcal{M}_{\{3\}^c}\left(\vectinf{ \textbf{p}^\ast_2(\hbm_{i-1})}{i}\right)\\ 
	&&+\Big[\hbfm\mapsto\mathcal{M}_{\{4\}^c}\left(\vectinf{ \textbf{p}^\ast_2(\hbm_{i-3})}{i}\right)\Big] 
\end{eqnarray*}

\noindent \textbf{Computation of $\left[\left[\left[ \bf{G}^{1,2},\bf{G}^{2,2} \right],\bf{G}^{1,3}\right],\bf{G}^{1,3}\right]$:}
\begin{eqnarray*}
	&&2^{2}\left[\left[\left[ \bf{G}^{1,2},\bf{G}^{2,2} \right],\bf{G}^{1,3}\right],\bf{G}^{1,3}\right]\\
	&=&\Big[\hbfm\mapsto\mathcal{M}_{\{0\}^c}\vectinf{ \textbf{h}^\ast_1(\hbm_{i+2})}{i}+\mathcal{M}_{\{1\}^c}\vectinf{ \textbf{h}^\ast_3(\hbm_{i+2})}{i}+\mathcal{M}_{\{1\}^c}\vectinf{ \textbf{g}^\ast_3(\hbm_{i})}{i}-\mathcal{M}_{\{2\}^c}\vectinf{ \textbf{g}^\ast_3(\hbm_{i})}{i}\\
	&&-\mathcal{M}_{\{2\}^c}\vectinf{ \textbf{h}^\ast_3(\hbm_{i-2})}{i}-\mathcal{M}_{\{3\}^c}\vectinf{ \textbf{h}^\ast_1(\hbm_{i-2})}{i},\hbfm\mapsto \vectinf{\textbf{g}^\ast_{3}(\hbm_{i-1})}{i}+\vectinf{\textbf{g}^\ast_{3}(\hbm_{i+1})}{i}\Big]\\
	&=&\Big[\hbfm\mapsto\mathcal{M}_{\{0\}^c}\left(\vectinf{ \textbf{h}^\ast_1(\hbm_{i+2})}{i}\right),\hbfm\mapsto \vectinf{\textbf{g}^\ast_{3}(\hbm_{i-1})}{i}\Big]\\ 
	&&+\Big[\hbfm\mapsto\mathcal{M}_{\{0\}^c}\left(\vectinf{ \textbf{h}^\ast_1(\hbm_{i+2})}{i}\right),\hbfm\mapsto \vectinf{\textbf{g}^\ast_{3}(\hbm_{i+1})}{i}\Big]\\
	&&+\Big[\hbfm\mapsto\mathcal{M}_{\{1\}^c}\left(\vectinf{ \textbf{h}^\ast_3(\hbm_{i+2})}{i}\right),\hbfm\mapsto \vectinf{\textbf{g}^\ast_{3}(\hbm_{i-1})}{i}\Big]\\ 
	&&+\Big[\hbfm\mapsto\mathcal{M}_{\{1\}^c}\left(\vectinf{ \textbf{h}^\ast_3(\hbm_{i+2})}{i}\right),\hbfm\mapsto \vectinf{\textbf{g}^\ast_{3}(\hbm_{i+1})}{i}\Big]\\
	&&-\Big[\hbfm\mapsto\mathcal{M}_{\{2\}^c}\left(\vectinf{ \textbf{h}^\ast_3(\hbm_{i-2})}{i}\right),\hbfm\mapsto \vectinf{\textbf{g}^\ast_{3}(\hbm_{i-1})}{i}\Big]\\ 
	&&-\Big[\hbfm\mapsto\mathcal{M}_{\{2\}^c}\left(\vectinf{ \textbf{h}^\ast_3(\hbm_{i-2})}{i}\right),\hbfm\mapsto \vectinf{\textbf{g}^\ast_{2}(\hbm_{i+1})}{i}\Big]\\
	&&-\Big[\hbfm\mapsto\mathcal{M}_{\{3\}^c}\left(\vectinf{ \textbf{h}^\ast_1(\hbm_{i-2})}{i}\right),\hbfm\mapsto \vectinf{\textbf{g}^\ast_{3}(\hbm_{i-1})}{i}\Big]\\ 
	&&-\Big[\hbfm\mapsto\mathcal{M}_{\{3\}^c}\left(\vectinf{ \textbf{h}^\ast_1(\hbm_{i-2})}{i}\right),\hbfm\mapsto \vectinf{\textbf{g}^\ast_{3}(\hbm_{i+1})}{i}\Big]\\
	&&+\Big[\hbfm\mapsto\mathcal{M}_{\{1\}^c}\left(\vectinf{ \textbf{g}^\ast_3(\hbm_{i})}{i}\right),\hbfm\mapsto \vectinf{\textbf{g}^\ast_{3}(\hbm_{i-1})}{i}\Big]\\ 
	&&+\Big[\hbfm\mapsto\mathcal{M}_{\{1\}^c}\left(\vectinf{ \textbf{g}^\ast_3(\hbm_{i})}{i}\right),\hbfm\mapsto \vectinf{\textbf{g}^\ast_{3}(\hbm_{i+1})}{i}\Big]\\
	&&-\Big[\hbfm\mapsto\mathcal{M}_{\{2\}^c}\left(\vectinf{ \textbf{g}^\ast_3(\hbm_{i})}{i}\right),\hbfm\mapsto \vectinf{\textbf{g}^\ast_{3}(\hbm_{i-1})}{i}\Big]\\ 
	&&-\Big[\hbfm\mapsto\mathcal{M}_{\{2\}^c}\left(\vectinf{ \textbf{g}^\ast_3(\hbm_{i})}{i}\right),\hbfm\mapsto \vectinf{\textbf{g}^\ast_{3}(\hbm_{i+1})}{i}\Big]\\
	&=&\Big[\hbfm\mapsto\mathcal{M}_{\{0\}^c}\left(\vectinf{ \textbf{p}^\ast_2(\hbm_{i+1})}{i}\right)-\mathcal{M}_{\{1\}^c}\left(\vectinf{ \textbf{p}^\ast_1(\hbm_{i+1})}{i}\right)\Big]\\ 
	&&+\Big[\hbfm\mapsto\mathcal{M}_{\{0\}^c}\left(\vectinf{ \textbf{p}^\ast_2(\hbm_{i+3})}{i}\right)\Big]\\
	&&-\Big[\hbfm\mapsto\mathcal{M}_{\{1\}^c}\left(\vectinf{ \textbf{p}^\ast_1(\hbm_{i+1})}{i}\right)-\mathcal{M}_{\{2\}^c}\left(\vectinf{ \textbf{p}^\ast_2(\hbm_{i+1})}{i}\right)\Big]\\  
	&&-\Big[\hbfm\mapsto\mathcal{M}_{\{1\}^c}\left(\vectinf{ \textbf{p}^\ast_1(\hbm_{i+3})}{i}\right)-\mathcal{M}_{\{0\}^c}\left(\vectinf{ \textbf{p}^\ast_2(\hbm_{i+3})}{i}\right)\Big]\\  
	&&+\Big[\hbfm\mapsto \mathcal{M}_{\{2\}^c}\left(\vectinf{ \textbf{p}^\ast_1(\hbm_{i-3})}{i}\right)-\mathcal{M}_{\{3\}^c}\left(\vectinf{ \textbf{p}^\ast_2(\hbm_{i-3})}{i}\right)\Big]\\
	&&+\Big[\hbfm\mapsto\mathcal{M}_{\{2\}^c}\left(\vectinf{ \textbf{p}^\ast_1(\hbm_{i-1})}{i}\right)-\mathcal{M}_{\{1\}^c}\left(\vectinf{ \textbf{p}^\ast_2(\hbm_{i-1})}{i}\right)\Big]\\
	&&-\Big[\hbfm\mapsto\mathcal{M}_{\{3\}^c}\left(\vectinf{ \textbf{p}^\ast_2(\hbm_{i-3})}{i}\right)-\mathcal{M}_{\{4\}^c}\left(\vectinf{ \textbf{p}^\ast_1(\hbm_{i-3})}{i}\right)\Big]\\ 
	&&-\Big[\hbfm\mapsto\mathcal{M}_{\{3\}^c}\left(\vectinf{ \textbf{p}^\ast_2(\hbm_{i-1})}{i}\right)-\mathcal{M}_{\{2\}^c}\left(\vectinf{ \textbf{p}^\ast_1(\hbm_{i-1})}{i}\right)\Big]\\
	&&+\Big[\hbfm\mapsto-\mathcal{M}_{\{1\}^c}\left(\vectinf{\textbf{p}^\ast_1(\hbm_{i-1})+ \textbf{p}^\ast_2(\hbm_{i-1})}{i}\right)+\mathcal{M}_{\{2\}^c}\left(\vectinf{\textbf{p}^\ast_1(\hbm_{i-1})+ \textbf{p}^\ast_2(\hbm_{i-1})}{i}\right)\Big]\\  
	&&+\Big[\hbfm\mapsto-\mathcal{M}_{\{1\}^c}\left(\vectinf{ \textbf{p}^\ast_1(\hbm_{i+1})+ \textbf{p}^\ast_2(\hbm_{i+1})}{i}\right)+\mathcal{M}_{\{0\}^c}\left(\vectinf{ \textbf{p}^\ast_1(\hbm_{i+1})+ \textbf{p}^\ast_2(\hbm_{i+1})}{i}\right)\Big]\\ 
	&&-\Big[\hbfm\mapsto-\mathcal{M}_{\{2\}^c}\left(\vectinf{\textbf{p}^\ast_1(\hbm_{i-1})+ \textbf{p}^\ast_2(\hbm_{i-1})}{i}\right)+\mathcal{M}_{\{3\}^c}\left(\vectinf{\textbf{p}^\ast_1(\hbm_{i-1})+ \textbf{p}^\ast_2(\hbm_{i-1})}{i}\right)\Big]\\  
	&&-\Big[\hbfm\mapsto-\mathcal{M}_{\{2\}^c}\left(\vectinf{ \textbf{p}^\ast_1(\hbm_{i+1})+ \textbf{p}^\ast_2(\hbm_{i+1})}{i}\right)+\mathcal{M}_{\{1\}^c}\left(\vectinf{ \textbf{p}^\ast_1(\hbm_{i+1})+ \textbf{p}^\ast_2(\hbm_{i+1})}{i}\right)\Big]\\
	&=&\Big[\hbfm\mapsto2\mathcal{M}_{\{0\}^c}\left(\vectinf{ \textbf{p}^\ast_2(\hbm_{i+1})}{i}\right)+2\mathcal{M}_{\{0\}^c}\left(\vectinf{ \textbf{p}^\ast_2(\hbm_{i+3})}{i}\right)+\mathcal{M}_{\{0\}^c}\left(\vectinf{ \textbf{p}^\ast_1(\hbm_{i+1})}{i}\right)\Big]\\ 
	&&-\Big[\hbfm\mapsto4\mathcal{M}_{\{1\}^c}\left(\vectinf{ \textbf{p}^\ast_1(\hbm_{i+1})}{i}\right)+\mathcal{M}_{\{1\}^c}\left(\vectinf{ \textbf{p}^\ast_1(\hbm_{i+3})}{i}\right)+2\mathcal{M}_{\{1\}^c}\left(\vectinf{ \textbf{p}^\ast_2(\hbm_{i-1})}{i}\right)\\
	&&\qquad+2\mathcal{M}_{\{1\}^c}\left(\vectinf{ \textbf{p}^\ast_2(\hbm_{i+1})}{i}\right)+\mathcal{M}_{\{1\}^c}\left(\vectinf{ \textbf{p}^\ast_1(\hbm_{i-1})}{i}\right)\Big]\\
	&&+\Big[\hbfm\mapsto2\mathcal{M}_{\{2\}^c}\left(\vectinf{ \textbf{p}^\ast_2(\hbm_{i+1})}{i}\right)+\mathcal{M}_{\{2\}^c}\left(\vectinf{ \textbf{p}^\ast_1(\hbm_{i-3})}{i}\right)+4\mathcal{M}_{\{2\}^c}\left(\vectinf{ \textbf{p}^\ast_1(\hbm_{i-1})}{i}\right)\\
	&&\qquad+2\mathcal{M}_{\{2\}^c}\left(\vectinf{ \textbf{p}^\ast_2(\hbm_{i-1})}{i}\right)+\mathcal{M}_{\{2\}^c}\left(\vectinf{ \textbf{p}^\ast_1(\hbm_{i+1})}{i}\right)\Big]\\
	&&-\Big[\hbfm\mapsto2\mathcal{M}_{\{3\}^c}\left(\vectinf{ \textbf{p}^\ast_2(\hbm_{i-3})}{i}\right)+2\mathcal{M}_{\{3\}^c}\left(\vectinf{ \textbf{p}^\ast_2(\hbm_{i-1})}{i}\right)+\mathcal{M}_{\{3\}^c}\left(\vectinf{ \textbf{p}^\ast_1(\hbm_{i-1})}{i}\right)\\ 
	&&+\Big[\hbfm\mapsto\mathcal{M}_{\{4\}^c}\left(\vectinf{ \textbf{p}^\ast_1(\hbm_{i-3})}{i}\right)\Big]  
\end{eqnarray*}

\end{appendix}

%\comanew{Please remove "sc" from the authors names. Use it only in book titles as in \cite{Bis} below. }

\end{document}